\documentclass[11pt]{article}

\usepackage{amsmath, amsthm, amssymb}
\usepackage{enumerate}
\usepackage{pdflscape}
\usepackage{caption}
\usepackage{bm}

\usepackage{ifpdf}
\ifpdf
\usepackage[pdftex]{graphicx}
\else
\usepackage[dvips]{graphicx}
\fi
\usepackage{tikz}
 	 \usetikzlibrary{arrows,backgrounds,patterns.meta}
\usepackage[all]{xy}

\usepackage{multicol}
\usepackage{enumitem}

\usepackage{tocvsec2}

\usepackage{bbm}

\input xy
\xyoption{all}

\usepackage[pdftex,plainpages=false,hypertexnames=false,pdfpagelabels]{hyperref}
\newcommand{\arxiv}[1]{\href{http://arxiv.org/abs/#1}{\tt arXiv:\nolinkurl{#1}}}
\newcommand{\arXiv}[1]{\href{http://arxiv.org/abs/#1}{\tt arXiv:\nolinkurl{#1}}}
\newcommand{\doi}[1]{\href{http://dx.doi.org/#1}{{\tt DOI:#1}}}

\newcommand{\googlebooks}[1]{(preview at \href{http://books.google.com/books?id=#1}{google books})}

\usepackage{xcolor}
\definecolor{dark-red}{rgb}{0.7,0.25,0.25}
\definecolor{dark-blue}{rgb}{0.15,0.15,0.55}
\definecolor{medium-blue}{rgb}{0,0,.8}
\definecolor{DarkGreen}{RGB}{0,150,0}
\definecolor{violet}{RGB}{148,0,211}
\hypersetup{
   colorlinks, linkcolor={purple},
   citecolor={medium-blue}, urlcolor={medium-blue}
}

\usepackage{longtable}
\usepackage{fullpage}

\theoremstyle{plain}
\newtheorem{thm}{Theorem}[section]
\newtheorem*{thm*}{Theorem}
\newtheorem{thmalpha}{Theorem}

\newtheorem{cor}[thm]{Corollary}

\newtheorem*{cor*}{Corollary}
\newtheorem{conj}[thm]{Conjecture}

\newtheorem*{conj*}{Conjecture}
\newtheorem{lem}[thm]{Lemma}
\newtheorem{prop}[thm]{Proposition}

\newtheorem*{quest*}{Question}
\newtheorem*{claim*}{Claim}

\theoremstyle{definition}
\newtheorem{defn}[thm]{Definition}

\newtheorem{alg}[thm]{Algorithm}
\newtheorem{assumption}[thm]{Assumption}
\newtheorem{nota}[thm]{Notation}

\newtheorem{fact}[thm]{Fact}
\newtheorem{facts}[thm]{Facts}

\newtheorem{example}[thm]{Example}

\newtheorem{sub-ex}[thm]{Sub-Example}
\newtheorem{rem}[thm]{Remark}
\newtheorem*{rem*}{Remark}
\newtheorem{remark}[thm]{Remark}

\newtheorem{warn}[thm]{Warning}

\DeclareMathOperator{\Ad}{Ad}
\DeclareMathOperator{\Aut}{Aut}
\DeclareMathOperator{\coev}{coev}
\DeclareMathOperator{\Dim}{Dim}
\DeclareMathOperator{\diag}{diag}
\DeclareMathOperator{\dist}{dist}

\DeclareMathOperator{\End}{End}
\DeclareMathOperator{\ev}{ev}

\DeclareMathOperator{\Hom}{Hom}

\DeclareMathOperator{\op}{op}

\DeclareMathOperator{\Ob}{Ob}

\let\mod\undefined
\DeclareMathOperator{\mod}{mod}
\DeclareMathOperator{\Mod}{Mod}
\DeclareMathOperator{\Mat}{Mat}
\DeclareMathOperator{\spann}{span}

\DeclareMathOperator{\id}{id}

\DeclareMathOperator{\Irr}{Irr}

\DeclareMathOperator{\Tr}{Tr}
\DeclareMathOperator{\tr}{tr}

\DeclareMathOperator{\vNdim}{vNdim}

\DeclareMathOperator{\Ind}{Ind}

\newcommand{\D}{\displaystyle}
\newcommand{\comment}[1]{}

\newcommand{\be}{\begin{enumerate}[label=(\arabic*)]}
\newcommand{\ee}{\end{enumerate}}

\newcommand{\R}{\mathbb{R}}
\newcommand{\C}{\mathbb{C}}

\newcommand{\set}[2]{\left\{#1 \middle| #2\right\}}

\newcommand{\alt}{\mathrm{alt}}

\def\semicolon{;}
\def\applytolist#1{
    \expandafter\def\csname multi#1\endcsname##1{
        \def\multiack{##1}\ifx\multiack\semicolon
            \def\next{\relax}
        \else
            \csname #1\endcsname{##1}
            \def\next{\csname multi#1\endcsname}
        \fi
        \next}
    \csname multi#1\endcsname}

\def\calc#1{\expandafter\def\csname c#1\endcsname{{\mathcal #1}}}
\applytolist{calc}QWERTYUIOPLKJHGFDSAZXCVBNM;
\def\bbc#1{\expandafter\def\csname bb#1\endcsname{{\mathbb #1}}}
\applytolist{bbc}QWERTYUIOPLKJHGFDSAZXCVBNM;
\def\bfc#1{\expandafter\def\csname bf#1\endcsname{{\mathbf #1}}}
\applytolist{bfc}QWERTYUIOPLKJHGFDSAZXCVBNM;
\def\sfc#1{\expandafter\def\csname s#1\endcsname{{\sf #1}}}
\applytolist{sfc}QWERTYUIOPLKJHGFDSAZXCVBNM;
\def\fc#1{\expandafter\def\csname f#1\endcsname{{\mathfrak #1}}}
\applytolist{fc}QWERTYUIOPLKJHGFDSAZXCVBNM;

\newcommand{\vNAlg}{{\sf vNAlg}}
\newcommand{\Bim}{{\sf Bim}}
\newcommand{\dBim}{{\sf Bim_{d}}}

\newcommand{\noshow}[1]{}
\newcommand{\MR}[1]{}

\newcommand{\Cstar}{{\rm C^*}}

\usepackage{mathtools}
\newcommand{\bim}[4][]{{}\prescript{\vphantom{#1}}{\vphantom{#4}#2}{#3}^{#1}_{\vphantom{#2}#4}}
\renewcommand{\bim}[3]{{}_{\vphantom{#3}#1}{#2}_{\vphantom{#1}#3}}

\usetikzlibrary{shapes}
\usetikzlibrary{cd}
\usetikzlibrary{backgrounds}
\usetikzlibrary{decorations,decorations.pathreplacing,decorations.markings}
\usetikzlibrary{fit,calc,through}
\usetikzlibrary{external}
\usetikzlibrary{arrows}
\tikzset{vertex/.style = {shape=circle,draw,fill=black,inner sep=0pt,minimum size=5pt}}
\tikzset{edge/.style = {->,> = latex', bend right}}
\tikzset{
	super thick/.style={line width=3pt}
}
\tikzset{
    quadruple/.style args={[#1] in [#2] in [#3] in [#4]}{
        #1,preaction={preaction={preaction={draw,#4},draw,#3}, draw,#2}
    }
}
\tikzstyle{shaded}=[fill=red!10!blue!20!gray!30!white]
\tikzstyle{unshaded}=[fill=white]
\tikzstyle{empty box}=[circle, draw, thick, fill=white, opaque, inner sep=2mm]
\tikzstyle{annular}=[scale=.7, inner sep=1mm, baseline]
\tikzstyle{rectangular}=[scale=.75, inner sep=1mm, baseline=-.1cm]
\tikzstyle{mid>}=[decoration={markings, mark=at position 0.5 with {\arrow{>}}}, postaction={decorate}]
\tikzstyle{mid<}=[decoration={markings, mark=at position 0.5 with {\arrow{<}}}, postaction={decorate}]
\tikzstyle{over}=[double, draw=white, super thick, double=]

\tikzdeclarepattern{
  name=primeddots,
  type=uncolored,
  bounding box={(-.6pt,-.6pt) and (.6pt,.6pt)},
  tile size={(5pt,5pt)},
  tile transformation={rotate=60},
  parameters={none}, 
  code={
    \fill(0pt,0pt) circle (.35pt);
  }
}

\tikzstyle{primedregion}[none]=[
	preaction={fill=#1},
	pattern=primeddots,
  draw=#1,
]

\newcommand{\roundNbox}[6]{
	\draw[rounded corners=5pt, very thick, #1] ($#2+(-#3,-#3)+(-#4,0)$) rectangle ($#2+(#3,#3)+(#5,0)$);
	\coordinate (ZZa) at ($#2+(-#4,0)$);
	\coordinate (ZZb) at ($#2+(#5,0)$);
	\node at ($1/2*(ZZa)+1/2*(ZZb)$) {#6};
}

\newcommand{\ScaledRoundNbox}[7]{
	\draw[rounded corners=#7pt, very thick, #1] ($#2+(-#3,-#3)+(-#4,0)$) rectangle ($#2+(#3,#3)+(#5,0)$);
	\coordinate (ZZa) at ($#2+(-#4,0)$);
	\coordinate (ZZb) at ($#2+(#5,0)$);
	\node at ($1/2*(ZZa)+1/2*(ZZb)$) {#6};
}

  \newcommand{\tikzmath}[2][]
     {\vcenter{\hbox{\begin{tikzpicture}[#1]#2
                     \end{tikzpicture}}}
     }

\newcommand{\AColor}{gray!20}
\newcommand{\BColor}{gray!55}

\newcommand{\APrimeColor}{gray!20}
\newcommand{\BPrimeColor}{gray!90}
\newcommand{\ATildeColor}{gray!20}
\newcommand{\BTildeColor}{gray!90}

\renewcommand{\subseteq}{\subset}

\begin{document}
\title{Distortion for multifactor bimodules and representations of multifusion categories}
\author{
Marcel Bischoff,
Ian Charlesworth,
Samuel Evington,
\\
Luca Giorgetti,
and
David Penneys
}
\date{\today}
\maketitle
\begin{abstract} 
We call a von Neumann algebra with finite dimensional center a multifactor.
We introduce an invariant of bimodules over $\rm II_1$ multifactors that we call modular distortion, and use it to formulate two classification results.

We first classify finite depth finite index connected hyperfinite $\rm II_1$ multifactor inclusions $A\subset B$
in terms of the standard invariant (a unitary planar algebra), together with the restriction to $A$ of the unique Markov trace on $B$.
The latter 
determines the modular distortion of the associated bimodule.
Three crucial ingredients are
Popa's 
uniqueness theorem for such inclusions which are also homogeneous, for which the standard invariant is a complete invariant,
a generalized version of the Ocneanu Compactness Theorem,
and the notion of Morita equivalence for inclusions.

Second, we classify fully faithful representations of unitary multifusion categories into bimodules over hyperfinite $\rm II_1$ multifactors 
in terms of
the modular distortion.
Every possible distortion arises from a representation, and we characterize the proper subset of distortions that arise from connected $\rm II_1$ multifactor inclusions.

\end{abstract}

\setcounter{tocdepth}{2}
\tableofcontents

\section{Introduction}

By a deep theorem of Popa \cite{MR1055708}, a finite depth finite index hyperfinite $\rm II_1$ subfactor $A\subset B$ is completely determined by its standard invariant.
This standard invariant has many equivalent axiomatizations, including
Ocneanu's paragroups/bi-unitary connections \cite{MR996454,MR1642584},
Popa's canonical commuting square \cite{MR1055708},
Popa's $\lambda$-lattices \cite{MR1334479}, and
Jones' planar algebras \cite{math.QA/9909027}.
Popa's classification theorem 
can be bootstrapped to show that every unitary fusion category admits an essentially unique fully faithful unitary tensor functor into $\Bim(R)$, where $R$ is the hyperfinite $\rm II_1$ factor \cite[\S3.2]{2004.08271}; see \cite{MR3635673} for the analogous statement for embedding into endomorphisms of the hyperfinite $\rm III_1$ factor based on
\cite{MR1339767}.

In this article, we extend these results to \emph{multifactor} inclusions and unitary \emph{multifusion} categories.
A unitary multifusion category is a semisimple rigid $\rm C^*$ tensor category with finitely many isomorphism classes of simple objects.
A multifactor is a von Neumann algebra with finite dimensional center.
A (unital) inclusion of finite multifactors $A\subset B$ is called \emph{connected} if $Z(A)\cap Z(B) = \C$, \emph{finite index} if the standard bimodule ${}_AL^2B{}_B$ is dualizable, and \emph{finite depth} if ${}_AL^2B{}_B$ generates a unitary multifusion subcategory of $\Bim(A\oplus B)$.

Our first main theorem gives a complete classification of finite depth finite index connected hyperfinite $\rm II_1$ multifactor inclusions.
By \cite[Thm.~3.7.3]{MR999799},
a finite index connected inclusion of finite multifactors $A\subset B$ has a unique \emph{Markov trace} $\tr^{\rm Markov}_B$, which is characterized by a certain Frobenius-Perron condition
(see \eqref{eq:MarkovTraceEigenvalueEquations} below),
and by \cite[3.6.4(i)]{MR999799}, $B{}_A$ is finitely generated and projective as a right $A$-module, so there is a finite Pimnser-Popa basis for $B$ over $A$ \cite{MR860811}.
The inclusion $A\subset (B,\tr_B^{\rm Markov})$ is said to be \emph{strongly Markov} following \cite{MR2812459}, and the standard invariant $\cP^{A\subset B}_\bullet$ is a 2-shaded unitary (i.e., $\rm C^*$ with finite dimensional box spaces) planar algebra.

\begin{thmalpha}
\label{thm:ClassificationOfFiniteDepthHyperfinite}
The map which takes $A\subset B$ to the pair $(\cP^{A\subset B}_\bullet, \tr^{\rm Markov}_B|_{Z(A)})$ descends to a bijection
$$
\frac{
\left\{
\parbox{5.8cm}{\rm
Finite depth finite index connected 
hyperfinite $\rm II_1$ multifactor inclusions
$A\subset B$ \phantom{$\cP_{0,+}$}
}\right\}
}
{
\parbox{5.8cm}{\rm
$\varphi: B_1 \xrightarrow{\sim} B_2$
taking $A_1$ onto $A_2$
}}
\cong
\frac{
\left\{
\parbox{6.5cm}{\rm
Pairs $(\cP_\bullet, \tau)$
with $\cP_\bullet$ a finite depth indecomposable unitary 2-shaded planar algebra and $\tau$ a faithful state on $\cP_{0,+}$
}\right\}
}{
\parbox{6.3cm}{\rm
$\varphi_\bullet : \cP^1_\bullet \xrightarrow{\sim} \cP^2_\bullet$ such that $\tau^2\circ \varphi_{0,+}=\tau^1$
}}
\,.
$$
The map which takes $A\subset B$ to $\cP^{A\subset B}_\bullet$ descends to a bijection
$$
\frac{
\left\{
\parbox{7.2cm}{\rm
Finite depth finite index connected 
hyperfinite $\rm II_1$ multifactor inclusions
$A\subset B$ \phantom{$\cP_{0,+}$}
}\right\}
}
{
\parbox{4cm}{\rm
Morita equivalence
}}
\cong
\frac{
\left\{
\parbox{5cm}{\rm
Finite depth indecomposable unitary 2-shaded planar algebras $\cP_\bullet$
}\right\}
}{
\parbox{5cm}{\rm
Planar $*$-algebra isomorphism
}}
\,.
$$
\end{thmalpha}

Importantly, the standard invariant $\cP^{A\subset B}_\bullet$ is only a complete invariant up to \emph{Morita equivalence} and not isomorphism of inclusions.
Here, two inclusions $A_1\subset B_1$ and $A_2\subset B_2$ are said to be \emph{Morita equivalent}
if there is an invertible bimodule ${}_{A_2}Y{}_{A_1}$ and a $*$-isomorphism $\psi : B_2 \to (B_1^{\op})' \cap B(Y\boxtimes_{A_1} L^2B_1)$ that restricts to the identity on $A_2$:
\begin{equation*}
\begin{tikzcd}[column sep=4em]
\hspace{-6mm}\psi:B_2
\arrow[r,rightarrow, "\cong"]
&
(B_1^{\op})' \cap B(Y\boxtimes_{A_1} L^2B_1)
\\
A_2
\arrow[r,rightarrow,"\id_{A_2}"']
\arrow[u,hookrightarrow]
&
A_2
\arrow[u,hookrightarrow]
\end{tikzcd}
\end{equation*}

We prove Theorem \ref{thm:ClassificationOfFiniteDepthHyperfinite} in two parts: the first part in Theorem \ref{thm:A-PartII}, preceded by the second part in Theorem \ref{thm:A-PartI}.
The other main tool we use, besides Morita equivalence, is Popa's uniqueness theorem for 
finite depth finite index 
\emph{homogeneous}
connected hyperfinite $\rm II_1$ multifactor inclusions, i.e., those such inclusions which admit a generating Jones tunnel.
See \S\ref{sec:Homogeneity} for other characterizations of homogeneity.
For completeness and convenience of the reader,
we provide a complete proof of Popa's theorem in the multifactor setting in Appendix~\ref{appendix:Popa}.

When the inclusion $A\subset B$ is not homogeneous, we may no longer have any Jones downward basic construction, let alone a generating tunnel.
An easy example is $A_0\otimes R \subset B_0\otimes R$ where $A_0\subset B_0$ is any finite dimensional inclusion with Bratteli diagram $A_4$ \cite[Ex.~1.2.8]{MR1339767}.
We treat this example in detail as Example \ref{ex:A4-example} below.
Nevertheless, we prove in Theorem \ref{thm:MultifactorInclusionMEtoHomogeneous} that any finite depth $\rm II_1$ inclusion is Morita equivalent to a homogeneous one.

To measure/quantify how an inclusion might fail to be homogeneous, and give further characterizations of homogeneity, we introduce Andr\'e Henriques' notion of the \emph{modular distortion} for $\rm II_1$ multifactor bimodules.
Suppose that $A$ and $B$ are hyperfinite $\rm II_1$ multifactors with minimal central projections $p_1, \dots, p_a$ and $q_1,\dots, q_b$, respectively, and define $A_i := p_iA$ and $B_j := q_jB$.
Given a connected dualizable $A$-$B$ bimodule $X$, we write $X_{ij}:= p_iXq_j$.
The modular distortion of $X$ is the partially defined function $\delta = \delta(X) : \{1,\dots, a\}\times \{1,\dots, b\} \to \bbR_{>0}$ given by
$$
\delta_{ij}
:=
\left(
\frac{\vNdim_L({}_{A_i} X_{ij} )}{\vNdim_R(X_{ij} {}_{B_j})}
\right)^{1/2}
\qquad\qquad
\text{whenever }
X_{ij} \neq (0).
$$
We analyze the behavior of the distortion under the Jones basic construction and the Jones downward basic construction.
This expands on 
\cite[Cor.~1.2.10]{MR1339767}
to give a quantitative answer
to when one can perform a downward basic construction.

\begin{rem*}
The notion of modular distortion for a $\rm II_1$ multifactor bimodule is closely related to Izumi's notion of Connes-Takesaki module for an endomorphism of a properly infinite factor; we explain the connection in Remark \ref{rem:ConnesTakesakiModules} below.
\end{rem*}

In Lemma \ref{lem:ExtendGraphWeighting}, we characterize those situations when $\delta$ extends (uniquely) to an everywhere defined function $\{1,\dots, a\}\times \{1,\dots, b\} \to \bbR_{>0}$ satisfying
\begin{equation}
\label{eq:DistortionExtensionCondition}
\delta_{ij}\delta_{i'j'}
=
\delta_{ij'}\delta_{i'j}
\qquad\qquad
\forall i,i'\leq a,\, j,j'\leq b.
\end{equation}
In those cases, there exist $\eta_i,\xi_j \in \bbR_{>0}$ (well defined up to global rescaling) such that $\delta_{ij}= \xi_j/\eta_i$.
When in addition the \emph{statistical dimension} $D_{ij}$ of each $X_{ij}$ is equal to its \emph{Jones dimension} (the square root of the Jones index), we call the bimodule ${}_A X{}_B$ \emph{extremal}; these two conditions together are the multifactor analog of the notion of extremality for $\rm II_1$ factor bimodules from \cite{MR3040370,MR3178106}.
In Corollary \ref{cor:FiniteDepthMultifactorBimoduleExtremal}, we show that any \emph{finite depth} bimodule (i.e., which generates a unitary multifusion category under taking finite direct sums, Connes fusion tensor products, subobjects, and contragredients) is automatically extremal.
In Section \ref{sec:ExtremalityAndMinimalExpectation}, when $A\subset B$ is a finite index connected inclusion of finite multifactors,  
we connect our definition of extremality for ${}_AL^2B{}_B$ with the minimality of the Markov trace-preserving conditional expectation $E: B \to A$.

Theorem \ref{thm:ClassificationOfFiniteDepthHyperfinite} tells us that not every distortion function satisfying \eqref{eq:DistortionExtensionCondition} can arise from a finite depth inclusion.
Indeed, 
by Corollary \ref{cor:trA and D determine trB and delta},
if $A\subset B$ is a finite depth finite index connected inclusion of finite multifactors,
then the distortion is determined by $\tr_A=\tr^{\rm Markov}_B|_A$, the restriction to $A$ of the unique Markov trace on $B$: 
\begin{equation}
\label{eq:delta in terms of trA and D}
\delta_{ij}
=
\left(\frac{\alpha_i}{\tr_A(p_i)}\right)
\sum_{h=1}^a 
\left(\frac{\tr_A(p_h)}{\alpha_h}\right)D_{hj}.
\end{equation}
Here, $\alpha_1,\ldots,\alpha_a$ are the first $a$ coordinates of a Frobenius-Perron eigenvector of $\big(\begin{smallmatrix}0&D\\D^T&0\end{smallmatrix}\big)$.
The space of possible distortion functions satisfying \eqref{eq:DistortionExtensionCondition} is paramterized by $\bbR^{a+b-1}_{>0}$, whereas the space of distortion functions realizable by an inclusion $A\subset B$ is parametrized by
the space of faithful traces on $A$, which is homeomorphic to $\bbR^{a-1}_{>0}$.

As an application of the modular distortion, we give a complete classification of representations of unitary multifusion categories into bimodules over hyperfinite $\rm II_1$ multifactors.
If $\cC$ is an indecomposable unitary multifusion category with $\dim(\End_\cC(1_\cC))=n$, then we call $\cC$ an $n\times n$ unitary multifusion category.
A \emph{representation} of an $n\times n$ unitary multifusion category consists of a hyperfinite $\rm II_1$ multifactor $A=\bigoplus_{i=1}^n A_i$ (where each $A_i$ is a hyperfinite $\rm II_1$ factor), together with a fully faithful unitary tensor functor $\alpha: \cC \to \Bim(A)$.
The \emph{modular distortion of $\alpha$} (an idea due to Andr\'e Henriques) is the matrix $\delta^\alpha \in M_n(\bbR_{>0})$ given by
$\delta^\alpha_{ij} :=\delta({}_{A_i}\alpha(c){}_{A_j})$
for $c\in \cC_{ij}$ --- this is independent of the choice of object $c \in \cC_{ij}$ (see Definition \ref{defn:DistortionOfRepresentation} below).
The modular distorsion of $\alpha$ is a groupoid homomorphism $\delta^\alpha:\cG_n\to \bbR_{>0}$ from the groupoid $\cG_n$ with $n$ objects and a unique isomorphism between any two objects to the group(oid) $\bbR_{>0}$, namely $\delta^{\alpha}_{ij}\delta^{\alpha}_{jk} = \delta^{\alpha}_{ik}$ for all $i,j,k\leq n$.

An \emph{isomorphism} between two representations $\alpha: \cC \to \Bim(A)$ and $\beta: \cC\to \Bim(B)$ consists of an invertible bimodule ${}_B\Phi{}_A$ together with a family of unitary natural isomorphisms $\phi=\{{}_B\Phi\boxtimes_A \alpha(c){}_A \to {}_B\beta(c)\boxtimes_B \Phi{}_A\}$
satisfying a certain coherence axiom \eqref{eq:HalfBraiding}.
We say that an isomorphism $(\Phi,\phi)$ is \emph{induced by} an algebra isomorphism $\varphi: A \to B$ if $\Phi = {}_BL^2B{}_{\varphi(A)}$ where the right $A$-action is transported by $\varphi$ (and $\phi$ is arbitrary).

Our second main theorem is as follows.

\begin{thmalpha}
\label{thm:ClassificationOfRepresentationsOfMultifusion}
Let $\cC$ be an $n\times n$ unitary multifusion category.
Then the map $\alpha\mapsto\delta^\alpha$ which assigns to a representation $\alpha$ its modular distortion descends to a bijection
$$
\frac{
\left\{
\parbox{5.9cm}{\rm
Representations $\alpha : \cC \to \Bim(R^{\oplus n})$
}\right\}
}
{
\parbox{6cm}{\rm
Iso $(\Phi,\phi)$ induced by
$\varphi \in \Aut(R^{\oplus n})$
}}
\,\cong\,
\left\{
\parbox{6.8cm}{\rm
Groupoid homomorphsims $\delta : \cG_{n} \to \bbR_{>0}$
}\right\}.
$$
Moreover, there is a unique representation $\cC \to \Bim(R^{\oplus n})$ up to isomorphism:
$$
\frac{
\left\{
\parbox{5.9cm}{\rm
Representations $\alpha : \cC \to \Bim(R^{\oplus n})$
}\right\}
}
{
\parbox{3.25cm}{\rm
Isomorphism $(\Phi,\phi)$
}}
\,\,\cong\,\,
\{ * \}
\,.
$$
\end{thmalpha}

\noindent We prove the first part of Theorem \ref{thm:ClassificationOfRepresentationsOfMultifusion} in Theorem \ref{thm:B-partI}, and the second part in Theorem \ref{thm:B-partII}.

\paragraph{Acknowledgements.}
This project began at the 2018 AMS MRC on Quantum Symmetries: Subfactors and Fusion Categories; this and further collaboration support was funded by NSF DMS grant 1641020. 
MB was supported by NSF DMS grant 1700192/1821162.
IC was supported by NSF DMS grant 1803557.
SE was supported by EPSRC grant EP/R025061/1.
LG was supported by EU H2020-MSCA-IF-2017 grant beyondRCFT 795151, ERC Advanced Grant QUEST 669240, MIUR Excellence Department Project awarded to the Department of Mathematics University of Rome Tor Vergata CUP E83C18000100006 and GNAMPA-INdAM.
DP was supported by NSF DMS grants 1500387/1655912 and 1654159.

The authors are indebted to Andr\'e Henriques for many ideas, including 
Definition \ref{defn:ModularDistortion} of the modular distortion,
the importance of Equation \eqref{eq:DistortionExtensionCondition},
Definition \ref{defn:MoritaEquivalence} of a Morita equivalent multifactor inclusion,
and Definition \ref{defn:DistortionOfRepresentation} of the distortion of a representation of a unitary multitensor category.
We also thank him for his help with many proofs in this manuscript including
Proposition \ref{prop:ConstantDistortionIffMinimal},
an older proof of Proposition \ref{sec:DistortionDynamicalForBasicConstruction},
Theorem \ref{thm:tunnelequivalencies}, and
Propositions \ref{prop:DistortionFormulaUnderMoritaEquivalence} and \ref{prop:MoritaEquivalentDistortions}.
The authors would also like to thank
Corey Jones,
Roberto Longo,
Cris Negron, and
Sorin Popa
for helpful conversations.

\section{Background}

\subsection{Unitary multitensor categories}

We refer the reader to \cite[\S2.1]{MR3663592}, \cite[\S2.1 and 2.2]{MR3687214}, and \cite[\S2.1 and 2.2]{2004.08271} for a rapid introduction to $\rm C^*$ and $\rm W^*$ tensor categories.
More references on ($\rm C^*$ and $\rm W^*$) tensor categories include \cite{MR808930,MR2767048,MR3242743}.
We refer the reader to \cite{MR3994584} for a background on $\rm C^*$ and $\rm W^*$ 2-categories.
It is well-known that 2-categories with exactly one object are monoidal categories \cite[Periodic Table in \S2.1 and \S5.6]{MR2664619}; a similar statement holds for $\rm C^*$ and $\rm W^*$ 2-categories.

There is a powerful graphical calculus of string diagrams for 2-categories where objects correspond to shaded regions, 1-morphisms correspond to strands, and 2-morphisms correspond to coupons 
\cite[\S8.1.2]{MR3971584}.
As monoidal categories can be viewed as 2-categories with one object, there is only one shading for regions in string diagrams for monoidal categories, where we denote objects by strands and morphisms by coupons \cite{MR1107651,MR2767048}.
In the graphical calculus, all associator and unitor isomorphisms in our 2-category/monoidal category are left implicit.

\begin{defn}
Let $\cC$ be a 2-category and let $a,b\in\cC$ be two objects.
A 1-morphism $X\in \cC(a\to b)$ is called \emph{dualizble} if there is a \emph{dual 1-morphism} $X^\vee\in \cC(b\to a)$ together with  
2-morphisms
\begin{align*}
\ev_{X} \in \cC(X^\vee \otimes X \Rightarrow 1_b)
\qquad\qquad
\coev_{X} \in \cC(1_a \Rightarrow X\otimes X^\vee)
\end{align*}
satisfying the \emph{zig-zag} or \emph{snake equations}:
\begin{align}
  \tikzmath{
  \begin{scope}
     \clip[rounded corners = 5] (-1,-.8) rectangle (1,.8);
    \fill[\AColor] (.6,-.8) -- (.6,0) arc (0:180:.3cm) arc (0:-180:.3cm) -- (-.6,.8) -- (-1,.8) -- (-1,-.8);
    \fill[\BColor] (.6,-.8) -- (.6,0) arc (0:180:.3cm) arc (0:-180:.3cm) -- (-.6,.8) -- (1,.8) -- (1,-.8);
  \end{scope}
	\draw (.6,-.8) -- (.6,0) arc (0:180:.3cm) arc (0:-180:.3cm) -- (-.6,.8);
	\node at (.8,-.6) {\scriptsize{$X$}};
	\node at (-.25,.0) {\scriptsize{$X^\vee$}};
	\node at (-.8,.6) {\scriptsize{$X$}};
  }
  =
  \tikzmath{
  \begin{scope}
    \clip[rounded corners = 5] (-.6,-.8) rectangle (.6,.8);
    \fill[\AColor] (0,-.8) rectangle (-.6,.8);
    \fill[\BColor] (0,-.8) rectangle (.6,.8);
  \end{scope}
	\draw (0,-.8) -- (0,.8) node [midway, right] {\scriptsize{$X$}};
  }\qquad\quad\,\,
    &(\id_X \otimes \ev_X)\circ (\coev_X \otimes \id_X) 
  =
  \id_X,
  \label{eq:ShadedZigZag1}
  \\
  \tikzmath[xscale=-1]{
  \begin{scope}
    \clip[rounded corners = 5] (-1.2,-.8) rectangle (1.2,.8);
  \fill[\AColor] (.6,-.8) -- (.6,0) arc (0:180:.3cm) arc (0:-180:.3cm) -- (-.6,.8) -- (-1.2,.8) -- (-1.2,-.8);
  \fill[\BColor] (.6,-.8) -- (.6,0) arc (0:180:.3cm) arc (0:-180:.3cm) -- (-.6,.8) -- (1.2,.8) -- (1.2,-.8);
  \end{scope}
	\draw (.6,-.8) -- (.6,0) arc (0:180:.3cm) arc (0:-180:.3cm) -- (-.6,.8);
	\node at (.9,-.6) {\scriptsize{$X^\vee$}};
	\node at (-.2,.0) {\scriptsize{$X$}};
	\node at (-.9,.6) {\scriptsize{$X^\vee$}};
  }
  =
  \tikzmath{
  \begin{scope}
    \clip[rounded corners = 5] (-.6,-.8) rectangle (.6,.8);
    \fill[\BColor] (0,-.8) rectangle (-.6,.8);
    \fill[\AColor] (0,-.8) rectangle (.6,.8);
  \end{scope}
	\draw (0,-.8) -- (0,.8) node [midway, right] {\scriptsize{$X^\vee$}};
  }\qquad
    (&\ev_X \otimes \id_{X^\vee}) \circ (\id_{X^\vee} \otimes \coev_X)
  =
  \id_{X^\vee}.
  \label{eq:ShadedZigZag2}
\end{align}
Here, the shaded regions denote the two objects
$\tikzmath{\filldraw[\AColor, rounded corners=5, very thin, baseline=1cm] (0,0) rectangle (.5,.5);}=a$,
$\tikzmath{\filldraw[\BColor, rounded corners=5, very thin, baseline=1cm] (0,0) rectangle (.5,.5);}=b$,
and we represent the evaluation and coevaluation morphisms by a cap and cup, respectively:
\[
\tikzmath{
  \begin{scope}
    \clip[rounded corners = 5] (-1,-.4) rectangle (1, .4);
    \fill[\BColor] (-1,-.4) rectangle (1, .4);
    \fill[\AColor] (-.3, -.4) -- ++ (0,.3) arc (180:0:.3cm) -- ++ (0, -.3);
  \end{scope}
  \draw (-.3, -.4) -- ++ (0,.3) node[midway,left] {\scriptsize{$X^\vee$}} arc (180:0:.3cm) -- ++ (0, -.3) node[midway,right] {\scriptsize{$X$}};
}
=
\ev_X
\qquad\quad
\tikzmath{
  \begin{scope}
    \clip[rounded corners = 5] (-1,-.4) rectangle (1, .4);
    \fill[\AColor] (-1,-.4) rectangle (1, .4);
    \fill[\BColor] (-.3, .4) -- ++ (0,-.3) arc (-180:0:.3cm) -- ++ (0, .3);
  \end{scope}
  \draw (-.3, .4) -- ++ (0,-.3) node[midway,left] {\scriptsize{$X$}} arc (-180:0:.3cm) -- ++ (0, .3) node[midway,right] {\scriptsize{$X^\vee$}};
}
=
\coev_X.
\]
We also require that the 1-morphism $X$ admits a \emph{predual 1-morphism} $X_\vee$ such that $(X_\vee)^\vee \cong X$.

We call $\cC$ \emph{rigid} if every 1-morphism in $\cC$ is dualizable.
Similarly, we call a monoidal category \emph{rigid} if every object is dualizble.
\end{defn}

\begin{defn}
A \emph{unitary multitensor category} is a semisimple rigid tensor $\Cstar$ category.
We call such a category \emph{indecomposable} if it is not the direct sum of two unitary multitensor categories.
An $r\times r$ unitary multitensor category is an indecomposable unitary multitensor category such that $\dim(\End_\cC(1_\cC))=r$. 
A \emph{unitary multifusion category} is a finitely semisimple unitary multitensor category.
\end{defn}

\begin{nota}
We let $1_\cC = \bigoplus_{i=1}^r 1_i$ be a decomposition into simples, and write $p_i\in \End_\cC(1_\cC)$ for the minimal projection onto $1_i$.

We write $\cC_{ij}:= 1_i \otimes \cC \otimes 1_j$, so that $\cC = \bigoplus \cC_{ij}$.
We may view $\cC$ as a rigid $\rm C^*/W^*$ 2-category with $r$ objects $1_1,\dots, 1_r$, and hom categories $\Hom(1_i \to 1_j) :=\cC_{ij}$.
In our graphical calculus for $\cC$, we may choose to use different shaded regions to denote various summands of $1_\cC$.
\end{nota}

Suppose $\cC$ is a unitary multitensor category.
A choice of triple $(c^\vee, \ev_c,\coev_c)$ for every object $c\in \cC$ gives rise to a monoidal \emph{dual functor} $\vee:\cC \to \cC^{\rm mop}:c\mapsto c^\vee$.
(Here, $\cC^{\rm mop}$ denotes the category with the opposite monoidal structure and opposite arrows.)
At the level of arrows, the dual functor is given by
$$
\cC(a\to b) \ni f
\longmapsto
\tikzmath{
  \draw (0,.3) arc (0:180:.3cm) -- (-.6,-.6) node [below] {$\scriptstyle b^\vee$};
  \draw (0,-.3) arc (-180:0:.3cm) -- (.6,.6) node [above] {$\scriptstyle a^\vee$};
	\roundNbox{unshaded}{(0,0)}{.3}{0}{0}{$f$};
}
=
(\ev_b \otimes \id_{a^\vee})
\circ
(\id_{b^\vee}\otimes f\otimes \id_{a^\vee})
\circ
(\id_{b^\vee}\otimes \coev_a)
\in
\cC(b^\vee \to a^\vee)
$$
and the tensorator $\nu_{a,b}: a^\vee \otimes b^\vee \to (b\otimes a)^\vee$ is given by
$$
\nu_{a,b}
:=
\tikzmath{
	\draw (.8,.6) -- (.8,0) arc (0:-180:.4cm) arc (0:180:.2cm) -- (-.4,-.6) node [below] {$\scriptstyle b^\vee$};
	\draw (.7,.6) -- (.7,0) arc (0:-180:.3cm) arc (0:180:.5cm) -- (-.9,-.6) node [below] {$\scriptstyle a^\vee$};
	\node at (.8,.8) {$\scriptstyle (b\otimes a)^\vee$};
}
=
(\ev_a\otimes \id_{(b\otimes a)^\vee})
\circ 
(\id_{a^\vee}\otimes \ev_b \otimes \id_a \otimes \id_{(b\otimes a)^\vee})
\circ
(\id_{a^\vee\otimes b^\vee} \otimes \coev_{b\otimes a}).
$$
If $\vee$ is a dagger functor, and the morphism $\nu_{a,b}$ is unitary for all $a,b\in \cC$, then we call $(\vee, \ev, \coev)$ a \emph{unitary dual functor}.
Following \cite[Lem.~7.5]{MR2767048}, given a unitary dual functor $(\vee, \ev, \coev)$, we get a unitary monoidal natural isomorphism $\varphi : \id_\cC \Rightarrow \vee\circ\vee$ by

$$
\varphi_c := (\coev_c^* \otimes \id_{c^{\vee\vee}})\circ (\id_c \otimes \coev_{c^\vee}).
$$
We call such a $\varphi$ (coming from a unitary dual functor) a \emph{unitary pivotal structure}.

Let $\cG_r$ be the groupoid with $r$ objects $\Ob(\cG_r)=\{1,\dots, r\}$, and a single isomorphism $e_{ij}: i\to j$ between any two objects (analogous to a system of matrix units $(e_{ij})$ for $M_r(\bbC)$).
An $r\times r$ unitary multitensor category $\cC$ is naturally graded by (the arrows of) $\cG_r$.

\begin{defn}
Given a morphism $f\in \cC(c\to c)$, the matrix-valued left and right pivotal traces $\Tr^\vee_L(f)$ and $\Tr^\vee_R(f)$ are determined respectively by
\begin{align*}
\Tr^\vee_L(f)_{ij} \id_{1_i} &= \coev_c^*\circ ( (p_i\otimes f\otimes p_j)\otimes \id_{c^\vee})\circ \coev_c
\qquad\text{and}
\\
\Tr^\vee_R(f)_{ij} \id_{1_j} &= \ev_c\circ (\id_{c^\vee}\otimes (p_i\otimes f\otimes p_j))\circ \ev_c^*.
\end{align*}
When $c\in \cC$ is homogeneous of degree $e_{ij}$ (namely when $c\in \cC_{ij}$), 
the only possibly non-zero entry of $\Tr^\vee_{L/R}(f)$ is the $(i,j)^{\rm th}$ one.
In that case, we set $\tr^\vee_{L/R}(f):=\Tr^\vee_{L/R}(f)_{ij}$.
Similarly, we define the left and right matrix-valued dimensions by 
$\Dim^\vee_{L/R}(c):= \Tr^\vee_{L/R}(\id_c)$.
When $c\in \cC_{ij}$, the matrix $\Dim^\vee_{L/R}(c)$ has exactly one non-zero entry, which we denote by $\dim^\vee_{L/R}(c)$.
\end{defn}

We recall the following classification theorem for unitary dual functors:

\begin{thm}[{\cite[Thm.~A]{MR4133163}}]
\label{thm:ClassificationOfUnitaryDualFunctors}
Let $\cC$ be a unitary multitensor category and let $\cU$ be its universal grading groupoid.
Then for $\vee$ a unitary dual functor on $\cC$, and $c\in \cC$ a non-zero object graded by $g\in\cU$, the quantity
$$
\pi(g) := \frac{\dim^\vee_L(c)}{\dim_R^\vee(c)}
$$
is independent of the choice of object $c$.

The map sending $\vee$ to $\pi$ establishes a bijection between unitary dual functors on $\cC$ up to unitary equivalence and groupoid homomorphisms $\cU \to \bbR_{>0}$.

If $\cC$ is an $r\times r$ unitary multifusion, then (since $\cU$ is finite) unitary equivalence classes of unitary dual functors are in canonical bijection with groupoid homomorphisms $\cG_r \to \bbR_{>0}$.
\end{thm}

The inverse map sends $\pi: \cU \to \bbR_{>0}$
to a unitary dual functor characterized by the so-called `$\pi$-balanced' solutions to the conjugate equations,
and the case $\pi = 1$ gives the unique unitary spherical structure on $\cC$.
We refer the reader to \cite{MR4133163} for more details.

\subsubsection{2-shadings}

Of particular importance to this article are 2-\emph{shaded} $r\times r$ unitary multitensor categories.

\begin{defn}\label{defn:2-shadingGeneratesConnected}
A 2-\emph{shading} on an $r\times r$ unitary multitensor category $\cC$ is an orthogonal decomposition $1_\cC = 1^+ \oplus 1^-$ of the unit object of $\cC$ into two non-zero objects (the objects $1^+$ and $1^-$ are not assumed to be simple).
\end{defn}

Let $a:= \dim(\End_\cC(1^+))$ and $b := \dim(\End_\cC(1^-))$ so that $r=a+b$,
let $1^+ = \bigoplus_{i=1}^a 1^+_i$ and 
$1^- = \bigoplus_{j=1}^b 1^-_j$
be orthogonal decompositions into simples, and let $p_i \in \End_\cC(1^+)$ be the minimal projection onto $1^+_i$ and $q_j \in \End_\cC(1^-)$ the minimal projection onto $1^-_j$.
We denote the objects $1^+$ and $1^-$ by the following two shadings:
$$
\tikzmath{\filldraw[\AColor, rounded corners=5, very thin, baseline=1cm] (0,0) rectangle (.5,.5);}=1^+
\quad\qquad
\tikzmath{\filldraw[\BColor, rounded corners=5, very thin, baseline=1cm] (0,0) rectangle (.5,.5);}=1^-.
$$

An object $X\in \cC^{+-}:=1^+ \otimes \,\cC\, \otimes 1^-$ is said to \emph{generate} $\cC$ if $\cC$ is Cauchy tensor generated by $X$ and $X^\vee$ (the dual object $X^\vee$ is well defined up to isomorphism).
In this setting, we write 
$X_{ij}:= 1^+_i \otimes X \otimes 1^-_j$ for the homogeneous components of $X$ of degree $e_{ij}$.
In the graphical calculus, we denote $X$ by a strand with the two shaded regions for $1^\pm$ on either side, and $X_{ij}$ is denoted by tensoring with $p_i$ and $q_j$ on the left and right
$$
X
=
\tikzmath{
  \begin{scope}
    \clip[rounded corners = 5] (0,-.5) rectangle (1,.5);
    \fill[\AColor] (0,-.5) rectangle (.5,.5);
    \fill[\BColor] (.5,-.5) rectangle (1,.5);
  \end{scope}
  \draw (.5,-.5) -- (.5,.5);
}
\qquad\qquad
X_{ij}
=
\tikzmath{
  \begin{scope}
    \clip[rounded corners = 5] (-.5,-.5) rectangle (1.5,.5);
    \fill[\AColor] (-.5,-.5) rectangle (.5,.5);
    \fill[\BColor] (.5,-.5) rectangle (1.5,.5);
  \end{scope}
  \draw (.5,-.5) -- (.5,.5);
  \roundNbox{unshaded}{(0,0)}{.25}{0}{0}{$p_i$}
  \roundNbox{unshaded}{(1,0)}{.25}{0}{0}{$q_j$}
}\,.
$$
Observe that since $X$ generates $\cC$, and the latter is indecomposable:
\begin{enumerate}[label=(\arabic*)]
\item 
we have inclusions of finite dimensional von Neumann algebras $\End_\cC(1^+)\hookrightarrow \End(X)$ and $\End_\cC(1^-) \hookrightarrow \End_\cC(X)$ given by
$$
\tikzmath{
  \fill[\AColor, rounded corners = 5] (-.5,-.5) rectangle (.5,.5);
  \roundNbox{unshaded}{(0,0)}{.25}{0}{0}{$p$}
}
\mapsto
\tikzmath{
  \begin{scope}
    \clip[rounded corners = 5] (-.5,-.5) rectangle (.8,.5);
    \fill[\AColor] (-.5,-.5) rectangle (.5,.5);
    \fill[\BColor] (.5,-.5) rectangle (.8,.5);
  \end{scope}
  \draw (.5,-.5) -- (.5,.5);
  \roundNbox{unshaded}{(0,0)}{.25}{0}{0}{$p$}
}
\qquad
\text{and}
\qquad
\tikzmath{
  \fill[\BColor, rounded corners = 5] (-.5,-.5) rectangle (.5,.5);
  \roundNbox{unshaded}{(0,0)}{.25}{0}{0}{$q$}
}
\mapsto
\tikzmath[xscale=-1]{
  \begin{scope}
    \clip[rounded corners = 5] (-.5,-.5) rectangle (.8,.5);
    \fill[\BColor] (-.5,-.5) rectangle (.5,.5);
    \fill[\AColor] (.5,-.5) rectangle (.8,.5);
  \end{scope}
  \draw (.5,-.5) -- (.5,.5);
  \roundNbox{unshaded}{(0,0)}{.25}{0}{0}{$q$}
}\,.
$$
\item
$X$ is \emph{connected}, that is, the intersection of the images of $\End_\cC(1^+)$ and $\End_\cC(1^-)$ in $\End_\cC(X)$ is $\bbC\id_X$. 
Equivalently, the bipartite graph with $a$ even vertices $b$ odd vertices and an edge from $i$ to $j$ whenever $X_{ij}\neq 0$ is connected.
\end{enumerate}

Let $\Dim(X)=\big(\begin{smallmatrix}0&D_X\\0&0\end{smallmatrix}\big)$ be the dimension matrix for $X$ with respect to the unique unitary spherical structure on $\cC$. It is block upper triangular, and its upper right corner $D_X$ an $a\times b$ matrix.
Let $d_X:=\|\Dim(X)\|=\|D_X\|$ be the Frobenius-Perron eigenvalue.
By Frobenius-Perron theory applied to the matrix $\big(\begin{smallmatrix}0&D_X\\D_X^T&0\end{smallmatrix}\big)$, there are unique vectors $\alpha \in \bbR_{>0}^a$ and $\beta \in \bbR_{>0}^b$ with strictly positive entries satisfying 
$D_X \beta = d_X\alpha$,
$D_X^T \alpha = d_X \beta$, and
$\|\alpha\|_2 = 1 = \|\beta\|_2$.
The object $X$ induces a \emph{standard unitary dual functor} on $\cC$, as follows:

\begin{defn} [{\cite[Def.~8.29]{MR3994584}}, {\cite[\S4.2]{MR4133163}}]
\label{defn:StandardDualFunctor}
Let $\cC$ be a 2-shaded multitensor category, and let $D_X=(D_{ij})$ and $d_X$ be as above.
The standard unitary dual functor associated to $X$ is determined by the following identities:
\begin{equation}
\label{eq:pi-q_j bubble}
\tikzmath{
  \fill[\AColor, rounded corners=5] (-.5,-.75) rectangle (1.75,.75);
  \filldraw[fill=\BColor] (1,0) circle (.5cm);
  \roundNbox{unshaded}{(0,0)}{.25}{0}{0}{$p_i$}
  \roundNbox{unshaded}{(1,0)}{.25}{0}{0}{$q_j$}
}
=
\frac{D_{ij}\beta_j}{\alpha_i}\,
\tikzmath{
  \fill[\AColor, rounded corners=5] (-.6, -.6) rectangle (.6, .6);
  \roundNbox{unshaded}{(0,0)}{.25}{0}{0}{$p_i$}
}
\qquad\qquad
\tikzmath[xscale=-1]{
  \fill[\BColor, rounded corners=5] (-.5,-.75) rectangle (1.75,.75);
  \filldraw[fill=\AColor] (1,0) circle (.5cm);
  \roundNbox{unshaded}{(0,0)}{.25}{0}{0}{$q_j$}
  \roundNbox{unshaded}{(1,0)}{.25}{0}{0}{$p_i$}
}
=
\frac{D_{ij}\alpha_i}{\beta_j}\,
\tikzmath{
  \fill[\BColor, rounded corners=5] (-.6, -.6) rectangle (.6, .6);
  \roundNbox{unshaded}{(0,0)}{.25}{0}{0}{$q_j$}
}
\,.
\end{equation}
It satisfies
the property that the two loop parameters for $X$ are both equal to the scalar $d_X$:
$$
\tikzmath{
  \fill[\AColor, rounded corners=5] (-.6, -.6) rectangle(.6, .6);
  \draw[fill=\BColor] (0,0) circle (.3cm);
}
:=
\coev_X^*\circ \coev_X
=
d_X
\id_{1^+}
\qquad\qquad
\tikzmath{
	\fill[\BColor, rounded corners=5] (-.6,-.6) rectangle (.6,.6);
	\draw[fill=\AColor] (0,0) circle (.3cm);
}
:=
\ev_X\circ \ev_X^*
=
d_X
\id_{1^-}.
$$
When $\cC$ is in addition multifusion, we will see in Theorem \ref{thm:UniqueUnitaryDualFunctor} below that this property uniquely characterizes the standard unitary dual functor.

Taking the ratio of the scalars for the left and right dimensions in \eqref{eq:pi-q_j bubble} above 
gives the formula for a groupoid homomorphism $\pi: \cG_{a+b} \to \bbR_{>0}$ describing the standard unitary dual functor associated to $X$:
\begin{equation}
\label{eq:UnitaryDualFunctorFormula}
\pi(e_{ij})
=
\frac{D_{ij}\alpha_i/\beta_j}{D_{ij}\beta_j/\alpha_i}
=
\frac{\alpha_i^2}{\beta_j^2}
\qquad
\forall\, 1\leq i \leq a
\text{ and }
\forall\, 1\leq j \leq b.
\end{equation}
This last formula appears in \cite[Lem.~4.5 and (27)]{MR4133163}.
By universality of the grading groupoid $\cU$, we get a groupoid homomorphism $\cU \to \cG_{a+b}\to \bbR_{>0}$.
\end{defn}

\subsection{Bimodules over multifactors}
\label{sec:BimodulesOverMultifactors}

In this section, $A,B,C,M,N$ denote von Neumann algbras, and $H,K,L$ denote separable Hilbert spaces.
All von Neumann algebras are assumed to have separable preduals.

\begin{defn}
An $A$-$B$ bimodule is a Hilbert space $H$ together with normal unital $*$-algebra homomorphisms $\lambda: A \to B(H)$ and $\rho: B^{\op} \to B(H)$ such that $[\lambda(a),\rho(b)]=0$ for all $a\in A$ and $b\in B$.
For notational simplicity, we suppress $\lambda,\rho$, and write $a\xi b :=\lambda(a) \rho(b)\xi$.

The collection $\vNAlg$ of von Neumann algebras with separable preduals, bimodules, and intertwiners forms a $\rm W^*$ 2-category, where composition of 1-morphisms is the  \emph{Connes fusion relative tensor product} \cite{MR703809}, \cite[Appendix~B.$\delta$]{MR1303779}, \cite{1705.05600}.
The tensor product $H\boxtimes_B K$ of ${}_A H{}_B$ and ${}_B K{}_L$ is the completion of the complex vector space
$$
\Hom_{-B}(L^2B \to H)
\otimes_B
L^2B
\otimes_B
\Hom_{B-}(L^2B \to K),
$$
under the sesquilinear form given by (the linear extension of)
$$
\langle 
f_1 \otimes \xi_1 \otimes g_1
,
f_2 \otimes \xi_2 \otimes g_2
\rangle 
:=
\langle
(f_2^* \circ f_1)\xi_1(g_2^*\circ g_1)
,
\xi_2
\rangle_{L^2B}.
$$
Here, $f_2^* \circ f_1 \in \End_{-B}(L^2B)\cong B$ where the identification is via the left action map, and $g_2^* \circ g_1 \in \End_{B-}(L^2B)\cong B$ where the identification is via the right action map.

Here above, $L^2B$ is the canonical Haagerup $L^2$ space \cite{MR0407615} which can be defined state-independently \cite{MR3342166}.
For every faithful normal state $\varphi$ on $B$, there is a canonical $B$-$B$ bimodule unitary isomorphism $L^2(B, \varphi)\cong L^2B$. We write $\sqrt{\varphi}\in L^2B$ for the image of the canonical cyclic vector $\Omega_\varphi\in L^2(B,\varphi)$.

Given an $A$-$B$ bimodule ${}_AH_B$, the \emph{conjugate} bimodule ${}_B\overline{H}_A$ is the complex conjugate Hilbert space $\overline{H}$ of $H$ (whose elements we denote by $\overline{\xi}$ for $\xi \in H$), equipped with the $B$-$A$ bimodule structure given by $b\overline{\xi} a:= \overline{a^*\xi b^*}$.
Given an intertwiner $f\in \Hom_{A-B}(H\to K)$, we define $\overline{f}\in \Hom_{B-A}(\overline{H} \to \overline{K})$ by $\overline{f}(\overline{\xi}):= \overline{f(\xi)}$.
It is straightforward to verify that $\overline{f^*} = \overline{f}^*$ and $\overline{f\circ g} = \overline{f}\circ \overline{g}$.
Hence, $\vNAlg$ is a \emph{bi-involutive} $\rm W^*$ 2-category in the sense of \cite[Def.~2.3]{MR3663592} and \cite[\S4.2 and 5.2]{2004.08271}.

Given a fixed von Neumann algebra $A$, we denote by $\Bim(A)$ the full bi-involutive $\rm W^*$ 2-subcategory of $\vNAlg$ whose only object is $A$. In other words, $\Bim(A)$ is a bi-involutive $\rm W^*$ tensor category.
\end{defn}

Recall that a \emph{factor} is a von Neumann with trivial center.
Factors are the fundamental building blocks of the theory of von Neuemann algebras. Indeed any von Neumann can be decomposed as a direct integral of factors.
In particular, a von Neumann algebra $A$ with finite dimensional center $Z(A) = \mathbb{C}^k$ decomposes as a finite direct sum of factors $A = \bigoplus_{i=1}^k p_iAp_i$, where $p_1,\ldots,p_k$ are the minimal central projections (and $p_iAp_i=p_iA$).

\begin{defn}
A \emph{multifactor} is a finite direct sum of factors.
A multifactor is called \emph{finite} if every summand is a finite von Neumann algebra.
A multifactor is called a $\rm II_1$ multifactor if every summand is a $\rm II_1$ factor.
\end{defn}

Note that every finite dimensional von Neumann algebra is a multifactor.

\begin{nota}
\label{nota:BimoduleNotation}
For the remainder of this article, unless stated otherwise, $M$ and $N$ are finite factors, ${}_MH{}_N$ is an $M$-$N$ bimodule, $A$ and $B$ are finite multifactors, and ${}_AX{}_B$ is an $A$-$B$ bimodule.
We let $\{p_i\}_{1\le i\le a}$ denote the minimal central projections of $A$, and $\{q_j\}_{1\le j\le b}$ the minimal central projections of $B$.
We define $A_i := p_iA$ and $B_j := q_jB$, and we let $X_{ij}:=p_iXq_j$, which is an $A_i$-$B_j$ bimodule.
\end{nota}

From the $M$-$N$ bimodule ${}_MH{}_N$, we can define subfactors $M \subset (N^{\op})'$ and $N^{\op} \subset M'$. 
Let $\vNdim_L(H)$ and $\vNdim_R(H)$ denote 
the left and right von Neumann dimensions of $H$.

\begin{defn}
\label{defn:JonesDimension}
The \emph{Jones dimension} of ${}_MH{}_N$ is
$$
\Delta(H) 
:= 
\sqrt{\vNdim_L(H)\vNdim_R(H)}.
$$
The \emph{Jones dimension matrix} of ${}_AX{}_B$ is
the $a\times b$ matrix $\Delta=\Delta(X)$ whose $(i,j)^{\rm th}$ entry is given by
$
\Delta_{ij}
:=
\Delta(X_{ij})$.
Note that we always have
$\Delta(\overline{H})=\Delta(H)$ and $\Delta(\overline{X})=\Delta(X)^T$.
\end{defn}

\begin{rem}
\label{rem:DeltaIdentity}
When both $\vNdim_L({}_MH)<\infty$
and $\vNdim_R(H_N)<\infty$,
recall from
\cite[(2.1.2) and Prop.~2.1.7]{MR0696688} that
\begin{align*}
[(N^{\op})':M]
&=
\frac{\vNdim_L({}_M H)}{\vNdim_L({}_{(N^{\op})'}H)}
=
\vNdim_L({}_M H)\vNdim_R(H {}_N)
=
\frac{\vNdim_L({}_{N^{\op}} H)}{\vNdim_L({}_{M'}H)}
=
[M':N^{\op}].
\end{align*}
Taking square roots,
$\Delta(X) = [(N^{\op})':M]^{1/2} = [M':N^{\op}]^{1/2}$.
\end{rem}

\begin{defn}
The \emph{statistical dimension} of ${}_MH{}_N$ is
$$
D(H):= \min\set{\sqrt{\operatorname{Ind}(E)}}{E:(N^{\rm op})' \to M\text{ is a conditional expectation}}
$$
where $\operatorname{Ind}(E) \geq 1$ is the Kosaki index of the faithful normal expectation $E$ \cite[Def.~2.1]{MR829381}. For subfactors, this number is finite for every $E$ as long as there exists one expectation with finite index.
We define the \emph{statistical dimension matrix} of ${}_AX{}_B$ to be the $a\times b$ matrix $D(X)$ with $(i,j)^{\rm th}$ entry 
$
D_{ij}
:=
D(X_{ij})
$.
Observe that 
$D(\overline{H}) = D(H)$
and
$D(\overline{X}) = D(X)^T$ \cite[Cor.~5.17]{MR3342166}, \cite[Cor.~6.8]{MR3994584}.
\end{defn}

Note that we always have $D(H) \leq \Delta(H)$, as $\Delta(H)^2$ is the index of the trace-preserving conditional expectation.
Moreover, if $H$ is \emph{simple} (i.e., if $\End_{N-M}(H)=\bbC$), then $D(H)=\Delta(H)$ because an irreducible subfactor admits at most one normal faithful conditional expectation.

\begin{lem}
\label{lem:DualizableBimoduleConditions}
For an $A$-$B$ bimodule ${}_{A}X{}_B$, the following are equivalent:
\begin{enumerate}[label=(\arabic*)]
\item $X$ is dualizable.
\item $X_{ij}$ is dualizable for all $i,j$.
\item Every entry of $D(X)$ is finite.
\item Every entry of $\Delta(X)$ is finite.
\item $X$ is finitely generated as both a left $A$-module and a right $B$-module.
\end{enumerate}
\end{lem}
\begin{proof}\textcolor{white}{.}\\
\noindent
$\underline{(1)\Leftrightarrow (2)}$:
Immediate from the fact that $X$ is the orthogonal direct sum of the $X_{ij}$.\\
$\underline{(2)\Leftrightarrow (3)}$:
By \cite[Prop.~7.3 and Prop.~7.5]{MR3342166}, $X_{ij}$ is dualizable if and only if there exists a conditional expectation $E:(B_j^{\op})' \to A_i$ whose index is finite.\\
$\underline{(3)\Leftrightarrow (4)}$: 
By \cite[Cor.~3.19]{MR945550}, if either $D(X)$ or $\Delta(X)$ is finite, then
the relative commutant $A_i' \cap (B_j^{\op})'$ is finite-dimensional.
It then follows by \cite[Thm. 6.6]{MR549119} that the existence of a conditional expectation $(B_j^{\op})' \to A_i$ of finite index implies that all conditional expectations have finite index.\\
$\underline{(4)\Leftrightarrow (5)}$: 
This follows by \cite[Prop.~9.3.2]{ClaireSorinII_1}, as the conditions (i)-(iii) of that proposition are equivalent for multifactors.
\end{proof}

\begin{rem}
Given a dualizable multifactor bimodule ${}_AX{}_B$ and any unitary dual functor $\vee$ on the unitary multitensor category $\Bim(A\oplus B)$, there is a canonical unitary isomorphism ${}_B(X^\vee){}_A \cong {}_B\overline{X}{}_A$ by \cite[Cor.~6.12]{MR3342166} and \cite[Cor.~B]{MR4133163}.
\end{rem}

\subsection{Connected multifactor inclusions}

We now consider the special case of 
$X = {}_AL^2B{}_B$
where $A\subset B$ is a unital inclusion of finite multifactors.
We call $A\subset B$ \emph{finite index} if ${}_AL^2B{}_B$ is dualizable (cf.~Lemma \ref{lem:DualizableBimoduleConditions}), and
\emph{connected} if 
${}_AL^2B{}_B$ is connected as an object $\Bim(A\oplus B)$ equivalently, if $Z(A) \cap Z(B) = \bbC 1$.

Let $A\subset B$ be a finite index connected inclusion of finite multifactors.
Let $\tr_B$ be a faithful normal trace on $B$, and let $E_A: B \to A$ be the unique trace-preserving conditional expectation.
By \cite[Prop.~3.5.2(ii)]{MR999799}, 
$\tr_A:=\tr_B|_A$
is characterized by the formula
\begin{equation}
\label{eq:TraceMatrixFormula}
\tr_A(p_i) = \sum_{j=1}^b T_{ij} \tr_B(q_j)
\qquad\qquad
\forall\,1\leq i\leq a
\end{equation}
where
$T \in M_{a\times b}(\bbR_{>0})$ is the \emph{trace matrix} \begin{equation}
\label{eq:TraceMatrix}
T_{ij} 
:=
\tr_{B_j}(p_iq_j)
=
\vNdim_R( (p_i(L^2B)q_j){}_{B_j})
\end{equation}
(the second equality in \eqref{eq:TraceMatrix} follows by \cite[Proof of 3.6.7 and 3.2.5(h)]{MR999799}).
We call $(\tr_A(p_i))_{i=1}^a$ and $(\tr_B(q_j))_{j=1}^b$
the \emph{trace vectors}.

The \emph{Jones projection} $e_A\in B(L^2(B,\tr_B))$ is the orthogonal projection onto $L^2(A, \tr_A)$. It satisfies $e_Ab\Omega = E_A(b)\Omega$, where $\Omega \in L^2(A, \tr_A)\subset L^2(B, \tr_B)$ is the image of $1\in A\subset B$.
The \emph{Jones basic construction} \cite[\S3]{MR0696688} is the von Neumman algebra generated by $B$ and by the Jones projection:
$$
\langle B, A\rangle
:=
\langle B, e_A\rangle
=
JA'J
\subset B(L^2(B,\tr_B)).
$$
By setting $A_0:=A$, $A_1:=B$, and $e_1:=e_A$, and inductively $A_n := \langle A_{n-1}, e_{n-1}\rangle = JA_{n-2}'J\subset B(L^2A_{n-1})$, we get the Jones tower
\begin{equation}\label{The Jones tower}
A_0 = A \subset B = A_1 \overset{e_1}{\subset} A_2 \overset{e_2}{\subset} A_3 \subset \cdots
\end{equation}

\begin{defn}
\label{defn:StronglyMarkov}
Let $A\subset B$ be a finite index connected inclusion of finite multifactors.
A faithful normal trace $\tr_B$ on $B$ is called \emph{Markov} if there exists a number $d>0$ and an extension $\tr_{\langle B, A\rangle}$ of $\tr_B$ to $\langle B, A\rangle$ that satisfies
$$
\tr_{\langle B, A\rangle}(xe_A) = d^{-2} \tr_B(x)
$$
for all $x\in B$.
\end{defn}

By \cite[\S2.7 and 3.7]{MR999799}, the Markov trace exists and is unique.
The number $d^2$ is called the \emph{Markov index} of the inclusion.
It is equal to the spectral radius of $\widetilde{T}T$, where $\widetilde{T} \in M_{b\times a}(\bbR_{\geq 0})$ is given by
\begin{equation}
\label{eq:TraceTildeMatrix}
\widetilde{T}_{ji}
:=
\left.\begin{cases}
\Delta_{ij}^2/T_{ij}
&\text{if }p_iq_j \neq 0
\\
0 &\text{if } p_iq_j=0
\end{cases}
\right\}
=
\vNdim_L({}_{A_i}(p_iL^2Bq_j)).
\end{equation}
By \cite[Proof of Thm.~3.7.3]{MR999799}, 
the Markov trace is completely determined by the equation
\begin{equation}
\label{eq:MarkovTraceEigenvalueEquations}
d^2\tr_B(q_j) = \sum_{i=1}^a \widetilde{T}_{ji} \tr_A(p_i).
\end{equation}

By definition, the minimal central projections of $\langle B, A\rangle$ are the $Jp_iJ$ for $i=1,\dots, a$.
By \cite[(3.7.3.1)]{MR999799},
the trace vector $(\tr_{\langle B, A\rangle}(Jp_iJ))_{i=1}^a$ is given by
\begin{equation}
\label{eq:TraceVectorForBasicConstruction}
\tr_{\langle B, A\rangle}(Jp_iJ)=
d^{-2}
\tr_A(p_i)
\sum_{\substack{
1\leq k\leq b
\\
p_iq_k\neq 0
}}
\frac{\Delta_{ik}^2}{T_{ik}}
\end{equation}
and by \cite[Prop.~3.6.8]{MR999799}, the trace matrix $T^{B\subset \langle B, A\rangle}\in M_{b\times a}(\bbR_{\geq 0})$ for the inclusion $B\subset \langle B, A\rangle$ is given by
\begin{equation}
\label{eq:TraceMatrixForBasicConstruction}
T_{ji}^{B\subset \langle B, A\rangle}
=
\frac{
\widetilde{T}_{ji}}
{
\sum_{\substack{
1\leq k\leq b
\\
p_iq_k\neq 0
}}
\frac{\Delta_{ik}^2}{T_{ik}}
}.
\end{equation}

From here on, $\tr_B$ is the unique Markov trace on $B$.
By \cite[Thm.~3.6.4.(i)]{MR999799} (see also \cite{MR860811}), there is a finite \emph{Pimsner-Popa basis} $\{b\}$ for $B$ over $A$ satisfying
$$
\sum_b bE_A(b^*x) =x \qquad\forall\, x\in B
\qquad
\text{and}
\qquad
\sum_{b} bb^* =d^2\in [1,\infty).
$$
This means $E_A$ is of \emph{index finite type} in the sense of \cite{MR996807},
and the Markov index $d^2=\sum_{b} bb^*$ is equal to the \emph{Watatani} index.

Iterating Jones' basic construction, we get a Jones tower
$$
A_0 = A \subset B = A_1 \overset{e_1}{\subset} A_2 \overset{e_2}{\subset} A_3 \subset \cdots
$$
abbreviated $(A_n, \tr_n, e_{n+1})_{n\geq 0}$, 
where each inclusion is strongly Markov.

\begin{example}[{\cite[\S3.2]{MR0696688}}]
An inclusion of finite dimensional von Neumann algebras $A\subset B$ equipped with a faithful trace $\tr_B$ is Markov if and only if the trace vectors $\vec{\lambda}_B$ for $B$ and $\vec{\lambda}_A$ for $A$ (whose $k$-th entry is the trace of a minimal projection in the $k$-th summand) satisfy $\Lambda^T\Lambda \vec{\lambda}_B = d^2 \vec{\lambda}_B$ and $\Lambda \Lambda^T \vec{\lambda}_A = d^2 \vec{\lambda}_A$ 
where $d>0$ such that
$d^2 = \|\Lambda \Lambda^T\|=\|\Lambda^T \Lambda\|$.
Here, $\Lambda$ is the bipartite adjacency matrix for the Bratteli diagram of the inclusion $A\subset B$, whose $(i,j)$-th entry $\Lambda_{i,j}$ is the number of edges between the $i$-th even vertex/simple summand of $A$ and $j$-th odd vertex/simple summand of $B$.
\end{example}

\begin{facts}
\label{facts:StronglyMarkov}
\mbox{}
\begin{itemize}
\item
(Multistep basic construction \cite[Prop.~2.20]{MR2812459})
For every $n\in \bbN$ and $0\leq k \leq n$, the inclusion
$A_{n-k} \subset (A_n,\tr_n) \subset (A_{n+k}, \tr_{n+k},f^n_{n-k})$
is \emph{standard} in the sense of \cite[Def.~2.14]{MR2812459}, i.e., isomorphic to a basic construction.
Here, $f^n_{n-k}$ is proportional to the unique word in the Jones projections
$e_n, \dots, e_{n-k+1}$
of maximal length \cite{MR965748}.
\item
(Conjugation by $J$ \cite[Rem.~2.21]{MR2812459})
Under the multistep basic construction isomorphisms above, on $L^2(A_n,\tr_n)$, $(J_n A_{n-k} J_n)' = A_{n+k}$.
\item
(Centralizer algebras as endomorphisms \cite[Prop.~2.20 Rem.~2.26]{MR2812459})
Using the multistep basic construction together with the bimodule isomorphisms $L^2A_n \cong L^2B^{\boxtimes_A n}$,
we get the following isomorphisms between centralizer and endomorphism algebras:
\begin{equation}
\label{eq:CentralizerAlgebrasAsEndomorphisms}
\begin{split}
A_0'\cap A_{2n}
&\cong
\End_{A-A}(L^2B^{\boxtimes_A n})
\qquad\qquad
A_1'\cap A_{2n+1}
\cong
\End_{B-B}(L^2B^{\boxtimes_A n+1})
\\
A_0'\cap A_{2n+1}
&\cong
\End_{A-B}(L^2B^{\boxtimes_A n+1})
\qquad\hspace{.373cm}
A_1'\cap A_{2n+2}
\cong
\End_{B-A}(L^2B^{\boxtimes_A n+1})
\end{split}
\end{equation}
\end{itemize}
\end{facts}

\subsection{Planar algebras}
\label{sec:PlanarAlgebras}

Suppose we have a
unitary multitensor category $\cC$
together with
a chosen unitary dual functor $\vee$,
a 2-shading $1_\cC = 1^+\oplus 1^-$, 
and a generator $X\in \cC^{+-}$.
Using \cite[\S5]{MR3178106} and \cite[\S4]{MR4133163},
we can construct a 
\emph{unitary 2-shaded planar algebra}, 
i.e.,
a 2-shaded $\rm C^*$-\emph{planar algebra}
\cite[Def.~1.37]{math.QA/9909027}
$\cP_\bullet=\cP(\cC,X,\vee)_\bullet$ with finite dimensional box spaces
\begin{align*}
\cP_{2n,+} &:= \End_\cC((X\otimes \overline{X})^{\otimes n})
&
\cP_{2n,-} &:= \End_\cC((\overline{X}\otimes X)^{\otimes n})
\\
\cP_{2n+1,+} &:= \End_\cC((X\otimes \overline{X})^{\otimes n}\otimes X)
&
\cP_{2n+1,-} &:= \End_\cC((\overline{X}\otimes X)^{\otimes n}\otimes \overline{X}).
\end{align*}
Given such a unitary 2-shaded planar algebra $\cP_\bullet$, one can recover the tuple $(\cC, \vee, X)$ by taking its \emph{category of projections} (see \cite[\S4.1]{MR2559686}, \cite[\S2.3]{MR3405915}, \cite[\S4]{MR4133163}, and \cite[\S3.3]{2004.08271}).
Moreover, these two constructions are mutually inverse.

\begin{thm}
\label{thm:UnitaryPAequivalence}
There is an equivalence of categories\,\footnote{\label{footnote:1Truncated}
The collection of triples $(\cC,\vee,X)$ forms a 2-category which is equivalent to a 1-category \cite[Lem.~3.5]{1607.06041}.} 
\[
\left\{\, 
\parbox{4cm}{\rm unitary 2-shaded planar algebras $\cP_\bullet$
}\,\left\}
\,\,\,\,\cong\,\,
\left\{\,\parbox{8cm}{\rm Triples $(\cC, \vee, X)$ with $\cC$ a unitary multitensor category, $\vee$ a unitary dual functor, 
$1_\cC = 1^+\oplus 1^-$ a 2-shading, and a generator $X \in \cC^{+-}$}\,\right\}.
\right.\right.
\]
\end{thm}

Using the language of strongly Markov inclusions of finite dimensional von Neumann algebras, we now prove that the standard unitary dual functor $\vee_{\rm standard}$ with respect to a chosen generator $X$ as in Definition \ref{defn:StandardDualFunctor} in a 2-shaded unitary multifusion category is the unique unitary dual functor whose loop parameters for $X$ are the same scalar.
As a corollary, if we add scalar loop parameter to the left hand side of the equivalence in Theorem \ref{thm:UnitaryPAequivalence}, we may remove the unitary dual functor from the right hand side as it is uniquely determined by the other data.

\begin{thm}
\label{thm:UniqueUnitaryDualFunctor}
Suppose $\cC$ is a 2-shaded indecomposable unitary multifusion category and $X\in \cC^{+-}$ generates $\cC$.
If $\vee$ is a unitary dual functor on $\cC$ whose loop parameters for $X$ are scalars
$$
\tikzmath{
  \fill[\AColor, rounded corners=5] (-.6, -.6) rectangle(.6, .6);
  \draw[fill=\BColor] (0,0) circle (.3cm);
}
=
\coev_X^*\circ \coev_X
=
\lambda
\id_{1^+}
\qquad\qquad
\tikzmath{
	\fill[\BColor, rounded corners=5] (-.6,-.6) rectangle (.6,.6);
	\draw[fill=\AColor] (0,0) circle (.3cm);
}
=
\ev_X\circ \ev_X^*
=
\lambda
\id_{1^-},
$$
then 
$\lambda =d_X$,
and $\vee$ is unitarily equivalent to the standard unitary dual functor with respect to $X$.
\end{thm}
\begin{proof}
Since $\cC$ is unitary multifusion, generated by $X$, there is an $n\in \bbN$ such that 
every simple object in $\cC^{++}= 1^+ \otimes \cC \otimes 1^+$ and $\cC^{+-}= 1^+\otimes \cC\otimes 1^-$
appears as a summand of
$(X\otimes \overline{X})^{\otimes n}$
or
$(X\otimes \overline{X})^{\otimes n}\otimes X$.
We define the following finite dimensional von Neumann algebras
$$
A=\End_\cC((X\otimes \overline{X})^{\otimes n})
\qquad
\text{and}
\qquad
B=\End_\cC((X\otimes \overline{X})^{\otimes n}\otimes X),
$$
and we observe there is an obvious inclusion map $A\hookrightarrow B$ by $-\otimes \id_X$, under which the inclusion $A\subset B$ is connected.
In fact, this connected inclusion is \emph{independent} of $\vee$!
What is not independent of $\vee$ is the choice of conditional expectation $E:B \to A$ given by
$$
x \mapsto 
\frac{1}{\lambda}\cdot 
\tikzmath{
 \fill[\AColor, rounded corners = 5] (-.7,-.8) rectangle (1,.8) ;
 \draw[very thick] (-.15,-.8) node [below] {$\scriptstyle 2n$} -- (-.15,.8) node [above] {$\scriptstyle 2n$};
 \filldraw[fill=\BColor] (.15,.3) arc (180:0:.3cm) -- (.75,-.3) arc (0:-180:.3cm);
 \roundNbox{unshaded}{(0,0)}{.3}{.15}{.15}{$x$}
}\,.
$$

By \cite[Thm.~D]{MR4133163}, there exists a state $\psi$ on $\End_\cC(1_\cC)$ such that for every $c\in \cC$ and $f:c\to c$,
$$
\psi\left(\,
\tikzmath{
 \draw (0,.3) arc (0:180:.3cm) -- (-.6,-.3) arc (-180:0:.3cm);
 \roundNbox{unshaded}{(0,0)}{.3}{0}{0}{$f$}
}
\right)
=
\psi\left(
\tikzmath[xscale=-1]{
 \draw (0,.3) arc (0:180:.3cm) -- (-.6,-.3) arc (-180:0:.3cm);
 \roundNbox{unshaded}{(0,0)}{.3}{0}{0}{$f$}
}\,
\right).
$$
(Observe that these pictures are not shaded as \cite[Thm.~D]{MR4133163} only applies to the unshaded case!)
Consider the tracial states on $A$ and $B$ given by
$$
\tr_A :=
\frac{1}{\lambda^{2n}\cdot \psi(\id_{1^+})}\cdot(\psi \circ \tr_R^\vee)
\qquad\qquad
\tr_B :=
\frac{1}{\lambda^{2n+1}\cdot \psi(\id_{1^+})}\cdot(\psi \circ \tr_R^\vee)
$$
and observe that $E: B \to A$ is the unique trace-preserving conditional expectation.

Since $\cC$ is multifusion, by the Recognition Lemma \cite[Lem.~5.3.1]{MR1473221} for the basic construction in finite dimensions, the von Neumann algebra
$$
C := \End_\cC((X\otimes \overline{X})^{\otimes n+1})
\qquad\qquad
e_A := \frac{1}{\lambda}\cdot
 \tikzmath{
  \fill[\AColor, rounded corners = 5] (-.4,-.5) rectangle (1.2,.5);
  \draw[very thick] (0,-.5) -- node[midway,below,rotate=270] {\tiny{$2n$}} (0,.5);
  \filldraw[fill= \BColor] (.3,.5) arc (-180:0:.3cm);
  \filldraw[fill= \BColor] (.3,-.5) arc (180:0:.3cm);
 }
$$
is isomorphic to the basic construction algebra $\langle A, B\rangle$ with Jones projection $e_A$, where the inclusion is given by $y\mapsto y\otimes \id_{\overline{X}}$ and conditional expectation
$F: C \to B$ given by
$$
y \mapsto 
\frac{1}{\lambda}\cdot 
\tikzmath{
 \fill[\AColor,rounded corners = 5] (-.7,-.8) rectangle (1,.8) ;
 \begin{scope}
   \clip[rounded corners = 5] (-.7,-.8) rectangle (1,.8) ;
   \fill[\BColor] (-.15,-.8) rectangle (1,.8) ;
 \end{scope}
 \draw[very thick] (-.15,-.8) node [below] {$\scriptstyle 2n+1$} -- (-.15,.8) node [above] {$\scriptstyle 2n+1$};
 \filldraw[fill=\AColor] (.15,.3) arc (180:0:.3cm) -- (.75,-.3) arc (0:-180:.3cm);
 \roundNbox{unshaded}{(0,0)}{.3}{.15}{.15}{$y$}
}\,.
$$
Moreover, there exists a trace on $C$ given by
$$
\tr_C:=
\frac{1}{\lambda^{2n+2}\cdot \psi(\id_{1^+})}\cdot(\psi \circ \tr_R^\vee).
$$
Since $\tr_C|_B = \tr_B$ and $F(e_A) = \lambda^{-1}$, we see that $\tr_C$ is a $(\lambda, B)$-trace on $C$ \cite[Def.~on p.8]{MR0696688}.
By \cite[Thm.~3.3.2]{MR0696688}, we have that $A\subset (B,\tr_B)$ is a connected Markov inclusion.
Thus $\tr_B$ is the unique Markov trace for the inclusion $A\subseteq B$.
From the existence of the Giorgetti-Longo unitary dual functor and induced Markov trace from \cite{MR3994584}, we immediately have that $\lambda = d_X$ and 
\begin{align*}
\alpha_i^2
=
\tr_A(p_i \otimes \id_{(X\otimes \overline{X})^{\otimes n}}) 
=
\frac{\psi(p_i)}{\psi(\id_{1^+})}
\qquad\qquad
\forall\, i=1,\dots,a.
\\
\beta_j^2
=
\tr_B(\id_{(X\otimes \overline{X})^{\otimes n}\otimes X}\otimes q_j) 
=
\frac{\psi(q_j)}{\psi(\id_{1^+})}
\qquad\qquad
\forall\, i=j,\dots, b.
\end{align*}

Finally, since $\cC$ is indecomposable multifusion, the relevant classfying grading groupoid for unitary dual functors is the matrix unit groupoid $\cG_{a+b}$.
We have that for all $i,j$, the classifying groupoid homomorphism
$\pi:\cG_{a+b}\to \bbR_{>0}$ is given by
$$
\pi(e_{ij})
=
\frac{\dim^\vee_L(X_{ij})}{\dim^\vee_R(X_{ij})}
=
\frac{
\psi\left(\,
\tikzmath[xscale=-1]{
  \fill[\BColor, rounded corners=5] (-.5,-.75) rectangle (1.75,.75);
  \filldraw[fill=\AColor] (1,0) circle (.5cm);
  \roundNbox{unshaded}{(0,0)}{.25}{0}{0}{$q_j$}
  \roundNbox{unshaded}{(1,0)}{.25}{0}{0}{$p_i$}
}
\,\right) / \psi(q_j)
}{
\psi\left(\,
\tikzmath{
  \fill[\AColor, rounded corners=5] (-.5,-.75) rectangle (1.75,.75);
  \filldraw[fill=\BColor] (1,0) circle (.5cm);
  \roundNbox{unshaded}{(0,0)}{.25}{0}{0}{$p_i$}
  \roundNbox{unshaded}{(1,0)}{.25}{0}{0}{$q_j$}
}
\,\right)/ \psi(p_i)
}
=
\frac{\alpha_i^2}{\beta_j^2}
$$
which is exactly the formula \eqref{eq:UnitaryDualFunctorFormula}.
Thus, $\vee$ is exactly the Giorgetti-Longo standard unitary dual functor by Theorem \ref{thm:ClassificationOfUnitaryDualFunctors}.
\end{proof}

\begin{cor}
\label{cor:UnitaryPAequivalence}
There is an equivalence of categories (see Footnote \ref{footnote:1Truncated})
\[
\left\{\, 
\parbox{6cm}{\rm Indecomposable finite depth unitary 2-shaded planar algebras $\cP_\bullet$ with
equal scalar loop moduli}\,\left\}
\,\,\,\,\cong\,\,
\left\{\,\parbox{7cm}{\rm Pairs $(\cC, X)$ with $\cC$ an indecomposable unitary multifusion category,
$1_\cC = 1^+\oplus 1^-$ a 2-shading, and a generator $X \in \cC^{+-}$}\,\right\}.
\right.\right.
\]
\end{cor}

\begin{warn}
While the unitary 2-shaded planar algebra $\cP_\bullet$ has scalar loop moduli which are the same for both shadings, $\cP_\bullet$ is not necessarily spherical.
Indeed, there can be endomorphisms of $X$ which have distinct left and right traces with respect to the standard unitary dual functor.
A typical example of this behavior is the planar algebra of a bipartite graph \cite{MR1865703}.
\end{warn}

\subsection{Standard invariants}
\label{sec:StandardInvariants}

\begin{defn}
Given multifactors $A,B$ and a dualizable bimodule ${}_AX{}_B$, the \emph{standard invariant} of ${}_AX{}_B$ is the abstract 2-shaded unitary multitensor category $\cC(X) \subset \Bim(A\oplus B)$ generated by $X$ together with the choice of generating object $X$
and the standard unitary dual functor with respect to $X$, whose loop moduli are both the scalar $d_X$ as in Definition \ref{defn:StandardDualFunctor}.
Here, we forget the existence of the forgetful fiber functor $\cC(X) \to \Bim(A\oplus B)$.
Observe that $\cC$ is 2-shaded with $1^+ = \bim A{L^2A}A$ and $1^-=\bim B{L^2B}B$.
We say that ${}_AX{}_B$ has \emph{finite depth} if $\cC(X)$ is a unitary multifusion category.

An \emph{equivalence} of standard invariants $(F,u):\cC(X)\to \cD(Y)$ is a unitary tensor equivalence $F: \cC \to \cD$ together with a unitary isomorphism $u\in \cD(Y \to F(X))$.

The \emph{standard invariant} of a
finite index
connected
inclusion $A\subset B$
of finite multifactors
is the standard invariant $\cC({}_AL^2B{}_B)$.
We say that the inclusion $A\subset B$ has \emph{finite depth} if ${}_AL^2B{}_B$ has finite depth, i.e., $\cC({}_AL^2B{}_B)$ is multifusion.
\end{defn}

By 
Theorem \ref{thm:UnitaryPAequivalence}, the standard invariant $\cC(X)$ of ${}_AX{}_B$ 
may also be viewed as a unitary 2-shaded planar algebra $\cP(X)_\bullet$ 
with equal scalar loop moduli.

Suppose now $A\subset (B, \tr_B)$
is a 
finite index
connected
inclusion of finite multifactors 
equipped with the unique Markov trace.
In addition to the unitary 2-shaded planar algebra $\cP({}_AL^2B{}_B)$, 
one can construct another unitary 2-shaded planar algebra 
$\cP^{A\subset B}_\bullet$
using the construction 
from \cite[\S3]{MR2812459}
whose box spaces are the higher centralizer algebras of the Jones tower:
$$
\cP^{A\subset B}_{n,+} := A_0' \cap A_{n}
\qquad\qquad
\cP^{A\subset B}_{n,-} := A_1' \cap A_{n+1}
\qquad
\qquad
\forall\, n \geq 0.
$$
In fact, the
two unitary planar algebras $\cP({}_AL^2B{}_B)_\bullet$ and $\cP^{A\subset B}_\bullet$
are
$*$-isomorphic by \cite[Pf.~of~Thm.~5.4 and Rem.~5.5]{MR3178106}, which do not rely in any substantial way on factoriality and can be adapted to our situation using 
\eqref{eq:CentralizerAlgebrasAsEndomorphisms}
from
Facts \ref{facts:StronglyMarkov} in place of the results 
from \cite{MR1424954}.
Hence we have a commutative diagram
\begin{equation}
\label{eq: Subfactor--planarAlg--TensorCat}
\begin{tikzcd}
\left\{
\parbox{5cm}{\rm 
Finite index connected
multifactor
inclusions $A\subset B_{\phantom{\bullet}}$
}
\right\}
\arrow[r, "\text{\scriptsize \cite{MR2812459}}"]
\arrow[d]
&
\left\{
\parbox{7cm}{\rm 
Indecomposable unitary
2-shaded planar algebras $\cP_\bullet$
with equal scalar loop moduli
}
\right\}
\\
\left\{
\parbox{5cm}{\rm 
Dualizable connected multifactor bimodules 
${}_AX{}_B$
}
\right\}
\arrow[r]
\arrow[ur,"\text{\scriptsize \cite{MR3178106}}"]
&
\left\{
\parbox{7.5cm}{\rm 
Triples $(\cC,X,\vee)$ with $\cC$ indecomposable such that $X\in \cC^{+-}$ has equal scalar loop values
}
\right\}
\arrow[u, leftrightarrow, "\text{\scriptsize 
Thm.~\ref{thm:UnitaryPAequivalence}}","\cong"']
\end{tikzcd}
\end{equation}

By \cite[Thm.~3.1]{MR1334479}, as explained in \cite[Pf.~of~Thm.~4.3.1]{math.QA/9909027}, 
given a spherical unitary 2-shaded planar algebra $\cP_\bullet$ which is \emph{connected}  
($\cP_{0,\pm}$ are 1-dimensional), there is a $\rm II_1$ subfactor $A\subset B$ which is \emph{extremal}
(the traces $\tr_{A'}$ and $\tr_B$ agree on $A'\cap B$)
whose standard invariant is $*$-isomorphic to $\cP_\bullet$.
If moreover $\cP_\bullet$ is finite depth, $A$ and $B$ can be taken to be hyperfinite.

Hence in the case of connected spherical unitary 2-shaded planar algebras, the top horizontal map in \eqref{eq: Subfactor--planarAlg--TensorCat} above is surjective.
We conjecture that this map is always surjective.

\begin{conj}
\label{conj:InclusionsToPlanarAlgberasIsSurjective}
Given a unitary 2-shaded planar algebra $\cP_\bullet$, there is a 
finite index 
homogeneous
connected 
$\rm II_1$ multifactor inclusion 
$A\subset B$ whose standard invariant is $*$-isomorphic to $\cP_\bullet$.
\end{conj}

We discuss \emph{homogeneous} inclusions in \S\ref{sec:Homogeneity} below cf.~\cite[Def.1.2.11]{MR1339767}. 
In \S\ref{sec:ConstructionOfInclusions} below, we prove Conjecture \ref{conj:InclusionsToPlanarAlgberasIsSurjective} for finite depth unitary 2-shaded planar algebras, in which case $A, B$ can be taken to be hyperfinite.

The commutative diagram \eqref{eq: Subfactor--planarAlg--TensorCat} is functorial for isomorphisms; this statement must be interpreted with the subtlety that the top line in \eqref{eq: Subfactor--planarAlg--TensorCat} consists of two 1-groupoids, while the second line consists of two 2-groupoids
(the second is 1-truncated and equivalent to a 1-groupoid by Theorem \ref{thm:UnitaryPAequivalence}, see also \cite[Lem.~3.5]{1607.06041}).
We now discuss the functoriality of these arrows in more detail.

An isomorphism of two finite index multifactactor inclusions $A\subset B$ and $\widetilde{A}\subset \widetilde{B}$ is a $*$-isomorphism $\varphi: B \to \widetilde{B}$ taking $A$ onto $\widetilde{A}$.
Given such an isomorphism $\varphi$ of inclusions, there is a unique extension of $\varphi$ taking the Jones tower of $A\subset B$ onto the Jones tower of $\widetilde{A}\subset \widetilde{B}$ preserving the Jones projections, i.e., $e^{A\subset B}_k \mapsto e^{\widetilde{A}\subset \widetilde{B}}_k$ for all $k$.
This gives us an induced planar algebra isomorphism
$\cP(\varphi)_\bullet:\cP^{A\subset B}_\bullet\to \cP^{\widetilde{A}\subset \widetilde{B}}_\bullet$
on the centralizer algebras.

To describe 1-isomorphisms between dualizable bimodules, we must introduce the notion of Morita equivalence.

\begin{defn}
\label{defn:MoritaEquivalence}
A \emph{Morita equivalence} between von Neumann algebras $M$ and $N$ is an invertible bimodule ${}_MH{}_N$, i.e., an $M-N$ bimodule $H$ equipped with unitary two bimodule isomorphisms ${}_MH\boxtimes_N \overline{H}{}_M\cong {}_ML^2M{}_M$
and
${}_N\overline{H}\boxtimes_M H{}_N\cong {}_NL^2N{}_N$
which satisfy the zig-zag relations (\ref{eq:ShadedZigZag1}, \ref{eq:ShadedZigZag2}).
The unitarity of these cups and caps is exactly the \emph{recabling relations}:
\begin{equation}
\label{eq:Recabling}
\begin{split}
\tikzmath{
\begin{scope}
\clip[rounded corners = 5] (-.5,-.5) rectangle (.5,.5);
\filldraw[\AColor] (-.5,-.5) rectangle (.5,.5);
\filldraw[\BColor] (-.25,-.5) arc (180:0:.25cm);
\filldraw[\BColor] (-.25,.5) arc (-180:0:.25cm);
\end{scope}
\draw (-.25,-.5) node[below] {$\scriptstyle \overline{H}$} arc (180:0:.25cm) node[below] {$\scriptstyle H$};
\draw (-.25,.5) node[above] {$\scriptstyle \overline{H}$} arc (-180:0:.25cm) node[above] {$\scriptstyle H$};
}
&=
\tikzmath{
\begin{scope}
\clip[rounded corners = 5] (-.5,-.5) rectangle (.5,.5);
\filldraw[\AColor] (-.5,-.5) rectangle (-.25,.5);
\filldraw[\BColor] (-.25,-.5) rectangle (.25,.5);
\filldraw[\AColor] (.25,-.5) rectangle (.5,.5);
\end{scope}
\draw (-.25,-.5) node[below] {$\scriptstyle \overline{H}$} -- (-.25,.5) node[above] {$\scriptstyle \overline{H}$};
\draw (.25,-.5) node[below] {$\scriptstyle H$} -- (.25,.5) node[above] {$\scriptstyle H$};
}
\qquad\qquad
\tikzmath{
\begin{scope}
\clip[rounded corners = 5] (-.5,-.5) rectangle (.5,.5);
\filldraw[\AColor] (-.5,-.5) rectangle (.5,.5);
\end{scope}
\filldraw[fill=\BColor] (0,0) circle (.25cm);
}
=
\tikzmath{
\begin{scope}
\clip[rounded corners = 5] (-.5,-.5) rectangle (.5,.5);
\fill[\AColor] (-.5,-.5) rectangle (.5,.5);
\end{scope}
}
=
\id_{L^2N}
\\
\tikzmath{
\begin{scope}
\clip[rounded corners = 5] (-.5,-.5) rectangle (.5,.5);
\filldraw[\BColor] (-.5,-.5) rectangle (.5,.5);
\filldraw[\AColor] (-.25,-.5) arc (180:0:.25cm);
\filldraw[\AColor] (-.25,.5) arc (-180:0:.25cm);
\end{scope}
\draw (-.25,-.5) node[below] {$\scriptstyle H$} arc (180:0:.25cm) node[below] {$\scriptstyle \overline{H}$};
\draw (-.25,.5) node[above] {$\scriptstyle H$} arc (-180:0:.25cm) node[above] {$\scriptstyle \overline{H}$};
}
&=
\tikzmath{
\begin{scope}
\clip[rounded corners = 5] (-.5,-.5) rectangle (.5,.5);
\filldraw[\BColor] (-.5,-.5) rectangle (-.25,.5);
\filldraw[\AColor] (-.25,-.5) rectangle (.25,.5);
\filldraw[\BColor] (.25,-.5) rectangle (.5,.5);
\end{scope}
\draw (-.25,-.5) node[below] {$\scriptstyle H$} -- (-.25,.5) node[above] {$\scriptstyle H$};
\draw (.25,-.5) node[below] {$\scriptstyle \overline{H}$} -- (.25,.5) node[above] {$\scriptstyle \overline{H}$};
}
\qquad\qquad
\tikzmath{
\begin{scope}
\clip[rounded corners = 5] (-.5,-.5) rectangle (.5,.5);
\filldraw[\BColor] (-.5,-.5) rectangle (.5,.5);
\end{scope}
\filldraw[fill=\AColor] (0,0) circle (.25cm);
}
=
\tikzmath{
\begin{scope}
\clip[rounded corners = 5] (-.5,-.5) rectangle (.5,.5);
\fill[\BColor] (-.5,-.5) rectangle (.5,.5);
\end{scope}
}
=
\id_{L^2M}
\end{split}
\end{equation}
By
\cite[Prop.\,3.1]{MR799587},
these unitary solutions to the conjugate equations are unique up to unique $M-N$ bilinear unitary isomorphism of ${}_MH{}_N$.

As an example, given a von Neumann algebra $N$, any faithful right $N$-module $H_N$ 
is canonically an $N'-N$ Morita equivalence bimodule by \cite[Prop.\,3.1]{MR799587} where $N'$ is the commutant of the right $N$-action.
When $A$ is a finite multifactor, 
such faithful right $A$-modules $Y_A$ are parametrized 
up to right $A$-linear unitary isomorphism 
by $\bbR_{>0}^a$
via the map $Y_A \mapsto (\vNdim_R(Yp_i {}_{A_i}))_{i=1}^a$.

Suppose now $N\subset M$ is a unital inclusion of von Neumann algebras.
A faithful right $N$-module $H{}_{N}$ induces a faithful right $M$-module $K{}_M:= H\boxtimes_N L^2M{}_M$.
Setting $N':=(N^{\op})'\cap B(H)$ and $M':= (M^{\op})'\cap B(K)$,
the induced left $N'$-action on $K$ commutes with the right $M$-action and is thus contained in $M'$.
The inclusion $N'\subset M'$ is called the \emph{Morita equivalent inclusion induced by} $H_{N}$.
(We warn the reader that the commutants $N'$ and $M'$ are taken in different representations, and thus $M'$ is \emph{not} contained in $N'$ even though $N\subset M$.)
\end{defn}

A 1-isomorphism between two dualizable bimodules ${}_AX^1{}_B$ and ${}_{A'} X^2 {}_{B'}$ consists of a triple $({}_{A'}Y{}_{A}, {}_{B'}Z{}_{B}, \psi)$ where ${}_{A'}Y{}_{A}$ and ${}_{B'}Z{}_{B}$ are Morita equivalence bimodules (which always come equipped with two distiguished unitary isomorphisms satisfying the zig-zag axioms) and an $A'-B$ bilinear unitary isomorphism 
$\psi : {}_{A'}Y\boxtimes_A X^1 {}_B \to {}_{A'}X^2 \boxtimes_{B'} Z{}_B$.
A 2-isomorphism between triples $({}_{A'}Y^1{}_{A}, {}_{B'}Z^1{}_{B}, \psi_1)$ and $({}_{A'}Y^2{}_{A}, {}_{B'}Z^2{}_{B}, \psi_2)$ consists of 
an $A'-A$ bilinear unitary 
$u: {}_{A'}Y^1{}_{A} \to {}_{A'}Y^2{}_{A}$
and a $B'-B$ bilinear unitary 
$v: {}_{B'}Z^1{}_{B} \to {}_{B'}Z^2{}_{B}$
satisfying the relation
\begin{equation}
\label{eq:Morita2IsoCompatibility}
\psi_2 \circ (u\boxtimes \id_{X^1}) = (v \boxtimes \id_{X^2}) \circ \psi_1
\qquad\qquad
\tikzmath{
\begin{scope}
\clip[rounded corners = 5] (-.9,-1.6) rectangle (.9,.6);
\filldraw[white] (-.3,-1.6) rectangle (.3,0);
\filldraw[\APrimeColor] (-.3,-1.6) rectangle (-.9,.6);
\filldraw[\BColor] (.3,-1.6) rectangle (.9,.6);
\filldraw[\BPrimeColor] (-.3,0) rectangle (.3,.6);
\end{scope}
\draw[thick, violet] (-.3,-1.6) node[below] {$\scriptstyle Y^1$} -- (-.3,-1);
\draw[thick, blue] (-.3,-1) -- node[left] {$\scriptstyle Y^2$} (-.3,0); 
\draw (-.3,0) -- (-.3,.6) node[above] {$\scriptstyle X^2$};
\draw (.3,-1.6) node[below] {$\scriptstyle X^1$} -- (.3,0);
\draw[thick, red] (.3,0) --  (.3,.6) node[above] {$\scriptstyle Z^2$};
\roundNbox{unshaded}{(0,0)}{.3}{.3}{.3}{$\psi_2$}
\roundNbox{unshaded}{(-.3,-1)}{.3}{0}{0}{$u$}
}
=
\tikzmath{
\begin{scope}
\clip[rounded corners = 5] (-.9,-.6) rectangle (.9,1.6);
\filldraw[white] (-.3,-.6) rectangle (.3,0);
\filldraw[\APrimeColor] (-.3,-.6) rectangle (-.9,1.6);
\filldraw[\BColor] (.3,-.6) rectangle (.9,1.6);
\filldraw[\BPrimeColor] (-.3,0) rectangle (.3,1.6);
\end{scope}
\draw[thick, violet] (-.3,-.6) node[below] {$\scriptstyle Y^1$} -- (-.3,0);
\draw (-.3,0) -- (-.3,1.6) node[above] {$\scriptstyle X^2$};
\draw (.3,-.6) node[below] {$\scriptstyle X^1$} -- (.3,0);
\draw[thick, orange] (.3,0) -- node[right] {$\scriptstyle Z^1$} (.3,1); 
\draw[thick, red] (.3,1) --  (.3,1.6) node[above] {$\scriptstyle Z^2$};
\roundNbox{unshaded}{(0,0)}{.3}{.3}{.3}{$\psi_1$}
\roundNbox{unshaded}{(.3,1)}{.3}{0}{0}{$v$}
}
\end{equation}
where we denote the four von Neumann algebras $A,B,A',B'$ by the shaded regions
$$
\tikzmath{\draw[fill=white, rounded corners=5, thin, dotted, baseline=1cm] (0,0) rectangle (.5,.5);}=A
\qquad
\tikzmath{\filldraw[\BColor, rounded corners=5, very thin, baseline=1cm] (0,0) rectangle (.5,.5);}=B
\qquad
\tikzmath{\filldraw[\APrimeColor, rounded corners=5, very thin, baseline=1cm] (0,0) rectangle (.5,.5);}=A'
\qquad
\tikzmath{\filldraw[\BPrimeColor, rounded corners=5, very thin, baseline=1cm] (0,0) rectangle (.5,.5);}=B'.
$$
By uniqueness of the unitary solutions to the zig-zag equations for $Y^1,Y^2$ and $Z^1,Z^2$ respectively, $u,v$ will automatically satisfy the relations
\begin{equation}
\label{eq:Morita2IsoCompatibility-Automatic}
\tikzmath{
\begin{scope}
   \clip[rounded corners = 5] (-.6,-.6) rectangle (.6,.6);
\filldraw[white] (0,-.6) rectangle (.6,.6);
\filldraw[\APrimeColor] (0,-.6) rectangle (-.6,.6);
\end{scope}
\draw[thick, blue] (0,-.6) node[below] {$\scriptstyle \overline{Y^2}$} -- (0,0);
\draw[thick, violet] (0,0) -- (0,.6) node[above] {$\scriptstyle \overline{Y^1}$};
\roundNbox{unshaded}{(0,0)}{.3}{0}{0}{$\overline{u}^*$}
}
=
\tikzmath{
\begin{scope}
   \clip[rounded corners = 5] (-.8,-.8) rectangle (.8,.8);
\filldraw[\AColor] (0,.3) arc (0:180:.3cm) -- (-.6,-.8) -- (-.8,-.8) -- (-.8,.8) -- (.6,.8) -- (.6,-.3) arc (0:-180:.3cm);
\filldraw[white] (0,.3) arc (0:180:.3cm) -- (-.6,-.8) -- (.8,-.8) -- (.8,.8) -- (.6,.8) -- (.6,-.3) arc (0:-180:.3cm);
\end{scope}
\draw[thick, blue] (0,.3) node[right, yshift=.2cm] {$\scriptstyle Y^2$} arc (0:180:.3cm) -- (-.6,-.8) node[below] {$\scriptstyle \overline{Y^2}$};
\draw[thick, violet] (0,-.3) node[left, yshift=-.2cm] {$\scriptstyle Y^1$} arc (-180:0:.3cm) -- (.6,.8) node[above] {$\scriptstyle \overline{Y^1}$};
\roundNbox{unshaded}{(0,0)}{.3}{0}{0}{$u$}
}
\qquad\qquad
\tikzmath{
\begin{scope}
   \clip[rounded corners = 5] (-.6,-.6) rectangle (.6,.6);
\filldraw[\BColor] (0,-.6) rectangle (.6,.6);
\filldraw[\BPrimeColor] (0,-.6) rectangle (-.6,.6);
\end{scope}
\draw[thick, red] (0,-.6) node[below] {$\scriptstyle \overline{Z^2}$} -- (0,0);
\draw[thick, orange] (0,0) -- (0,.6) node[above] {$\scriptstyle \overline{Z^1}$};
\roundNbox{unshaded}{(0,0)}{.3}{0}{0}{$\overline{v}^*$}
}
=
\tikzmath{
\begin{scope}
   \clip[rounded corners = 5] (-.8,-.8) rectangle (.8,.8);
\filldraw[\BPrimeColor] (0,.3) arc (0:180:.3cm) -- (-.6,-.8) -- (-.8,-.8) -- (-.8,.8) -- (.6,.8) -- (.6,-.3) arc (0:-180:.3cm);
\filldraw[\BColor] (0,.3) arc (0:180:.3cm) -- (-.6,-.8) -- (.8,-.8) -- (.8,.8) -- (.6,.8) -- (.6,-.3) arc (0:-180:.3cm);
\end{scope}
\draw[thick, red] (0,.3) node[right, yshift=.2cm] {$\scriptstyle Z^2$} arc (0:180:.3cm) -- (-.6,-.8) node[below] {$\scriptstyle \overline{Z^2}$};
\draw[thick, orange] (0,-.3) node[left, yshift=-.2cm] {$\scriptstyle Z^1$} arc (-180:0:.3cm) -- (.6,.8) node[above] {$\scriptstyle \overline{Z^1}$};
\roundNbox{unshaded}{(0,0)}{.3}{0}{0}{$v$}
}\,.
\end{equation}
We leave it to the reader to define composition of 1-morphisms and the associator isomorphisms in this 2-groupoid.

Suppose now that $\varphi: (A\subset B) \to (\widetilde{A} \subset \widetilde{B})$ is an isomorphism of finite index connected inclusions of finite multifactors.
Denote these four von Neumann algebras $A,B,\widetilde{A}, \widetilde{B}$ by the shaded regions
$$
\tikzmath{\draw[fill=white, rounded corners=5, thin, dotted, baseline=1cm] (0,0) rectangle (.5,.5);}=A
\qquad
\tikzmath{\filldraw[\BColor, rounded corners=5, very thin, baseline=1cm] (0,0) rectangle (.5,.5);}=B
\qquad
\tikzmath{\filldraw[\APrimeColor, rounded corners=5, very thin, baseline=1cm] (0,0) rectangle (.5,.5);}=\widetilde{A}
\qquad
\tikzmath{\filldraw[\BPrimeColor, rounded corners=5, very thin, baseline=1cm] (0,0) rectangle (.5,.5);}=\widetilde{B},
$$
and the standard and Morita equivalence bimodules by
$$
\tikzmath{
\begin{scope}
\clip[rounded corners=5pt] (-.4,-.4) rectangle (.4,.4);
\fill[\BColor] (0,-.4) rectangle (.4,.4);
\fill[white] (-.4,-.4) rectangle (0,.4);
\end{scope}
\draw[dotted,rounded corners=5pt] (0,-.4) -- (-.4,-.4) -- (-.4,.4) -- (0,.4);
\draw (0,-.4) -- (0,.4);
}
=
{}_AL^2B{}_B
\qquad
\tikzmath{
\begin{scope}
\clip[rounded corners=5pt] (-.4,-.4) rectangle (.4,.4);
\fill[\ATildeColor] (-.4,-.4) rectangle (0,.4);
\fill[\BTildeColor] (.4,-.4) rectangle (0,.4);
\end{scope}
\draw (0,-.4) -- (0,.4);
}
=
{}_{\widetilde{A}}L^2\widetilde{B}{}_{\widetilde{B}}
\qquad
\tikzmath{
\begin{scope}
\clip[rounded corners=5pt] (-.4,-.4) rectangle (.4,.4);
\fill[\ATildeColor] (0,-.4) rectangle (-.4,.4);
\fill[white] (0,-.4) rectangle (.4,.4);
\end{scope}
\draw[dotted,rounded corners=5pt] (0,-.4) -- (.4,-.4) -- (.4,.4) -- (0,.4);
\draw[thick, blue, dashed] (0,-.4) -- (0,.4);
}
={}_{\widetilde{A}}L^2\widetilde{A}_{\varphi(A)}
\qquad
\tikzmath{
\begin{scope}
\clip[rounded corners=5pt] (-.4,-.4) rectangle (.4,.4);
\fill[\BColor] (.4,-.4) rectangle (0,.4);
\fill[\BTildeColor] (-.4,-.4) rectangle (0,.4);
\end{scope}
\draw[thick, red, dashed] (0,-.4) -- (0,.4);
}
={}_{\widetilde{B}}L^2\widetilde{B}_{\varphi(B)}
$$
We denote the conjugate bimodules by the horizontal reflection, and the restriction to $A,\widetilde{A}$ respectively by changing the shading.
The map $L^2\varphi :{}_B L^2B {}_B \to {}_{\varphi(B)} L^2\widetilde{B}{}_{\varphi(B)}$ 
is the isomorphism
$x\sqrt{\tr_B} \mapsto \varphi(x) \sqrt{\tr_{\widetilde{B}}}$
where $\tr_B,\tr_{\widetilde{B}}$ are the unique Markov traces respectively.
We may and do view $L^2\varphi$ as the canonical $B-B$ bimodule isomorphism $L^2B \to {}_\varphi L^2\widetilde{B} \boxtimes_{\widetilde{B}} L^2\widetilde{B}{}_\varphi$ 
from \cite[Prop.~3.1]{MR703809}
denoted by a cup:
$$
\tikzmath{
\fill[\BColor, rounded corners=5pt] (-.5,-.5) rectangle (.5,.5);
\fill[\BTildeColor] (-.3,.5) -- (-.3,.1) arc (-180:0:.3cm) -- (.3,.5);
\draw[thick, red, dashed] (-.3,.5) node[above, xshift=-.2cm] {$\scriptstyle {}_\varphi L^2\widetilde{B}$} -- (-.3,.1) arc (-180:0:.3cm) -- (.3,.5) node[above, xshift=.2cm] {$\scriptstyle L^2\widetilde{B}{}_\varphi $};
\node at (0,-.7) {$\scriptstyle L^2B$};
}
:=
\tikzmath{
\fill[\BColor, rounded corners=5pt] (-.7,-.7) rectangle (.7,.7);
\fill[\BTildeColor] (-.3,.3) rectangle (.3,.7);
\draw[thick, red, dashed] (-.3,0) -- (-.3,.7) node[above, xshift=-.2cm] {$\scriptstyle {}_\varphi L^2\widetilde{B}$};
\draw[thick, red, dashed] (.3,0) -- (.3,.7) node[above, xshift=.2cm] {$\scriptstyle L^2\widetilde{B}{}_\varphi $};
\roundNbox{unshaded}{(0,0)}{.3}{.2}{.2}{$L^2\varphi$}
\node at (0,-.9) {$\scriptstyle L^2B$};
}
\,.
$$
We can restrict $L^2\varphi$ to an $A-B$ bimodule map still denoted $L^2\varphi$ 
by restricting the left $A$-action using the canonical isomorphism ${}_AL^2B{}_B \boxtimes_B {}_\varphi L^2\widetilde{B}{}_B \cong {}_{\varphi(A)} L^2\widetilde{B}{}_B$, which we denote by a trivalent vertex:
$$
L^2\varphi
=
\tikzmath{
\begin{scope}
\clip[rounded corners=5pt] (-.7,-.5) rectangle (.5,.5);
\fill[white] (-.7,-.5) -- (-.5,-.5) -- (-.5,0) arc (180:90:.2cm) -- (-.3,.5) -- (-.7,.5);
\fill[\BColor] (-.5,-.5) -- (-.5,0) arc (180:90:.2cm) -- (-.3,.1) arc (-180:0:.3cm) -- (-.3,.5) -- (.5,.5) -- (.5,.-.5);
\fill[\BTildeColor] (-.3,.5) -- (-.3,.1) arc (-180:0:.3cm) -- (.3,.5);
\end{scope}
\draw[thick, red, dashed] (-.3,.5) node[above, xshift=-.2cm] {$\scriptstyle {}_\varphi L^2\widetilde{B}$} -- (-.3,.1) arc (-180:0:.3cm) -- (.3,.5) node[above, xshift=.2cm] {$\scriptstyle L^2\widetilde{B}{}_\varphi $};
\draw (-.5,-.5 ) node[below] {$\scriptstyle L^2B$} -- (-.5,0) arc (180:90:.2cm);
\filldraw (-.3,.2) circle (.05cm);
}
:
{}_A L^2B{}_B \to {}_{\varphi(A)}L^2\widetilde{B}\boxtimes_{\widetilde{B}} L^2\widetilde{B}{}_{\varphi(B)}
$$
We get an invertible 1-morphism $(L^2\widetilde{A}_\varphi, L^2\widetilde{B}_\varphi, \psi_\varphi) : {}_AL^2B{}_B \to {}_{\widetilde{A}} L^2\widetilde{B}{}_{\widetilde{B}}$ 
in the bimodule
2-groupoid
where 
$\psi_\varphi$
is the isomorphism
\begin{equation}
\label{eq:PsiFromIsomorphism}
\psi_\varphi:=
\tikzmath{
\begin{scope}
\clip[rounded corners=5pt] (-.9,-.8) rectangle (.9,.8);
\fill[white] (-.2,-.8) -- (-.2,-.4) arc (180:90:.2cm) -- (0,.2) arc (0:180:.3cm) -- (-.6,-.8);
\fill[\ATildeColor] (.2,.8) -- (.2,.4) arc (0:-90:.2cm) arc (0:180:.3cm) -- (-.6,-.8) -- (-.9,-.8) -- (-.9,.8);
\fill[\BTildeColor] (.2,.8) -- (.2,.4) arc (0:-90:.2cm) -- (0,-.2) arc (-180:0:.3cm) -- (.6,.8);
\fill[\BColor] (-.2,-.8) -- (-.2,-.4) arc (180:90:.2cm) arc (-180:0:.3cm) -- (.6,.8) -- (.9,.8) -- (.9,.-.8);
\end{scope}
\draw[thick, blue, dashed] (-.6,-.8) node[below, xshift=-.2cm] {$\scriptstyle L^2\widetilde{A}{}_\varphi $} -- (-.6,.2) arc (180:0:.3cm);
\draw[thick, red, dashed] (.6,.8) node[above, xshift=.2cm, yshift=-.1cm] {$\scriptstyle L^2\widetilde{B}{}_\varphi $} -- (.6,-.2) arc (0:-180:.3cm) -- (0,.2);
\draw (-.2,-.8) node[below, xshift=.2cm] {$\scriptstyle L^2B$} -- (-.2,-.4) arc (180:90:.2cm);
\draw (.2,.8) node[above, xshift=-.2cm] {$\scriptstyle L^2\widetilde{B}$} -- (.2,.4) arc (0:-90:.2cm);
\filldraw (0,.2) circle (.05cm);
\filldraw (0,-.2) circle (.05cm);
}
:
{}_{\widetilde{A}} L^2\widetilde{A}_\varphi \boxtimes_A L^2B{}_B 
\to 
{}_{\widetilde{A}}L^2\widetilde{B} \boxtimes_{\widetilde{B}} L^2\widetilde{B}_{\varphi(B)}.
\end{equation}
The assignment 
$(A\subset B)\mapsto {}_AL^2B{}_B$
and
$[\varphi: (A\subset B) \to (\widetilde{A} \subset \widetilde{B})] \mapsto (L^2\widetilde{A}{}_\varphi, L^2\widetilde{B} {}_\varphi, \psi_\varphi)$
can be endowed with the structure of a 2-functor; we leave the details to the reader.

Suppose now that 
$(Y,Z,\psi) :{}_AX^1{}_B\to {}_{A'} X^2 {}_{B'}$
is a 1-isomorphism between dualizable multifactor bimodules.
Denote these four von Neumann algebras $A,B,A', B'$ by the shaded regions
$$
\tikzmath{\draw[fill=white, rounded corners=5, thin, dotted, baseline=1cm] (0,0) rectangle (.5,.5);}=A
\qquad
\tikzmath{\filldraw[\BColor, rounded corners=5, very thin, baseline=1cm] (0,0) rectangle (.5,.5);}=B
\qquad
\tikzmath{\draw[primedregion=white, draw=black, rounded corners=5, thin, dotted, baseline=1cm] (0,0) rectangle (.5,.5);}=A'
\qquad
\tikzmath{\filldraw[primedregion=\BColor, rounded corners=5, very thin, baseline=1cm] (0,0) rectangle (.5,.5);}=B',
$$
and the standard and Morita equivalence bimodules by
$$
\tikzmath{
\begin{scope}
\clip[rounded corners=5pt] (-.4,-.4) rectangle (.4,.4);
\fill[\BColor] (0,-.4) rectangle (.4,.4);
\fill[white] (-.4,-.4) rectangle (0,.4);
\end{scope}
\draw[dotted,rounded corners=5pt] (0,-.4) -- (-.4,-.4) -- (-.4,.4) -- (0,.4);
\draw (0,-.4) -- (0,.4);
}
=
{}_AL^2B{}_B
\qquad
\tikzmath{
\begin{scope}
\clip[rounded corners=5pt] (-.4,-.4) rectangle (.4,.4);
\fill[primedregion=\BColor] (0,-.4) rectangle (.4,.4);
\fill[primedregion=white] (-.4,-.4) rectangle (0,.4);
\end{scope}
\draw[dotted,rounded corners=5pt] (0,-.4) -- (-.4,-.4) -- (-.4,.4) -- (0,.4);
\draw (0,-.4) -- (0,.4);
}
=
{}_{A'}L^2B'{}_{B'}
\qquad
\tikzmath{
\begin{scope}
\clip[rounded corners=5pt] (-.4,-.4) rectangle (.4,.4);
\fill[primedregion=white] (0,-.4) rectangle (-.4,.4);
\fill[white] (0,-.4) rectangle (.4,.4);
\end{scope}
\draw[dotted,rounded corners=5pt] (-.4,-.4) rectangle (.4,.4);
\draw[thick, violet] (0,-.4) -- (0,.4);
}
={}_{A'}Y{}_{A}
\qquad
\tikzmath{
\begin{scope}
\clip[rounded corners=5pt] (-.4,-.4) rectangle (.4,.4);
\fill[primedregion=\BColor] (-.4,-.4) rectangle (0,.4);
\fill[\BColor] (.4,-.4) rectangle (0,.4);
\end{scope}
\draw[thick, orange] (0,-.4) -- (0,.4);
}
={}_{B}Z{}_{B}
$$
We denote the conjugate bimodules by the horizontal reflection, and the restriction to $A,\widetilde{A}$ respectively by changing the shading.
We get a planar algebra $*$-isomorphism by `encircling' the box spaces of $\cP({}_AX^1{}_B)_\bullet$ using 
$\psi, \psi^*,\overline{\psi},\overline{\psi}^*$, 
and the standard cups and caps.
For example, 
abbreviating isomorphisms $\psi, \overline{\psi}, \psi^*\overline{\psi}^*$ by 4-valent vertices
$$
\tikzmath{
\begin{scope}
\clip[rounded corners=5pt] (-.4,-.4) rectangle (.4,.4);
\fill[white] (-.4,-.4) -- (0,0) -- (.4,-.4);
\fill[primedregion=white] (-.4,-.4) -- (0,0) -- (-.4,.4);
\fill[\BColor] (.4,-.4) -- (0,0) -- (.4,.4);
\fill[primedregion=\BColor] (-.4,.4) -- (0,0) -- (.4,.4);
\draw (.4,-.4) -- (0,0);
\draw (-.4,.4) -- (0,0);
\draw[thick, violet] (-.4,-.4) -- (0,0);
\draw[thick, orange] (0,0) -- (.4,.4);
\end{scope}
}
:=
\tikzmath{
\begin{scope}
\clip[rounded corners=5pt] (-.7,-.8) rectangle (.7,.8);
\fill[white] (-.2,-.8) rectangle (.2,0);
\fill[primedregion=white] (-.7,.8) rectangle (-.2,-.8);
\fill[primedregion=\BColor] (-.2,.8) rectangle (.2,0);
\fill[\BColor] (.2,-.8) rectangle (.7,.8);
\end{scope}
\draw (.2,-.8) -- (.2,0);
\draw (-.2,.8) -- (-.2,0);
\draw[thick, violet] (-.2,-.3) -- (-.2,-.8);
\draw[thick, orange] (.2,.3) -- (.2,.8);
\roundNbox{unshaded}{(0,0)}{.3}{.1}{.1}{$\psi$}
}
\qquad
\qquad
\tikzmath[xscale=-1]{
\begin{scope}
\clip[rounded corners=5pt] (-.4,-.4) rectangle (.4,.4);
\fill[white] (-.4,-.4) -- (0,0) -- (.4,-.4);
\fill[primedregion=white] (-.4,-.4) -- (0,0) -- (-.4,.4);
\fill[\BColor] (.4,-.4) -- (0,0) -- (.4,.4);
\fill[primedregion=\BColor] (-.4,.4) -- (0,0) -- (.4,.4);
\draw (.4,-.4) -- (0,0);
\draw (-.4,.4) -- (0,0);
\draw[thick, violet] (-.4,-.4) -- (0,0);
\draw[thick, orange] (0,0) -- (.4,.4);
\end{scope}
}
:=
\tikzmath[xscale=-1]{
\begin{scope}
\clip[rounded corners=5pt] (-.7,-.8) rectangle (.7,.8);
\fill[white] (-.2,-.8) rectangle (.2,0);
\fill[primedregion=white] (-.7,.8) rectangle (-.2,-.8);
\fill[primedregion=\BColor] (-.2,.8) rectangle (.2,0);
\fill[\BColor] (.2,-.8) rectangle (.7,.8);
\end{scope}
\draw (.2,-.8) -- (.2,0);
\draw (-.2,.8) -- (-.2,0);
\draw[thick, violet] (-.2,-.3) -- (-.2,-.8);
\draw[thick, orange] (.2,.3) -- (.2,.8);
\roundNbox{unshaded}{(0,0)}{.3}{.1}{.1}{$\overline{\psi}$}
}
\qquad
\qquad
\tikzmath[yscale=-1]{
\begin{scope}
\clip[rounded corners=5pt] (-.4,-.4) rectangle (.4,.4);
\fill[white] (-.4,-.4) -- (0,0) -- (.4,-.4);
\fill[primedregion=white] (-.4,-.4) -- (0,0) -- (-.4,.4);
\fill[\BColor] (.4,-.4) -- (0,0) -- (.4,.4);
\fill[primedregion=\BColor] (-.4,.4) -- (0,0) -- (.4,.4);
\draw (.4,-.4) -- (0,0);
\draw (-.4,.4) -- (0,0);
\draw[thick, violet] (-.4,-.4) -- (0,0);
\draw[thick, orange] (0,0) -- (.4,.4);
\end{scope}
}
:=
\tikzmath[yscale=-1]{
\begin{scope}
\clip[rounded corners=5pt] (-.7,-.8) rectangle (.7,.8);
\fill[white] (-.2,-.8) rectangle (.2,0);
\fill[primedregion=white] (-.7,.8) rectangle (-.2,-.8);
\fill[primedregion=\BColor] (-.2,.8) rectangle (.2,0);
\fill[\BColor] (.2,-.8) rectangle (.7,.8);
\end{scope}
\draw (.2,-.8) -- (.2,0);
\draw (-.2,.8) -- (-.2,0);
\draw[thick, violet] (-.2,-.3) -- (-.2,-.8);
\draw[thick, orange] (.2,.3) -- (.2,.8);
\roundNbox{unshaded}{(0,0)}{.3}{.1}{.1}{$\psi^*$}
}
\qquad\qquad
\tikzmath[xscale=-1, yscale=-1]{
\begin{scope}
\clip[rounded corners=5pt] (-.4,-.4) rectangle (.4,.4);
\fill[white] (-.4,-.4) -- (0,0) -- (.4,-.4);
\fill[primedregion=white] (-.4,-.4) -- (0,0) -- (-.4,.4);
\fill[\BColor] (.4,-.4) -- (0,0) -- (.4,.4);
\fill[primedregion=\BColor] (-.4,.4) -- (0,0) -- (.4,.4);
\draw (.4,-.4) -- (0,0);
\draw (-.4,.4) -- (0,0);
\draw[thick, violet] (-.4,-.4) -- (0,0);
\draw[thick, orange] (0,0) -- (.4,.4);
\end{scope}
}
:=
\tikzmath[xscale=-1, yscale=-1]{
\begin{scope}
\clip[rounded corners=5pt] (-.7,-.8) rectangle (.7,.8);
\fill[white] (-.2,-.8) rectangle (.2,0);
\fill[primedregion=white] (-.7,.8) rectangle (-.2,-.8);
\fill[primedregion=\BColor] (-.2,.8) rectangle (.2,0);
\fill[\BColor] (.2,-.8) rectangle (.7,.8);
\end{scope}
\draw (.2,-.8) -- (.2,0);
\draw (-.2,.8) -- (-.2,0);
\draw[thick, violet] (-.2,-.3) -- (-.2,-.8);
\draw[thick, orange] (.2,.3) -- (.2,.8);
\roundNbox{unshaded}{(0,0)}{.3}{.1}{.1}{$\overline{\psi}^*$}
}\,,
$$
given $x \in \cP({}_AX^1{}_B)_{3,+}$, we define
$$
\cP(Y,Z,\psi)_{3,+}(x)
:=
\tikzmath{
\begin{scope}
\clip[rounded corners = 5] (-1.1,-1.1) rectangle (1.1,1.1) ;
\filldraw[primedregion=white] (-1.1,-1.1) rectangle (.4,1.1);
\filldraw[primedregion=\BColor] (-.4,-1.1) rectangle (0,1.1);
\filldraw[primedregion=white] (0,-1.1) rectangle (.4,1.1);
\filldraw[primedregion=\BColor] (.4,-1.1) rectangle (1.1,1.1);
\filldraw[white] (-120:.8cm) arc (240:120:.8cm);
\filldraw[\BColor] (0,-0.8) arc (-90:-120:.8cm) -- (120:.8cm) arc (120:90:.8cm);
\filldraw[\BColor] (60:.8cm) arc (60:-60:.8cm);
\filldraw[white] (0,-0.8) arc (-90:-60:.8cm) -- (60:.8cm) arc (60:90:.8cm);
\end{scope}
\draw[thick, violet] (-120:.8cm) arc (240:120:.8cm);
\draw[thick, orange] (60:.8cm) arc (60:-60:.8cm);
\draw[thick, violet] (60:.8cm) arc (60:90:.8cm);
\draw[thick, orange] (90:.8cm) arc (90:120:.8cm);
\draw[thick, violet] (-60:.8cm) arc (-60:-90:.8cm);
\draw[thick, orange] (-90:.8cm) arc (-90:-120:.8cm);
\draw (.4,-1.1) -- (.4,1.1);
\draw (0,-1.1) -- (0,1.1);
\draw (-.4,-1.1) -- (-.4,1.1);
\roundNbox{unshaded}{(0,0)}{.3}{.25}{.25}{$x$}
}
:=
\tikzmath{
\begin{scope}
\clip[rounded corners = 5] (-1.8,-1.3) rectangle (1.8,1.3) ;
\filldraw[primedregion=white] (-1.8,-1.3) rectangle (-1,1.3);
\filldraw[primedregion=\BColor] (-1,-1.3) rectangle (0,1.3);
\filldraw[primedregion=white] (0,-1.3) rectangle (1,1.3);
\filldraw[primedregion=\BColor] (1,-1.3) rectangle (1.8,1.3);
\filldraw[white] (-1,-.8) .. controls ++(135:.3cm) and ++(270:.3cm) .. (-1.6,0) .. controls ++(90:.4cm) and ++(-135:.4cm) .. (-1,.8) -- (-1,-.8);
\filldraw[\BColor] (-1,.8) .. controls ++(45:.4cm) and ++(135:.4cm) .. (0,.8) -- 
(0,-.8) .. controls ++(-135:.4cm) and ++(-45:.4cm) .. (-1,-.8) -- (-1,.8);
\filldraw[white] (0,.8) .. controls ++(-45:.4cm) and ++(-135:.4cm) .. (1,.8) -- (1,-.8) .. controls ++(135:.4cm) and ++(45:.4cm) .. (0,-.8) -- (0,.8);
\filldraw[\BColor] (1,.8) .. controls ++(45:.4cm) and ++(90:1cm) .. (1.6,0) .. controls ++(270:1cm) and ++(-45:.4cm) .. (1,-.8) -- (1,.8);
\end{scope}
\draw (-1,-1.3) -- (-1,1.3);
\draw (0,-1.3) -- (0,1.3);
\draw (1,-1.3) -- (1,1.3);
\draw[thick, violet] (-1,-.8) .. controls ++(135:.3cm) and ++(270:.3cm) .. (-1.6,0) .. controls ++(90:.4cm) and ++(-135:.4cm) .. (-1,.8);
\draw[thick, orange] (-1,.8) .. controls ++(45:.4cm) and ++(135:.4cm) .. (0,.8);
\draw[thick, violet] (0,.8) .. controls ++(-45:.4cm) and ++(-135:.4cm) .. (1,.8);
\draw[thick, orange] (-1,-.8) .. controls ++(-45:.4cm) and ++(-135:.4cm) .. (0,-.8);
\draw[thick, violet] (0,-.8) .. controls ++(45:.4cm) and ++(135:.4cm) .. (1,-.8);
\draw[thick, orange] (1,.8) .. controls ++(45:.4cm) and ++(90:1cm) .. (1.6,0) .. controls ++(270:1cm) and ++(-45:.4cm) .. (1,-.8);
\roundNbox{unshaded}{(0,0)}{.3}{1}{1}{$x$}
}\,.
$$
One verifies this `encircling' action gives a planar algebra isomorphism using the recabling relation \eqref{eq:Recabling}.
It is a straightforward exercise using \eqref{eq:Morita2IsoCompatibility} and \eqref{eq:Morita2IsoCompatibility-Automatic} that if there exists a 2-isomorphism $(u,v) : (Y^1,Z^1,\psi_1) \Rightarrow (Y^2,Z^2,\psi_2)$, the planar algebra $*$-isomorphisms $\cP(Y^1,Z^1,\psi_1)_\bullet$ and $\cP(Y^2,Z^2,\psi_2)_\bullet$ 
from 
$\cP({}_AX^1{}_B)_\bullet \to \cP({}_{A'}X^2{}_{B'})_\bullet$
are equal.
Using this `encircling action', we also get an equivalence of 
the non-idempotent complete full subcategories 
$\widetilde{\cC}({}_AX^1{}_B)\subset \cC({}_AX^1{}_B)$ and 
$\widetilde{\cC}({}_{A'}X^2{}_{B'})\subset \cC({}_{A'}X^2{}_{B'}$
whose objects are the alternating tensor powers of $X^j$ and $\overline{X^j}$ for $j=1,2$ respectively, which descends to an equivalence of the idempotent completions $\cC({}_AX^1{}_B) \simeq \cC({}_{A'}X^2{}_{B'})$.

Hence the construction \cite{MR3178106} actually gives a 1-functor from the 1-truncation of the 2-groupoid of bimodules to the 1-groupoid of unitary 2-shaded planar algebras.
Moreover, the result \cite[Pf.~of Thm.~5.4 and Rem.~5.5]{MR3178106} can then be reinterpreted as the statement that the 1-functor from the 1-groupoid of inclusions to the 1-groupoid of planar algebras 
is naturally isomorphic to the composite of the 
2-functor from the 1-groupoid of inclusions to 2-groupoid of bimodules followed 
and 
this 1-functor from the 1-truncation of the 2-groupoid of bimodules to the 1-groupoid of planar algebras.
That is, given a $*$-isomorphism $\varphi: (A\subset B) \to (\widetilde{A} \subset \widetilde{B})$,
we have the following commutative diagram:
\begin{equation}
\label{eq:SF-PA-TC-isos}
\begin{tikzcd}[column sep=4em]
\cP^{A\subset B}_\bullet
\arrow[r,leftrightarrow,"\cong"]
\arrow[d,"\cP(\varphi)_\bullet"]
&
\cP({}_{A}L^2B{}_{B})_\bullet
\arrow[d,"{\cP(L^2\widetilde{A}_\varphi, L^2\widetilde{B}_\varphi, \psi_\varphi)_\bullet}"]
\\
\cP^{\widetilde{A}\subset \widetilde{B}}_\bullet
\arrow[r,leftrightarrow,"\cong"]
&
\cP({}_{\widetilde{A}}L^2\widetilde{B}{}_{\widetilde{B}})_\bullet
\end{tikzcd}
\end{equation}

We now collect some important results on Morita equivalent inclusions and their standard invariants.

\begin{lem}
\label{lem:ME InclusionStdInv}
Suppose $A\subset B$ is a finite index connected inclusion of finite multifactors and $Y_A$ is a faithful right $A$-module.
The
induced Morita equivalent inclusion
$A'\subset B'$
is a connected inclusion of finite multifactors with the same Jones dimension matrix.
Moreover
$({}_{A'}Y_A, {}_{B'}Z{}_{B}, \psi)$ 
is an invertible 1-morphism from ${}_AL^2B{}_B \to {}_{A'}L^2B' {}_{B'}$ where $\psi$ is the composite
$$
{}_{A'}Y\boxtimes_A L^2B{}_B 
=
{}_{A'}Z{}_{B}
\cong
{}_{A'}L^2B'\boxtimes_{B'} Z_{B}.
$$
In particular, Morita equivalent inclusions have canonically isomorphic standard invariants.
\end{lem}
\begin{proof}
By construction, $A',B'$ are clearly finite multifactors with the same centers as $A,B$ respectively.
By \cite[Prop.~3.1]{MR703809},
there is a canonical $B'-B'$ bimodule isomorphism
$L^2 B'
\cong
Z\boxtimes_B \overline{Z}$,
which restricts to an 
$A'-B'$ bimodule isomorphism
\begin{equation}
\label{eq:L2 of ME}
{}_{A'}L^2B'{}_{B'}
\cong
{}_{A'}Y\boxtimes_A L^2B \boxtimes_B \overline{Z}{}_{B'}. \end{equation}
Since left and right von Neumann dimension are multiplicative, we see
$\Delta({}_{A'}L^2B'{}_{B'})
=
\Delta({}_{A}L^2 B{}_{B})$,
so $A'\subset B'$ is finite index and connected. 
The rest is straightforward and left to the reader.
\end{proof}

\begin{prop}
\label{prop:TransportIsoAlongCompatibleME}
Let $\varphi:(A\subset B)\to (\widetilde{A}\subset \widetilde{B})$
be an isomorphism of finite index connected inclusions of finite multifactors, and $Y_A, \widetilde{Y}_{\widetilde{A}}$ two faithful right modules.
Let $A'\subset B'$ and $\widetilde{A}'\subset \widetilde{B}'$ be the induced Morita equivalent inclusions, where $Z:= Y\boxtimes_A L^2B$ and $\widetilde{Z}:=\widetilde{Y}\boxtimes_{\widetilde{A}}L^2\widetilde{B}$.

For every right $A$-linear unitary $w: Y_A \to \widetilde{Y}_{\varphi(A)}$ (if one exists),
there is an isomorphism 
$\varphi': (A'\subset B') \to (\widetilde{A}'\subset \widetilde{B}')$
such that $\varphi'|_{A'} = \Ad(w)$.
Moreover, there exists an invertible 2-morphism
$$
\begin{tikzcd}
{}_AL^2B{}_B
\arrow[rr]
\arrow[d,swap,"{(L^2\widetilde{A}_\varphi,L^2\widetilde{B}_\varphi,\psi_{\varphi})}"]
\arrow[rr,"{(Y,Z,\psi)}"]
&&
{}_{A'}L^2B'{}_{B'}
\arrow[d,"{(L^2\widetilde{A}'_{\varphi'},L^2\widetilde{B}'_{\varphi'},\psi_{\varphi'})}"]
\arrow[dll,Rightarrow,shorten <= 2em, shorten >= 2em,"{(u,v)}"]
\\
{}_{\widetilde{A}}L^2\widetilde{B}{}_{\widetilde{B}}
\arrow[rr,swap,"{(\widetilde{Y},\widetilde{Z},\widetilde{\psi})}"]
&&
{}_{\widetilde{A}'}L^2\widetilde{B}'{}_{\widetilde{B}'}
\end{tikzcd}
$$
which only depends on $w,\varphi$.
In particular, the following square commutes:
\begin{equation}
\begin{tikzcd}
\cP({}_AL^2B{}_B)_\bullet
\arrow[r,leftrightarrow,"\cong"]
\arrow[d,swap,"{\cP(L^2\widetilde{A}_\varphi,L^2\widetilde{B}_\varphi,\psi_{\varphi})_\bullet}"]
&
\cP({}_{A'}L^2B'{}_{B'})_\bullet
\arrow[d,"{\cP(L^2\widetilde{A}'_{\varphi'},L^2\widetilde{B}'_{\varphi'},\psi_{\varphi'})_\bullet}"]
\\
\cP({}_{\widetilde{A}}L^2\widetilde{B}{}_{\widetilde{B}})_\bullet
\arrow[r,leftrightarrow,"\cong"]
&
\cP({}_{\widetilde{A}'}L^2\widetilde{B}'{}_{\widetilde{B}'})_\bullet
\end{tikzcd}
\end{equation}
\end{prop}
\begin{proof}
We denote the von Neumann algebras
$A, A', \widetilde{A},\widetilde{A}',B, B', \widetilde{B}, \widetilde{B}'$ by the shaded regions
$$
\tikzmath{\draw[fill=white, rounded corners=5, thin, dotted, baseline=1cm] (0,0) rectangle (.5,.5);}=A
\qquad
\tikzmath{\draw[primedregion=white, draw=black, rounded corners=5, thin, dotted, baseline=1cm] (0,0) rectangle (.5,.5);}=A'
\qquad
\tikzmath{\filldraw[\ATildeColor, rounded corners=5, very thin, baseline=1cm] (0,0) rectangle (.5,.5);}=\widetilde{A}
\qquad
\tikzmath{\draw[primedregion=\ATildeColor, rounded corners=5, very thin, baseline=1cm] (0,0) rectangle (.5,.5);}=\widetilde{A}'
\qquad
\tikzmath{\filldraw[\BColor, rounded corners=5, very thin, baseline=1cm] (0,0) rectangle (.5,.5);}=B
\qquad
\tikzmath{\filldraw[primedregion=\BColor, rounded corners=5, very thin, baseline=1cm] (0,0) rectangle (.5,.5);}=B'
\qquad
\tikzmath{\filldraw[\BTildeColor, rounded corners=5, very thin, baseline=1cm] (0,0) rectangle (.5,.5);}=\widetilde{B}
\qquad
\tikzmath{\filldraw[primedregion=\BTildeColor, rounded corners=5, very thin, baseline=1cm] (0,0) rectangle (.5,.5);}=\widetilde{B}'
$$
the standard bimodules
$
{}_AL^2B{}_B, 
{}_{A'}L^2B'{}_{B'}, 
{}_{\widetilde{A}}L^2\widetilde{B}{}_{\widetilde{B}}, 
{}_{\widetilde{A}'}L^2\widetilde{B}'{}_{\widetilde{B}'}
$ 
by
$$
\tikzmath{
\begin{scope}
\clip[rounded corners=5pt] (-.4,-.4) rectangle (.4,.4);
\fill[\BColor] (0,-.4) rectangle (.4,.4);
\fill[white] (0,-.4) rectangle (-.4,.4);
\end{scope}
\draw[dotted,rounded corners=5pt] (0,-.4) -- (-.4,-.4) -- (-.4,.4) -- (0,.4);
\draw (0,-.4) -- (0,.4);
}
=
{}_AL^2B{}_B
\qquad
\tikzmath{
\begin{scope}
\clip[rounded corners=5pt] (-.4,-.4) rectangle (.4,.4);
\fill[primedregion=white] (0,-.4) rectangle (-.4,.4);
\fill[primedregion=\BColor] (0,-.4) rectangle (.4,.4);
\end{scope}
\draw[dotted,rounded corners=5pt] (0,-.4) -- (-.4,-.4) -- (-.4,.4) -- (0,.4);
\draw (0,-.4) -- (0,.4);
}
=
{}_{A'}L^2B'{}_{B'}
\qquad
\tikzmath{
\begin{scope}
\clip[rounded corners=5pt] (-.4,-.4) rectangle (.4,.4);
\fill[\ATildeColor] (-.4,-.4) rectangle (0,.4);
\fill[\BTildeColor] (.4,-.4) rectangle (0,.4);
\end{scope}
\draw (0,-.4) -- (0,.4);
}
=
{}_{\widetilde{A}}L^2\widetilde{B}{}_{\widetilde{B}}
\qquad
\tikzmath{
\begin{scope}
\clip[rounded corners=5pt] (-.4,-.4) rectangle (.4,.4);
\fill[primedregion=\ATildeColor] (-.4,-.4) rectangle (0,.4);
\fill[primedregion=\BTildeColor] (.4,-.4) rectangle (0,.4);
\end{scope}
\draw (0,-.4) -- (0,.4);
}
=
{}_{\widetilde{A}'}L^2\widetilde{B}'{}_{\widetilde{B}'},
$$
and the Morita equivalences $Y, Z,\widetilde{Y}, \widetilde{Z}$ by the colored strands
$$
\tikzmath{
\begin{scope}
\clip[rounded corners=5pt] (-.4,-.4) rectangle (.4,.4);
\fill[white] (-.4,-.4) rectangle (.4,.4);
\fill[primedregion] (-.4, -.4) rectangle (0, .4);
\end{scope}
\draw[dotted,rounded corners=5pt] (-.4,-.4) rectangle (.4,.4);
\draw[thick, violet] (0,-.4) -- (0,.4);
}
=
{}_{A'}Y{}_{A}
\qquad
\tikzmath{
\begin{scope}
\clip[rounded corners=5pt] (-.4,-.4) rectangle (.4,.4);
\fill[\BColor] (-.4,-.4) rectangle (.4,.4);
\fill[primedregion] (-.4, -.4) rectangle (0, .4);
\end{scope}
\draw[thick, orange] (0,-.4) -- (0,.4);
}
=
{}_{B'}Z{}_{B}
\qquad
\tikzmath{
\begin{scope}
\clip[rounded corners=5pt] (-.4,-.4) rectangle (.4,.4);
\fill[\ATildeColor] (-.4,-.4) rectangle (.4,.4);
\fill[primedregion] (-.4, -.4) rectangle (0, .4);
\end{scope}
\draw[thick, blue] (0,-.4) -- (0,.4);
}
=
{}_{\widetilde{A}'}\widetilde{Y}_{\widetilde{A}}
\qquad
\tikzmath{
\begin{scope}
\clip[rounded corners=5pt] (-.4,-.4) rectangle (.4,.4);
\fill[\BTildeColor] (-.4,-.4) rectangle (.4,.4);
\fill[primedregion] (-.4, -.4) rectangle (0, .4);
\end{scope}
\draw[thick, red] (0,-.4) -- (0,.4);
}
=
{}_{\widetilde{B}'}\widetilde{Z}_{\widetilde{B}}
$$
We denote the conjugates of these bimodules by taking horizontal reflections.
We denote the restrictions of $B,B',\widetilde{B},\widetilde{B}'$-actions 
to
$A,A',\widetilde{A},\widetilde{A}'$ by changing the shading on the appropriate side.
For example, when we restrict the left $B,B'$-action
on $Z,\widetilde{Z}$ 
to $A', \widetilde{A}'$,
we have obvious identification isomorphisms
\begin{align*}
{}_{A'}Z{}_B &= {}_{A'}Y\boxtimes_A L^2B{}_B
\qquad\qquad
\tikzmath{
\begin{scope}
\clip[rounded corners=5pt] (-.6,-.4) rectangle (.6,.4);
\filldraw[primedregion=white] (0,-.4) -- (0,0) arc (-90:-180:.4cm) -- (-.6,.4) -- (-.6,-.4) -- (0,-.4);
\filldraw[white] (-.4,.4) arc (-180:0:.4cm);
\filldraw[\BColor] (0,-.4) -- (0,0) arc (-90:0:.4cm) -- (.6,.4) -- (.6,-.4) -- (0,-.4);
\end{scope}
\draw[thick, violet] (-.4,.4) arc (-180:-90:.4cm);
\draw (0,0) arc (-90:0:.4cm);
\draw[thick, orange] (0,-.4) -- (0,0);
\filldraw (0,0) circle (.05cm);
}
\\
{}_{\widetilde{A}'}\widetilde{Z}{}_{\widetilde{B}} 
&= 
{}_{\widetilde{A}'}\widetilde{Y}
\boxtimes_{\widetilde{A}} 
L^2\widetilde{B}
{}_{\widetilde{B}}
\qquad\qquad
\tikzmath{
\begin{scope}
\clip[rounded corners=5pt] (-.6,-.4) rectangle (.6,.4);
\filldraw[primedregion=\ATildeColor] (0,-.4) -- (0,0) arc (-90:-180:.4cm) -- (-.6,.4) -- (-.6,-.4) -- (0,-.4);
\filldraw[\ATildeColor] (-.4,.4) arc (-180:0:.4cm);
\filldraw[\BTildeColor] (0,-.4) -- (0,0) arc (-90:0:.4cm) -- (.6,.4) -- (.6,-.4) -- (0,-.4);
\end{scope}
\draw[thick, blue] (-.4,.4) arc (-180:-90:.4cm);
\draw (0,0) arc (-90:0:.4cm);
\draw[thick, red] (0,-.4) -- (0,0);
\filldraw (0,0) circle (.05cm);
}
\end{align*}

We now define a right $B$-linear unitary
$x:Z{}_{B} \to \widetilde{Z}{}_{\varphi(B)}$ by
$$
Z{}_{B} 
= 
Y\boxtimes_A L^2B{}_B
\xrightarrow{w \boxtimes L^2\varphi}
\widetilde{Y}_{\varphi}
\boxtimes_A 
{}_\varphi L^2\widetilde{B}{}_{\varphi(B)}
\cong
\widetilde{Y}
\boxtimes_{\widetilde{A}}
L^2\widetilde{B}{}_{\varphi(B)}
=
\widetilde{Z}{}_{\varphi(B)}.
$$
This right $B$-linear unitary induces an isomorphism
of right $B$-commutants $\varphi':= \Ad(x): B' \to \widetilde{B}'$.
By construction, $\varphi'|_{A'} = \Ad(w)$.
We denote the Morita equivalences $L^2\widetilde{A}{}_\varphi, L^2\widetilde{B}{}_\varphi, L^2\widetilde{A}'{}_{\varphi'}, L^2\widetilde{B}'{}_{\varphi'}$
by
$$
\tikzmath{
\begin{scope}
  \clip[rounded corners=5pt] (-.4,-.4) rectangle (.4,.4);
  \fill[\ATildeColor] (-.4,-.4) rectangle (0,.4);
  \fill[white] (.4,-.4) rectangle (0,.4);
\end{scope}
\draw[dotted,rounded corners=5pt] (0,-.4) -- (.4,-.4) -- (.4,.4) -- (0,.4);
\draw[thick, blue, dashed] (0,-.4) -- (0,.4);
}
=
{}_{\widetilde{A}}L^2\widetilde{A}{}_{\varphi(A)}
\qquad
\tikzmath{
\begin{scope}
  \clip[rounded corners=5pt] (-.4,-.4) rectangle (.4,.4);
  \fill[\BTildeColor] (-.4,-.4) rectangle (0,.4);
  \fill[\BColor] (.4,-.4) rectangle (0,.4);
\end{scope}
\draw[thick, red, dashed] (0,-.4) -- (0,.4);
}
=
{}_{\widetilde{B}}L^2\widetilde{B}{}_{\varphi(B)}
\qquad
\tikzmath{
\begin{scope}
  \clip[rounded corners=5pt] (-.4,-.4) rectangle (.4,.4);
  \fill[primedregion=\ATildeColor] (-.4,-.4) rectangle (0,.4);
  \fill[primedregion=white] (.4,-.4) rectangle (0,.4);
\end{scope}
\draw[dotted,rounded corners=5pt] (0,-.4) -- (.4,-.4) -- (.4,.4) -- (0,.4);
\draw[thick, violet, dashed] (0,-.4) -- (0,.4);
}
=
{}_{\widetilde{A}'}L^2\widetilde{A}'{}_{\varphi'(A')}
\qquad
\tikzmath{
\begin{scope}
  \clip[rounded corners=5pt] (-.4,-.4) rectangle (.4,.4);
  \fill[primedregion=\BTildeColor] (-.4,-.4) rectangle (0,.4);
  \fill[primedregion=\BColor] (.4,-.4) rectangle (0,.4);
\end{scope}
\draw[thick, orange, dashed] (0,-.4) -- (0,.4);
}
=
{}_{\widetilde{B}'}L^2\widetilde{B}'{}_{\varphi'(B')}
$$
their conjugates by again taking horizontal reflection, and restrictions by changing the shading.
Using $\varphi,\varphi'$, we view
$w$ as an $A'-A$ bilinear unitary
and 
$x$ as a $B'-B$ bilinear unitary
\begin{align*}
\tikzmath{
\begin{scope}
  \clip[rounded corners=5pt] (-.8,-.8) rectangle (.8,.8);
  \fill[primedregion=white] (-.8,-.8) rectangle (0,.8);
  \fill[white] (0,-.8) rectangle (.8,.8);
  \fill[primedregion=\ATildeColor] (-.3,0) rectangle (0,.8);
  \fill[\ATildeColor] (0,0) rectangle (.3,.8);
\end{scope}
\draw[thick, violet] (0,-.8) -- (0,0);
\draw[thick, blue] (0,0) -- (0,.8);
\draw[thick, violet, dashed] (-.3,0) -- (-.3,.8);
\draw[thick, blue, dashed] (.3,0) -- (.3,.8);
\roundNbox{unshaded}{(0,0)}{.3}{.2}{.2}{$w$}
}
&:
{}_{A'}Y{}_A 
\to 
{}_{\varphi'(A')}
L^2\widetilde{A}'
\boxtimes_{\widetilde{A}'}
\widetilde{Y} 
\boxtimes_{\widetilde{A}}
L^2\widetilde{A} 
{}_{\varphi(A)}
\\
\tikzmath{
\begin{scope}
  \clip[rounded corners=5pt] (-.8,-.8) rectangle (.8,.8);
  \fill[primedregion=\BColor] (-.8,-.8) rectangle (0,.8);
  \fill[\BColor] (0,-.8) rectangle (.8,.8);
  \fill[primedregion=\BTildeColor] (-.3,0) rectangle (0,.8);
  \fill[\BTildeColor] (0,0) rectangle (.3,.8);
\end{scope}
\draw[thick, orange] (0,-.8) -- (0,0);
\draw[thick, red] (0,0) -- (0,.8);
\draw[thick, orange, dashed] (-.3,0) -- (-.3,.8);
\draw[thick, red, dashed] (.3,0) -- (.3,.8);
\roundNbox{unshaded}{(0,0)}{.3}{.2}{.2}{$x$}
}
&:
{}_{B'}Z{}_B 
\to 
{}_{\varphi'(B')}
L^2\widetilde{B}'
\boxtimes_{\widetilde{B}'}
\widetilde{Z}
\boxtimes_{\widetilde{B}}
L^2\widetilde{B}
{}_{\varphi(B)}.
\end{align*}

We now define our component unitaries of our 2-morphism by
\begin{align*}
u:=
\tikzmath{
\begin{scope}
\clip[rounded corners=5pt] (-1.2,-.8) rectangle (.8,.8);
\fill[\ATildeColor] (0,.3) rectangle (.3,.8);
\fill[primedregion=\ATildeColor] (-.3,.3) arc (0:180:.3cm) -- (-.9,-.8) -- (-1.2,-.8) -- (-1.2,.8) -- (0,.8) -- (0,.3);
\fill[primedregion=white] (-.3,.3) arc (0:180:.3cm) -- (-.9,-.8) -- (0,-.8) -- (0,0);
\fill[white] (0,-.8) -- (.8,-.8) -- (.8,.8) -- (.3,.8) -- (.3,0) -- (0,0);
\end{scope}
\draw[thick, violet] (0,0) -- (0,-.8);
\draw[thick, blue] (0,0) -- (0,.8);
\draw[thick, blue, dashed] (.3,0) -- (.3,.8);
\draw[thick, violet, dashed] (-.3,.3) arc (0:180:.3cm) -- (-.9,-.8);
\roundNbox{unshaded}{(0,0)}{.3}{.2}{.2}{$w$}
}
&:
{}_{\widetilde{A}'}
L^2\widetilde{A}'_{\varphi'}
\boxtimes_{A'} 
Y
{}_A
\longrightarrow
{}_{\widetilde{A}'}
\widetilde{Y}
\boxtimes_{\widetilde{A}} 
L^2\widetilde{A}
{}_{\varphi(A)}
\\
v:=
\tikzmath{
\begin{scope}
\clip[rounded corners=5pt] (-1.2,-.8) rectangle (.8,.8);
\fill[\BTildeColor] (0,.3) rectangle (.3,.8);
\fill[primedregion=\BTildeColor] (-.3,.3) arc (0:180:.3cm) -- (-.9,-.8) -- (-1.2,-.8) -- (-1.2,.8) -- (0,.8) -- (0,.3);
\fill[primedregion=\BColor] (-.3,.3) arc (0:180:.3cm) -- (-.9,-.8) -- (0,-.8) -- (0,0);
\fill[\BColor] (0,-.8) -- (.8,-.8) -- (.8,.8) -- (.3,.8) -- (.3,0) -- (0,0);
\end{scope}
\draw[thick, orange] (0,0) -- (0,-.8);
\draw[thick, red] (0,0) -- (0,.8);
\draw[thick, red, dashed] (.3,0) -- (.3,.8);
\draw[thick, orange, dashed] (-.3,.3) arc (0:180:.3cm) -- (-.9,-.8);
\roundNbox{unshaded}{(0,0)}{.3}{.2}{.2}{$x$}
}
&:
{}_{\widetilde{B}'}
L^2\widetilde{B}'_{\varphi'}
\boxtimes_{B'} 
Z
{}_B
\longrightarrow
{}_{\widetilde{B}'}
\widetilde{Z}
\boxtimes_{\widetilde{B}} 
L^2\widetilde{B}
{}_{\varphi(B)}.
\end{align*}
It remains to verify \eqref{eq:Morita2IsoCompatibility} holds, which in string diagrams is as follows:
$$
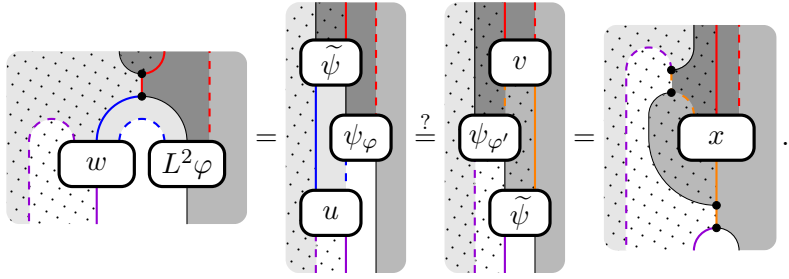
\tikzmath{
\begin{scope}
\clip[rounded corners=5pt] (-1.2,-.8) rectangle (2,1.5);
\fill[\BColor] (2,1.5) -- (2,-.8) -- (1.2,-.8) -- (1.2,0) -- (1.5,.3) -- (1.5,1.5);
\fill[\BTildeColor] (1.5,.3) -- (1.2,.3) arc (0:90:.6cm) -- (.6,1.2) arc (-90:0:.3cm) -- (1.5,1.5);
\fill[primedregion=\BTildeColor] (.3,1.5) arc (-180:0:.3cm);
\fill[\AColor] (1.2,.3) arc (0:180:.6cm) -- (.3,.3) arc (180:0:.3cm);
\fill[white] (0,-.8) -- (0,0) -- (.3,.3) arc (180:0:.3cm) -- (1.2,0) -- (1.2,-.8);
\fill[primedregion=white] (-.9,-.8) -- (-.9,.3) arc (180:0:.3cm) -- (0,0) -- (0,-.8);
\fill[primedregion=\AColor] (-.9,-.8) -- (-.9,.3) arc (180:0:.3cm) -- (0,.3) arc (180:90:.6cm) -- (.6,1.2) arc (-90:-180:.3cm) -- (-1.2,1.5) -- (-1.2,-.8);
\end{scope}
\draw[thick, violet, dashed] (-.3,.3) arc (0:180:.3cm) -- (-.9,-.8);
\draw[thick, blue] (0,.3) arc (180:90:.6cm);
\draw (.6,.9) arc (90:0:.6cm);
\draw[thick, blue, dashed] (.3,.3) arc (180:0:.3cm);
\draw[thick, red, dashed] (1.5,0) -- (1.5,1.5);
\draw[thick, red] (.6,.9) -- (.6,1.2);
\draw (.6,1.2) arc (-90:-180:.3cm);
\draw[thick, red] (.6,1.2) arc (-90:0:.3cm);
\draw[thick, violet] (0,-.3) -- (0,-.8);
\draw (1.2,-.3) -- (1.2,-.8);
\filldraw (.6,.9) circle (.05cm);
\filldraw (.6,1.2) circle (.05cm);
\roundNbox{unshaded}{(0,0)}{.3}{.2}{.2}{$w$}
\roundNbox{unshaded}{(1.2,0)}{.3}{.2}{.2}{$L^2\varphi$}
}
=
\tikzmath{
\begin{scope}
\clip[rounded corners=5pt] (-.5,-.8) rectangle (1.1,2.8);
\fill[primedregion=white] (-.1,-.8) rectangle (.3,0);
\fill[white] (.7,-.8) rectangle (.3,1);
\fill[primedregion=\ATildeColor] (-.5,-.8) rectangle (-.1,2.8);
\fill[\ATildeColor] (-.1,0) rectangle (.3,2);
\fill[\BColor] (.7,-.8) rectangle (1.1,2.8);
\fill[primedregion=\BTildeColor] (-.1,2) rectangle (.3,2.8);
\fill[\BTildeColor] (.7,1) rectangle (.3,2.8);
\end{scope}
\draw[thick, violet, dashed] (-.1,-.8) -- (-.1,0);
\draw[thick, violet] (.3,-.8) -- (.3,0);
\draw (.7,-.8) -- (.7,1);
\draw[thick, blue] (-.1,0) -- (-.1,2);
\draw[thick, blue, dashed] (.3,0) -- (.3,1);
\draw (.3,1) -- (.3,2);
\draw[thick, red, dashed] (.7,2.8) -- (.7,1);
\draw (-.1,2.8) -- (-.1,2);
\draw[thick, red] (.3,2.8) -- (.3,2);
\roundNbox{unshaded}{(0,0)}{.3}{0}{.2}{$u$}
\roundNbox{unshaded}{(.6,1)}{.3}{.2}{0}{$\psi_\varphi$}
\roundNbox{unshaded}{(0,2)}{.3}{0}{.2}{$\widetilde{\psi}$}
}
\overset{?}{=}
\tikzmath{
\begin{scope}
\clip[rounded corners=5pt] (-.5,-.8) rectangle (1.1,2.8);
\fill[primedregion=white] (-.1,-.8) rectangle (.3,1);
\fill[white] (.7,-.8) rectangle (.3,0);
\fill[primedregion=\ATildeColor] (-.5,-.8) rectangle (-.1,2.8);
\fill[primedregion=\BColor] (.3,0) rectangle (.7,2);
\fill[\BColor] (.7,-.8) rectangle (1.1,2.8);
\fill[primedregion=\BTildeColor] (-.1,1) rectangle (.3,2.8);
\fill[\BTildeColor] (.7,2) rectangle (.3,2.8);
\end{scope}
\draw[thick, violet, dashed] (-.1,-.8) -- (-.1,1);
\draw[thick, violet] (.3,-.8) -- (.3,0);
\draw (.7,-.8) -- (.7,0);
\draw (.3,0) -- (.3,1);
\draw[thick, orange] (.7,0) -- (.7,2);
\draw[thick, orange, dashed] (.3,1) -- (.3,2);
\draw[thick, red, dashed] (.7,2.8) -- (.7,2);
\draw (-.1,2.8) -- (-.1,1);
\draw[thick, red] (.3,2.8) -- (.3,2);
\roundNbox{unshaded}{(.6,0)}{.3}{.2}{0}{$\widetilde{\psi}$}
\roundNbox{unshaded}{(0,1)}{.3}{0}{.2}{$\psi_{\varphi'}$}
\roundNbox{unshaded}{(.6,2)}{.3}{.2}{0}{$v$}
}
=
\tikzmath{
\begin{scope}
\clip[rounded corners=5pt] (-1.5,-1.5) rectangle (.8,1.5);
\fill[primedregion=\AColor] (-1.5,-1.5) -- (-1.2,-1.5) -- (-1.2,.9) arc (180:0:.3cm) arc (-90:0:.3cm) -- (-.3,1.5) -- (-1.5,1.5);
\fill[\BColor] (.3,-1.5) arc (0:90:.3cm) -- (0,0) -- (.3,.3) -- (.3,1.5) -- (.8,1.5) -- (.8,-1.5);
\fill[\BTildeColor] (0,0) rectangle (.3,1.5);
\fill[primedregion=\BTildeColor] (-.3,.3) arc (0:90:.3cm) -- (-.6,.9) arc (-90:0:.3cm) -- (-.3,1.5) -- (0,1.5) -- (0,.3);
\fill[primedregion=\BColor] (-.6,.6) arc (90:180:.3cm) -- (-.9,0) arc (180:270:.9cm) -- (0,0) -- (-.3,.3) arc (0:90:.3cm);
\fill[white] (-.3,-1.5) arc (180:0:.3cm);
\fill[primedregion=white] (-.6,.6) arc (90:180:.3cm) -- (-.9,0) arc (180:270:.9cm) -- (0,-1.2) arc (90:180:.3cm) -- (-1.2,-1.5) -- (-1.2,.9) arc (180:0:.3cm) -- (-.6,.6);
\end{scope} 
\draw[thick, red, dashed] (.3,0) -- (.3,1.5);
\draw[thick, red] (0,0) -- (0,1.5);
\draw (-.3,1.5) -- (-.3,1.2) arc (0:-90:.3cm);
\draw[thick, violet, dashed] (-.6,.9) arc (0:180:.3cm) -- (-1.2,-1.5);
\draw[thick, violet] (-.3,-1.5) arc (180:90:.3cm);
\draw (.3,-1.5) arc (0:90:.3cm);
\draw[thick, orange] (0,0) -- (0,-1.2);
\draw[thick, orange, dashed] (-.3,.3) arc (0:90:.3cm) -- (-.6,.9);
\draw (-.6,.6) arc (90:180:.3cm) -- (-.9,0) arc (180:270:.9cm);
\roundNbox{unshaded}{(0,0)}{.3}{.2}{.2}{$x$}
\filldraw (-.6,.6) circle (.05cm);
\filldraw (-.6,.9) circle (.05cm);
\filldraw (0,-.9) circle (.05cm);
\filldraw (0,-1.2) circle (.05cm);
}\,.
$$
Applying isotopy and composing with obvious trivalent vertex isomorphisms, the above equation is equivalent to the following equation:
$$
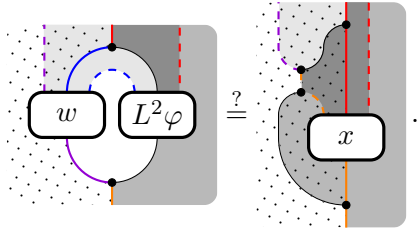
\tikzmath{
\begin{scope}
\clip[rounded corners=5pt] (-.8,-1.2) rectangle (2,1.2);
\fill[\BColor] (2,1.2) -- (2,-1.2) -- (.6,-1.2) -- (.6,-.9) arc (-90:0:.6cm) -- (1.5,.3) -- (1.5,1.2);
\fill[\BTildeColor] (1.5,.3) -- (1.2,.3) arc (0:90:.6cm) -- (.6,1.2) -- (1.5,1.2);
\fill[\AColor] (1.2,.3) arc (0:180:.6cm) -- (.3,.3) arc (180:0:.3cm);
\fill[white] (1.2,-.3) arc (0:-180:.6cm) -- (.3,.3) arc (180:0:.3cm) -- (1.2,0) -- (1.2,-.3);
\fill[primedregion=white] (-.8,-1.2) -- (.6,-1.2) -- (.6,-.9) arc (-90:-180:.6cm) -- (-.3,.3) -- (-.3,1.2) -- (-.8,1.2);
\fill[primedregion=\AColor] (0,.3) arc (180:90:.6cm) -- (.6,1.2) -- (-.3,1.2) -- (-.3,.3);
\end{scope}
\draw[thick, violet, dashed] (-.3,.3) -- (-.3,1.2);
\draw[thick, blue] (0,.3) arc (180:90:.6cm);
\draw (.6,.9) arc (90:0:.6cm);
\draw[thick, blue, dashed] (.3,.3) arc (180:0:.3cm);
\draw[thick, red, dashed] (1.5,0) -- (1.5,1.2);
\draw[thick, red] (.6,.9) -- (.6,1.2);
\draw[thick, violet] (0,-.3) arc (-180:-90:.6cm);
\draw (.6,-.9) arc (-90:0:.6cm);
\draw[thick, orange] (.6,-.9) -- (.6,-1.2);
\filldraw (.6,.9) circle (.05cm);
\filldraw (.6,-.9) circle (.05cm);
\roundNbox{unshaded}{(0,0)}{.3}{.2}{.2}{$w$}
\roundNbox{unshaded}{(1.2,0)}{.3}{.2}{.2}{$L^2\varphi$}
}
\overset{?}{=}
\tikzmath{
\begin{scope}
\clip[rounded corners=5pt] (-1.2,-1.2) rectangle (.8,1.8);
\fill[primedregion=\AColor] (-.9,1.8) -- (-.9,1.2) arc (-180:0:.3cm) arc (180:90:.3cm) -- (0,1.8);
\fill[\BColor] (0,-1.2) -- (0,0) -- (.3,.3) -- (.3,1.8) -- (.8,1.8) -- (.8,-1.2);
\fill[\BTildeColor] (0,0) rectangle (.3,1.8);
\fill[primedregion=\BTildeColor] (-.3,.3) arc (0:90:.3cm) -- (-.6,.9) arc (-90:0:.3cm) arc (180:90:.3cm) -- (0,.3);
\fill[primedregion=\BColor] (-.6,.6) arc (90:180:.3cm) -- (-.9,0) arc (180:270:.9cm) -- (0,0) -- (-.3,.3) arc (0:90:.3cm);
\fill[primedregion=white] (-.6,.6) arc (90:180:.3cm) -- (-.9,0) arc (180:270:.9cm) -- (0,-1.2) -- (-1.2,-1.2) -- (-1.2,1.8) -- (-.9,1.8) -- (-.9,1.2) arc (-180:-90:.3cm) -- (-.6,.6);
\end{scope}
\draw[thick, red, dashed] (.3,0) -- (.3,1.8);
\draw[thick, red] (0,0) -- (0,1.8);
\draw (-.6,.9) arc (-90:0:.3cm) arc (180:90:.3cm);
\draw[thick, violet, dashed] (-.6,.9) arc (-90:-180:.3cm) -- (-.9,1.8);
\draw[thick, orange] (0,0) -- (0,-1.2);
\draw[thick, orange, dashed] (-.3,.3) arc (0:90:.3cm) -- (-.6,.9);
\draw (-.6,.6) arc (90:180:.3cm) -- (-.9,0) arc (180:270:.9cm);
\roundNbox{unshaded}{(0,0)}{.3}{.2}{.2}{$x$}
\filldraw (-.6,.6) circle (.05cm);
\filldraw (-.6,.9) circle (.05cm);
\filldraw (0,-.9) circle (.05cm);
\filldraw (0,1.5) circle (.05cm);
}\,.
$$
Finally, this equation holds by definition of the morphism $x$ when restricted to an $A'-B$ bimodule map.
\end{proof}

\section{Distortion and extremality}

In this section, we introduce the notion of the modular distortion for bimodules over finite multifactors.
The notion of distortion for bimodules over a $\rm II_1$ factor $N$ is closely related with the notion of Connes-Takesaki module for an endomorphism of $B(\ell^2)\otimes N$; we refer the reader to Remark \ref{rem:ConnesTakesakiModules} for a detailed discussion.

Using the notion of distortion, we introduce the notion of extremality for multifactor bimodules, and we give many equivalent characterizations.
 We reconcile our definition based on \cite{MR3040370} and the definition for a $\rm II_1$ factor bimodule from \cite[p.51]{MR3178106} in Corollary \ref{cor:DGG-extremality} below.

We connect our definition to extremality of a 
finite index
connected 
inclusion of finite multifactors $A\subset (B,\tr_B)$ with its Markov trace and trace-preserving conditional expectation in \S\ref{sec:ExtremalityAndMinimalExpectation} below.

\subsection{Distortion and extremality for \texorpdfstring{$\rm II_1$}{II1} factor bimodules}

In this section, $M,N,P$ will denote $\rm II_1$ factors, unless stated otherwise.

\begin{defn}
\label{defn:ModularDistortion}
The \emph{modular distortion} of a dualizable $M-N$ bimodule $H$ is defined to be
$$
\delta 
= 
\delta(H) 
:= 
\frac{\sqrt{\vNdim_L(H)}}{\sqrt{\vNdim_R(H)}}.
$$
Observe that
\begin{equation}
\label{eq:Distortion, Jones dim, and von Neumann dim}
\vNdim_L({}_{M}H)
=\delta(H)\Delta(H)
\qquad\qquad
\vNdim_R(H{}_N)
=\frac{\Delta(H)}{\delta(H)},
\end{equation}
where $\Delta(H)$ is the Jones dimension.
Moreover, distortion is multiplicative, i.e., if ${}_MH{}_N$ and ${}_N K{}_P$ are dualizable bimodules, then
\begin{equation}
\label{eq:DistortionMultiplicative}
\delta(H \boxtimes_N K) = \delta(H)\delta(K).
\end{equation}
We say $H$ has \emph{constant distortion} $\delta$ if for every $M-N$ sub-bimodule $K\subseteq H$, we have $\delta(K)=\delta(H)$.
\end{defn}

\begin{prop}
\label{prop:ConstantDistortionIffMinimal}
For a dualizable $M-N$ bimodule $H$, the following are equivalent.
\begin{enumerate}[label=(\arabic*)]
\item $H$ has constant distortion, and
\item $\Delta(H)=D(H)$, i.e., the Jones dimension 
is equal to the statistical dimension.
\end{enumerate}
\end{prop}
\begin{proof}
Since $H$ is dualizable, we may write $H$ as a finite direct sum of simple $M-N$ bimodules $H=\bigoplus_{i=1}^n K_i$.
On the one hand, we have 
$$
D(H)
=
\sum_{i=1}^n D(K_i) 
=
\sum_{i=1}^n \Delta(K_i)
=
\sum_{i=1}^n \sqrt{\vNdim_L(K_i)\vNdim_R(K_i)}
=
\sum_{i=1}^n \delta(K_i)\vNdim_R(K_i),
$$
using the fact that $D$ is additive and the Jones dimension of each simple summand agrees with the statistical dimension. On the other hand, we have
$$
\scalebox{.9}{$
\displaystyle
\Delta(H)
=
\sqrt{\left(
\sum_{i=1}^n\vNdim_L(K_i)\right)
\left(\sum_{i=1}^n\vNdim_R(K_i)
\right)}
=
\sqrt{\left(
\sum_{i=1}^n \delta(K_i)^2 \vNdim_R(K_i)\right)
\left(\sum_{i=1}^n\vNdim_R(K_i)
\right)}.
$}
$$
Therefore, it's immediate that $D(H) = \Delta(H)$ when $H$ has constant distortion.

The reverse implication follows by recognising the expressions for $D(H)$ and $\Delta(H)$ computed above as related by the Cauchy--Schwarz inequality in $\R^n$.
Indeed, if $\mathbf{e}_i\in\bbR^n$ denotes the $i$-th standard basis vector, setting 
$$
\mathbf{x} := \sum_{i=1}^n \delta(K_i)\vNdim{(K_i)}^{1/2}\mathbf{e}_i
\qquad\qquad
\mathbf{y} := \sum_{i=1}^n \vNdim{(K_i)}^{1/2}\mathbf{e}_i,
$$ 
we have $D(H) = \left\langle\mathbf{x},\mathbf{y}\right\rangle$ and $\Delta(H) = \Vert\mathbf{x}\Vert\Vert\mathbf{y}\Vert$.
Suppose now that $D(H) = \Delta(H)$. Then we have $\left\langle\mathbf{x},\mathbf{y}\right\rangle = \Vert\mathbf{x}\Vert\Vert\mathbf{y}\Vert$, and so $\mathbf{x} = \lambda \mathbf{y} $ for some $\lambda \geq 0$.
Hence, $\delta(K_i) = \lambda$ is constant independent of $i$.
Since $\delta$ is constant on all simple summands of $H$, it follows (since $\vNdim_L$ and $\vNdim_R$ are additive) that $\delta$ is constant on all sub-bimodules of $H$.
\end{proof}

\begin{cor}
\label{cor:ConstantDistortionMultiplicativeForFactors}
If ${}_M H{}_N$ and ${}_NK{}_P$ are dualizable $\rm II_1$ factor bimodules with constant distortion, then so is $H\boxtimes_N K$.
\end{cor}
\begin{proof}
Apply Proposition \ref{prop:ConstantDistortionIffMinimal} to the equality
\begin{equation*}
\Delta(H\boxtimes_N K)
=
\Delta(H)\Delta(K)
=
D(H)D(K)
=
D(H\boxtimes_N K).
\qedhere
\end{equation*}
\end{proof}

\begin{rem}
Given a $\rm II_1$ factor $N$, the modular distortion $\delta$ induces a grading on the unitary tensor category $\dBim(N)$ of dualizable $N-N$ bimodules by
$$
\dBim(N) = \bigoplus_{r>0} \delta^{-1}(r)
$$
where $\delta^{-1}(r)$ is the semisimple subcategory of $\dBim(H)$ whose simple objects have distortion $r$.
(Observe that $\delta^{-1}(r)$ may contain only the zero object, so this $\bbR_{>0}$-grading may not be faithful.)
By combining \eqref{eq:DistortionMultiplicative} and Corollary \ref{cor:ConstantDistortionMultiplicativeForFactors}, we see that $\delta^{-1}(r) \otimes \delta^{-1}(s) \subseteq \delta^{-1}(rs)$.
\end{rem}

\begin{defn}[\cite{MR3040370}]
\label{defn:ExtremalityForII1FactorBimodule}
An $N-N$ bimodule $H$ is called \emph{extremal} if for all $N-N$ sub-bimodules $K\subseteq H$, we have
$\vNdim_L(K)=\vNdim_R(K)$.
When $H$ is dualizable, $H$ is extremal if and only if $H$ has constant distortion equal to $1$.
\end{defn}

\begin{cor}
\label{cor:ExtremalImpliesMinimal}
If $H$ is an extremal dualizable $N-N$ bimodule, then $D(H)=\Delta(H)$.
\end{cor}

\begin{rem}
Notice that the converse of Corollary \ref{cor:ExtremalImpliesMinimal} is not true, as any simple non-distortion 1 bimodule is a counterexample.
\end{rem}

Using the universal grading group, we get an extremely short proof of the following result, which can also be deduced from \cite[3.7.1]{MR1055708}.

\begin{prop}
\label{prop:FiniteDepthImpliesExtremal}
If ${}_NH{}_N$ is finite depth, then ${}_NH{}_N$ is extremal.
\end{prop}
\begin{proof}
Since the unitary tensor category $\mathcal{C}(H)$ is fusion, its universal grading group $\mathcal{U}$ is finite \cite[\S4.14]{MR3242743}.
Since $\delta$ induces a grading on $\cC(H)$,
$\delta$ descends to a group homomorphism $\cU \to \bbR_{>0}$ by universality.
Hence $\delta(\mathcal{U}) \subset \mathbb{R}_{>0}$ is a finite group, so it must be $\{1\}$.
\end{proof}

\begin{rem}
\label{rem:ConnesTakesakiModules}
Given a $\rm II_\infty$ factor $M$, in \cite[Rem.~4.6]{MR1953517}, 
Izumi introduces the notion of a scalar-valued module 
$\Mod(\rho)$ for every dualizable endomorphism $\rho \in \End(M)$ determined by the formula
\begin{equation}
\label{eq:ScalarValuedModule}
\Tr_M\circ \rho 
= 
d_\rho \cdot\Mod(\rho)\cdot \Tr_M
\end{equation}
where $\Tr_M$ is any faithful semifinite normal trace on $M$.
When $M=B(\ell^2)\otimes N$ for a $\rm II_1$ factor $N$, there is a well-known equivalence of rigid $\rm C^*$ tensor categories 
$
\End(M) \cup \{0\}
\cong
\Bim(M)
$
\cite[\S3.2]{2004.08271}
and an equivalence
$
\Bim(M) \cong \Bim(N)
$
afforded by the Morita equivalence invertible $M-N$ bimodule ${}_M(\ell^2 \otimes L^2N){}_N$.
Using the composite equivalence
$$
\End(M) \cup \{0\} \cong \Bim(N),
$$
we can relate $\Mod(\rho)$ to the modular distortion $\delta({}_NH{}_N)$ where $H\in \Bim(N)$ is any dualizable bimodule corresponding to the dualizable endomorphism $\rho \in \End(M)$.

Fix a dualizble ${}_NH{}_N \in \Bim(N)$, and let $\Tr_M=\Tr\otimes \tr_N$.
We get a dualizable endomorphism $\rho \in \End(M)$
as follows.
First choose any right $N$-lineay unitary
$u: \ell^2 \otimes H \to \ell^2 \otimes L^2N$,
and observe that for any $x \in M$, $uxu^*$ commutes with the right $N$-action on $\ell^2 \otimes L^2N$ and thus lies in $M$.
Hence $\rho:=\Ad(u) \in \End(M)$.
Moreover, given any other right $N$-linear unitary $v: \ell^2\otimes H \to \ell^2\otimes L^2N$,
we have
$$
(uxu^*)(uv^*) = uxv^* = (uv^*)vxv^*,
$$
so $\Ad(v)$ is equivalent to $\Ad(u)$ in $\End(M)$ via the unitary intertwiner $uv^*$.

Recall now that for any right $N$-linear isometry
$
u : H{}_N \to (\ell^2 \otimes L^2N){}_N,
$
we have
$\Tr_M(uu^*)=\vNdim_R(H{}_N)$
\cite[Prop.~10.1.3 and Def.~10.1.4]{JonesVNA}.
This implies that for the projection $e_{11}\otimes 1_N \in M$, 
we may view $u(e_{11}\otimes 1_N) : e_1 \otimes H \to \ell^2 \otimes L^2N$
as an $N$-linear isometry, and thus
$$
\Tr_M(\rho(e_{11}\otimes 1_N))
=
\Tr_M(u(e_{11}\otimes 1_N)u^*)
=
\vNdim_R(H{}_N).
$$
By \eqref{eq:ScalarValuedModule}, we conclude that 
$$
\Mod(\rho) 
= 
\frac{\vNdim_R(H{}_N)}{d_\rho}
= 
\frac{\vNdim_R(H{}_N)}{D(H)}.
$$
When $D(H) = \Delta(H)$, by \eqref{eq:Distortion, Jones dim, and von Neumann dim}, we have
\begin{equation}
\label{eq:ScalarModuleAndDistortion}
\Mod(\rho) = \delta(H)^{-1}.
\end{equation}

Now again by \cite[Rem.~4.6]{MR1953517}, $\rho$ has a \emph{Connes-Takesaki module}
$\mod(\rho)\in \Aut(Z(\widetilde{M}))$
where $\widetilde{M} := M\rtimes_{\sigma^{\Tr_M}} \bbR$
if and only if the minimal conditional expectation $E_\rho : M \to \rho(M)$ is $\Tr_M$-preserving.
In this case, 
letting $\lambda^{\Tr_M}(t) \in \widetilde{M}$ be the unitary implementing $\sigma^{\tr_M}_t$,
$\mod(\rho)$ is determined by the formula
$$
\mod(\rho)(\lambda^{\Tr_M}(t))
=
\Mod(\rho)^{-it}\cdot \lambda^{\Tr_M}(t).
$$
We claim that
\begin{itemize}
\item 
$\rho$ has a Connes-Takesaki module if and only if $H$ has constant distortion, and
\item
in this case, $\mod(\rho)$ is multiplication by $\delta(H)^{it}$.
\end{itemize}
Indeed, by \cite[Prop.~4.2(2) and Rem.~4.6]{MR1953517}, it suffices to consider the case when $H$ and $\rho$ are simple.
Since $H$ simple implies constant distortion, we have $D(H) = \Delta(H)$, and $\Mod(\rho)=\delta(H)^{-1}$ by \eqref{eq:ScalarModuleAndDistortion}.
\end{rem}

\subsection{Distortion and extremality for \texorpdfstring{$\rm II_1$}{II1} multifactor bimodules}

We now extend the notions of distortion and extremality to bimodules over $\rm II_1$ multifactors.
Let ${}_AX{}_B$ be a connected dualizable bimodule, where we assume Notation \ref{nota:BimoduleNotation}.

\begin{defn}
\label{defn:DistortionForBimodules}
The \emph{modular distortion} of $X$, denoted $\delta = \delta(X)$, is the partially defined $a\times b$ matrix whose $ij$-th entry is given by
$$
\delta_{ij}
:=
\sqrt{\frac{\vNdim_{L}({}_{A_i}(X_{ij}))}{\vNdim_{R}((X_{ij}) {}_{B_j})}}
\qquad\qquad
\text{when }X_{ij}\neq 0.
$$
\end{defn}

The following corollary follows immediately from Proposition \ref{prop:ConstantDistortionIffMinimal}.

\begin{cor}
\label{cor:JonesEqualsStatistical}
The following are equivalent for a dualizable connected $A-B$ bimodule $X$.
\begin{enumerate}[label=(\arabic*)]
\item
For every $i,j$, the bimodule $X_{ij}$ has constant distortion (which may depend on $i$ and $j$).
\item
$D(X) = \Delta(X)$.
\end{enumerate}
\end{cor}

\begin{defn}\label{defn:ExtremalBimod}
We call ${}_AX{}_B$ \emph{extremal} if for all 
$1\leq i\leq a$ and $1\leq j\leq b$, all $A_i-A_i$ and $B_j-B_j$ bimodules generated by $X$ and $\overline{X}$ are extremal in the sense of Definition \ref{defn:ExtremalityForII1FactorBimodule}.
\end{defn}

\begin{rem}
\label{rem:X extremal iff XbarX extremal}
An immediate consequence of Definition \ref{defn:ExtremalBimod} is that
extremality of one of
${}_AX{}_B$,
${}_A X \boxtimes_B \overline{X}{}_A$,
${}_B\overline{X}{}_A$,
or
${}_B\overline{X}\boxtimes_A X{}_B$
is equivalent to extremality of all of these bimodules.
\end{rem}

\begin{lem}
\label{lem:ExtremalImpliesDistortionExtends}
Suppose ${}_AX{}_B$ is extremal.
For any $A-B$ bimodules $Y,Z$ generated by $X$ and $\overline{X}$ such that $Y_{ij}:=p_iYq_j \neq 0\neq p_i Z q_j=:Z_{ij}$, we have
\begin{equation}
\label{eq:ExtendDistortion}
\frac{\vNdim_{L}({}_{A_i}(Y_{ij}))}{\vNdim_{R}((Y_{ij}){}_{B_j})}
=
\frac{\vNdim_{L}({}_{A_i}(Z_{ij}))}{\vNdim_{R}((Z_{ij}) {}_{B_j})}.
\end{equation}
\end{lem}
\begin{proof}
Notice that \eqref{eq:ExtendDistortion} above is equivalent to
$$
\vNdim_{L}({}_{A_i}(Y_{ij}))
\vNdim_{R}((Z_{ij}) {}_{B_j})
=
\vNdim_{L}({}_{A_i}(Z_{ij}))
\vNdim_{R}((Y_{ij}) {}_{B_j}),
$$
which is equivalent to
$$
\vNdim_{L}({}_{A_i}(Y_{ij}))
\vNdim_{L}( {}_{B_j} (\overline{Z}_{ji}))
=
\vNdim_{R}((Y_{ij}) {}_{B_j})
\vNdim_{R}((\overline{Z}_{ji}) {}_{A_i}),
$$
which is equivalent to
$$
\vNdim_{L}({}_{A_i}(Y_{ij}\boxtimes_{B_j} \overline{Z}_{ji}))
=
\vNdim_{R}(Y_{ij} \boxtimes_{B_j} \overline{Z}_{ji})_{A_i}),
$$
which follows by extremality of $X$.
\end{proof}

When ${}_AX{}_B$ is extremal,
Lemma \ref{lem:ExtremalImpliesDistortionExtends} allows us to uniquely extend the definition of $\delta_{ij}$ when $p_iq_j = 0$ by defining 
$$
\delta_{ij}
:=
\sqrt{\frac{\vNdim_{L}({}_{A_i}(Y_{ij}))}{\vNdim_{R}((Y_{ij}) {}_{B_j})}}
$$
for any $A-B$ bimodule $Y$ generated by $X$ and $\overline{X}$ such that $Y_{ij}:= p_iYq_j \neq 0$.

\begin{cor}
\label{cor:DistortionExtendsToGroupoidHom}
Suppose ${}_AX{}_B$ is extremal.
The extension of the distortion function $\delta$ satisfies
\begin{equation*}
\delta_{ij}\delta_{i'j'} = \delta_{ij'}\delta_{i'j}
\qquad\qquad
\forall
\,\,
1\leq i,i' \leq a
\qquad\text{and}\qquad
1\leq j,j'\leq b.
\tag{\ref{eq:DistortionExtensionCondition}}
\end{equation*}
\end{cor}
\begin{proof}
Pick $A-B$ bimodules $U,V,Y,Z$ generated by $X$ and $\overline{X}$ such that 
$U_{ij}$, $V_{i'j'}$, $Y_{ij'}$, and $Z_{i'j}$ are all non-zero.
By Lemma \ref{lem:ExtremalImpliesDistortionExtends},
$\delta_{ij}\delta_{i'j'} = \delta_{ij'}\delta_{i'j}$ 
if and only if
$$
\frac{\vNdim_{L}({}_{A_i}(U_{ij})}{\vNdim_{R}((U_{ij}) {}_{B_j})}
\frac{\vNdim_{L}({}_{A_{i'}}(V_{i'j'}))}{\vNdim_{R}((V_{i'j'}) {}_{B_{j'}})}
=
\frac{\vNdim_{L}({}_{A_i}(Y_{ij'}))}{\vNdim_{R}((Y_{ij'}) {}_{B_{j'}})}
\frac{\vNdim_{L}({}_{A_{i'}}(Z_{i'j}))}{\vNdim_{R}((Z_{i'j}) {}_{B_j})}
$$
which holds if and only if
$$
{}_{A_i} U_{ij} \boxtimes_{B_j} \overline{Z}_{ji'} \boxtimes_{A_{i'}} Vq_{i'j'} \boxtimes_{B_{j'}} \overline{Y}_{j'i} {}_{A_i}
$$
has the same left and right von Neumann dimension, 
which follows by extremality of $X$.
\end{proof}

\begin{lem}
\label{lem:ExtendGraphWeighting}
Suppose we have a connected bipartite graph 
$\Gamma$
with $a$ even vertices and $b$ odd vertices and no double edges.
Suppose we have a weighting $\delta_{ij}\in \bbR_{>0}$ 
for each edge $(i, j)\in \Gamma$.  
The following conditions are equivalent for $\delta$.
\begin{enumerate}[label=(\arabic*)]
\item 
For any cycle 
$(i_1, j_1, i_2, j_2, \dots, i_n, j_n)$ in $\Gamma$ we have 
\begin{equation}
\label{eq:LoopsMultiplyTo1}
\prod_{k=1}^n
\delta_{i_kj_k}
=
\prod_{k=1}^n
\delta_{i_{k+1}j_{k}}
\end{equation}
where indices are taken modulo $n$.
\item
There is a weighting $\eta_i,\xi_j \in \bbR_{>0}$ for each even vertex $i$ and odd vertex $j$ of $\Gamma$ such that $\delta_{ij}= \xi_j/\eta_i$.
Moreover this weighting is unique up to simultaneous uniform scaling of all $\eta_i, \xi_j$.
\item
There is an extension of $\delta$ to the complete bipartite graph $K_{a,b}$ which satisfies Condition \eqref{eq:DistortionExtensionCondition} above.
Moreover, any such extension is unique.
\item
There is an extension of $\delta$ to a groupoid homomorphism $\cG_{a+b}\to \bbR_{>0}$
where $\cG_{a+b}$ is the groupoid consisting of $a+b$ objects and a unique isomorphism between any two objects.
Moreover, any such extension is unique.
\end{enumerate}
\end{lem}
\begin{proof}
\item[$\underline{(1)\Rightarrow (2):}$]
Suppose (1) holds.
Fix an arbitrary even vertex $i_1$ of $\Gamma$, and set $\eta_{i_1}:=1$.
For an arbitrary odd vertex $j$ of $\Gamma$, pick an arbitrary path
$(i_1,j_1,\dots, i_n,j_n)$ from $i_1$ to $j_n=j$ (one exists as $\Gamma$ is connected), and define
$$
\xi_j
:=
\frac{
\prod_{k=1}^n\delta_{i_k j_k}
}{
\prod_{\ell=1}^{n-1}\delta_{i_{\ell+1}j_{\ell}}
}.
$$
By \eqref{eq:LoopsMultiplyTo1}, $\xi_j$ is independent of the choice of path.
Similarly, for an arbitrary even vertex $i$ of $\Gamma$, we pick an arbitrary path
$(i_1,j_1,\dots, i_n,j_n,i_{n+1})$ from $i_1$ to $i_{n+1}=i$, and we define
$$
\eta_{i}
:=
\frac{
\prod_{k=1}^n\delta_{i_k j_k}
}{
\prod_{\ell=1}^{n}\delta_{i_{\ell+1}j_{\ell}}
}
$$
which is again independent of the choice of path by \eqref{eq:LoopsMultiplyTo1}.
It is clear that $\delta_{ij}=\xi_j/\eta_i$ for all $(i ,j)\in \Gamma$ by construction.

Now suppose $\eta_i'$ and $\xi_j'$ is another choice of vertex weighting such that $\delta_{ij}=\xi_j'/\eta_i'$ for all $(i ,j)\in \Gamma$.
Setting $\lambda:=\eta_{i_1}'$, we claim that $\eta_i' = \lambda \eta_i$ and $\xi_j' = \lambda \xi_j$ for all $i,j$.
Indeed, fixing a path $(i_1,j_1,\dots, i_n,j_n)$ from $i$ to $j$ on $\Gamma$, we have
$$
\xi_j 
=
\xi_{j_n} 
= 
\frac{
\prod_{k=1}^n\delta_{i_k j_k}
}{
\prod_{\ell=1}^{n-1}\delta_{i_{\ell+1}j_{\ell}}
}
=
\frac{
\prod_{k=1}^n\frac{\xi_{j_k}'}{\eta_{i_k}'}
}{
\prod_{\ell=1}^{n-1}\frac{\xi_{j_\ell}'}{\eta_{i_{\ell+1}}'}
}
=
\frac{\xi_{j_n}'}{\eta_{i_1}'}
=
\frac{\xi_{j}'}{\lambda}
$$
and thus $\xi_j' = \lambda \xi_j$ as claimed.
Similarly, $\eta_i'=\lambda \eta_i$.

\item[$\underline{(2)\Rightarrow (3):}$]
Suppose (2) holds.
Setting $\delta_{ij}:= \xi_j/\eta_i$ for all $(i,j)\in K_{a,b}$ is an extension of $\delta$ to $K_{a,b}$ which clearly satisfies \eqref{eq:DistortionExtensionCondition} as 
$$
\delta_{ij}\delta_{i'j'}
=
\frac{\xi_j}{\eta_i}
\frac{\xi_{j'}}{\eta_{i'}}
=
\frac{\xi_{j'}}{\eta_i}
\frac{\xi_{j}}{\eta_{i'}}
=
\delta_{ij'}\delta_{i'j}
$$
for all $i,i',j,j'$.

Now suppose $\delta'$ is another weighting on the edges of $K_{a,b}$ 
satisfying \eqref{eq:DistortionExtensionCondition}
such that 
$\delta_{ij}=\delta_{ij}'$ for every $(i,j)\in \Gamma$.
We claim that $\delta'_{ij} = \xi_j/\eta_i$ for all $(i,j)\in K_{a,b}$.
Indeed, picking an arbitrary path  $(i_1,j_1,\dots, i_n,j_n)$ from $i=i_1$ to $j=j_n$ in $\Gamma$, we see
$$
\delta_{ij}'
=
\frac{
\prod_{k=1}^n\delta_{i_k j_k}'
}{
\prod_{\ell=1}^{n-1}\delta_{i_{\ell+1}j_{\ell}}'
}
=
\frac{
\prod_{k=1}^n\delta_{i_k j_k}
}{
\prod_{\ell=1}^{n-1}\delta_{i_{\ell+1}j_{\ell}}
}
=
\frac{
\prod_{k=1}^n\frac{\xi_{j_k}}{\eta_{i_k}}
}{
\prod_{\ell=1}^{n-1}\frac{\xi_{j_\ell}}{\eta_{i_{\ell+1}}}
}
=
\frac{\xi_{j_n}}{\eta_{i_1}}
=
\frac{\xi_{j}}{\eta_{i}}
$$
as claimed.

\item[$\underline{(3)\Rightarrow (4):}$]
Suppose (3) holds.
Setting $\delta_{ii'} := \delta_{ij}/\delta_{i'j}$ for any odd vertex $j$ in independent of the choice of $j$ by  \eqref{eq:DistortionExtensionCondition}.
We define $\delta_{jj'}$ analogously.
It is straightforward to verify this is the only possible extension of 
$\delta$ to $\cG_{a+b}$ satisfying $\delta_{rs}\delta_{st} = \delta_{rt}$ for all $r,s,t$.

\item[$\underline{(4)\Rightarrow (1):}$]
Suppose $\delta$ extends to a groupoid homomorphism.
For any cycle $(i_1,j_1,\dots, i_n,j_n)$ in $\Gamma$, we have
\begin{equation}
\label{eq:GroupoidLoopProductIs1}
\prod_{k=1}^n \delta_{i_k j_k}\delta_{j_k i_{k+1}}
=
1
\end{equation}
where indices are taken modulo $n$.
Since $\delta_{ji}= \delta_{ij}^{-1}$,
the equation \eqref{eq:GroupoidLoopProductIs1} is equivalent to \eqref{eq:LoopsMultiplyTo1}.
\end{proof}

Similar techniques prove the following useful corollary.

\begin{cor}
\label{cor:EquivalentGroupoidHomCharacterizations}
The following are equivalent for a matrix $\delta \in M_n(\bbR_{>0})$.
\begin{enumerate}[label=(\arabic*)]
\item
$\delta$
gives a groupoid homomorphism $\cG_n \to \bbR_{>0}$, i.e.,
$\delta_{ij}\delta_{jk}=\delta_{ik}$ for all $1\leq i,j,k\leq n$.
\item
There are $(\lambda_i)_{i=1}^n \in \bbR^n_{>0}$ such that
$\delta_{ij}=\lambda_j / \lambda_i$.
Moreover this weighting is unique up to simultaneous uniform scaling of all $\lambda_i$.
\end{enumerate}
\end{cor}
\begin{proof}
Obviously $(2)\Rightarrow (1)$.
Suppose $(1)$ holds.
Setting $\lambda_i := \delta_{1i}$, we have
$\delta_{ij} = \delta_{i1}\delta_{1j} = \delta_{1j}/\delta_{1i}$ as desired.
If $(\eta_i)\in \bbR^n_{>0}$ is any other such vector such that $\delta_{ij} = \eta_j/\eta_i= \lambda_j /\lambda_i$, then clearly
$\lambda_j/\eta_j = \lambda_i/\eta_i$ for all $1,\leq i,j\leq n$.
Hence $(2)$ holds.
\end{proof}

\begin{thm}
\label{thm:ExtremalCharacterization}
The following are equivalent for a connected  dualizable bimodule ${}_AX{}_B$.
\begin{enumerate}[label=(\arabic*)]
\item  
$X$ is extremal
\item
$X_{ij}$ has constant distortion $\delta_{ij}$ whenever $X_{ij}\neq 0$, and $\delta$ satisfies \eqref{eq:LoopsMultiplyTo1} above.
\item
$D(X)=\Delta(X)$ and $\delta$ satisfies \eqref{eq:LoopsMultiplyTo1} above.
\end{enumerate}
\end{thm}
\begin{proof}
\item[$\underline{(1)\Rightarrow (2):}$]
Suppose $X$ is extremal.
That $X_{ij}$ has constant distortion follows immediately from Lemma \ref{lem:ExtremalImpliesDistortionExtends}.
Define a bipartite graph $\Gamma=\Gamma(X)$ with $a$ even vertices and $b$ odd vertices, where $(i,j)\in \Gamma$ if and only if $X_{ij} \neq 0$.
By Corollary \ref{cor:DistortionExtendsToGroupoidHom}, the extension of $\delta$ to $K_{a,b}$ satisfies \eqref{eq:DistortionExtensionCondition}.
Finally, we may apply Lemma \ref{lem:ExtendGraphWeighting} to see \eqref{eq:LoopsMultiplyTo1} holds for $\delta$ restricted back to $\Gamma$.

\item[$\underline{(2)\Rightarrow (1):}$]
Suppose (2) holds.
Then as $\delta$ satisfies condition \eqref{eq:LoopsMultiplyTo1}, 
by Lemma \ref{lem:ExtendGraphWeighting}, $\delta$ extends uniquely to a groupoid homomorphism $\cG_{a+b}\to \bbR_{>0}$, so $\delta_{ii}=1=\delta_{jj}$ for all $1\leq i\leq a$ and $1\leq j\leq b$.
Now since $X$ has constant distortion, 
given any loop $(i_1,j_1,\dots, i_n,j_n)$ with $i_1=i$ in $\Gamma$, the $A_i-A_i$ bimodule
$$
X_{i_1j_1}
\boxtimes_{B_{j_i}} 
\overline{X}_{j_1i_2}
\boxtimes_{A_{i_2}}
\cdots
\boxtimes_{A_{i_n}}
X_{i_nj_n}
\boxtimes_{B_{j_n}}
\overline{X}_{j_ni_1}
$$
has constant distortion equal to $\delta_{ii}=1$ by Corollary \ref{cor:ConstantDistortionMultiplicativeForFactors}.

Moreover this constant distortion must be equal to 1 independent of the choice of loop of length $2n$ starting at $i$
by \eqref{eq:LoopsMultiplyTo1}.
Hence 
$$
p_i(X\boxtimes_B \overline{X})^{\boxtimes_A n}p_i
=
\bigoplus_{
\substack{
\text{loops of length $2n$}
\\
\text{based at $i$}
}}
X_{i_1j_1}
\boxtimes_{B_{j_i}} 
\overline{X}_{j_1i_2}
\boxtimes_{A_{i_2}}
\cdots
\boxtimes_{A_{i_n}}
X_{i_nj_n}
\boxtimes_{B_{j_n}}
\overline{X}_{j_ni_1}
$$
has constant distortion equal to 1 for all $n\in\bbN$ and $1\leq i\leq a$.
Similarly, 
$q_j(\overline{X}\boxtimes_A X)^{\boxtimes_B n}q_j$ has constant distortion 1 for all $n\in\bbN$ and $1\leq j\leq b$.
Thus $X$ is extremal.

\item[$\underline{(2)\Rightarrow (3):}$]
This follows immediately by Corollary \ref{cor:JonesEqualsStatistical}.
\end{proof}

As a corollary, our definition of extremality agrees with that from \cite[p.51]{MR3178106} in the case of a dualizable bimodule over $\rm II_1$ factors.

\begin{cor}
\label{cor:DGG-extremality}
Suppose $M,N$ are $\rm II_1$ factors and ${}_M H{}_N$ is a dualizable $M-N$ bimodule.
The following are equivalent.
\begin{itemize}
\item 
${}_MH{}_N$ is extremal in the sense of Definition \ref{defn:ExtremalBimod}.
\item
${}_MH{}_N$ is extremal in the sense of \cite[p.51]{MR3178106}, i.e., the traces $\tr_{N'}$ and $\tr_{M'}$ agree on $N'\cap M' \cap B(H)$.
\end{itemize}
\end{cor}
\begin{proof}
First, we note that \eqref{eq:LoopsMultiplyTo1} is always satisfied when $M,N$ are factors, so $H$ is extremal in the sense of Definition \ref{defn:ExtremalBimod} if and only if it has constant distortion.
Second, we note that $\tr_{N'}=\tr_{M'}$ on $N'\cap M'$ if and only if they agree on projections.
Now observe that for every projection $p\in N'\cap M'$, we have
$$
\delta(pH)
=
\left(
\frac{\vNdim_L({}_M (pH))}{\vNdim_R( (pH){}_N)}
\right)^{1/2}
=
\left(
\frac{\tr_{M'}(p)\vNdim_L({}_MH)}{\tr_{N'}(p)\vNdim_R(H{}_N)}
\right)^{1/2}
=
\delta(H)
\left(
\frac{\tr_{N'}(p)}{\tr_{M'}(p)}
\right)^{1/2}.
$$
Thus $\delta(pH)= \delta(H)$ if and only if $\tr_{N'}(p) = \tr_{M'}(p)$.
We conclude that $H$ has constant distortion if and only if $\tr_{N'}=\tr_{M'}$.
\end{proof}

\begin{cor}
\label{cor:FiniteDepthMultifactorBimoduleExtremal}
Every finite depth connected dualizable bimodule ${}_AX{}_B$ is extremal.
\end{cor}
\begin{proof}
We may apply Proposition \ref{prop:FiniteDepthImpliesExtremal} to the diagonal component unitary fusion categories $\cC(X)_{kk}$ of the standard invariant $\cC(X)\subset \Bim(A\oplus B)$ to see that every $A_i-A_i$ and $B_j-B_j$ bimodule generated by $X$ is extremal.
\end{proof}

\subsection{Distortion and extremality for multifactor inclusions}
\label{sec:ExtremalityAndMinimalExpectation}

In this section, we study the special case of the bimodule $X = {}_AL^2B{}_B$ for a 
finite index connected inclusion $A\subset B$ of finite multifactors.
Our main goal for this section is to characterize extremality of ${}_AL^2B{}_B$ in terms of the Markov and minimal conditional expectations.

\begin{nota}
\label{nota:InclusionNotation}
For this section, fix a connected finite index inclusion $A\subset B$ of finite multifactors.
In addition to assuming Notation \ref{nota:BimoduleNotation}, we fix the following notation.
We write $\Delta=\Delta(X)$ for the Jones dimension matrix and $D=D(X)$ for the statistical dimension matrix.
Since $D^TD$ has positive entries and is irreducible, by the Frobenius-Perron Theorem, there are unique \emph{row} vectors $\vec{\alpha}\in \bbR^a_{>0}$ and $\vec{\beta}\in \bbR^b_{>0}$ with $\|\vec{\alpha}\|_2 = 1 = \|\vec{\beta}\|_2$ such that
$\vec{\alpha}D = d_X \vec{\beta}$ and $\vec{\beta}D^T  = d_X \vec{\alpha}$ where $d_X^2$ is the largest eigenvalue of $D^TD$ \cite{MR3994584}.
\end{nota}

\begin{fact}
We can now express the trace matrices in terms of the modular distortion.
First, observe that $p_iq_j =p_i Jq_jJ$, so $p_iq_j \neq 0$ if and only if $Jp_iq_jJ = Jp_iJq_j \neq 0$.
Combining 
\eqref{eq:TraceMatrix},
\eqref{eq:TraceTildeMatrix}, and
\eqref{eq:TraceMatrixForBasicConstruction} 
respectively with 
\eqref{eq:Distortion, Jones dim, and von Neumann dim}, 
whenever $p_iq_j \neq 0$,
\begin{equation}
\label{eq:TraceMatricesInTermsOfDistortions}
T_{ij}=\frac{\Delta_{ij}}{\delta_{ij}},
\qquad\qquad
\widetilde{T}_{ji}=\delta_{ij}\Delta_{ij},
\qquad\qquad\text{and}\qquad\qquad
T^{B\subset \langle B, A\rangle}_{ji} =
\frac{\delta_{ij}\Delta_{ij}}
{\sum_{k=1}^b \delta_{ik}\Delta_{ik}}.
\end{equation}
\end{fact}

\begin{defn}
Given a conditional expectation $E:B \to A$, we define $\lambda_{ij}=\lambda^E_{ij}\in [0,\infty)$ by the formula
$$
\lambda_{ij} p_i = E(p_i q_j).
$$
The conditional expectation $E$ induces 
conditional expectations $E_{ij}: p_iq_jBp_iq_j \to p_iq_j A$ whenever $p_iq_j \neq0$ by
$$
E_{ij}(p_iq_jbp_iq_j):=\lambda_{ij}^{-1}E(bq_j)p_iq_j.
$$
We will be most interested in the unique Markov trace-preserving conditional expectation $E^{\rm Markov}: B \to A$
and the unique minimal expectation $E^0:B \to A$ \cite{MR976765,MR1027496,MR3994584}.
By \cite[Lem.~3.4]{MR1086543}, the Markov trace-preserving expectation is uniquely determined by $E^{\rm Markov}_{ij}$ is trace-preserving whenever $p_iq_j\neq 0$, and
\begin{equation}
\label{eq:LambdaMarkov}
\lambda^{\rm Markov}_{ij}
=
T_{ij}\left(\frac{\tr_B(q_j)}{\tr_A(p_i)}\right)
\underset{\text{\eqref{eq:TraceMatricesInTermsOfDistortions}}}{=}
\left(\frac{\Delta_{ij} }{\delta_{ij}}\right) \left(\frac{\tr_B(q_j)}{\tr_A(p_i)}\right)
\qquad
\forall\,p_iq_j\neq0
\end{equation}
where $\tr_B$ is the unique Markov trace and $\tr_A=\tr_B|_A$.
By \cite[Thm.~2.6]{MR3994584} and \cite[Thm.~2.3]{1908.09121},
the minimal expectation is uniquely determined by
\begin{equation}
\label{eq:UniqueMarkovAndMinimalCE}
\Ind(E^0_{ij})=D_{ij}
\qquad
\text{and}
\qquad
\lambda^0_{ij} = \frac{D_{ij}\beta_j}{d\alpha_i}
\qquad\qquad
\forall\,i,j.
\end{equation}
\end{defn}

\begin{thm}
\label{thm:Minimal=Extremal}
Suppose $A\subset B$ is a finite index connected inclusion of finite multifactors.
The following are equivalent:
\begin{enumerate}[label=\rm (E\arabic*)]
\item 
\label{Extremal:Markov=Minimal}
$E^{\rm Markov}=E^0$, i.e., $A\subset B$ is extremal in the sense of \cite[Def.~4.1]{MR3994584}.
\item
\label{Extremal:Distortion}
$D=\Delta$ and the Markov trace $\tr_B$ on $B$ 
and its restriction $\tr_A:=\tr_B|_A$ satisfy
\begin{equation}
\label{eq:DistortionCharacterizationOfMarkovTrace}
\delta_{ij}
=
d \left(\frac{\tr_B(q_j)}{\beta_j}\right) \left(\frac{\alpha_i}{\tr_A(p_i)}\right)
\qquad\qquad
\forall\,p_iq_j \neq 0.
\end{equation}
\item
\label{Extremal:Bimodule}
${}_AL^2B{}_B$ is extremal.
\end{enumerate}
\end{thm}
\begin{proof}
\mbox{}
\item[\underline{\ref{Extremal:Markov=Minimal}$\Rightarrow$\ref{Extremal:Distortion}:}]
Suppose that $E^{\rm Markov}$ is minimal.
By \eqref{eq:UniqueMarkovAndMinimalCE},
$\Ind(E_{ij})=\Delta_{ij}=D_{ij}$ for all $i,j$, so $\Delta=D$.
Combining \eqref{eq:UniqueMarkovAndMinimalCE} and \eqref{eq:LambdaMarkov},
whenever $p_iq_j \neq 0$,
$$
\frac{D_{ij} \beta_j}{d \alpha_i}
\underset{\text{\eqref{eq:UniqueMarkovAndMinimalCE}}}{=}
\lambda^{\rm Markov}_{ij} 
\underset{\text{\eqref{eq:LambdaMarkov}}}{=}
\left(\frac{D_{ij}}{\delta_{ij}}\right)\left(\frac{\tr_B(q_j)}{\tr_A(p_i)}\right).
$$
Solving for $\delta_{ij}$ gives \eqref{eq:DistortionCharacterizationOfMarkovTrace}.

\item[\underline{\ref{Extremal:Distortion}$\Rightarrow$\ref{Extremal:Markov=Minimal}:}]
Since $D=\Delta$, $\Ind(E_{ij})=\Delta_{ij}=D_{ij}$ for all $i,j$.
Now combining \eqref{eq:LambdaMarkov}, \eqref{eq:DistortionCharacterizationOfMarkovTrace}, and $D=\Delta$,
we see that
$$
\lambda^{\rm Markov}_{ij} 
\underset{\text{\eqref{eq:LambdaMarkov}}}{=}
\left(\frac{D_{ij}}{\delta_{ij}}\right)\left(\frac{\tr_B(q_j)}{\tr_A(p_i)}\right)
\underset{\text{\eqref{eq:DistortionCharacterizationOfMarkovTrace}}}{=}
\frac{D_{ij} \beta_j}{d \alpha_i}.
$$
By \eqref{eq:UniqueMarkovAndMinimalCE}, $E^{\rm Markov}$ is the unique minimal expectation.

\item[\underline{\ref{Extremal:Distortion}$\Rightarrow$\ref{Extremal:Bimodule}:}]
By Lemma \ref{lem:ExtendGraphWeighting}, $\delta$ satisfies \eqref{eq:LoopsMultiplyTo1}.
The result now follows by Theorem \ref{thm:ExtremalCharacterization}.

\item[\underline{\ref{Extremal:Bimodule}$\Rightarrow$\ref{Extremal:Distortion}:}]
Suppose ${}_AL^2B{}_B$ is extremal.
By Theorem \ref{thm:ExtremalCharacterization}, $D=\Delta$ and $\delta$ satisfies \eqref{eq:LoopsMultiplyTo1}.
By Lemma \ref{lem:ExtendGraphWeighting}, there are vectors $\eta\in \bbR_{>0}^a$ and $\xi\in \bbR_{>0}^b$, unique up to uniformly scaling both $\eta,\xi$ by the same positive scalar, such that $\delta_{ij} = \xi_j/\eta_i$ for all $1\leq i \leq a$ and $1\leq j\leq b$. 

Now choose the unique trace $\tr_B$ on $B$ such that
$\tr_B(q_j) = \xi_j \beta_j /d$, where we have replaced $\eta,\xi$ with the unique uniform scaling such that $\sum_{j=1}^b \tr_B(q_j) = 1$.
Set $\tr_A:=\tr_B|_A$.
We now calculate that
$$
\tr_A(p_i)
=
\sum_{j=1}^bT_{ij}\tr_B(q_j)
\underset{\text{\eqref{eq:TraceMatricesInTermsOfDistortions}}}{=}
\sum_{j=1}^b \left(\frac{D_{ij}}{\delta_{ij}}\right) \frac{\xi_j \beta_j}{d}
=
\frac{1}{d}\sum_{j=1}^b D_{ij}\left(\frac{\eta_i}{\xi_j}\right) \xi_j \beta_j
=
\frac{\eta_i}{d}\sum_{j=1}^b D_{ij} \beta_j
=
\eta_i \alpha_i
$$
which implies that $\eta_i = \tr_A(p_i)/ \alpha_i$.
Hence we have
$$
\delta_{ij}
=
\frac{\xi_j}{\eta_i}
=
d
\left(
\frac{\tr_B(q_j)}{\beta_j}
\right)
\left(
\frac{\alpha_i}{\tr_A(p_i)}
\right).
$$
To see that $\tr_B$
is the Markov trace, we verify \eqref{eq:MarkovTraceEigenvalueEquations}:
\begin{equation*}
\sum_{i=1}^a \widetilde{T}_{ji} \tr_{A}(p_i)
\underset{\text{\eqref{eq:TraceMatricesInTermsOfDistortions}}}{=}
\sum_{i=1}^a 
\delta_{ij}D_{ij} \tr_A(p_i)
=
d
\left(
\frac{\tr_B(q_j)}{\beta_j}
\right)
\sum_{i=1}^a 
D_{ij}
\alpha_i
=
d^2\tr_B(q_j).
\qedhere
\end{equation*}
\end{proof}

In light of Theorem \ref{thm:Minimal=Extremal}, we make the following definition.

\begin{defn}
A connected finite index multifactor inclusion is called \emph{extremal} if either of the equivalent conditions of Theorem \ref{thm:Minimal=Extremal} holds.
\end{defn}

\begin{cor}
The connected dualizble multifactor bimodule ${}_AX{}_B$ is extremal
if and only if
the 
finite index
connected 
inclusion of finite multifactors
$A \subset M:=(B^{\op})'\cap B(X)$ is extremal.
\end{cor}
\begin{proof}
By \cite[Prop.~3.1]{MR703809}, we have a canonical $M-M$ bimodule isomorphism $X\boxtimes_B \overline{X} \cong L^2M$, which restricts to an $A-M$ bimodule isomorphism.
By Remark \ref{rem:X extremal iff XbarX extremal},
${}_AX{}_B$ is extremal
if and only if
${}_AX \boxtimes_B \boxtimes \overline{X}{}_A$ is extremal
if and only if
${}_A L^2M {}_A$ is extremal
if and only if
${}_A L^2M \boxtimes_M L^2M {}_A$ is extremal
if and only if
${}_A L^2M {}_M$ is extremal.
\end{proof}

\section{Modular distortion and the Jones tower/tunnel}
\label{section:connectedinclusions}

Let $A\subset B$ be a connected finite index inclusion  of finite multifactors, and assume Notations \ref{nota:BimoduleNotation} and \ref{nota:InclusionNotation}.
In \S\ref{sec:DistortionDynamicalForBasicConstruction} and \ref{sec:DistortionForDownwardBasicConstruction} respectively, we characterize the behavior of the modular distortion under basic construction and downward basic construction.
We then make connections to Popa's homogeneity \cite[Def.~1.2.11]{MR1339767} in \S\ref{sec:Homogeneity}, and we show that the notion of \emph{super-extremality} \cite[Def.~4.1]{MR3994584} for a finite index extremal connected Markov inclusion $A\subset (B, \tr_B)$ of finite multifactors is exactly Popa's homogeneity.

\subsection{The distortion dynamical system for Jones' basic construction}
\label{sec:DistortionDynamicalForBasicConstruction}

We now compute the behavior of the distortion under taking Jones' basic construction.
Recall that $p_iq_j \neq 0$ if and only if $Jp_iJq_j \neq 0$, and that the Jones dimension matrix for $B\subset \langle B , A\rangle$ is $\Delta^T$ where $\Delta$ is the Jones dimension matrix of $A\subset B$ \cite[Prop.~3.6.6]{MR999799}.

\begin{prop}
The distortion 
$\delta^{B\subset \langle B,A\rangle}= \delta({}_B L^2\langle B, A\rangle {}_{\langle B, A\rangle})$
is related to $\delta=\delta({}_AL^2B{}_B)$ by
\begin{equation}
\label{eq:DistortionBasicConstructionFormula}
\delta^{B\subset \langle B,A\rangle}_{ji} = \frac{1}{\delta_{ij}}\sum_{k=1}^b \delta_{ik}\Delta_{ik}
\qquad\qquad
\forall\, p_iq_j \neq 0 \Leftrightarrow Jp_iJq_j\neq0.
\end{equation}
\end{prop}
\begin{proof}
Whenever $p_iq_j \neq 0 \Leftrightarrow Jp_iJq_j \neq 0$, 
by using \eqref{eq:TraceMatricesInTermsOfDistortions} twice, 
we have
\begin{equation*}
\delta^{B \subset \langle B, A\rangle}_{ji}
=
\frac{\Delta^T_{ji}}{T^{B \subset \langle B, A\rangle}_{ji}}
=
\delta_{ij}^{-1} \sum_{k=1}^b \delta_{ik}\Delta_{ik}.
\qedhere
\end{equation*}
\end{proof}


\begin{assumption}
We assume for the remainder of this section that $A\subset B$ is extremal, i.e., ${}_AL^2B {}_B$ is extremal so that
$\Delta = D$ and $\delta$ satisfies the equivalent conditions of Lemma \ref{lem:ExtendGraphWeighting}.
In particular, there are row vectors $\vec{\eta}=(\eta_i)_{i=1}^a\in \bbR_{>0}^a$ and $\vec{\xi}=(\xi_j)_{j=1}^b\in \bbR_{>0}^b$, unique up to uniform scaling both $\eta,\xi$ by the same positive scalar, such that $\delta_{ij} = \xi_j/\eta_i$ for all $1\leq i \leq a$ and $1\leq j\leq b$.
\end{assumption}

We now iterate the basic construction one more time to get the first 4 algebras of the Jones tower: $A=A_0 \subset B=A_1 \subset A_2 \subset A_3$.

\begin{cor}
The distortions $\delta^{A_1\subset A_2}= \delta({}_{A_1} L^2A_2 {}_{A_2})$ and $\delta^{A_2\subset A_3}= \delta({}_{A_2} L^2A_3 {}_{A_3})$ are related to $\delta = \delta({}_AL^2B {}_B)$ by
\begin{equation}
\label{eq:DistortionOrder1and2Formula}
\delta^{A_1\subset A_2}_{ji} 
= 
\frac{(\vec{\xi} D^T)_i}{(\vec{\xi})_j},
\qquad
\delta^{A_2\subset A_3}_{ij} 
= 
\frac{(\vec{\xi} D^TD)_j}{(\vec{\xi} D^T)_i}.
\end{equation}
(Observe that the above formulas are independent of the choice of $\vec{\xi}$ up to uniform positive scaling.)
\end{cor}
\begin{proof}
The first equality is proven by substituting $\delta_{ij} = \xi_j/\eta_i$ in \eqref{eq:DistortionBasicConstructionFormula}:
\begin{equation*}
\delta^{A_1\subset A_2}_{ji} = \frac{\eta_i}{\xi_j}\sum_{k=1}^b \frac{\xi_k}{\eta_i}D_{ik} = \frac{1}{\xi_j}\sum_{k=1}^b \xi_k D_{ik}
\qquad
\forall\, p_iq_j \neq 0 \Leftrightarrow Jp_iJq_j\neq0
\end{equation*}
and then extending $\delta^{A_1\subset A_2}_{ji}$ for every $j,i$ as in Lemma \ref{lem:ExtendGraphWeighting}.
The second equality follows by iteration with $D^T = D({}_{A_1} L^2A_2 {}_{A_2})$ in place of $D = D({}_{A} L^2B {}_{B})$ and $\vec{\xi}D^T$ in place of $\vec{\xi}$.
\end{proof}

\begin{thm}
\label{thm:UniqueFixedPoint}
Consider the topological space
$$
\set{(\delta_{ij})\in \operatorname{Mat}_{a\times b}(\bbR_{>0})}{ \delta_{ij}\delta_{i'j'} = \delta_{ij'}\delta_{i'j} \quad\forall\, 1\leq i,i' \leq a,\quad \forall\, 1\leq j,j'\leq b} \cong \bbR_{>0}^{a+b-1}.
$$
The distortion dynamical system
\begin{equation}
\label{eq:DistortionDynamicalSystem}
  (\delta_{ij}) \xmapsto{\Phi} \left(\frac{(\vec{\xi}D^TD)_j}{(\vec{\xi}D^T)_i}\right)
\end{equation}
where $\vec{\xi}\in \bbR^b_{>0}$ is as in Lemma \ref{lem:ExtendGraphWeighting}(2) for $\delta$,
has a unique fixed point $\sigma$ given by $\sigma_{ij}:= d_X \beta_j / \alpha_i$,
where $\vec{\alpha}$, $\vec{\beta}$, $d_X$ are as in Notation \ref{nota:InclusionNotation}.
\end{thm}
\begin{proof}
Arguing by induction, we see that iterating this dynamical system $n$ times starting with $\delta_{ij}= \xi_j / \eta_i$ yields
$$
\Phi^n(\delta)
=
\left(\frac{(\vec{\xi}(D^TD)^n)_j}{(\vec{\xi}(D^TD)^{n-1}D^T)_i}\right).
$$
Since $\delta_{ij}>0$ for all $1\leq i\leq a$ and $1\leq j\leq b$,
it is well-known that
$$
d_X^{-2n} \vec{\xi}(D^TD)^n \xrightarrow{n\to \infty} c\vec{\beta}
$$
for some $c>0$ as the matrix $d_X^{-2n}(D^TD)^n$ tends to the projection onto $\bbR\vec{\beta}$.
This means that
given any input $\delta$ to the dynamical system, iterating yields
$$
\Phi^n(\delta)
\xrightarrow{n\to \infty}
\frac{c (\vec{\beta} D^TD)_j}{c (\vec{\beta} D^T)_i} = \frac{d_X^2 \beta_j}{d_X\alpha_i} = \frac{d_X\beta_j}{\alpha_i}.
$$
Hence there is at most one fixed point given by $\sigma_{ij}:= d_X\beta_j/\alpha_i$.
One immediately sees that this $\sigma$ is indeed a fixed point under the dynamical system $\Phi$.
\end{proof}

\begin{defn}
Given a connected, extremal $A-B$ bimodule ${}_AX{}_B$
we say that $X$ has \emph{standard distortion} if $\delta = \sigma$, the unique fixed point under the distortion dynamical system \eqref{eq:DistortionDynamicalSystem}.
\end{defn}

\subsection{Distortion and the downward basic construction}
\label{sec:DistortionForDownwardBasicConstruction}

We now 
compare the various notions of downward basic construction 
for a
finite index 
connected
finite
multifactor inclusion $A\subset B$ 
that appear in the literature.

\begin{defn}
Suppose $A\subset B$ is a finite index connected finite
multifactor inclusion.
Let $\tr_B$ be the unique Markov trace on $B$ and $d^2$ the Markov index.
\begin{itemize}
\item 
A \emph{trace independent downward basic construction}
consists of a
von Neumann subalgebra $C \subset A$ and a unital $*$-algebra isomorphism $\psi:B\to JC'J \cap B(L^2A)$ such that $\psi|_{A}=\id_{A}$.
\item
A \emph{Jones downward basic construction}
cf.~\cite[Lem.~3.1.8]{MR0696688}
consists of a projection $f\in B$ with central support 1
together with a left $A$-module unitary $u : L^2A \to L^2B f$.
\item
A \emph{Popa downward basic construction}
cf.~\cite[Prop.~1.2.7]{MR1339767}
consists of a projection $e\in B$ with $E_A(e)=d^{-2}$
such that
setting
$C:=\{e\}'\cap A$,
$C\subset (A,\tr_A) \subset (B, \tr_B, e)$ 
is the Jones basic construction, i.e., the map $ae_C b \mapsto aeb$ for $a,b\in A$ extends to a trace-preserving isomorphism $\langle A, C\rangle \cong B$.
Observe that in this case,
$E_C(a):=d^{2}E_A(eae)$ is the unique Markov trace-preserving conditional expectation;
indeed, 
setting $\tr_C:=\tr_A|_C$,
for all $a\in A$ and $c\in C$,
\begin{equation}
\label{eq:TunnelExpectation}
\begin{split}
\tr_C(E_C(a)c)
&=
\tr_A(d^{2}E_A(eae)c)
=
d^{2}\tr_B(eaec)
\\&=
d^{2}\tr_B(eac)
=
d^{2}\tr_B(E_A(e)ac)
=
\tr_A(ac).
\end{split}
\end{equation}
\end{itemize}
\end{defn}

It is relatively straightforward to show that a 
Popa downward basic construction 
is a 
Jones downward basic construction
which, in turn, is a
trace independent basic construction.
The 
main result of this section
builds on \cite[Prop.~1.2.7]{MR1339767}
to prove the equivalence of all these notions
and further quantifies the existence of a downward basic construction in terms of the distortion $\delta$.

\begin{lem}
\label{lem:CutDownDimension}
Suppose $A\subset B$ is a
finite index
connected
inclusion of finite multifactors.
For any projection $f\in B$,
$$
\vNdim_L({}_{A_i}(p_i L^2B fq_j))
=
\pi_j \delta_{ij} \Delta_{ij}
\qquad\qquad
\forall\, 1\leq i\leq a
$$
where $\pi =\tr^Z_B(f)$, i.e, $\pi_j = \tr_{B_j}(f q_j)$.
\end{lem}
\begin{proof}
Fix $1\leq i \leq a$, and recall that
$\delta_{ij} \Delta_{ij} 
= 
\vNdim_L( {}_{A_i}(p_i L^2B q_j) )$ by \eqref{eq:Distortion, Jones dim, and von Neumann dim}.
Thus by \cite[(2.1.3)]{MR0696688}, 
$$
\vNdim_L( {}_{A_i}(p_iL^2B f q_j) )
=
\tr_{A_i'}(Jfq_jJ)
\delta_{ij} \Delta_{ij}
=
\tr_{JA_i'J}(fq_j)
\delta_{ij} \Delta_{ij}.
$$
But since $B_j\subset JA_i'J$ is a $\rm II_1$ subfactor (whose unit is also $q_j$), we must have that
$$
\tr_{JA_i'J}(fq_j) 
= 
\tr_{B_j}(f q_j) 
= 
\pi_j
$$
by the uniqueness of the trace on $B_j$.
The result follows.
\end{proof}

\begin{fact}
\label{fact:CenterValuedTraceOfJonesProjection}
Since the minimal central projections of the basic construction $\langle B, A\rangle$ are the $Jp_iJ$,
by \eqref{eq:TraceVectorForBasicConstruction},
the center-valued trace of $e_A\in \langle B, A\rangle$ has $i$-th component
$$
\tr_{\langle B, A \rangle_i}(e_A Jp_iJ)
=
\frac{
\tr_{\langle B, A \rangle}(e_A Jp_iJ)
}{
\tr_{\langle B, A \rangle}(Jp_iJ)
}
=
\frac{
\tr_{\langle B, A \rangle}(e_A p_i)
}{
\tr_{\langle B, A \rangle}(Jp_iJ)
}
=
\frac{
d^{-2}\tr_{A}(p_i)
}{
\tr_{\langle B, A \rangle}(Jp_iJ)
}
\underset{\text{\eqref{eq:TraceVectorForBasicConstruction}}}{=}
\frac{1}{\sum_{j=1}^b \delta_{ij}\Delta_{ij}}.
$$
\end{fact}

\begin{prop}
\label{prop:JonesDownwardEquivalentConditions}
Let $A\subset B$ be a finite index connected inclusion of finite multifactors, and let $\tr_B$ be the unique Markov trace.
For a projection $f\in B$, the following are equivalent.
\begin{enumerate}[label=(\arabic*)]
\item
$L^2A \cong L^2Bf$ as left $A$-modules.
\item
the vector $\pi:=\tr^Z_B(f) \in [0,1]^b$ satisfies
$\sum_{j=1}^b \pi_j \delta_{ij}\Delta_{ij} = 1$ for all $1\leq i\leq a$.
\item
$f$ is equivalent to $e_A$ in $\langle B, A\rangle \subset B(L^2(B, \tr_B))$.
\end{enumerate}
\end{prop}
\begin{proof}
\item[\underline{$(1)\Leftrightarrow (2)$}]
This follows immediately from Lemma \ref{lem:CutDownDimension}:
\begin{align*}
\vNdim_L({}_{A_i}(p_iL^2Bf))
&=
\vNdim_L\left(
\bigoplus_{j=1}^b
{}_{A_i}(p_iL^2B f q_j)
\right)
\\&=
\sum_{j=1}^b 
\vNdim_L({}_{A_i}(p_iL^2B f q_j))
=
\sum_{j=1}^b 
\pi_j \delta_{ij} \Delta_{ij}.
\end{align*}
Hence $L^2Bf \cong L^2A$ as left $A$-modules if and only if
$\sum_{j=1}^b \pi_j \delta_{ij}\Delta_{ij} = 1$ for all $1\leq i\leq a$.

\item[\underline{$(2)\Rightarrow (3)$}]
Whenever $q_j Jp_iJ \neq 0$,
by uniqueness of the trace on $B_j$,
for all $b\in B$,
$$
\tr_{B_j}(bq_j)
=
\frac{
\tr_{\langle B, A\rangle_i}(b q_j Jp_iJ)
}
{\tr_{\langle B, A\rangle_i}(q_j Jp_iJ)}
=
\frac{
\tr_{\langle B, A\rangle_i}(b q_j Jp_iJ)
}
{T^{B\subset \langle B, A\rangle}_{ij}}.
$$
Hence for our projection $f\in B$, 
\begin{align*}
\tr_{\langle B, A\rangle_i}(f Jp_iJ)
&=
\sum_{j=1}^b \tr_{\langle B, A\rangle_i}(f q_j Jp_iJ)
=
\sum_{j=1}^b \pi_j T^{B\subset \langle B, A\rangle}_{ij}
\underset{\text{\eqref{eq:TraceMatricesInTermsOfDistortions}}}{=}
\frac{\sum_{j=1}^b \pi_j \delta_{ij}\Delta_{ij}}
{\sum_{k=1}^b \delta_{ik}\Delta_{ik}}
=
\frac{1}{\sum_{k=1}^b \delta_{ik}\Delta_{ik}}.
\end{align*}
By Fact \ref{fact:CenterValuedTraceOfJonesProjection}, $\tr^Z_{\langle B, A\rangle}(f) = \tr^Z_{\langle B, A\rangle}(e_A)$.

\item[\underline{$(3)\Rightarrow (1)$}]
Recall that two projections in a finite multifactor are equivalent if and only if they are unitarily conjugate.
Suppose $f\in B$ and $u \in U(\langle B, A\rangle)$ such that $e_A = ufu^*$.
Since the commutant of the left $A$-action on $L^2B$ is the right-$\langle A, B\rangle$ action,
and since $[J, e_A]=0$ \cite[Prop.~3.6.1(ii)]{MR999799},
we have isomorphisms of left $A$-modules
\begin{equation*}
L^2A
\cong
e_AL^2B
=
L^2Be_A
=
L^2Bufu^*
=
(L^2Bf)u^*
\cong
L^2Be.
\qedhere
\end{equation*}
\end{proof}

\begin{thm}[{cf.~\cite[Prop.~1.2.7]{MR1339767}}]
\label{thm:ExistenceOfDownward}
Suppose $A\subset B$ is a
finite index
connected
inclusion of finite multifactors.
The following are equivalent.
\begin{enumerate}[label=(D\arabic*)]
\item
\label{downward:Independent}
$A\subset B$ admits a trace independent downward basic construction.

\item
\label{downward:JonesModule}
$A\subset B$ admits a Jones downward basic construction, i.e.,
there is a projection $f\in B$ 
with central support 1
such that $L^2A \cong L^2B f$ as left $A$-modules.
\item
\label{downward:JonesVector}
There exists a projection $f\in B$
with central support 1
such that the vector $\pi:= \tr^Z_B(f)\in (0,1]^b$
satisfies
$\sum_{j=1}^b \pi_j\delta_{ij}\Delta_{ij}=1$ for every $1\leq i\leq a$.
\item
\label{downward:JonesEquivalent}
There is a projection $f\in B$ 
with central support 1
which is equivalent to $e_A \in \langle B, A\rangle \subset B(L^2(B, \tr_B))$ where $\tr_B$ is the unique Markov trace.

\item
\label{downward:Popa}
$A\subset B$ admits a Popa downward basic construction, i.e., there is a projection $e\in B$ with 
$E_A(e)=d^{-2}$
such that setting $C:=\{e\}'\cap A$,
$C\subset (A,\tr_A)\subset (B,e, \tr_B)$
is the Jones basic construction
and
$E_C^A(a):=d^{2}E_A(eae)$ is the Markov trace-preserving conditional expectation.
\end{enumerate}
\end{thm}
\begin{proof}
\item[\underline{\ref{downward:Independent}$\Rightarrow$\ref{downward:Popa}:}]
Suppose $C\subset A \subset B$ is an independent downward basic construction, and suppose there is a $*$-isomorphism $B \cong \langle A, C\rangle$ on $L^2A$.
Identify $L^2A = L^2(A, \tr_A)$ where $\tr_A$ is the unique Markov trace for $C\subset A$, and identify $B=\langle A, C\rangle = \langle A, e_C\rangle$ where $e_C$ is the orthogonal projection onto $L^2(C, \tr_C)$ and $\tr_C := \tr_A|_C$.
Then by \cite[Prop.~3.1.4 and 3.4.1]{MR0696688},
$E_A(e_C)=d^{-2}$, $C=\{e_C\}'\cap A$, and $E_C(a)e_C = e_Cae_C$ for all $a\in A$.
Taking $E_A$, we get $E_C(a)=d^{-2}E_A(e_Cae_C)$.

\item[\underline{\ref{downward:Popa}$\Rightarrow$\ref{downward:JonesModule}}:]
We set $f=e$, which we may identify with the Jones projection $e_C$ for $C\subset (A, \tr_A)$.
By \cite[Prop.~3.1.5(iv)]{MR0696688}, the central support of $e_C$ is $1$.
We then have the following isomorphisms of left $A$-modules:
$$
L^2B e_C
\cong
L^2A \boxtimes_C L^2A e_C
\cong
L^2A \boxtimes_C L^2C
\cong
L^2A.
$$

\item[\underline{\ref{downward:JonesModule}$\Rightarrow$\ref{downward:Independent}}:]
This follows from the argument in \cite[Lem.~3.1.8]{MR0696688}.
Using the isomorphism $L^2A \cong L^2B e$ as left $A$-modules, we get a normal left $B$-action on $L^2A$ extending the left $A$-action.
Now define $C := (J_ABJ_A)'$ on $L^2A$, which is manifestly a downward basic construction.

\item[\underline{\ref{downward:JonesModule}$\Leftrightarrow$\ref{downward:JonesVector}$\Leftrightarrow$\ref{downward:JonesEquivalent}:}]
This follows by Proposition \ref{prop:JonesDownwardEquivalentConditions} together with the assumption that the central support of $f$ is 1.
\end{proof}

\begin{rem}
\label{rem:DownwardMarkovTunnel}
Suppose $A\subset B$ is a finite index connected inclusion of finite multifactors.
The Jones tower $(A_n,\tr_n,e_{n+1})_{n\geq 0}$ is a \emph{Markov tower} in the sense of \cite[Def.~3.1]{1810.07049}.
Suppose $f\in B$ is a projection satisfying the equivalent conditions of Proposition \ref{prop:JonesDownwardEquivalentConditions}
such that $E_A(f)=d^{-2}$, except the central support $z$ of $f$ in $B$ is possibly not 1.
We can perform a \emph{one step downward Markov tunnel} by defining $A_{-1}:=\{f\}'\cap A$, $\tr_{-1}:=\tr_A|_{A_{-1}}$, and $e_0:=f$.
Observe that
$E_{-1}(a):=d^{2}E_A(e_0ae_0)$ 
still defines the unique trace-preserving conditional expectation by
\eqref{eq:TunnelExpectation}.
Note, however, that $zBz$ (which is possibly not $B$) is the basic construction of $A_{-1}\subset (A,\tr_A)$.
\end{rem}

For (non-)examples, we refer the reader to Examples \ref{ex:A4-example} and \ref{ex:SE-example} below.

\begin{cor}
\label{cor:DistortionFormulaForDownwardBC}
Suppose $A\subset B$ is a
finite index
connected
$\rm II_1$ multifactor inclusion
such that 
\ref{downward:JonesModule} holds.
The distortion $\gamma$ of the inclusion $C\subset A$ corresponding to projection $f\in B$ is given by 
$\gamma_{ji} = \frac{1}{\pi_j \delta_{ij}}$.
\end{cor}
\begin{proof}
To compute $\gamma_{ji}$, we compute 
$$
\left(
\frac{\vNdim_R((p_iL^2Ar_j){}_{C_j})}{\vNdim_L({}_{A_i}(p_iL^2Ar_j))}
\right)^{1/2}
$$
where the $r_j$ are the minimal central projections of $C$.
To define $C$, we transported the left $B$ action on ${}_AL^2Bf$ to ${}_AL^2A$ under the left $A$-module isomorphism.
This means the left $A$-module isomorphism extends to an $A-B'$ bimodule isomorphism, where the right action of $B'$ is exactly the right action of $C$.
Hence we may identify
${}_AL^2A{}_C \cong {}_A L^2Bf{}_{fBf}$.
Since $Z(fBf) = fZ(B)$, we may identify the $r_j$ with $fq_j$.

First, by Lemma \ref{lem:CutDownDimension},
$$
\vNdim_L({}_{A_i}(p_iL^2Ar_j))
=
\vNdim_L({}_{A_i}(p_iL^2Bfq_j))
=
\pi_j \delta_{ij}\Delta_{ij}.
$$
Second, by \text{\cite[(2.1.4)]{MR0696688}} and \eqref{eq:Distortion, Jones dim, and von Neumann dim},
$$
\vNdim_R((p_iL^2Ar_j){}_{C_j})
=
\vNdim_R({}_{A_i}(p_iL^2Bfq_j){}_{fB_jf})
=
\frac{1}{\tr_{B_j}(fq_j)} 
\vNdim_R({}_{A_i}(p_iL^2Bq_j){}_{B_j})
=
\frac{\Delta_{ij}}{\pi_j\delta_{ij}}.
$$
We conclude that $\gamma_{ji} = \frac{1}{\pi_j \delta_{ij}}$ as claimed.
\end{proof}

\begin{rem}
Under the hypotheses of Corollary \ref{cor:DistortionFormulaForDownwardBC}, observe that $\delta$ and $\gamma$ satisfy \eqref{eq:DistortionBasicConstructionFormula}:
\begin{align*}
\frac{1}{\gamma_{ji}}
\sum_{k=1}^a \gamma_{jk} \Delta^T_{jk}
&=
\pi_j \delta_{ij}
\sum_{k=1}^a \frac{1}{\pi_j \delta_{kj}}\Delta_{kj}
=
\delta_{ij}
\sum_{k=1}^a
\frac{\Delta_{kj}}{\delta_{kj}}
\\&=
\delta_{ij}
\sum_{k=1}^a
\vNdim(
(p_kL^2Bq_j)_{B_j}
)
=
\delta_{ij}
\vNdim((L^2Bq_j) {}_{B_j})
=
\delta_{ij}.
\end{align*}
Moreover, we can easily see that for all $1\leq i\leq a$,
$$
\sum_{j=1}^b
\pi_j \delta_{ij}\Delta_{ij}
=
\sum_{j=1}^b
\frac{\Delta_{ji}^T}{\gamma_{ji}}
=
\vNdim(
(L^2Ap_i)_{A_i}
)
=
1.
$$
\end{rem}

\begin{cor}
\label{cor:RelateDownwardBasicConstructions}
Suppose $A\subset B$ is a 
finite index
connected
inclusion of finite multifactors
and suppose $e_1,e_2 \in B$ satisfy \ref{downward:Popa}, with $C_1:= \{e_1\}'\cap A$ and $C_2:=\{e_2\}'\cap A$.
The following are equivalent.
\begin{enumerate}[label=(\arabic*)]
\item
There is a $u\in A$ such that $ue_1u^*=e_2$.
\item 
There is a $u\in A$ such that $uC_1u^* = C_2$.
\item
$\tr^Z_B(e_1) = \tr^Z_B(e_2)$.
\end{enumerate}
\end{cor}
\begin{proof}
\item[\underline{$(1)\Rightarrow(2)$:}]
Suppose $u\in A$ such that $ue_1u^* = e_2$.
there is a unitary $u\in B$ such that $ue_1u^*=e_2$.
Then for all $c\in C_1$,
$$
ucu^* e_2 
= 
uc(u^*e_2u)u^* 
= 
uce_1u^* 
= 
ue_1cu^* 
= 
ue_1u^*ucu^* 
=
e_2ucu^*,
$$
so $ucu^* \in C_2$.
Similarly, for all $c\in C_2$, $u^*cu \in C_1$.
Hence $uC_1u^*=C_2$.

\item[\underline{$(2)\Rightarrow(3)$:}]
We prove the contrapositive.
If $\tr^Z_B(e_1)\neq \tr^Z_B(e_2)$, then 
by Corollary \ref{cor:DistortionFormulaForDownwardBC}, the distortions $\delta^{C_1\subset A}$ and $\delta^{C_2\subset A}$ will not agree.
We conclude no such unitary can exist.

\item[\underline{$(3)\Rightarrow(1)$:}]
The proof is similar to \cite[Prop.~1.7(i) and Cor.~1.8(ii)]{MR860811}.

First, a similar argument to \cite[Prop.~1.7(i)]{MR860811} proves that for any finite index connected inclusion $A\subset B$ of finite multifactors, the map
$u\mapsto ue_Au^*$ descends to a bijection from $U(B)/U(A)$ to the set of projections 
$$
\set{f\in P(\langle B,A\rangle)}{E_B(f)=d^{-2}\text{ and }\tr^Z_{\langle B,A\rangle}(f)=\tr^Z_{\langle B,A\rangle}(e_A)}.
$$

Now suppose $e_1,e_2\in B$ with $E_A(e_1)=E_A(e_2)=d^{-2}$ and $\tr^Z_B(e_1)=\tr^Z_B(e_2)$.
Similar to the proof of \cite[Cor.~1.8(ii)]{MR860811},
it follows from the above bijection applied to $C_1\subset A$ with $\langle C_1,A\rangle=B$ that there is a unitary $u \in A$ such that $ue_1u^*=e_2$.
\end{proof}

\subsection{Popa's homogeneity}
\label{sec:Homogeneity}

We now study Popa's notion of \emph{homogeneity} for finite index connected inclusions $A\subset B$ of finite multifactors.
We obtained the following definition 
by combining \cite[Def.~1.2.11]{MR1339767} and \cite[Ex.~2.7 and Prop.~3.1(1)]{MR1073519}.

\begin{defn}
We say $A\subset B$ is \emph{homogeneous of index $\lambda$} if $\vNdim_L({}_{A_i}p_iL^2B)=\lambda$ for all $1\leq i\leq a$.
\end{defn}

\begin{thm}
\label{thm:tunnelequivalencies}
Suppose that $A\subset B$ is a
finite index
extremal 
connected
inclusion of finite multifactors.
Let $\tr_B$ be the unique Markov trace, $\tr_A=\tr_B|_A$, $E$ the unique Markov/minimal conditional expectation, and $e_1\in B(L^2(B, \tr_B))$ the Jones projection.
The following are equivalent.
\begin{enumerate}[label={\rm(H\arabic*)}]
\item
\label{condition:homogeneous}
$A\subset B$ is homogeneous of index $d^2$.
\item
\label{condition:homogeneous2}
$\sum_{j=1}^b
\delta_{ij}D_{ij} = d^2$
for all $1\leq i\leq a$.
\item
\label{condition:fixedpoint}
$\Phi(\delta)=\delta$, i.e., $\delta$ is a fixed point of the dynamical system \eqref{eq:DistortionDynamicalSystem} (which is unique by Theorem \ref{thm:UniqueFixedPoint}).
\item 
\label{condition:StandardDistortion} 
$\delta_{ij} = d\beta_j / \alpha_i$.
\item 
\label{thm:tuneq-scalartrace} 
$\tr_{\langle B, A\rangle}^Z(e_1) \in \C$, where $\tr^Z_{\langle B, A\rangle}$ is the canonical center-valued trace;
\item 
\label{thm:tuneq-unchangedtrace} 
$\tr_A(p_i) = \tr_{\langle B, A\rangle}(Jp_iJ)$ for all $i$.
\item 
\label{thm:tuneq-superext} 
$A \subset B$ is \emph{super-extremal} in the sense of \cite[Def.~4.1, Lem.~4.2]{MR3994584}: $\tr_A(p_i)=\alpha_i^2$ where $\vec{\alpha} \in \bbR_{>0}^a$ is the eigenvector for $DD^T$ with eigenvalue $d^2$ normalized so that $\sum_{i=1}^a \alpha_i^2 = 1$.
\end{enumerate}
In addition, if $A$ and $B$ are of type $\mathrm{II}_1$, the above are equivalent to:
\begin{enumerate}[label={\rm(H8)}]
\item 
\label{thm:tuneq-tunnel} 
the inclusion $A \subset B$ admits an infinite Jones tunnel, i.e., we can iterate the downward basic construction infinitely many times.
\end{enumerate}
\end{thm}
\begin{proof}
\item[\underline{$\ref{condition:homogeneous}\Leftrightarrow \ref{condition:homogeneous2}$:}]
Observe that 
$\vNdim_L({}_{A_i}p_iL^2B)
=
\sum_{j=1}^b 
\vNdim_L({}_{A_i}p_iL^2Bq_j{}_{B_j})
=
\sum_{j=1}^b
\delta_{ij}D_{ij}$.

\item[\underline{$\ref{condition:homogeneous2}\Rightarrow \ref{condition:fixedpoint}$:}]
Combining \ref{condition:homogeneous2} and \eqref{eq:DistortionBasicConstructionFormula},
we have $\delta_{ji}^{B\subset \langle B,A\rangle} = d^2/\delta_{ij}$.
Applying \eqref{eq:DistortionBasicConstructionFormula} again with $\delta^{B\subset \langle B,A\rangle}$ in place of $\delta$ and $D^T$ in place of $D$, we have
\begin{align*}
\Phi(\delta)_{ij} 
=
\frac{\delta_{ij}}{d^2} 
\sum_{\ell=1}^a 
\frac{d^2}{\delta_{\ell j}} D^T_{j\ell}
=
\delta_{ij}
\sum_{\ell=1}^a 
\frac{D_{\ell j}}{\delta_{\ell j}}
\underset{\text{\eqref{eq:Distortion, Jones dim, and von Neumann dim}}}{=}
\delta_{ij}
\sum_{\ell=1}^a 
\vNdim_R({}_{A_\ell}(p_\ell L^2Bq_j){}_{B_j})
=
\delta_{ij}.
\end{align*}
Hence $\delta$ is the unique fixed point under the dynamical system \eqref{eq:DistortionDynamicalSystem}.

\item[\underline{$\ref{condition:fixedpoint}\Rightarrow \ref{condition:StandardDistortion}$:}]
This follows from Theorem \ref{thm:UniqueFixedPoint}.

\item[\underline{$\ref{condition:StandardDistortion}\Rightarrow \ref{condition:homogeneous2}$:}]
If $\delta_{ij} = d\beta_j/\alpha_i$, then for all $i$,
$
\sum_{j=1}^b
\delta_{ij}D_{ij}
=
\sum_{j=1}^b
d \frac{\beta_j}{\alpha_i}D_{ij}
=
d^2.
$

\item[\underline{$\ref{condition:homogeneous2}\Leftrightarrow\ref{thm:tuneq-scalartrace}$:}]
By Fact \ref{fact:CenterValuedTraceOfJonesProjection}, the $i$-th component of the center-valued trace of $e_1$ is given by
$$
\tr_{\langle B, A\rangle_i}(e_1Jp_iJ) 
=
\frac{1}{\sum_{j=1}^b \delta_{ij}\Delta_{ij}}.
$$
Hence $e_1$ has scalar central trace if and only if $\sum_{j=1}^b \delta_{ij}\Delta_{ij}$ is the same scalar for all $1\leq i\leq a$.
Since $\tr_{\langle B, A\rangle}(e_1)=d^{-2}$, this scalar must be equal to $d^{-2}$.

\item[\underline{$\ref{thm:tuneq-unchangedtrace} \Leftrightarrow \ref{condition:homogeneous2}$:}]
By \eqref{eq:TraceVectorForBasicConstruction} cf.~\cite[3.7.3.1]{MR999799},
$$
d^2\tr_{\langle B,A\rangle}(r_i)
=
\tr_A(p_i) 
\sum_{j=1}^b D_{ij}\delta_{ij}.
$$
Hence $\tr_A(p_i) = \tr_{\langle B,A\rangle}(r_i)$ for all $1\leq i\leq a$ if and only if
\ref{condition:homogeneous2} holds.

\item[\underline{$\ref{condition:StandardDistortion} \Leftrightarrow \ref{thm:tuneq-superext}$:}]
Since $A\subset B$ is extremal by Theorem \ref{thm:Minimal=Extremal}, we have $D=\Delta$
and
$$
\delta_{ij}
=
d\left(\frac{\tr_B(q_j)}{\beta_j}\right)\left(\frac{\alpha_i}{\tr_A(p_i)}\right)
\qquad\forall\, 1\leq i\leq a
\qquad\forall \, 1\leq j \leq b.
$$
where $\tr_B$ is the unique Markov trace and $\tr_A=\tr_B|_A$.
By uniqueness in Lemma \ref{lem:ExtendGraphWeighting}(2), $\delta_{ij}=d\beta_j/\alpha_i$ if and only if
$\tr_B(q_j)=\beta_j^2$ and $\tr_A(p_i)=\alpha_i^2$ for all $i,j$.

\item[\underline{$\ref{thm:tuneq-scalartrace} \Leftrightarrow \ref{thm:tuneq-tunnel}$:}]
This is \cite[Cor.~1.2.10]{MR1339767}.
\end{proof}

\begin{cor}
Let $A\subset B$ be extremal and fulfill either of the equivalent conditions of Theorem \ref{thm:tunnelequivalencies}. Denote by $(A_n)_{n\geq 0}$ the Jones tower for $A=A_0 \subset B=A_1$. 
Then
\begin{equation}
\label{eq:DistortionEvenOddExtFormulaHomogeneousCase}
\delta^{A_{2n}\subset A_{2n+1}}_{ij} 
= 
\frac{d\beta_j}{\alpha_i}
\qquad\text{and}\qquad
\delta^{A_{2n+1}\subset A_{2n+2}}_{ji} 
= 
\frac{d\alpha_i}{\beta_j}\qquad\qquad \forall n\geq 0,
\end{equation}
i.e., the distortion matrices oscillate between two fixed matrices depending on the parity of $n$.
\end{cor}
\begin{proof}
Immediate by \ref{condition:StandardDistortion} ($\delta_{ij} = d\beta_j/\alpha_i$) and \eqref{eq:DistortionOrder1and2Formula}.
\end{proof}

\begin{example}
\label{ex:SE-example}
Suppose $A\subset B$ is a connected inclusion of finite dimensional von Neumann algebras.
Denote by $\vec{\mu}=(\mu_i)_{i=1}^a\in \bbN^a$ and $\vec{\nu}=(\nu_j)_{j=1}^b\in \bbN^b$ the row vectors with $\mu_i^2$ and $\nu_j^2$ equal to the algebraic dimensions of the full matrix algebras $A_i$ and $B_j$. 
In this case, the inclusion is always extremal \cite[Cor.~2.2]{MR3994584}, and the Bratteli diagram (bipartite adjacency matrix) of the inclusion is equal to $\Delta=D$. 
The equality $\vec{\mu} D = \vec{\nu}$ corresponds to unitality of the inclusion $A\subset B$.

By \cite[Prop.~4.4]{MR3994584}, $A\subset B$ is \emph{super-extremal} if and only if $\vec{\nu} D^T = d^2 \vec{\mu}$, so that
$\begin{pmatrix}
d\vec{\mu}
&
\vec{\nu}
\end{pmatrix}$
is a Frobenius-Perron eigenvector for right multiplication by
\begin{equation*}
\begin{pmatrix}
0 & D
\\
D^T & 0
\end{pmatrix}.
\end{equation*}
In particular, $d = \|\vec{\nu}\|_2/\|\vec{\mu}\|_2$ and the normalized Frobenius-Perron eigenvectors in this case are $\vec{\alpha} = \vec{\mu}/\|\vec{\mu}\|_2$ and $\vec{\beta} = \vec{\nu}/\|\vec{\nu}\|_2$.
We then calculate from \eqref{eq:DistortionEvenOddExtFormulaHomogeneousCase} the distortion matrices for the Jones tower:
$$
\delta^{A_{2n}\subset A_{2n+1}}_{ij} = \frac{\nu_j}{\mu_i}
\qquad\text{and}\qquad
\delta^{A_{2n+1}\subset A_{2n+2}}_{ji} = \frac{\mu_i}{\nu_j}
\qquad\qquad
\forall n \geq 0.
$$
\end{example}

\begin{example}[{\cite[Ex.~1.2.8]{MR1339767}}]
\label{ex:A4-example}
Consider the finite dimensional (and hence extremal) inclusion $P=\mathbb{C} \oplus \mathbb{C} \subset M_2(\mathbb{C}) \oplus \mathbb{C} = Q$ whose bipartite adjacency matrix and distortion matrices are given by
$$
\Delta=D=
\begin{pmatrix}
1 & 0
\\
1 & 1
\end{pmatrix}
\qquad\text{and}\qquad
\delta
=
\begin{pmatrix}
2 & 1 
\\
2 & 1
\end{pmatrix}.\footnote{Initially, the $(1,2)$-entry of $\delta$ is not defined as $\Delta_{12}=0$.
This $\delta$ is the unique extension afforded by Lemma \ref{lem:ExtendGraphWeighting}(3).}
$$
The inclusion $A= P\otimes R \subset Q\otimes R=B$ does not admit any downward basic construction.
Indeed, using Theorem \ref{thm:ExistenceOfDownward}, one easily verifies there is no strictly positive solution to 
$$
\begin{pmatrix}
1 \\ 1
\end{pmatrix}
=
\begin{pmatrix}
\pi_1\delta_{11}\Delta_{11}+\pi_2\delta_{12}\Delta_{12}
\\ 
\pi_1\delta_{21}\Delta_{21}+\pi_2\delta_{22}\Delta_{22}
\end{pmatrix}
=
\begin{pmatrix}
2 & 0 
\\
2 & 1
\end{pmatrix}
\begin{pmatrix}
\pi_1
\\
\pi_2
\end{pmatrix}
=
\begin{pmatrix}
2\pi_1
\\
2\pi_1 + \pi_2
\end{pmatrix}.
$$
Taking the next two steps in the Jones tower $A_0 \subset A_1 \subset A_2 \subset A_3$,
we get a Morita equivalent inclusion $A_2\subset A_3$ with the same standard invariant which manifestly admits two downward basic constructions.
One quickly observes these inclusions have different distortions:
$$
\delta({}_{A_0}L^2A_1{}_{A_1})
=
\delta({}_PL^2Q{}_Q)
=
\begin{pmatrix}
2 & 1
\\
2 & 1
\end{pmatrix}
\qquad\qquad
\delta({}_{A_2}L^2A_3{}_{A_3})
=
\begin{pmatrix}
5/2 & 3/2
\\
5/3 & 1
\end{pmatrix}.
$$
One calculates that for the Jones tower $(A_n)_{n\geq 0}$,
$$
\delta({}_{A_{2n}}L^2A_{2n+1}{}_{A_{2n+1}})
=
\begin{pmatrix}
F_{2n+2}/F_{2n} & F_{2n+1}/ F_{2n}
\\
F_{2n+2}/F_{2n+1}& 1
\end{pmatrix}
\xrightarrow{n\to \infty}
\begin{pmatrix}
\phi^2 & \phi
\\
\phi & 1
\end{pmatrix},
$$
where $F_n$ is the $n$-th Fibonacci number ($F_0=F_1=1$)
and $\phi$ is the golden ratio.
\end{example}

\begin{rem}
In the previous example, although there is no downward basic construction, there is a two-step downward Markov tunnel as discussed in Remark \ref{rem:DownwardMarkovTunnel}.
The projection $e_{11}\in M_2(\bbC)\oplus \bbC$ has central support $(1,0)$, center-valued trace $(\pi_1,\pi_2) = (1/2,0)$, and trace-preserving expectation $E_P(e_{11})=d^{-2}$, giving a solution to $\sum_{j=1}^2 \pi_j \delta_{ij}\Delta_{ij}=1$ for $i=1,2$.
Setting $N:=\{e_{11}\}'\cap P \cong \bbC$
and $E_N(x):=d^{-2}E_P(e_{11}xe_{11})$, the inclusion $N\subset P$ has bipartite adjacency matrix and distortion matrix given by
$$
\Delta^{N\subset P}
=
D^{N\subset P}
=
\begin{pmatrix}
1 & 1
\end{pmatrix}
=
\delta^{N\subset P}.
$$
There is again no strictly positive solution to
$$
1 
= 
\tr^Z_P(p)_1 \cdot \delta^{N\subset P}_{11}\Delta^{N\subset P}_{11}
+
\tr^Z_P(p)_2 \cdot \delta^{N\subset P}_{12}\Delta^{N\subset P}_{12}
=
\tr^Z_P(p)_1 + \tr^Z_P(p)_2
$$
as the center-valued trace $\tr^Z_P(p)\in Z(P)$ of a projection $p\in P=\bbC^2$ can only have entries in $\{0,1\}$.
Only $p:=(1,0)$ satisfies 
$E_N(p)=d^{-2}$
(one verifies $E_N((0,1))=0$),
which 
gives a further one-step Markov tunnel $M:=\{p\}'\cap N = \bbC$ and $E_M(x):=d^{-2}E_N(pxp)$.
Experts will identify $M\subset N\subset P\subset Q$ as exactly the first 4 algebras in the Temperley-Lieb-Jones $A_4$ subfactor planar algebra with $d=\phi$.
\end{rem}

\section{Classification of finite depth hyperfinite multifactor inclusions}

In \S\ref{sec:ConstructionOfInclusions}, we show that given a
indecomposable unitary 2-shaded planar algebra 
$\cP_\bullet$ with scalar loop parameters,
there is a 
finite index
homogeneous
connected
hyperfinite 
$\rm II_1$ multifactor inclusion
$A\subset B$
whose standard invariant is $*$-isomorphic to $\cP_\bullet$.
Then in \S\ref{sec:DistortionsRealizableByInclusions}, we determine how distortion varies under Morita equivalence of bimodules, and we use this to characterize all finite depth finite index connected hyperfinite $\rm II_1$ multifactor inclusions.

\subsection{Construction of multifactor inclusions}
\label{sec:ConstructionOfInclusions}

Suppose $\cC$ is an indecomposable unitary multifusion category.
As in 
Definition \ref{defn:2-shadingGeneratesConnected},
fix a 2-shading 
$1_\cC = 1^+\oplus 1^-$, 
and suppose $X=\bigoplus X_{ij} \in \cC^{+-}$ generates $\cC$.
Let $D=D_X$, $d=d_X$, $\alpha$, and $\beta$ be as in Definition \ref{defn:StandardDualFunctor}.
Let $\cP(X)_\bullet$ be the indecomposable unitary 2-shaded planar algebra constructed from $(\cC,X)$ from \S\ref{sec:PlanarAlgebras}.

\begin{thm}[Existence of homogeneous inclusions]
\label{thm:ExistsStandardInclusion}
Suppose $\cC(X)\sim \cP(X)_\bullet$ is a finite depth standard invariant.
There exists a finite index homogeneous connected hyperfinite $\rm II_1$ multifactor inclusion $A\subset B$ whose standard invariant is equivalent to $\cC(X)\sim \cP(X)_\bullet$.
Hence by Theorem \ref{thm:tunnelequivalencies}, the distortion $\delta({}_A L^2B{}_B)$ is the standard distortion $\sigma$ with respect to $X$.
\end{thm}
\begin{proof}
Here, we give an outline of the proof, together with forward references for some technical lemmas whose proofs appear in Appendix \ref{sec:CommutingSquares}.
\begin{enumerate}[label=(\Alph*)]

\item 
Suppose we are given $(\cC,X)$ where $X=1^+\otimes X \otimes 1^-$ generates $\cC$ and $1^+\oplus 1^-= 1_\cC$, but $1^+,1^-$ are not necessarily simple.
Choose the standard unitary dual functor $\vee$ with respect to $X$ from Definition \ref{defn:StandardDualFunctor}.
The loop parameters are both equal to $d_X$ times the identities of $1^+$ or $1^-$ depending on shading:
$$
\tikzmath{
  \fill[\AColor, rounded corners=5] (-.6, -.6) rectangle(.6, .6);
  \draw[fill=\BColor] (0,0) circle (.3cm);
}
=
\coev_X^*\circ \coev_X
=
d_X
\id_{1^+}
\qquad\qquad
\tikzmath{
	\fill[\BColor, rounded corners=5] (-.6,-.6) rectangle (.6,.6);
	\draw[fill=\AColor] (0,0) circle (.3cm);
}
=
\ev_X\circ \ev_X^*
=
d_X
\id_{1^-}.
$$
Moreover, for all $i,j$, we have \eqref{eq:pi-q_j bubble}:
\begin{equation*}
\tikzmath{
  \fill[\AColor, rounded corners=5] (-.5,-.75) rectangle (1.75,.75);
  \filldraw[fill=\BColor] (1,0) circle (.5cm);
  \roundNbox{unshaded}{(0,0)}{.25}{0}{0}{$p_i$}
  \roundNbox{unshaded}{(1,0)}{.25}{0}{0}{$q_j$}
}
=
\frac{D_{ij}\beta_j}{\alpha_i}\,
\tikzmath{
  \fill[\AColor, rounded corners=5] (-.6, -.6) rectangle (.6, .6);
  \roundNbox{unshaded}{(0,0)}{.25}{0}{0}{$p_i$}
}
\qquad\qquad
\tikzmath[xscale=-1]{
  \fill[\BColor, rounded corners=5] (-.5,-.75) rectangle (1.75,.75);
  \filldraw[fill=\AColor] (1,0) circle (.5cm);
  \roundNbox{unshaded}{(0,0)}{.25}{0}{0}{$q_j$}
  \roundNbox{unshaded}{(1,0)}{.25}{0}{0}{$p_i$}
}
=
\frac{D_{ij}\alpha_i}{\beta_j}\,
\tikzmath{
  \fill[\BColor, rounded corners=5] (-.6, -.6) rectangle (.6, .6);
  \roundNbox{unshaded}{(0,0)}{.25}{0}{0}{$q_j$}
}
\,.
\end{equation*}
There is a unique \emph{spherical state}\footnote{
Observe that the spherical state is normalized so that $\psi(\id_{1_\cC})=2$.
However, this is the correct normalization so that $\psi(\id_{X^{\alt\otimes n}})=1 = \psi(\id_{\overline{X}^{\alt\otimes n}})$ for all $n$.
} 
$\psi$ on $\End_\cC(1_\cC)$ given by
\begin{equation}
\label{eq:SphericalState}
\psi(p_i):= \alpha_i^2
\qquad\qquad
\psi(q_j):= \beta_j^2
\end{equation}
which satisfies for all $c\in \cC$ and $f: c\to c$, 
$\psi(\tr^\vee_L(f)) = \psi(\tr^\vee_R(f))$.

\item
We now construct a tower of commuting squares of finite dimensional von Neumann algebras together with a common Markov trace for canonical Jones projections.
To do this, we define an alternating tensor product \emph{to the left} by
\begin{align*}
\overline{X} \otimes X \otimes \cdots \otimes \overline{X}\otimes X
&=:
X^{\alt \otimes 2n}
&&
\text{($2n$ tensorands)}
\\
{X}\otimes \overline{X}\otimes {X}\otimes \cdots \otimes \overline{X} \otimes {X}
&=:
{X}^{\alt \otimes 2n+1}
&&
\text{($2n+1$ tensorands)}
\end{align*}
We define $\overline{X}^{\alt\otimes k}$ similarly.
By convention, we define $X^{\alt \otimes 0} = 1^-$ and $\overline{X}^{\alt\otimes 0} := 1^+$.

For $n\geq 0$, we now define
$$
\cQ_{n,+}
:=
\End_\cC(X^{\alt \otimes n})
\qquad\qquad
\cQ_{n,-}
:=
\End_\cC(\overline{X}^{\alt \otimes n}),
$$
and observe we get a tower of commuting squares
\begin{equation}
\label{eq:CommutingSquaresFromEndomorphisms}
\xymatrix@C=5pt@R=2pt{
\cdots & \supset&\cQ_{3,+} &\supset & \cQ_{2,+} &\supset & \cQ_{1,+} &\subset & \cQ_{0,+}
\\&&\cup&&\cup&&\cup&&
\\
\cdots & \supset& \cQ_{2,-}  &\supset & \cQ_{1,-} &\supset & \cQ_{0,-}.
}
\end{equation}
Above, horizontal inclusions are given by tensoring by the appropriate identity morphisms on the \emph{left}, and vertical inclusions are given by tensoring by $\id_{X}$ to the \emph{right}.
Expanding the definitions, we have \eqref{eq:CommutingSquaresFromEndomorphisms} is exactly
\begin{equation}
\label{eq:TowerForReconstruction}
\xymatrix@C=5pt@R=2pt{
\cdots &\supset & \End_\cC(X\otimes \overline{X}\otimes X) &\supset & \End_\cC(\overline{X}\otimes X) &\supset & \End_\cC(X)   &\supset & \End_\cC(1^-)
\\&&\cup&&\cup&&\cup&&
\\
\cdots &\supset & \End_\cC(X\otimes \overline{X}) &\supset & \End_\cC(\overline{X}) &\supset  & \End_\cC(1^+)
}
\end{equation}
The Markov trace $\tr$ is given by the inductive limit of $d_X^{-n}(\psi\circ\tr^\vee_L)$ on $\cQ_{n,+}$ where $\psi$ is the spherical state from \eqref{eq:SphericalState}.
The Jones projections $e_{n,\pm}\in \cQ_{n+1,\pm}$ which implement the canonical trace-preserving conditional expectations $\cQ_{n,\pm}\to \cQ_{n-1,\pm}$ are given by
\begin{align*}
 e_{2k+1,+} &:=
 \frac{1}{d}
 \cdot
 \tikzmath[xscale=-1]{
    \begin{scope}
    \clip[rounded corners = 5] (-.4,-.5) rectangle (1.2,.5);
  \fill[\BColor] (-.4,-.5) rectangle (1.2,.5);
  \end{scope}
  \draw[very thick] (0,-.5) -- node[midway,above,rotate=270] {\tiny{$2k$}} (0,.5);
  \filldraw[fill= \AColor] (.3,.5) arc (-180:0:.3cm);
  \filldraw[fill= \AColor] (.3,-.5) arc (180:0:.3cm);
 }
 =
 d^{-1}
(\ev_X^* \circ \ev_X) 
 \otimes
 \id_{X^{\alt \otimes 2k}}
 \\
 e_{2k+2,+}
 &:=
 \frac{1}{d}
 \cdot
 \tikzmath[xscale=-1]{
   \begin{scope}
    \clip[rounded corners = 5] (-.4,-.5) rectangle (1.2,.5);
  \fill[\BColor] (-.4,-.5) rectangle (0,.5);
  \fill[\AColor] (0,-.5) rectangle (1.2,.5);
  \end{scope}
  \draw[very thick] (0,-.5) -- node[midway,above,rotate=270] {\tiny{$2k+1$}} (0,.5);
  \filldraw[fill=\BColor] (.3,.5) arc (-180:0:.3cm);
  \filldraw[fill=\BColor] (.3,-.5) arc (180:0:.3cm);
 }
=
d^{-1}
(\coev_X \circ \coev_X^*)
\otimes
\id_{X^{\alt \otimes 2k+1}}
\\
 e_{2k+1,-} &:=
 \frac{1}{d}
 \cdot
 \tikzmath[xscale=-1]{
   \begin{scope}
    \clip[rounded corners = 5] (-.4,-.5) rectangle (1.2,.5);
  \fill[\AColor] (-.4,-.5) rectangle (1.2,.5);
  \end{scope}
  \draw[very thick] (0,-.5) -- node[midway,above,rotate=270] {\tiny{$2k$}} (0,.5);
  \filldraw[fill= \BColor] (.3,.5) arc (-180:0:.3cm);
  \filldraw[fill= \BColor] (.3,-.5) arc (180:0:.3cm);
 }
 =
 d^{-1}
 (\coev_X \circ \coev_X^*)
 \otimes
 \id_{\overline{X}^{\alt \otimes 2k}}
 \\
 e_{2k+2,-}
 &:=
 \frac{1}{d}
 \cdot
 \tikzmath[xscale=-1]{
  \begin{scope}
    \clip[rounded corners = 5] (-.4,-.5) rectangle (1.2,.5);
    \fill[\AColor] (-.4,-.5) rectangle (0,.5);
    \fill[\BColor] (0,-.5) rectangle (1.2,.5);
  \end{scope}
  \draw[very thick] (0,-.5) -- node[midway,above,rotate=270] {\tiny{$2k+1$}} (0,.5);
  \filldraw[fill=\AColor] (.3,.5) arc (-180:0:.3cm);
  \filldraw[fill=\AColor] (.3,-.5) arc (180:0:.3cm);
 }
=
d^{-1}
(\ev_X^* \circ \ev_X)
\otimes
\id_{\overline{X}^{\alt \otimes 2k+1}}
\end{align*}

\item
Since $\cC$ is multifusion and $X$ generates $\cC$, by Lemma \ref{lem:MultifusionCommutingSquaresNondegenerate} (see Example \ref{ex:EventuallyNondegenerate}) below, eventually these commuting squares are \emph{nondegenerate} in the sense of Definition \ref{defn:Nondegenerate} in Appendix~\ref{sec:CommutingSquares}.
We truncate the sequence of commuting squares \eqref{eq:CommutingSquaresFromEndomorphisms} so that the first commuting square is nondegenerate.
\begin{equation}
\label{eq:HMCommutingSquaresFromEndomorphisms}
\xymatrix@C=5pt@R=2pt{
M_1=\cQ_{2k+2,+} &\supset & M_0=\cQ_{2k+1,+}
\\\cup&&\cup
\\
N_1=\cQ_{2k+1,-}  &\supset & N_0=\cQ_{2k,-}.
}
\end{equation}
Moreover, by Lemma \ref{lem:MultifusionCommutingSquaresNondegenerate} (see Example \ref{ex:EventuallyNondegenerate}), the sequence of commuting squares after this point is isomorphic to the basic construction commuting square at each level.
(Observe that it does not matter where we truncate as we will take inductive limits.)

\item
\label{eq:Construct A in B}
The GNS construction with respect to $\tr_B$ gives an inductive limit inclusion
$$
A:=\varinjlim \cQ_{n,-}
\hookrightarrow
\varinjlim \cQ_{n,+} =:B,
$$
of $\rm II_1$ multifactors.
By Lemma \ref{lem:II_1Inclusion} below,
$Z(A) = N_0'\cap N_0 = \End(1^+)$ acting on the left, and
$Z(B) = M_0'\cap M_0 = \End(1^-)$ acting on the right.

By the Ocneanu Compactness Theorem \ref{thm:OcneanuCompactness}, 
$A'\cap B = N_1'\cap M_0$ inside $M_1 = \End_\cC(X^{\alt \otimes 2k+2})$.
By Corollary \ref{cor:CommutationCorollary} below, we see that 
$$
N_1'\cap M_0=\id_{\overline{X}^{\alt \otimes 2k+1}}\otimes \End_\cC(X)\cong \End_\cC(X).
$$
Under this identification, the inclusions $Z(A), Z(B)\hookrightarrow N_1'\cap M_0=\End_\cC(X)$ are given respectively by
$$
\tikzmath{
  \fill[\AColor] (-.5,-.5) rectangle (.5,.5);
  \roundNbox{unshaded}{(0,0)}{.25}{0}{0}{$p$}
}
\mapsto
\tikzmath{
  \fill[\AColor] (-.5,-.5) rectangle (.5,.5);
  \fill[\BColor] (.5,-.5) rectangle (.8,.5);
  \draw (.5,-.5) -- (.5,.5);
  \roundNbox{unshaded}{(0,0)}{.25}{0}{0}{$p$}
}
\qquad
\text{and}
\qquad
\tikzmath{
  \fill[\BColor] (-.5,-.5) rectangle (.5,.5);
  \roundNbox{unshaded}{(0,0)}{.25}{0}{0}{$q$}
}
\mapsto
\tikzmath[xscale=-1]{
  \fill[\BColor] (-.5,-.5) rectangle (.5,.5);
  \fill[\AColor] (.5,-.5) rectangle (.8,.5);
  \draw (.5,-.5) -- (.5,.5);
  \roundNbox{unshaded}{(0,0)}{.25}{0}{0}{$q$}
}\,.
$$
It follows that $Z(A)\cap Z(B) = \bbC$, and the inclusion is connected.

\item
Observe now that the sequence of commuting squares \eqref{eq:TowerForReconstruction} fits into a doubly infinite lattice of commuting squares by alternately tensoring on the \emph{right} by $\overline{X}$ and $X$ to obtain new rows above those in \eqref{eq:TowerForReconstruction}.
\begin{equation}
\label{eq:LatticeForReconstruction}
\xymatrix@C=5pt@R=2pt{
&& \vdots && \vdots && \vdots && \vdots
\\&&\cup&&\cup&&\cup&&\cup
\\
 \cdots & \supset &
 \End_\cC(X^{\alt\otimes 5})&\supset & \End_\cC(X^{\alt\otimes 4}) &\supset & \End_\cC(X^{\alt\otimes 3}) &\supset & \End_\cC(\overline{X} \otimes X)   
\\&&\cup&&\cup&&\cup&&\cup
\\
 \cdots & \supset &
 \End_\cC(\overline{X}^{\alt\otimes 4}) 
 &\supset & \End_\cC(\overline{X}^{\alt\otimes 3}) 
&\supset & \End_\cC(X\otimes \overline{X}) &\supset &\End_\cC(\overline{X}) 
\\&&\cup&&\cup&&\cup&&\cup
\\ \cdots &\supset & \End_\cC(X^{\alt\otimes3}) &\subset & \End_\cC(\overline{X}\otimes X)
&\supset & \End_\cC(X) & \supset & \End_\cC(1^-) 
\\&&\cup&&\cup&&\cup&&
\\\cdots &\supset & \End_\cC(X\otimes \overline{X})& \supset & \End_\cC(\overline{X}) &\supset
&\End_\cC(1^+) 
}
\end{equation}
Moreover, truncating this lattice as in \eqref{eq:HMCommutingSquaresFromEndomorphisms}, 
by Lemma \ref{lem:MultifusionCommutingSquaresNondegenerate} (see Example \ref{ex:EventuallyNondegenerate}),
all the commuting squares on the left are nondegenerate,
and obtained by iterated basic constructions.
\begin{equation}
\label{eq:CommutingSquareLattice}
\xymatrix@C=5pt@R=2pt{
\vdots &&&&\vdots && \vdots && \vdots && \vdots
\\\cup &&&&\cup&&\cup&&\cup&&\cup
\\
A_3 &\supset & \cdots & \supset & \cQ_{2k+6,+} &\supset & \cQ_{2k+5,+} &\supset & \cQ_{2k+4,+} &\supset & \cQ_{2k+3,+} 
\\
\cup &&&&\cup&&\cup&&\cup&&\cup
\\
A_2 &\supset & \cdots & \supset & \cQ_{2k+5,-} &\supset & \cQ_{2k+4,-} &\supset & \cQ_{2k+3,-} &\supset & \cQ_{2k+2,-} 
\\
\cup &&&&\cup&&\cup&&\cup&&\cup
\\
B = A_1 &\supset & \cdots & \supset & \cQ_{2k+4,+} &\supset & \cQ_{2k+3,+} &\supset & \cQ_{2k+2,+} &\supset & \cQ_{2k+1,+} 
\\
\cup &&&&\cup&&\cup&&\cup&&\cup
\\
A = A_0 &\supset & \cdots & \supset & \cQ_{2k+3,-} &\supset & \cQ_{2k+2,-} &\supset & \cQ_{2k+1,-} &\supset & \cQ_{2k,-}         
}
\end{equation}

\item
\label{step:InductiveLimitJonesTower}
Again by Lemma \ref{lem:II_1Inclusion} below, the inductive limit $A_j$ of the $j$-th row of \eqref{eq:CommutingSquareLattice} is a finite direct sum of $\rm II_1$ factors whose centers are isomorphic to $\End_\cC(1^+)$ or $\End_\cC(1^-)$ depending on the parity.
Moreover, they come equipped with inductive limit traces $\tr_j$ coming from the normalized left traces on $\cC$ composed with the spherical state $\psi$.
By the basic construction recognition lemma \cite[Lem.~2.15]{MR2812459}, the tracial von Neumann algebras $(A_j,\tr_j)_{j\geq 0}$ form a Jones tower where
the Jones projections $f_j\in A_{j+1}$ such that $A_{j+1} = \langle A_j , f_j\rangle$ are given by
\begin{align*}
 f_{2\ell+1} &:=
 \frac{1}{d}
 \cdot
 \tikzmath{
  \fill[\AColor] (-.4,-.5) rectangle (1.2,.5);
  \draw[very thick] (0,-.5) -- node[midway,above,rotate=90] {\tiny{$2k{+}2\ell$}} (0,.5);
  \filldraw[fill= \BColor] (.3,.5) arc (-180:0:.3cm);
  \filldraw[fill= \BColor] (.3,-.5) arc (180:0:.3cm);
 }
 =
 d^{-1}
 \id_{\overline{X}^{\alt \otimes (2k+2\ell)}}
 \otimes
 (\coev_X \circ \coev_X^*)
 &\in \cQ_{2k+2\ell+2,-} \subset A_{2\ell+2}
\\
 f_{2\ell+2}
 &:=
 \frac{1}{d}
 \cdot
 \tikzmath{
  \fill[\AColor] (-.4,-.5) rectangle (0,.5);
  \fill[\BColor] (0,-.5) rectangle (1.2,.5);
  \draw[very thick] (0,-.5) -- node[midway,above,rotate=90] {\tiny{$2k{+}2\ell{+}1$}} (0,.5);
  \filldraw[fill=\AColor] (.3,.5) arc (-180:0:.3cm);
  \filldraw[fill=\AColor] (.3,-.5) arc (180:0:.3cm);
 }
=
d^{-1}
\id_{X^{\alt \otimes (2k+2\ell+1)}}
\otimes
(\ev_X^* \circ \ev_X)
&\in \cQ_{2k+2\ell+3,+} \subset A_{2\ell+3}.
\end{align*}

\item
For each $j\geq 2$, we can look at the composite sequence of commuting squares consisting of the $0$-th and $j$-th rows of \eqref{eq:CommutingSquareLattice}.
Again by Lemma \ref{lem:MultifusionCommutingSquaresNondegenerate} (see Example \ref{ex:EventuallyNondegenerate}),
the first composite square is nondegenerate,
and all subsequent commuting squares are obtained by iterating the basic construction.

Again by the Ocneanu Compactness Theorem \ref{thm:OcneanuCompactness}, we have that
\begin{align*}
\cP^{A\subset B}_{2n,+}
&:=
A_0'\cap A_{2n}
=
\cQ_{2k+1,-}' \cap \cQ_{2k+2n,-}
\\
\cP^{A\subset B}_{2n+1,+}
&:=
A_0'\cap A_{2n+1}
=
\cQ_{2k+1,-}' \cap \cQ_{2k+2n+1,+}
\end{align*}

Similarly, looking at composite sequence of commuting squares consisting of the $1$-st and $j$-th rows of \eqref{eq:CommutingSquareLattice}
and applying the Ocneanu Compactness Theorem \ref{thm:OcneanuCompactness}, we have
\begin{align*}
\cP^{A\subset B}_{2n,-}
&:=
A_1'\cap A_{2n+1}
=
\cQ_{2k+2,+}' \cap \cQ_{2k+2n+1,+}
\\
\cP^{A\subset B}_{2n+1,-}
&:=
A_1'\cap A_{2n+2}
=
\cQ_{2k+2,+}' \cap \cQ_{2k+2n+2,-}
\end{align*}

\item
By Corollary \ref{cor:CommutationCorollary}, the map 
$\varphi_{n,+}:=\id_{(X\otimes \overline{X})^{\otimes k}} \otimes -$
which adds $2k$ strands to the left
is a unital $*$-algebra isomorphism
\begin{align*}
\cP(X)_{2n,+}
&:=
\End_\cC((X\otimes \overline{X})^{\otimes n})
\longrightarrow
\cQ_{2k+1,-}' \cap \cQ_{2k+2n,-}
=
A_0'\cap A_{2n}
=
\cP^{A\subset B}_{2n,+} 
\\
\cP(X)_{2n+1,+}
&:=
\End_\cC((X\otimes \overline{X})^{\otimes n}\otimes X)
\longrightarrow
\cQ_{2k+1,-}' \cap \cQ_{2k+2n+1,+}
=
A_0'\cap A_{2n+1}
=
\cP^{A\subset B}_{2n+1,+} 
\end{align*}
which maps the Jones projections $f_n\in A_0'\cap A_{n+1}=\cP^{A\subset B}_{n+1,+}$ from \ref{step:InductiveLimitJonesTower} to $e_{n,+}\in \cP(X)_{n+1,+}$ and is compatible with the right inclusion and the partial trace (conditional expectation).

Similarly, the map 
$\varphi_{n,-}:=\id_{(X\otimes \overline{X})^{\otimes k}\otimes X}\otimes -$
which adds $2k+1$ strands to the left
is a unital $*$-algebra isomorphism:
\begin{align*}
\cP(X)_{2n,-}
&:=
\End_\cC((\overline{X} \otimes X)^{\otimes n})
\longrightarrow
\cQ_{2k+2,+}' \cap \cQ_{2k+2n+1,+}
=
A_1'\cap A_{2n+1}
=
\cP^{A\subset B}_{2n,-}
\\
\cP(X)_{2n+1,-}
&:=
\End_\cC((\overline{X}\otimes X)^{\otimes n}\otimes \overline{X})
\longrightarrow
\cQ_{2k+2,+}' \cap \cQ_{2k+2n+2,-}
=
A_1'\cap A_{2n+2}
=
\cP^{A\subset B}_{2n+1,-}
\end{align*}
which is compatible with the right inclusion
and the left inclusion.
Indeed, for all $y\in \cP(X)_{n,-}$,
$$
\varphi_{n+1,+}(\id_X\otimes y)
=
\id_{(X\otimes \overline{X})^{\otimes k} \otimes X}\otimes y
=
\varphi_{n,-}(y)
\in 
\cP^{A\subset B}_{n,-}
\subset
\cP^{A\subset B}_{n+1,+}
$$
as $A_1'\cap A_{n+1}\subset A_0'\cap A_{n+1}$.

To see that the $\varphi_\bullet$ assemble into a planar $*$-algebra isomorphism, by \cite[Pf.~of~Lem.~2.49]{MR2812459}, it suffices to check 
that capping on the left is compatible with $\varphi_{n,+}$.
This argument (and indeed this entire part of the proof) is identical to \cite[Pf.~of~Thm.~4.1(iii)]{MR2812459}.
For $n\geq 1$, the left capping map 
$\cP^{A\subset B}_{n,+} \to \cP^{A\subset B}_{n-1,-}$ is given by
$x\mapsto d^{-1} \sum_b bxb^*$ \cite[Thm.~250]{MR2812459}.
By non-degeneracy of the composite commuting squares of \eqref{eq:CommutingSquareLattice} and Proposition \ref{rem:HorizontallyConnected}, 
there is a Pimsner-Popa basis $\{b\}$ for $\cQ_{2k+1,+}\subset B$ over $\cQ_{2k,-}\subset A$ which is also a Pimsner-Popa basis for $B$ over $A$.
We now employ
Vaughan Jones' diagrammatic trick from \cite[Pf.~of~Thm.~4.1(iii)]{MR2812459}.
For all $y\in \cP(X)_{n,+}$ with $n\geq 1$,
$$
\frac{1}{d}
\sum_b b\cdot \varphi_{n,+}(y)\cdot b^*
=
\frac{1}{d}
\sum_b
\tikzmath{
\fill[\AColor] (-.9,-1.25) rectangle (-.1,1.25);
\fill[\BColor] (-.1,-1.25) rectangle (.2,1.25);
\draw (-.1,-1.25) -- (-.1,1.25);
\draw[very thick] (-.4,1.25) node[left, yshift=-.2cm] {$\scriptstyle 2k$} -- (-.4,-1.25) node[left, yshift=.2cm] {$\scriptstyle 2k$};
\draw[very thick] (.2,1.25) node[right, yshift=-.2cm] {$\scriptstyle n-1$} -- (.2,-1.25) node[right, yshift=.2cm] {$\scriptstyle n-1$};
\roundNbox{unshaded}{(0,0)}{.25}{0}{.1}{$y$}
\roundNbox{unshaded}{(-.2,.75)}{.25}{.1}{0}{$b$}
\roundNbox{unshaded}{(-.2,-.75)}{.25}{.1}{0}{$b^*$}
}
=
\frac{1}{d}
\sum_b
\tikzmath{
\fill[\AColor] (-1.6,1.1) -- (-.9,1.1) -- (-.9,.25) arc (-180:0:.15cm) -- (-.6,.75) arc (180:0:.25cm) -- (-.1,-.75) arc (0:-180:.25cm) -- (-.6,-.25) arc (0:180:.15cm) -- (-.9,-1.1) --(-1.6,-1.1);
\fill[\BColor] (.1,1.1) -- (-.9,1.1) -- (-.9,.25) arc (-180:0:.15cm) -- (-.6,.75) arc (180:0:.25cm) -- (-.1,-.75) arc (0:-180:.25cm) -- (-.6,-.25) arc (0:180:.15cm) -- (-.9,-1.1) --(.1,-1.1);
\draw[very thick] (-1.1,1.1) node[above] {$\scriptstyle 2k$} -- (-1.1,-1.1) node[below] {$\scriptstyle 2k$};
\draw[very thick] (.1,1.1) node[above] {$\scriptstyle n-1$} -- (.1,-1.1) node[below] {$\scriptstyle n-1$};
\draw (-.9,1.1) -- (-.9,.25) arc (-180:0:.15cm) -- (-.6,.75) arc (180:0:.25cm) -- (-.1,-.75) arc (0:-180:.25cm) -- (-.6,-.25) arc (0:180:.15cm) -- (-.9,-1.1);
\roundNbox{unshaded}{(0,0)}{.25}{0}{0}{$y$}
\roundNbox{unshaded}{(-1,.5)}{.25}{0}{0}{$b$}
\roundNbox{unshaded}{(-1,-.5)}{.25}{0}{0}{$b^*$}
\draw[very thick, dashed, rounded corners=5pt] (-1.4,-.9) rectangle (-.4,.9);
}
=
\tikzmath{
\fill[\AColor] (-.9,-.6) rectangle (-.6,.6);
\fill[\BColor] (-.6,-.6) rectangle (.1,.6);
\draw[very thick] (-.6,.6) node[above] {$\scriptstyle 2k+1$} -- (-.6,-.6) node[below] {$\scriptstyle 2k+1$};
\draw[very thick] (.1,.6) node[above] {$\scriptstyle n-1$} -- (.1,-.6) node[below] {$\scriptstyle n-1$};
\filldraw[fill=\AColor] (-.1,.25) arc (0:180:.15cm) -- (-.4,-.25) arc (-180:0:.15cm);
\roundNbox{unshaded}{(0,0)}{.25}{0}{0}{$y$}
}=
\varphi_{n-1,-}\left(
\tikzmath{
\fill[\BColor] (-.6,-.6) rectangle (.1,.6);
\draw[very thick] (.1,.6) node[above] {$\scriptstyle n-1$} -- (.1,-.6) node[below] {$\scriptstyle n-1$};
\filldraw[fill=\AColor] (-.1,.25) arc (0:180:.15cm) -- (-.4,-.25) arc (-180:0:.15cm);
\roundNbox{unshaded}{(0,0)}{.25}{0}{0}{$y$}
}
\right).
$$
Hence $\varphi_\bullet : \cP(X)_\bullet \to \cP^{A\subset B}_\bullet$ is a planar $*$-algebra isomorphism.

\item
We claim the distortion for ${}_AL^2B{}_B$ is the standard distortion
$\sigma_{ij}
=
d_X \beta_j/\alpha_i$.
Indeed,
since ${}_AL^2B {}_B$ is finite depth and thus extremal, it suffices to prove this formula holds when $p_iq_j \neq 0$.
In this case, by \eqref{eq:TraceMatrix} above, we have $\delta_{ij} = D_{ij} / T_{ij}$.
Now by 
\eqref{eq:pi-q_j bubble} and \eqref{eq:SphericalState},
$$
T_{ij}
=
\tr_{B_j}(p_iq_j)
=
\frac{\tr_{B}(p_iq_j)}{\tr_B(q_j)}
=
\frac{
d_X^{-1} \psi\left(
\,
\tikzmath[xscale=-1]{
  \fill[\BColor, rounded corners=5] (-.5,-.75) rectangle (1.75,.75);
  \filldraw[fill=\AColor] (1,0) circle (.5cm);
  \roundNbox{unshaded}{(0,0)}{.25}{0}{0}{$q_j$}
  \roundNbox{unshaded}{(1,0)}{.25}{0}{0}{$p_i$}
}
\,
\right)
}
{\beta_j^2}
=
\frac{d_X^{-1}
\frac{D_{ij}\alpha_i}{\beta_j}
\psi(q_j)}
{\beta_j^2}
=
\frac{D_{ij} \alpha_i}{d_X \beta_j}.
$$
But combining 
\eqref{eq:TraceMatrix},
\eqref{eq:Distortion, Jones dim, and von Neumann dim},
and extremality of ${}_AL^2B{}_B$,
we also have
$T_{ij}=\Delta_{ij}/\delta_{ij} = D_{ij}/\delta_{ij}$.
The result follows.
\qedhere
\end{enumerate}
\end{proof}

\begin{rem}
Observe that the inclusion $A\subset B$ constructed above in the proof of Theorem \ref{thm:ExistsStandardInclusion} admits an infinite Jones tunnel.
Indeed,
By removing a copy of $\overline{X}$ on the right, we can add another row below the doubly infinite lattice of commuting squares \eqref{eq:LatticeForReconstruction}.
The argument from step
\ref{step:InductiveLimitJonesTower} 
above in the proof of Theorem \ref{thm:ExistsStandardInclusion}
can be repeated for this new row to obtain a multifactor $A_{-1}\subset A_0$ such that $A_1$ is isomorphic to the basic construction algebra $\langle A_0, A_{-1}\rangle$.
As one can keep removing tensor factors from the right, a simple induction argument shows that we may continue in this fashion to obtain an infinite Jones tunnel
$$
\xymatrix@C=5pt@R=2pt{
\cdots &\subset & A_{-2} &\subset & A_{-1} &\subset & A_0 &\subset & A_1 &\subset & A_2 &\subset & \cdots.
}
$$
\end{rem}

\subsection{Distortions realizable by multifactor inclusions}
\label{sec:DistortionsRealizableByInclusions}

For this section, we assume both Notations \ref{nota:BimoduleNotation} and \ref{nota:InclusionNotation} for a finite index inclusion $A\subset B$ of finite multifactors.
Recall from Definition \ref{defn:MoritaEquivalence} in \S\ref{sec:StandardInvariants} that given an invertible $A'-A$ bimodule ${}_{A'}Y{}_A$ (so that $A'=(A^{\op})'\cap B(Y)$), the Morita equivalent inclusion is $A' \subset B'$ where $B' := (B^{\op})'\cap B(Z)$ and $Z_B := Y_A \boxtimes L^2B{}_B$.
Remember that the commutants of $A$ and $B$ are taken in different representations, so $A'\subset B'$ instead of the reverse inclusion.

\begin{prop}
\label{prop:DistortionFormulaUnderMoritaEquivalence}
Suppose $A\subset B$ 
is a finite index connected inclusion of finite multifactors,
${}_{A'}Y{}_A$ is a Morita equivalence bimodule, and $A'\subset B'$ is the induced Morita equivalent inclusion.
Denote
$\delta=\delta({}_AL^2B{}_B)$, and $\Delta=\Delta({}_AL^2B{}_B)$.
The distortion $\delta'$ of $A'\subset B'$ is given by 
$$
\delta'_{ij}=\delta_{ij}\rho_i^{-1} \sum_{h=1}^{a} \rho_h \Delta_{hj}\delta_{hj}^{-1}
$$
where
$\rho_i :=\vNdim_R(Yp_i {}_{A_i})$.
\end{prop}
\begin{proof}
For $1\leq i\leq a$, let $p'_i\in Z(A')$ be the minimal central projection corresponding to $p_i \in A$, and for  
$1\leq j\leq b$, let $q'_j\in Z(B')$ be the minimal central projection in $B'$ corresponding to $q_j \in B$.
We calculate that
\begin{align*}
\vNdim_L({}_{A'_i} p'_i L^2B' q'_j)
&=
\vNdim_L({}_{A'_i} p'_i Y \boxtimes_A L^2B \boxtimes_A \overline{Y} q'_j {}_{B'_j} )
&&
\text{\cite[Prop.~3.1]{MR703809}}
\displaybreak[1]\\&=
\vNdim_L( {}_{\widetilde{A}_i} Y p_i )
\vNdim_L({}_{A_i}p_i L^2B \boxtimes_A \overline{Y} q'_j )
\displaybreak[1]\\&=
\rho_i^{-1}
\vNdim_R(q'_j Y \boxtimes_A L^2B p_i {}_{A_i})
\displaybreak[1]\\&=
\rho_i^{-1}
\vNdim_R( Y \boxtimes_A L^2B q_j p_i {}_{A_i})
\displaybreak[1]\\&=
\rho_i^{-1}
\sum_{h=1}^a 
\vNdim_R( Yp_h \boxtimes_{A_h} p_h L^2B q_j p_i {}_{A_i})
\displaybreak[1]\\&=
\rho_i^{-1}
\sum_{h=1}^a 
\vNdim_R( Yp_h) \vNdim_R(p_h L^2B q_j p_i {}_{A_i})
\displaybreak[1]\\&=
\rho_i^{-1}
\sum_{h=1}^a 
\rho_h \vNdim_R(p_h L^2B q_j \boxtimes_{B_j} q_j L^2B p_i {}_{A_i})
\displaybreak[1]\\&=
\rho_i^{-1}
\sum_{h=1}^a 
\rho_h \vNdim_R(p_h L^2B q_j {}_{B_j}) \vNdim_R(q_j L^2B p_i {}_{A_i})
\displaybreak[1]\\&=
\rho_i^{-1}
\sum_{h=1}^a 
\rho_h \frac{\Delta_{hj}}{\delta_{hj}} \vNdim_L({}_{A_i}p_i L^2B q_j)
&&
\text{\eqref{eq:Distortion, Jones dim, and von Neumann dim}}
\displaybreak[1]\\&=
\rho_i^{-1}
\sum_{h=1}^a 
\rho_h \frac{\Delta_{hj}}{\delta_{hj}} \delta_{ij}\Delta_{ij}
&&
\text{\eqref{eq:Distortion, Jones dim, and von Neumann dim}}
\displaybreak[1]\\&=
\rho_i^{-1}
\delta_{ij}\Delta_{ij}
\sum_{h=1}^a 
\rho_h \frac{\Delta_{hj}}{\delta_{hj}}.
\end{align*}
By a similar calculation, we have
$$
\vNdim_R(p'_i L^2B' q'_j {}_{B'_j} )
=
\rho_i
\delta_{ij}^{-1}\Delta_{ij} \left(\sum_{h=1}^a \rho_h \frac{\Delta_{hj}}{\delta_{hj}}\right)^{-1}.
$$
We conclude that
the distortion $\delta'_{ij}$ for $A'\subset B'$ is given by
$$
\delta'_{ij}
=
\sqrt{\frac{\vNdim_L({}_{A'_i} p'_i L^2B' q'_j )}{\vNdim_R(p'_i L^2B' q'_j {}_{B'_j} )}}
=
\delta_{ij}\rho_i^{-1} \sum_{h=1}^{a} \rho_h \Delta_{hj}\delta_{hj}^{-1}
$$
as claimed.
\end{proof}

\begin{cor}
\label{cor:MEPreservesExtremality}
Suppose $A\subset B$ and $A' \subset B'$ are as in the hypotheses of Proposition \ref{prop:DistortionFormulaUnderMoritaEquivalence}.
If $A\subset B$ is extremal, then so is $A' \subset B'$.
\end{cor}
\begin{proof}
By Theorem \ref{thm:ExtremalCharacterization}, we must
show ${}_{A'} L^2B'{}_{B'}$ has constant distortion and that $\delta'$ satisfies \eqref{eq:LoopsMultiplyTo1}.
Observe that each of 
${}_{A'} Y_{A}$,
${}_AL^2B{}_B$,
and
${}_{B'} Z_{B}$ 
have constant distortion,
so ${}_{A'} L^2B'{}_{B'}$
has constant distortion by \eqref{eq:L2 of ME} and Corollary \ref{cor:ConstantDistortionMultiplicativeForFactors}.
Now to prove $\delta'$ satisfies \eqref{eq:LoopsMultiplyTo1}, we use Lemma \ref{lem:ExtendGraphWeighting}.
Since $A\subset B$ is extremal, there are $\eta \in \bbR^a_{>0}$ and $\xi \in \bbR^b_{>0}$ (unique up to a simultaneous uniform scaling) such that $\delta_{ij} = \xi_j/ \eta_i$.
By Proposition \ref{prop:DistortionFormulaUnderMoritaEquivalence}, we have
$$
\delta'_{ij}
=
\delta_{ij}
\rho_i^{-1} \sum_{h=1}^a \rho_h \Delta_{hj} \delta_{hj}^{-1}
=
\frac{\xi_j\sum_{h=1}^a \rho_h \Delta_{hj} \delta_{hj}^{-1}}{ \eta_i \rho_i}
=:
\frac{\xi'_j}{\eta'_i}.
$$
Thus by Lemma \ref{lem:ExtendGraphWeighting}(2) again, $\delta'$ satisfies \eqref{eq:LoopsMultiplyTo1}.
\end{proof}

The next corollary follows immediately.

\begin{cor}
\label{cor:FormulaForDistortionMEtoStandard}
Suppose $A\subset B$ and $A' \subset B'$ are as in the hypotheses of Proposition \ref{prop:DistortionFormulaUnderMoritaEquivalence}.
Moreover, assume $A\subset B$ is extremal and $X={}_AL^2B{}_B$ has the standard distortion. $\sigma_{ij}=d_X \beta_j / \alpha_i$.
Then $A' \subset B'$ 
is extremal, and the distortion $\delta'$ is given by $\delta'_{ij} = \alpha_i^{-1}\rho_i^{-1}\sum_{h=1}^a \rho_h D_{hj}\alpha_h$.
\end{cor}

\begin{thm}
\label{thm:MultifactorInclusionMEtoHomogeneous}
Suppose 
$A\subset B$ is a 
finite index
connected
hyperfinite
$\rm II_1$ multifactor
inclusion with
finite depth standard invariant $\cC({}_AL^2B{}_B)$.
Then there is a Morita equivalence ${}_{A'} Y_{A}$ such that the induced Morita equivalent inclusion $A' \subset B'$ is homogeneous.
That is, $A'\subset B'$ has the standard distortion $\sigma$ with respect to ${}_AL^2B{}_B$.
\end{thm}
\begin{proof}
Let $\delta:=\delta({}_AL^2B{}_B)$.
Since $\delta/\sigma$ satisfies \eqref{eq:DistortionExtensionCondition}, by Lemma \ref{lem:ExtendGraphWeighting}(2), there are $(\eta_i)\in \bbR_{>0}^a$ and $(\xi_j)\in \bbR_{>0}^b$ such that
$$
\frac{\xi_j}{\eta_i} = \frac{\delta_{ij}}{\sigma_{ij}}
\Longleftrightarrow
\eta_i \delta_{ij} = \xi_j \sigma_{ij}.
$$
Since the fundamental group $\cF(R)$ of the hyperfinite $\rm II_1$ factor is $\bbR_{>0}$ \cite{MR0009096}, there exists
a Morita equivalence
${}_{A'} Y_{A}$
such that $\vNdim(Yp_i{}_{A_i}) = \eta_{i}^{-1}$.
By Proposition \ref{prop:DistortionFormulaUnderMoritaEquivalence},
the distortion of the induced Morita equivalent inclusion
$\widetilde{A}\subset \widetilde{B}$
is given by
\begin{align*}
\delta'_{ij}
&=
\delta_{ij}\eta_i \sum_{h=1}^a \eta_h^{-1}\delta_{hj}^{-1} D_{hj}
=
\xi_j \sigma_{ij} \sum_{h=1}^a \xi_j^{-1}\sigma_{hj}^{-1} D_{hj}
=
\sigma_{ij} \sum_{h=1}^a \sigma_{hj}^{-1} D_{hj}
\\&
=
\sigma_{ij} \sum_{h=1}^a \frac{\alpha_h}{d_X \beta_j} D_{hj}
=
\frac{\sigma_{ij}}{d_X\beta_j} \sum_{h=1}^a \alpha_h D_{hj}
=
\frac{\sigma_{ij}}{d_X\beta_j} d_X\beta_j
=
\sigma_{ij}.
\end{align*}
We conclude that 
$A'\subset B'$
has the standard distortion with respect to $X$, and thus this inclusion is homogeneous by Theorem \ref{thm:tunnelequivalencies}.
\end{proof}

\begin{rem}
Observe that Theorem \ref{thm:MultifactorInclusionMEtoHomogeneous} does not hold for $A\subset B$ finite dimensional.
Indeed, no inclusion of finite dimensional von Neumann algebras with Bratteli diagram the $A_4$ Coxeter-Dynkin diagram can be homogeneous, since $\delta$ will always have rational entries, but the standard distortion function $\sigma$ has irrational entries cf.~Example \ref{ex:A4-example}.
\end{rem}

Using Theorem \ref{thm:MultifactorInclusionMEtoHomogeneous}, given a 2-shaded indecomposable unitary multifusion category $\cC$ and a generator $X\in \cC^{+-}$,
we can say exactly which distortions of $X$ arise from realizations of $\cC(X)$ as a standard invariant of a finite index
connected
inclusion of finite multifactors $A\subset B$.

\begin{prop}
\label{prop:MoritaEquivalentDistortions}
Let $\cC$ be a 2-shaded indecomposable unitary multifusion category with generator $X\in \cC^{+-}$ as in Definition \ref{defn:2-shadingGeneratesConnected},
and let $D=D_X$, $d=d_X$, and $\alpha$ be as in Definition \ref{defn:StandardDualFunctor}.
Suppose $\delta: \{1,\dots, a\}\times \{1,\dots, b\}\to \bbR_{>0}$ is an arbitrary  function satisfying \eqref{eq:DistortionExtensionCondition}, i.e.,
$$
\delta_{ij}\delta_{i'j'}
=
\delta_{ij'}\delta_{i'j}
\qquad\qquad
\forall
1\leq i,i'\leq a,
\qquad
\text{and}
\qquad
1\leq j,j'\leq b.
$$
The following are equivalent.
\begin{enumerate}[label=(\arabic*)]
\item
There is a
finite index
connected
inclusion of hyperfinite $\rm II_1$ multifactors
$A\subset B$ with standard invariant equivalent to $\cC$
such that $\delta = \delta({}_AL^2B{}_B)$.
\item
There exists $(\rho_i)\in \bbR_{>0}^a$ such that
$\delta_{ij}= \alpha_i^{-1}\rho_i^{-1}\sum_{h=1}^a \rho_h D_{hj}\alpha_h$.
\item
Writing $\delta_{ij}= \xi_j / \eta_i$
(which is unique up to uniformly scaling all $\xi_j, \eta_i$ by Lemma \ref{lem:ExtendGraphWeighting}(2)), we have
$\xi_j = \sum_{h=1}^a \eta_h D_{hj}$.
\end{enumerate}
\end{prop}
\begin{proof}
\item[\underline{$(1)\Rightarrow (2):$}]
Suppose $\delta = \delta({}_AL^2B{}_B)$ for some finite index connected inclusion $A\subset B$ of hyperfinite finite multifactors.
By Theorem \ref{thm:MultifactorInclusionMEtoHomogeneous}, 
$A\subset B$ is Morita equivalent to an inclusion with standard distortion.
Hence (1) follows by Corollary \ref{cor:FormulaForDistortionMEtoStandard}.

\item[\underline{$(2)\Rightarrow (1):$}]
By Theorem \ref{thm:ExistsStandardInclusion}, there exists a finite index connected homogeneous hyperfinite $\rm II_1$ multifactor inclusion $A\subset B$ with standard distortion.
Let $Y_A$ be any faithful right $A$-module such that $\rho_i = \vNdim_R(Yp_i{}_{A_i})$, and let $\widetilde{A}\subset \widetilde{B}$ be the induced Morita equivalent inclusion.
By Corollary \ref{cor:FormulaForDistortionMEtoStandard},
$\widetilde{\delta}_{ij}= \alpha_i^{-1}\rho_i^{-1}\sum_{h=1}^a \rho_h D_{hj}\alpha_h$ as desired.

\item[\underline{$(2)\Rightarrow (3)$:}]
Setting $\xi_j := \sum_{h=1}^a \rho_h D_{hj} \alpha_h$ and $\eta_i := \alpha_i \rho_i$, we clearly have $\xi_j = \sum_{h=1}^a \eta_h D_{hj}$.

\item[\underline{$(3)\Rightarrow (2)$:}]
Setting $\rho_i := \eta_i/\alpha_i$ gives the desired formula for $\delta_{ij}$.
\end{proof}

\begin{cor}
\label{cor:trA and D determine trB and delta}
Suppose $A\subset B$ is a
finite depth finite index
connected inclusion of finite multifactors and $\tr_B$ is the unique Markov trace on $B$.
The distortion $\delta$ of ${}_AL^2B{}_B$ is given by
\begin{equation}
\tag{\ref{eq:delta in terms of trA and D}}
\delta_{ij}
=
\left(\frac{\alpha_i}{\tr_A(p_i)}\right)
\sum_{h=1}^a 
\left(\frac{\tr_A(p_h)}{\alpha_h}\right)D_{hj}.
\end{equation}
\end{cor}
\begin{proof}
By Corollary \ref{cor:FiniteDepthMultifactorBimoduleExtremal}, $A\subset B$ is extremal, so by Theorem \ref{thm:Minimal=Extremal},
$$
\delta_{ij}
=
d\left(\frac{\tr_B(q_j)}{\beta_j}\right)\left(\frac{\alpha_i}{\tr_A(p_i)}\right).
$$
By Proposition \ref{prop:MoritaEquivalentDistortions}(3),
$$
d\left(\frac{\tr_B(q_j)}{\beta_j}\right)
=
\sum_{h=1}^a \left(\frac{\tr_A(p_h)}{\alpha_h}\right)D_{hj}.
$$
The result follows.
\end{proof}

\subsection{Classification of finite depth connected hyperfinite multifactor inclusions}

We now recall Popa's theorem for 
finite index 
finite depth 
homogeneous 
connected 
hyperfinite $\rm II_1$ multifactor inclusions.
We provide a proof here and in Appendix \ref{appendix:Popa} for completeness and convenience of the reader.

\begin{thm}
\label{thm:PopaTheoremHomogeneousMultifactor}
Suppose $A\subset B$ and $\widetilde{A} \subset \widetilde{B}$ are two finite index homogeneous connected hyperfinite $\rm II_1$ multifactor inclusions with isomorphic finite depth standard invariants.
For every isomorphism $\varphi_\bullet : \cP_\bullet^{A\subset B} \to \cP_\bullet^{\widetilde{A}\subset \widetilde{B}}$,
there exists a (non-unique) $*$-isomorphism $\varphi: B \to \widetilde{B}$ taking $A$ onto $\widetilde{A}$ which induces the original $*$-isomorphism $\varphi_\bullet$ of standard invariants.
\end{thm}
\begin{proof}
By Theorem \ref{thm:ExistsGeneratingHomogeneousTunnel}, there are generating tunnels $(A_{-n})_{n\in \bbN}$ and $(\widetilde{A}_{-n})_{n \in \bbN}$.
By tunnel/tower duality (see Fact \ref{fact:TunnelTowerDuality}),\footnote{
We warn the reader that there are two numbering conventions for the Jones tunnel.
If one sets $B=A_1$ and $A=A_0$ as we do in the main body of this article, then the box spaces 
of standard invariant $\cP^{A\subset B}_{n,\pm}$ 
are given by
$\cP^{A\subset B}_{n,+}:=A_0'\cap A_n$
and
$\cP^{A\subset B}_{n,-}:=A_1'\cap A_{n+1}$,
which are anti-isomorphic to
$A_{-n-1}'\cap A_0$ 
and
$A_{-n}'\cap A_1$
respectively.
However, in Fact \ref{fact:TunnelTowerDuality}, we set $B=B_0$ and $A=B_{-1}$, under which we get that $B_0'\cap B_n$ is anti-isomorphic to $B_0' \cap B_{-n}$.
Thus while the first convention is more practical for discussing the standard invariant, this second convention is more practical for discussing tunnel/tower duality.
} 
$\varphi_\bullet$ induces isomorphisms
$$
\varphi^{\op}_{n,+}:A_{-n-1}'\cap A \to \widetilde{A}_{-n-1}'\cap \widetilde{A}
\qquad\qquad
\varphi^{\op}_{n,-}:A_{-n}'\cap B \to \widetilde{A}_{-n}'\cap \widetilde{B}
$$
for all $n\in \bbN$ which are compatible with the inclusions, conditional expectations, and unique Markov traces.
By the generating property, 
given this fixed generating tunnel, 
there is a unique unitary isomorphism
$$
u_\varphi : L^2\left(\bigcup A_{-n}'\cap B, \tr\right) \cong L^2B 
\to 
L^2\left(\bigcup \widetilde{A}_{-n}'\cap \widetilde{B}, \tr\right)\cong L^2\widetilde{B}
$$
which intertwines 
the left $A_{-n}'\cap B$ and $\widetilde{A}_{-n}'\cap \widetilde{B}$ actions
and
the left $A_{-n-1}'\cap A$ and $\widetilde{A}_{-n-1}'\cap \widetilde{A}$ actions.
We conclude that $\Ad(u_\varphi)$ is the desired isomorphism.
Since $\varphi$ restricts to the isomorphisms $\varphi^{\op}_{n,\pm}$ on the relative commutants of the Jones tunnel, again by tunnel/tower duality, $\varphi$ induces the original $*$-isomorphism $\varphi_\bullet$ on the standard invariant.
\end{proof}

Given two finite index extremal inclusions of finite multifactors $A\subset B$ and $\widetilde{A}\subset \widetilde{B}$ together with isomorphisms $\varphi_{0,+}: Z(A) \to Z(\widetilde{A})$ and $\varphi_{0,-}: Z(B) \to Z(\widetilde{B})$, we say that $\varphi_{0,\pm}$ \emph{preserves distortion} if
$$
\delta^{\widetilde{A}\subset \widetilde{B}}_{\varphi_{0,+}(p), \varphi_{0,-}(q)}
=
\delta^{A\subset B}_{p,q}
$$
for all minimal projections $p\in Z(A)$ and $q\in Z(B)$.
Here, we view $\delta^{A\subset B}$ and $\delta^{\widetilde{A}\subset \widetilde{B}}$ as matrices indexed by minimal projections rather than some enumerations of these projections.

\begin{cor}
\label{cor:FiniteDepthHyperfiniteInclusionsIsoIffSameDistortion}
Suppose $A\subset B$ and $\widetilde{A}\subset \widetilde{B}$ are two 
finite depth finite index connected hyperfinite $\rm II_1$ multifactor inclusions.
For every $*$-isomorphism of standard invariants $\varphi_\bullet: \cP^{A\subset B}_\bullet \to \cP^{\widetilde{A}\subset \widetilde{B}}_\bullet$ which preserves distortion,
there exists a (non-unique)
$*$-isomorphism $\varphi: B \to \widetilde{B}$ taking $A$ onto $\widetilde{A}$ which induces the original $*$-isomorphism $\varphi_\bullet$ of standard invariants.
\end{cor}
\begin{proof}
Since $\varphi_\bullet$ preserves distortion, using Theorem \ref{thm:MultifactorInclusionMEtoHomogeneous},
there exist faithful right modules $Y_A$ and $\widetilde{Y}_{\widetilde{A}}$ such that:
\begin{itemize}
\item
$\vNdim_R(Yp) = \vNdim_R(\widetilde{Y}\varphi(p))$ for every minimal projection $p\in Z(A)$, and
\item
the induced Morita equivalent inclusions $A'\subset B'$ and $\widetilde{A}' \subset \widetilde{B}'$ have the standard distortion, and are thus homogeneous by Theorem \ref{thm:tunnelequivalencies}.
(Recall that $A':= (A^{\op})'\cap B(Y)$, $Z:=Y\boxtimes_A L^2B$, $B':=(B^{\op})'\cap B(Z)$, and similarly for $\widetilde{A}', \widetilde{Z}$, $\widetilde{B}'$.)
\end{itemize}
By the first bullet point above, there exists a right $A$-linear unitary $w: Y_Z \to \widetilde{Y}_{\varphi(A)}$.
By Lemma \ref{lem:ME InclusionStdInv}, we have an isomorphism of standard invariants $\varphi'_\bullet:\cP_\bullet^{A' \subset B'} \to \cP_\bullet^{\widetilde{A}' \subset \widetilde{B}'}$ given by the following commutative diagram:
\begin{equation*}
\begin{tikzcd}[column sep=3em]
\cP^{A\subset B}_\bullet
\arrow[r,leftrightarrow,"\cong"]
\arrow[d,"\varphi_\bullet"]
&
\cP({}_AL^2B{}_B)_\bullet
\arrow[r,"{\cP(Y,Z,\psi)_\bullet}"]
&
\cP({}_{A'}L^2B'{}_{B'})_\bullet
\arrow[r,leftrightarrow,"\cong"]
&
\cP^{A'\subset B'}_\bullet
\arrow[d,"\varphi'_\bullet"]
\\
\cP^{\widetilde{A}\subset \widetilde{B}}_\bullet
\arrow[r,leftrightarrow,"\cong"]
&
\cP({}_{\widetilde{A}}L^2\widetilde{B}{}_{\widetilde{B}})_\bullet
\arrow[r,"{\cP(\widetilde{Y},\widetilde{Z},\widetilde{\psi})_\bullet}"]
&
\cP({}_{\widetilde{A}'}L^2\widetilde{B}'{}_{\widetilde{B}'})_\bullet
\arrow[r,leftrightarrow,"\cong"]
&
\cP^{\widetilde{A}'\subset \widetilde{B}'}_\bullet
\end{tikzcd}
\end{equation*}
By the second bullet point above and Theorem \ref{thm:PopaTheoremHomogeneousMultifactor}, there exists an isomorphism $\varphi' : (A'\subset B') \to (\widetilde{A}'\subset \widetilde{B}')$ such that
$\cP(\varphi')_\bullet = \varphi'_\bullet$.
By \eqref{eq:SF-PA-TC-isos}, the isomorphism $\cP(L^2\widetilde{A}'_{\varphi'}, L^2\widetilde{B}'_{\varphi'}, \psi_{\varphi'})_\bullet$ fits in the following commutative diagram:
\begin{equation*}
\begin{tikzcd}[column sep=3em]
\cP^{A\subset B}_\bullet
\arrow[r,leftrightarrow,"\cong"]
\arrow[d,"\varphi_\bullet"]
&
\cP({}_AL^2B{}_B)_\bullet
\arrow[r,"{\cP(Y,Z,\psi)_\bullet}"]
&
\cP({}_{A'}L^2B'{}_{B'})_\bullet
\arrow[d,"{\cP(L^2\widetilde{A}'_{\varphi'}, L^2\widetilde{B}'_{\varphi'}, \psi_{\varphi'})_\bullet}"]
\arrow[r,leftrightarrow,"\cong"]
&
\cP^{A'\subset B'}_\bullet
\arrow[d,"\cP(\varphi')_\bullet=\varphi'_\bullet"]
\\
\cP^{\widetilde{A}\subset \widetilde{B}}_\bullet
\arrow[r,leftrightarrow,"\cong"]
&
\cP({}_{\widetilde{A}}L^2\widetilde{B}{}_{\widetilde{B}})_\bullet
\arrow[r,"{\cP(\widetilde{Y},\widetilde{Z},\widetilde{\psi})_\bullet}"]
&
\cP({}_{\widetilde{A}'}L^2\widetilde{B}'{}_{\widetilde{B}'})_\bullet
\arrow[r,leftrightarrow,"\cong"]
&
\cP^{\widetilde{A}'\subset \widetilde{B}'}_\bullet
\end{tikzcd}
\end{equation*}
Now since 
$A$ is the commutant of the right $A'$-action on ${}_A\overline{Y}{}_{A'}$ 
and
$B$ is the commutant of the right $B'$ action on 
$\overline{Y}\boxtimes_{A'} L^2B'_{B'} \cong \overline{Z}{}_{B'}$,
we may invoke Proposition \ref{prop:TransportIsoAlongCompatibleME}
with the roles of $A,B,\widetilde{A}, \widetilde{B},Y,Z$ swapped with $A',B', \widetilde{A}', \widetilde{B}', \overline{Y}, \overline{Z}$ to get an isomorphism $\varphi: (A\subset B) \to (\widetilde{A} \subset \widetilde{B})$ such that the following diagram commutes:
\begin{equation*}
\begin{tikzcd}[column sep=3em]
\cP^{A\subset B}_\bullet
\arrow[r,leftrightarrow,"\cong"]
\arrow[d,"\varphi_\bullet"]
&
\cP({}_AL^2B{}_B)_\bullet
\arrow[d,"{\cP(L^2\widetilde{A}_{\varphi}, L^2\widetilde{B}_{\varphi}, \psi_{\varphi})_\bullet}"]
\arrow[r,"{\cP(Y,Z,\psi)_\bullet}"]
&
\cP({}_{A'}L^2B'{}_{B'})_\bullet
\arrow[d,"{\cP(L^2\widetilde{A}'_{\varphi'}, L^2\widetilde{B}'_{\varphi'}, \psi_{\varphi'})_\bullet}"]
\arrow[r,leftrightarrow,"\cong"]
&
\cP^{A'\subset B'}_\bullet
\arrow[d,"\cP(\varphi')_\bullet=\varphi'_\bullet"]
\\
\cP^{\widetilde{A}\subset \widetilde{B}}_\bullet
\arrow[r,leftrightarrow,"\cong"]
&
\cP({}_{\widetilde{A}}L^2\widetilde{B}{}_{\widetilde{B}})_\bullet
\arrow[r,"{\cP(\widetilde{Y},\widetilde{Z},\widetilde{\psi})_\bullet}"]
&
\cP({}_{\widetilde{A}'}L^2\widetilde{B}'{}_{\widetilde{B}'})_\bullet
\arrow[r,leftrightarrow,"\cong"]
&
\cP^{\widetilde{A}'\subset \widetilde{B}'}_\bullet
\end{tikzcd}
\end{equation*}
Finally, by \eqref{eq:SF-PA-TC-isos}, the isomorphisms $\cP(L^2\widetilde{A}_\varphi, L^2\widetilde{B}_\varphi, \psi_\varphi)_\bullet$ and $\cP(\varphi)_\bullet$ satisfy a commutative square, and we conclude that $\cP(\varphi)_\bullet = \varphi_\bullet$ as desired.
\end{proof}

We now prove Theorem \ref{thm:ClassificationOfFiniteDepthHyperfinite}
in two parts.

\begin{thm}[{Theorem \ref{thm:ClassificationOfFiniteDepthHyperfinite}, Part $\rm I$}]
\label{thm:A-PartI}
The map $A\subset B \mapsto (\cP_\bullet^{A\subset B},\tr^{\rm Markov}_B|_{Z(A)})$ descends to a well-defined bijection
$$
\frac{
\left\{
\parbox{5.8cm}{\rm
Finite depth finite index connected 
hyperfinite $\rm II_1$ multifactor inclusions
$A\subset B$ \phantom{$\cP_{0,+}$}
}\right\}
}
{
\parbox{5.8cm}{\rm
$\varphi: B_1 \xrightarrow{\sim} B_2$
taking $A_1$ onto $A_2$
}}
\cong
\frac{
\left\{
\parbox{6.5cm}{\rm
Pairs $(\cP_\bullet, \tau)$
with $\cP_\bullet$ a finite depth indecomposable unitary 2-shaded planar algebra and $\tau$ a faithful state on $\cP_{0,+}$
}\right\}
}{
\parbox{6.3cm}{\rm
$\varphi_\bullet : \cP^1_\bullet \xrightarrow{\sim} \cP^2_\bullet$ such that $\tau^2\circ \varphi_{0,+}=\tau^1$
}}
\,.
$$
\end{thm}
\begin{proof}
\mbox{}
\item[\underline{Well-defined:}]
Suppose $A\subset B$ and $\widetilde{A}\subset \widetilde{B}$ are 
finite index finite depth connected hyperfinite $\rm II_1$ multifactor inclusions.
Let $\tr_B$ and $\tr_{\widetilde{B}}$ denote the unique respective Markov traces, and denote by $\tau_A$ and $\tau_{\widetilde{A}}$ their restrictions to $Z(A)$ and $Z(\widetilde{A})$ respectively.

Suppose there is a $*$-isomorphism $\varphi: B \to \widetilde{B}$ taking $A$ onto $\widetilde{A}$.
Then $\varphi$ induces an isomorphism $\varphi_\bullet : \cP^{A\subset B}_\bullet \to \cP^{\widetilde{A}\subset \widetilde{B}}_\bullet$.
Moreover, by uniqueness of the Markov trace, we must have $\tr_{\widetilde{B}} \circ \varphi = \tr_B$, so $\tau_{\widetilde{A}} \circ \varphi_{0,+} = \tau_A$ on $\cP^{A\subset B}_{0,+}=Z(A)$.

\item[\underline{Injective:}]
Now suppose $A\subset B$ and $\widetilde{A}\subset \widetilde{B}$ have isomorphic pairs, i.e., there is a planar $*$-algebra isomorphism $\varphi_\bullet : \cP^{A\subset B}_\bullet \to \cP^{\widetilde{A}\subset \widetilde{B}}_\bullet$ such that $\tau_A \circ \varphi_{0,+} = \tau_{\widetilde{A}}$.
Since the distortions $\delta^{A\subset B}$ and $\delta^{\widetilde{A}\subset \widetilde{B}}$
are determined by $\tau_A$ and $\tau_{\widetilde{A}}$ respectively by \eqref{eq:delta in terms of trA and D}, we see the isomorphism $\varphi_\bullet$ preserves distortion.
By Corollary \ref{cor:FiniteDepthHyperfiniteInclusionsIsoIffSameDistortion}, there is a $*$-isomorphism $\varphi: B \to \widetilde{B}$ which maps $A$ onto $\widetilde{A}$.

\item[\underline{Surjective:}]
Suppose $\cP_\bullet$ is a finite depth unitary 2-shaded planar algebra and $\tau$ is some faithful state on $\cP_{0,+}$.
By Theorem \ref{thm:ExistsStandardInclusion}, there is a finite index homogeneous connected hyperfinite $\rm II_1$ multifactor inclusion $A_0\subset B_0$ whose standard invariant is $*$-isomorphic to $\cP_\bullet$.
Taking $\eta_i := \tau(p_i)/\alpha_i$, by Proposition \ref{prop:MoritaEquivalentDistortions}, there is a finite index connected hyperfinite $\rm II_1$ inclusion $A\subset B$ 
with standard invariant $*$-isomorphic to $\cP_\bullet$ 
with distortion
$\delta_{ij} = \eta_i^{-1}\sum_{h=1}^a \eta_h D_{hj}$.
By Proposition \ref{prop:MoritaEquivalentDistortions}(3) and \eqref{eq:delta in terms of trA and D}, the Markov trace $\tr_B$ on $B$ restricts to $\tau$ on $Z(A)\cong \cP_{0,+}$.
\end{proof}

\begin{thm}[{Theorem \ref{thm:ClassificationOfFiniteDepthHyperfinite}, Part $\rm II$}]
\label{thm:A-PartII}
The map $A\subset B \mapsto \cP_\bullet^{A\subset B}$ descends to a well-defined bijection
$$
\frac{
\left\{
\parbox{7.2cm}{\rm
Finite depth finite index connected 
hyperfinite $\rm II_1$ multifactor inclusions
$A\subset B$ \phantom{$\cP_{0,+}$}
}\right\}
}
{
\parbox{4cm}{\rm
Morita equivalence
}}
\cong
\frac{
\left\{
\parbox{5.7cm}{\rm
Finite depth indecomposable unitary 2-shaded planar algebras $\cP_\bullet$
}\right\}
}{
\parbox{5cm}{\rm
Planar $*$-algebra isomorphism
}}
\,.
$$\end{thm}
\begin{proof}
\mbox{}
\item[\underline{Well-defined:}]
Morita equivalent multifactor inclusions have isomorphic standard invariants by Lemma \ref{lem:ME InclusionStdInv}.

\item[\underline{Injective:}]
Suppose $A\subset B$ and $\widetilde{A}\subset \widetilde{B}$ are two finite depth finite index connected hyperfinite $\rm II_1$ multifactor inclusions with isomorphic standard invariants.
By Theorem \ref{thm:MultifactorInclusionMEtoHomogeneous} above, both $A\subset B$ and $\widetilde{A}\subset \widetilde{B}$ are Morita equivalent to homogeneous hyperfinite inclusions $A' \subset B'$ and $\widetilde{A}' \subset \widetilde{B}'$ respectively.
By Popa's Uniqueness Theorem \ref{thm:PopaTheoremHomogeneousMultifactor},
there is a $*$-algebra isomorphism $\varphi': B' \to \widetilde{B}'$ taking $A'$ onto $\widetilde{A}'$ and inducing an isomorphism $\varphi'_\bullet$ of standard invariants.
Thus $L^2\widetilde{A}'{}_{\varphi'}$ is a Morita equivalence $\widetilde{A}'-A'$ bimodule witnessing the Morita equivalence of the inclusions $A'\subset B'$ and $\widetilde{A}'\subset \widetilde{B}'$.
We get our desired Morita equivalence by composing Morita equivalences:
$$
(A\subset B)
\sim
(A' \subset B')
\sim
(\widetilde{A}' \subset \widetilde{B}')
\sim
(\widetilde{A}\subset \widetilde{B}).
$$

\item[\underline{Surjective:}]
Surjectivity is immediate from the Existence Theorem \ref{thm:ExistsStandardInclusion}.
\end{proof}

\section{Representations of unitary multifusion categories}
\label{sec:Representations}

For this section, $\cC$ denotes a unitary multitensor category, and $R$ denotes the hyperfinite $\rm II_1$ factor.

\begin{defn}[{\cite[Def.~3.1]{2004.08271}}]
A \emph{representation} of $\cC$ is a tensor dagger functor
$\alpha : \cC \to \Bim(M)$
for some von Neumann algebra $M$.
Given two representations $\alpha: \cC \to \Bim(M)$ and $\beta: \cC \to \Bim(N)$, a \emph{morphism} from $\alpha$ to $\beta$ consists of an $N-M$ bimodule ${}_N\Phi{}_M$ together with a family of unitary natural isomorphisms 
$\{
\phi_c : \Phi \boxtimes_M \alpha(c) \to \beta(c) \boxtimes_N \Phi
\}_{c\in \cC}$
satisfying the following coherence axiom:
\begin{equation}
\label{eq:HalfBraiding}
\begin{tikzcd}
\Phi \boxtimes_M \alpha(c) \boxtimes_M \alpha(d)
\arrow[r,"\phi_c \boxtimes \id"]
\arrow[d, "\id\boxtimes \mu^\alpha_{c,d}"]
&
\beta(c)\boxtimes_N \Phi \boxtimes_M \alpha(d)
\arrow[r,"\id\boxtimes \phi_d"]
&
\beta(c) \boxtimes_N \beta(d) \boxtimes_N \Phi
\arrow[d, "\mu^\beta_{c,d}\boxtimes \id"]
\\
\Phi \boxtimes_M \alpha(c \otimes d)
\arrow[rr,"\phi_{c\otimes d}"]
&&
\beta(c\otimes d) \boxtimes_N \Phi
\end{tikzcd}
\end{equation}
We call such a morphism an \emph{isomorphism} if ${}_N\Phi{}_M$ is an invertible $N-M$ bimodule.
\end{defn}

\begin{rem}
As discussed further in \cite[\S5.4]{2004.08271}, 
there are additional structures for representations, like bi-involutivity and positivity, when $\cC$ has the corresponding structure.
In the presence of such structures, one can ask for an additional coherence for representations.
Since we are only interested in the case $\cC$ multifusion in this article, by \cite[Thm.~A]{2004.08271}, it suffices to restrict our attention to representations of unitary multitensor categories in the absence of any additional structures and coherences.
\end{rem}

\begin{defn}
\label{defn:DistortionOfRepresentation}
Suppose $\cC$ is an $n\times n$ unitary multitensor category, $\Irr(\cC)$ a set of representatives of simple objects of $\cC$, and $\alpha : \cC \to \Bim(R^{\oplus n})$ is a representation where $R$ is a finite factor.
The \emph{modular distortion} of $\alpha$ is
the function $\delta^\alpha: \Irr(\cC)\to \bbR_{>0}$ given by $\delta^\alpha(c):=\delta({}_R\alpha(c){}_R)$.

Since the distortion is multiplicative for finite factor bimodules by \eqref{eq:DistortionMultiplicative}, 
$\delta^\alpha$ induces a grading on $\cC$, 
and thus
$\delta^\alpha$ descends to a groupoid homomorphism $\delta^\alpha: \cU \to \bbR_{>0}$ by universality.

When $\cC$ is unitary multifusion, $\cU$ is finite.
Thus all groupoid homomorphisms $\cU \to \bbR_{>0}$ factor uniquely through $\cG_n$ similar to Theorem \ref{thm:ClassificationOfUnitaryDualFunctors}.
Hence we may identify $\delta^\alpha$ with a groupoid homomorphism $\cG_n \to \bbR_{>0}$. 
\end{defn}

\begin{example}
\label{ex:AnyDistortionForRepsIntoBimRn}
Suppose $\alpha : \cC \to \Bim(R^{\oplus n})$ is a representation with distortion $\delta^\alpha$.
Since the fundamental group $\cF(R)$ of the hyperfinite $\rm II_1$ factor is $\bbR_{>0}$ \cite{MR0009096},
for any vector $(\lambda_i)\in \bbR^n_{>0}$,
there is an invertible Morita equivalence $R^{\oplus n} - R^{\oplus n}$ bimodule $\Phi$ with $\lambda_i := \vNdim_L(p_i\Phi ) =\delta(p_i \Phi )$.
Since distortion is multiplicative for $\rm II_1$ factor bimodules by \eqref{eq:DistortionMultiplicative}, for all $c\in \cC_{ij}$,
$$
\delta(p_i \Phi \boxtimes \alpha(c) \boxtimes \overline{\Phi} p_j)
=
\delta(p_i \Phi ) \delta(\alpha(c)) \delta(p_j \overline{\Phi} )
=
\lambda_i \delta^\alpha_{ij} \lambda_j^{-1}.
$$
Then conjugating $\alpha$ by $\Phi$ yields a representation $\Ad(\Phi)\circ \alpha: c\mapsto \Phi\boxtimes \alpha(c)\boxtimes \overline{\Phi}$ with distortion given by
\begin{equation}
\label{eq:DistortionOfConjugatedRepresentation}
\delta^{\Ad(\Phi)\circ \alpha}_{ij}
=
\lambda_i \delta^\alpha_{ij} \lambda_j^{-1}.
\end{equation}
Indeed, since $\Phi$ is invertible, by 
\cite[Prop.\,3.1]{MR799587}
and
\cite[Thm.~4.7]{MR2091457}
(see also \cite[\S4]{MR3342166} and \cite[Cor.~3.34]{MR4133163}), there is an $R^{\oplus n}-R^{\oplus n}$ bilinear unitary $u_\Phi : \overline{\Phi}\boxtimes_{R^{\oplus n}} \Phi \to L^2R^{\oplus n}$, unique up to unique unitary automorphism of $\Phi$,  
such that
$(\overline{\Phi}, u_\Phi, \overline{u_\Phi}^*)$ exhibits $\overline{\Phi}$ as the unitary spherical dual of $\Phi$.
Denoting $u_\Phi$ by a cap, the tensorator 
$$
\mu^{\Ad(\Phi)\circ\alpha}_{a,b}
:
\Phi\boxtimes \alpha(a)\boxtimes \overline{\Phi}
\boxtimes
\Phi\boxtimes \alpha(b)\boxtimes \overline{\Phi}
\longrightarrow
\Phi\boxtimes \alpha(a\otimes b)\boxtimes \overline{\Phi}
$$
is given by
$$
\mu^{\Ad(\Phi)\circ \alpha}_{a,b}
:=
\tikzmath{
\begin{scope}
\clip[rounded corners=5pt] (-1.6,-1) rectangle (1.6,.6);
\filldraw[\BColor] (-1.6,-1) rectangle (1.6,.6);
\filldraw[\AColor] (-1.3,-1) .. controls ++(90:.5cm) and ++(270:.5cm) .. (-.8,.6) -- (.8,.6) .. controls ++(270:.5cm) and ++(90:.5cm) .. (1.3,-1) -- (.3,-1) arc (0:180:.3cm);
\end{scope}
\draw (-1.3,-1) node[below] {$\scriptstyle \Phi$} .. controls ++(90:.5cm) and ++(270:.5cm) .. (-.8,.6) node[above] {$\scriptstyle \Phi$};
\draw (1.3,-1) node[below] {$\scriptstyle \overline{\Phi}$} .. controls ++(90:.5cm) and ++(270:.5cm) .. (.8,.6) node[above] {$\scriptstyle \overline{\Phi}$};
\draw (-.3,-1) node[below] {$\scriptstyle \overline{\Phi}$} arc (180:0:.3cm) node[below] {$\scriptstyle \Phi$};
\draw (-.8,-1) node[below] {$\scriptstyle \alpha(a)$} .. controls ++(90:.3cm) and ++(270:.3cm) .. (-.2,-.3); 
\draw (.8,-1) node[below] {$\scriptstyle \alpha(b)$} .. controls ++(90:.3cm) and ++(270:.3cm) .. (.2,-.3); 
\draw[double] (0,0) -- (0,.6) node[above] {$\scriptstyle \alpha(a\otimes b)$};
\roundNbox{unshaded}{(0,0)}{.3}{.1}{.1}{$\mu^\alpha_{a,b}$}
}
=
\id_\Phi 
\boxtimes
(\mu^\alpha_{a,b} \circ (\id_{\alpha(a)} \boxtimes u_\Phi \boxtimes \id_{\alpha(b)}))
\boxtimes 
\id_{\overline{\Phi}}.
$$
Moreover, $\Phi$ induces a canonical isomorphism $(\Phi,\phi):\alpha \to \Ad(\Phi)\circ \alpha$ by defining
$$
\phi_c := 
\tikzmath{
\begin{scope}
\clip[rounded corners=5pt] (-1.3,-.6) rectangle (.9,.6);
\filldraw[\BColor] (-1.3,-.6) rectangle (-1,.6);
\filldraw[\AColor] (-1,-.6) rectangle (.9,.6);
\filldraw[\BColor] (0,.6) arc (-180:0:.3cm);
\end{scope}
\draw (-1,-.6) node[below] {$\scriptstyle \Phi$} -- (-1,.6) node[above] {$\scriptstyle \Phi$};
\draw (-.5,-.6) node[below] {$\scriptstyle \alpha(c)$} -- (-.5,.6) node[above] {$\scriptstyle \alpha(c)$};
\draw (0,.6) node[above] {$\scriptstyle \overline{\Phi}$} arc (-180:0:.3cm) node[above] {$\scriptstyle \Phi$};
}
=
\id_{\Phi}\boxtimes \id_{\alpha(c)}\boxtimes u_\Phi^*
: \Phi\boxtimes \alpha(c) 
\longrightarrow 
\Phi\boxtimes \alpha(c)\boxtimes\overline{\Phi}\boxtimes\Phi.
$$
\end{example}

\subsection{Existence of representations of unitary multifusion categories}

Let $\cC$ be an $n\times n$ unitary multifusion category, let $R$ be a hyperfinite $\rm II_1$ factor, and denote the minimal central projections of $R^{\oplus n}$ by $\{p_1,\dots, p_n\}$.
It follows directly from Theorem~\ref{thm:ExistsStandardInclusion} that representations of unitary multifusion categories exist.

\begin{prop}
\label{prop:ExistsRepresentation}
Let $\cC$ be an $n\times n$ unitary multifusion category.
There exists a representation $\alpha : \cC \to \Bim(R^{\oplus n})$.
\end{prop}
\begin{proof}
We may reduce to the case that $n\geq 2$ by replacing $\cC$ with $\Mat_2(\cC)$ if $n=1$.
Let $1_\cC = 1^+\oplus 1^-$ be a non-trivial 2-shading of $\cC$, where 
$a=\dim(\End_\cC(1^+))\geq 1$,
$b=\dim(\End_\cC(1^-))\geq 1$,
and $n=a+b$.
Pick an arbitrary generator $X\in \cC^{+-}$ (e.g., we may take $X = \bigoplus_{c\in \Irr(\cC^{+-})} c$).
By Theorem~\ref{thm:ExistsStandardInclusion}, there is a 
finite index homogeneous connected hyperfinite $\rm II_1$ multifactor inclusion $A=R^{\oplus a}\subset 
R^{\oplus b}=B$ such that $\cC({}_AL^2B{}_B)$ is equivalent to $\cC(X)$.
The equivalence $\alpha:\cC(X) \hookrightarrow \cC({}_AL^2B{}_B)$ is such a representation.
\end{proof}

\begin{prop}
\label{prop:DistortionsAreRealizable}
Suppose $\cC$ is an $n\times n$ unitary multifusion category
and $\alpha: \cC \to \Bim(R^{\oplus n})$ is a representation.
For any groupoid homomorphism $\delta: \cG_n \to \bbR_{>0}$, 
there is an invertible bimodule $\Phi$ inducing an isomorphism $(\Phi, \phi) : \alpha \to \Ad(\Phi)\circ \alpha$
such that
$\delta^{\Ad(\Phi)\circ \alpha}=\delta$.
\end{prop}
\begin{proof}
Since $\delta^\alpha$ and $\delta$ are groupoid homomorphisms $\cG_n \to \bbR_{>0}$, so is the ratio $\delta^\alpha/\delta$.
By Corollary \ref{cor:EquivalentGroupoidHomCharacterizations}, there exists $(\lambda_i)\in \bbR^n_{>0}$, unique up to uniform scaling, such that $\delta^\alpha_{ij}/\delta_{ij} = \lambda_j/\lambda_i$ for all $1\leq i,j\leq n$.
As discussed in Example \ref{ex:AnyDistortionForRepsIntoBimRn}, there is an invertible $R^{\oplus n}-R^{\oplus n}$ bimodule $\Phi$ with $\lambda_i = \delta(p_i\Phi )$ for all $1\leq i\leq n$.
By \eqref{eq:DistortionOfConjugatedRepresentation}, 
the representation $\Ad(\Phi)\circ \alpha : \cC \to \Bim(R^{\oplus n})$ has distortion $\delta$ as desired.
\end{proof}

The next corollary now follows immediately by combining Propositions \ref{prop:ExistsRepresentation} and \ref{prop:DistortionsAreRealizable}.

\begin{cor}
\label{cor:DistortionsAreRealizable}
Suppose $\cC$ is an $n\times n$ unitary multifusion category.
For any groupoid homomorphism $\delta: \cG_n \to \bbR_{>0}^n$, 
there is a unitary tensor functor
$\alpha:\cC \to \Bim(R^{\oplus n})$ with distortion $\delta^\alpha=\delta$.
\end{cor}

\begin{rem}
The realization result for unitary multifusion categories obtained in Proposition \ref{prop:ExistsRepresentation} appears also in \cite[Cor.~3.10]{RealTwoCat}. The proof in \cite{RealTwoCat} follows different ideas and techniques, namely it is shown that every unitary multifusion category $\cC$ is equivalent to the category of special bimodules over a fixed standard C$^*$-Frobenius algebra $A$ in a unitary fusion category, the latter can be realized in $\Bim(R)$, and the realization of $\cC$ into $\Bim(R^{\oplus n})$ is obtained by considering the non-factorial extension of $R$ given by $A$, cf.~Theorem \ref{thm:ExistsStandardInclusion}. In \cite{RealTwoCat} it is also shown that every unitary multitensor category and every rigid $\rm C^*$ 2-category with finitely decomposable horizontal units can be realized as well into $\Bim(N^{\oplus n})$ where $N$ is a $\rm II_1$ factor, not necessarily hyperfinite.
\end{rem}

\subsection{Uniqueness of representations of unitary multifusion categories}
\label{sec:UniquenessForRepresentations}

In \cite[Thm.~2.2]{MR3635673}, Izumi adapts Popa's uniqueness theorem for finite depth hyperfinite $\rm III_1$ subfactors \cite[Cor.~6.11]{MR1339767}
to prove uniqueness of representations of unitary fusion categories as endomorphisms of the hyperfinite $\rm III_1$ factor.
We now adapt Izumi's proof
using Corollary \ref{cor:FiniteDepthHyperfiniteInclusionsIsoIffSameDistortion} 
to the setting of representations of a $n\times n$ unitary multifusion category $\cC$ into $\Bim(R^{\oplus n})$ where $R$ is the hyperfinite $\rm II_1$ factor.

\begin{defn}
\label{defn:IsomorphismInducedByAlgebraIso}
Suppose $\cC$ be an $n\times n$ unitary multifusion category, 
$A$ is a $\rm II_1$ multifactor with $n$-dimensional center,
$\alpha : \cC \to \Bim(A)$ is a representation, and
$\varphi: A \to B$ is a $*$-isomorphism.
Consider the $B-A$ bimodule $L^2B{}_\varphi$ with right $A$-action transported via the isomorphism $\varphi$.
The \emph{representation induced by} $\varphi$ is $\Ad(L^2B{}_\varphi)\circ \alpha: \cC \to \Bim(B)$, and we call $L^2B{}_\varphi$ the \emph{isomorphism induced by} $\varphi$ from $\alpha$ to $\Ad(L^2B{}_\varphi)\circ \alpha$.
Since $\delta(L^2B{}_\varphi)$ is the $n\times n$ identity matrix, we see that
$\Ad(L^2B{}_\varphi) \circ \alpha$ and $\alpha$ have the same distortion by \eqref{eq:DistortionOfConjugatedRepresentation}.
Hence an isomorphism induced by a $*$-algebra isomorphism can never change the distortion.
\end{defn}

\begin{thm}
\label{thm:RepresentationIsoInducedByAlgebraIso}
Suppose $\alpha: \cC \to \Bim(A)$ and $\beta: \cC\to \Bim(B)$ are two representations with $\delta^\alpha = \delta^\beta$, where $A$ and $B$ are both hyperfinite type ${\rm II}_1$ multifactors with $n$-dimensional centers.
Then there is an isomorphism $(\Phi, \phi)$ from $\alpha$ to $\beta$
which is induced by a $*$-algebra isomorphism $\varphi: B \to A$.
\end{thm}

\begin{proof}
Fix an arbitrary generator $X\in \cC$ such that every object of $\cC$ is isomorphic to a direct sum of summands of tensor powers of $X\otimes \overline{X}$  (e.g., we may take $X = \bigoplus_{c\in \Irr(\cC)} c$).
Consider the full subcategory $\widetilde{\cC} \subset \cC$ whose objects are the tensor powers of $X\otimes \overline{X}$.
To construct an equivalence from $\alpha$ to $\beta$, it suffices to construct an equivalence $(\Phi, \phi)$ from $\alpha|_{\widetilde{\cC}}$ to $\beta|_{\widetilde{\cC}}$. 
The result will then follow by idempotent completion.

Define $M$ to be the commutant of the right $A$-action on $\alpha(X)$, and define $N$ to be the commutant of the right $B$-action on $\beta(X)$.
By \cite[Prop.\,3.1]{MR799587}, $\alpha(X)$ is an invertible $M-A$ bimodule, and $\beta(X)$ is an invertible $N-B$ bimodule.
Since $\alpha(X)$ and $\beta(X)$ are dualizable, we immediately have that $A\subset M$ and $B\subset N$ are finite index multifactor inclusions.
Since $X$ generates $\cC$, both these inclusions are connected.

Consider the amplification $\Mat_2(\cC)$ with generator $X_{12}$ corresponding to $X$ in the component $\cC_{12}$, i.e.,
$$
\Mat_2(\cC)=
\begin{pmatrix}
\cC & \cC
\\
\cC & \cC
\end{pmatrix}
\qquad\qquad
X_{12}
=
\begin{pmatrix}
0 & X
\\
0 & 0
\end{pmatrix}.
$$
Denote by
$\cC(X_{12})$ the abstract standard invariant and $\cP(X_{12})_\bullet$ the associated unitary 2-shaded planar algebra.

We now construct a distinguished isomorphism between the standard invariants $\cC({}_AL^2M{}_M)$ and $\cC({}_BL^2N{}_N)$ which passes through $\cC(X_{12})$.
First, observe that amplifying $\alpha, \beta$ gives representations
$\Mat_2(\alpha) : \cC(X_{12}) \to \Bim(A\oplus A)$ and $\Mat_2(\beta): \cC(X_{12}) \to \Bim(B\oplus B)$
whose $++$ corners may be identifed with $\alpha$ and $\beta$ respectively.
Moreover, $\delta^{\Mat_2(\alpha)} = \delta^{\Mat_2(\beta)}$ by construction.
We now compose these representations with $\Ad(L^2A \oplus \alpha(X))$ and $\Ad(L^2B \oplus \beta(X))$ respectively to get representations $\widetilde{\alpha}:\cC(X_{12}) \to \Bim(A\oplus M)$ and $\widetilde{\beta}: \cC(X_{12}) \to \Bim(B\oplus N)$.
That is,
$$
\widetilde{\alpha}
\begin{pmatrix}
a & b
\\
c & d
\end{pmatrix}
:=
\begin{pmatrix}
\alpha(a)
&
\alpha(b) \boxtimes_A \overline{\alpha(X)}_M
\\
{}_M \alpha(X) \boxtimes_A \alpha(c)
&
{}_M \alpha(X) \boxtimes_A \alpha(d) \boxtimes_A \overline{\alpha(X)}_M
\end{pmatrix}
\in
\Bim(A\oplus M)
$$
and similarly for $\widetilde{B}$.
Since $\delta^\alpha = \delta^\beta$ and $\delta^{\Mat_2(\alpha)} = \delta^{\Mat_2(\beta)}$, we also have $\delta^{\widetilde{\alpha}} =\delta^{\widetilde{\beta}}$.
Observe now that $\widetilde{\alpha}(X) = {}_A\alpha(X)\boxtimes_A \overline{\alpha(X)}{}_M$ which is canonically isomorphic to ${}_AL^2M{}_M$ by  \cite[Prop.\,3.1]{MR799587};
similarly, $\widetilde{\beta}(X)\cong {}_BL^2N{}_N$.
We thus have the following zig-zag of isomorphisms of standard invariants:
$$
\cC({}_AL^2M{}_M)
\xleftarrow{\widetilde{\alpha}}
\cC(X_{12})
\xrightarrow{\widetilde{\beta}}
\cC({}_BL^2N{}_N).
$$
Inverting $\widetilde{\alpha}$ on its essential image $\cC({}_AL^2M{}_M)$
and passing to planar algebras
gives us a distinguished isomorphism $\varphi_\bullet: \cP^{A\subset M}_\bullet \to \cP^{B\subset N}_\bullet$ which preserves distortion.
Since $\varphi_\bullet$ came from a zig-zag of isomorphisms, we have $\varphi_{n,\pm}(\widetilde{\alpha}(f)) = \widetilde{\beta}(f)$ for every $f\in \cP(X)_{n,\pm}$.
Looking at the principal even part, we have $\varphi_{2n,+}(\alpha(f)) = \beta(f)$ for all $f\in \End_{\widetilde{\cC}}((X\otimes \overline{X})^{\otimes n})$.

By Corollary \ref{cor:FiniteDepthHyperfiniteInclusionsIsoIffSameDistortion}, there is a $*$-isomorphism $\varphi: M\to N$ taking $A$ onto $B$ 
which induces the above isomorphism of standard invariants $\cP_\bullet^{A\subset M} \cong \cP_{\bullet}^{B\subset N}$.
Since the functor from the 1-groupoid of inclusions to the 1-groupoid of standard invariants factors through the 2-groupoid of standard bimodules as in \eqref{eq: Subfactor--planarAlg--TensorCat}, the invertible 1-morphism $(L^2B{}_{\varphi}, L^2N{}_{\varphi}, \psi_\varphi) : {}_AL^2N{}_N \to {}_B L^2M{}_M$ produces the isomorphism of standard invariants $\varphi_\bullet$ by the `encircling' action.
Here, the isomorphism $\psi_\varphi : {}_B L^2B{}_\varphi \boxtimes_A L^2M{}_M \to {}_B L^2N \boxtimes_N L^2N{}_\varphi$ is as in \eqref{eq:PsiFromIsomorphism}.

In more detail, denote the four von Neumann algebras $A,B,M,N$ by the shaded regions
$$
\tikzmath{\draw[fill=white, rounded corners=5, thin, dotted, baseline=1cm] (0,0) rectangle (.5,.5);}=A
\qquad
\tikzmath{\filldraw[\BColor, rounded corners=5, very thin, baseline=1cm] (0,0) rectangle (.5,.5);}=M
\qquad
\tikzmath{\filldraw[\APrimeColor, rounded corners=5, very thin, baseline=1cm] (0,0) rectangle (.5,.5);}=B
\qquad
\tikzmath{\filldraw[\BPrimeColor, rounded corners=5, very thin, baseline=1cm] (0,0) rectangle (.5,.5);}=N,
$$
and the standard and Morita equivalence bimodules by
$$
\tikzmath{
\begin{scope}
\clip[rounded corners=5pt] (-.4,-.4) rectangle (.4,.4);
\fill[\BColor] (0,-.4) rectangle (.4,.4);
\fill[white] (-.4,-.4) rectangle (0,.4);
\end{scope}
\draw[dotted,rounded corners=5pt] (0,-.4) -- (-.4,-.4) -- (-.4,.4) -- (0,.4);
\draw (0,-.4) -- (0,.4);
}
=
{}_AL^2M{}_M
\qquad
\tikzmath{
\begin{scope}
\clip[rounded corners=5pt] (-.4,-.4) rectangle (.4,.4);
\fill[\ATildeColor] (-.4,-.4) rectangle (0,.4);
\fill[\BTildeColor] (.4,-.4) rectangle (0,.4);
\end{scope}
\draw (0,-.4) -- (0,.4);
}
=
{}_{B}L^2N{}_{N}
\qquad
\tikzmath{
\begin{scope}
\clip[rounded corners=5pt] (-.4,-.4) rectangle (.4,.4);
\fill[\ATildeColor] (0,-.4) rectangle (-.4,.4);
\fill[white] (0,-.4) rectangle (.4,.4);
\end{scope}
\draw[dotted,rounded corners=5pt] (0,-.4) -- (.4,-.4) -- (.4,.4) -- (0,.4);
\draw[thick, blue] (0,-.4) -- (0,.4);
}
={}_{B}L^2B_{\varphi(A)}
\qquad
\tikzmath{
\begin{scope}
\clip[rounded corners=5pt] (-.4,-.4) rectangle (.4,.4);
\fill[\BColor] (.4,-.4) rectangle (0,.4);
\fill[\BTildeColor] (-.4,-.4) rectangle (0,.4);
\end{scope}
\draw[thick, red] (0,-.4) -- (0,.4);
}
={}_{N}L^2N_{\varphi(M)}
$$
We denote the conjugate bimodules by the horizontal reflection, and the restriction to $A,\widetilde{A}$ respectively by changing the shading.
We abbreviate the isomorphisms $\psi_\varphi, \overline{\psi_\varphi}, \psi_\varphi^*,\overline{\psi_\varphi}^*$ by 4-valent vertices:
$$
\tikzmath{
\begin{scope}
\clip[rounded corners=5pt] (-.4,-.4) rectangle (.4,.4);
\fill[white] (-.4,-.4) -- (0,0) -- (.4,-.4);
\fill[\ATildeColor] (-.4,-.4) -- (0,0) -- (-.4,.4);
\fill[\BColor] (.4,-.4) -- (0,0) -- (.4,.4);
\fill[\BTildeColor] (-.4,.4) -- (0,0) -- (.4,.4);
\draw (.4,-.4) -- (0,0);
\draw (-.4,.4) -- (0,0);
\draw[thick, blue] (-.4,-.4) -- (0,0);
\draw[thick, red] (0,0) -- (.4,.4);
\end{scope}
}
:=
\tikzmath{
\begin{scope}
\clip[rounded corners=5pt] (-.7,-.7) rectangle (.7,.7);
\fill[white] (-.2,-.7) rectangle (.2,0);
\fill[\ATildeColor] (-.7,.7) rectangle (-.2,-.7);
\fill[\BTildeColor] (-.2,.7) rectangle (.2,0);
\fill[\BColor] (.2,-.7) rectangle (.7,.7);
\end{scope}
\draw (.2,-.7) -- (.2,0);
\draw (-.2,.7) -- (-.2,0);
\draw[thick, blue] (-.2,-.3) -- (-.2,-.7);
\draw[thick, red] (.2,.3) -- (.2,.7);
\roundNbox{unshaded}{(0,0)}{.3}{.1}{.1}{$\psi_\varphi$}
}
\qquad
\qquad
\tikzmath[xscale=-1]{
\begin{scope}
\clip[rounded corners=5pt] (-.4,-.4) rectangle (.4,.4);
\fill[white] (-.4,-.4) -- (0,0) -- (.4,-.4);
\fill[\ATildeColor] (-.4,-.4) -- (0,0) -- (-.4,.4);
\fill[\BColor] (.4,-.4) -- (0,0) -- (.4,.4);
\fill[\BTildeColor] (-.4,.4) -- (0,0) -- (.4,.4);
\draw (.4,-.4) -- (0,0);
\draw (-.4,.4) -- (0,0);
\draw[thick, blue] (-.4,-.4) -- (0,0);
\draw[thick, red] (0,0) -- (.4,.4);
\end{scope}
}
:=
\tikzmath[xscale=-1]{
\begin{scope}
\clip[rounded corners=5pt] (-.7,-.7) rectangle (.7,.7);
\fill[white] (-.2,-.7) rectangle (.2,0);
\fill[\ATildeColor] (-.7,.7) rectangle (-.2,-.7);
\fill[\BTildeColor] (-.2,.7) rectangle (.2,0);
\fill[\BColor] (.2,-.7) rectangle (.7,.7);
\end{scope}
\draw (.2,-.7) -- (.2,0);
\draw (-.2,.7) -- (-.2,0);
\draw[thick, blue] (-.2,-.3) -- (-.2,-.7);
\draw[thick, red] (.2,.3) -- (.2,.7);
\roundNbox{unshaded}{(0,0)}{.3}{.1}{.1}{$\overline{\psi_\varphi}$}
}
\qquad
\qquad
\tikzmath[yscale=-1]{
\begin{scope}
\clip[rounded corners=5pt] (-.4,-.4) rectangle (.4,.4);
\fill[white] (-.4,-.4) -- (0,0) -- (.4,-.4);
\fill[\ATildeColor] (-.4,-.4) -- (0,0) -- (-.4,.4);
\fill[\BColor] (.4,-.4) -- (0,0) -- (.4,.4);
\fill[\BTildeColor] (-.4,.4) -- (0,0) -- (.4,.4);
\draw (.4,-.4) -- (0,0);
\draw (-.4,.4) -- (0,0);
\draw[thick, blue] (-.4,-.4) -- (0,0);
\draw[thick, red] (0,0) -- (.4,.4);
\end{scope}
}
:=
\tikzmath[yscale=-1]{
\begin{scope}
\clip[rounded corners=5pt] (-.7,-.7) rectangle (.7,.7);
\fill[white] (-.2,-.7) rectangle (.2,0);
\fill[\ATildeColor] (-.7,.7) rectangle (-.2,-.7);
\fill[\BTildeColor] (-.2,.7) rectangle (.2,0);
\fill[\BColor] (.2,-.7) rectangle (.7,.7);
\end{scope}
\draw (.2,-.7) -- (.2,0);
\draw (-.2,.7) -- (-.2,0);
\draw[thick, blue] (-.2,-.3) -- (-.2,-.7);
\draw[thick, red] (.2,.3) -- (.2,.7);
\roundNbox{unshaded}{(0,0)}{.3}{.1}{.1}{$\psi_\varphi^*$}
}
\qquad\qquad
\tikzmath[xscale=-1, yscale=-1]{
\begin{scope}
\clip[rounded corners=5pt] (-.4,-.4) rectangle (.4,.4);
\fill[white] (-.4,-.4) -- (0,0) -- (.4,-.4);
\fill[\ATildeColor] (-.4,-.4) -- (0,0) -- (-.4,.4);
\fill[\BColor] (.4,-.4) -- (0,0) -- (.4,.4);
\fill[\BTildeColor] (-.4,.4) -- (0,0) -- (.4,.4);
\draw (.4,-.4) -- (0,0);
\draw (-.4,.4) -- (0,0);
\draw[thick, blue] (-.4,-.4) -- (0,0);
\draw[thick, red] (0,0) -- (.4,.4);
\end{scope}
}
:=
\tikzmath[xscale=-1, yscale=-1]{
\begin{scope}
\clip[rounded corners=5pt] (-.7,-.7) rectangle (.7,.7);
\fill[white] (-.2,-.7) rectangle (.2,0);
\fill[\ATildeColor] (-.7,.7) rectangle (-.2,-.7);
\fill[\BTildeColor] (-.2,.7) rectangle (.2,0);
\fill[\BColor] (.2,-.7) rectangle (.7,.7);
\end{scope}
\draw (.2,-.7) -- (.2,0);
\draw (-.2,.7) -- (-.2,0);
\draw[thick, blue] (-.2,-.3) -- (-.2,-.7);
\draw[thick, red] (.2,.3) -- (.2,.7);
\roundNbox{unshaded}{(0,0)}{.3}{.1}{.1}{$\overline{\psi_\varphi}^*$}
}\,.
$$
We illustrate the `encircling' action for $f \in \widetilde{\cC}((X\otimes \overline{X})^{\otimes 1} \to (X\otimes \overline{X})^{\otimes 2})$:
$$
\beta(f)
=
\tikzmath{
\begin{scope}
\clip[rounded corners = 5] (-1.1,-1.1) rectangle (1.1,1.1) ;
\filldraw[\ATildeColor] (-1.1,-1.1) rectangle (1.1,1.1);
\filldraw[\BTildeColor] (-.2,-1.1) rectangle (.2,0);
\filldraw[\BTildeColor] (-.3,0) rectangle (-.1,1.1);
\filldraw[\BTildeColor] (.3,0) rectangle (.1,1.1);
\filldraw[white] (0,0) circle (.8cm);
\filldraw[\BColor] (-.2,0) -- (-104.5:.8cm) arc (-104.5:-75.5:.8cm) -- (.2,0);
\filldraw[\BColor] (-.3,0) -- (112:.8cm) arc (112:98:.8cm) -- (-.1,0);
\filldraw[\BColor] (.3,0) -- (68:.8cm) arc (68:82:.8cm) -- (.1,0);
\end{scope}
\draw[thick, blue] (97:.8cm) arc (97:83:.8cm);
\draw[thick, blue] (-75.5:.8cm) arc (-75.5:68:.8cm);
\draw[thick, blue] (112:.8cm) arc (112:255.5:.8cm);
\draw[thick, red] (83:.8cm) arc (83:68:.8cm);
\draw[thick, red] (97:.8cm) arc (97:112:.8cm);
\draw[thick, red] (-104.5:.8cm) arc (-104.5:-75.5:.8cm);
\draw (-.3,0) -- (-.3,1.1);
\draw (-.1,0) -- (-.1,1.1);
\draw (.1,0) -- (.1,1.1);
\draw (.3,0) -- (.3,1.1);
\draw (-.2,0) -- (-.2,-1.1);
\draw (.2,0) -- (.2,-1.1);
\roundNbox{unshaded}{(0,0)}{.3}{.2}{.2}{$\alpha(f)$}
}
=
\widetilde{\cC}(\varphi)(\alpha(f))\,.
$$

We now define ${}_B\Phi{}_A := {}_BL^2B_{\varphi(A)}$ 
and
$$
\phi:=
\phi_{X\otimes \overline{X}}
:=
\tikzmath{
\begin{scope}
\clip[rounded corners=5pt] (-.8,-.8) rectangle (.8,.8);
\fill[white] (-.6,-.8) .. controls ++(90:.2cm) and ++(-135:.3cm) .. (-.2,-.2) -- (.2,.2) .. controls ++(45:.3cm) and ++(270:.2cm) .. (.6,.8) -- (.8,.8) -- (.8,-.8);
\fill[\ATildeColor] (-.6,-.8) .. controls ++(90:.2cm) and ++(-135:.3cm) .. (-.2,-.2) -- (.2,.2) .. controls ++(45:.3cm) and ++(270:.2cm) .. (.6,.8) -- (-.8,.8) -- (-.8,-.8);
\fill[\BColor] (-.2,-.8) -- (-.2,-.2) -- (.2,.2) -- (.2,-.8);
\fill[\BTildeColor] (-.2,.8) -- (-.2,-.2) -- (.2,.2) -- (.2,.8);
\end{scope}
\draw (-.2,-.8) -- (-.2,.8);
\draw (.2,-.8) -- (.2,.8);
\draw[thick, blue] (-.6,-.8) .. controls ++(90:.2cm) and ++(-135:.3cm) .. (-.2,-.2);
\draw[thick, red] (-.2,-.2) -- (.2,.2);
\draw[thick, blue] (.2,.2) .. controls ++(45:.3cm) and ++(270:.2cm) .. (.6,.8);
}
:=
\tikzmath{
\begin{scope}
\clip[rounded corners=5pt] (-.7,-.8) rectangle (2.1,.8);
\fill[white] (-.2,-.8) rectangle (.2,0);
\fill[white] (.8,-.8) -- (.8,-.3) -- (1.2,-.3) arc (-180:0:.3cm) -- (1.8,.8) -- (2.1,.8) -- (2.1,-.8);
\fill[\ATildeColor] (1.2,.8) -- (1.2,-.3) arc (-180:0:.3cm) -- (1.8,.8);
\fill[\ATildeColor] (-.7,.8) rectangle (-.2,-.8);
\fill[\BTildeColor] (-.2,.8) rectangle (1.2,0);
\fill[\BColor] (.2,-.8) -- (.2,.3) arc (180:0:.3cm) -- (.8,-.8);
\end{scope}
\draw (.2,-.8) -- (.2,0);
\draw (-.2,.8) -- (-.2,0);
\draw (.8,-.8) -- (.8,0);
\draw (1.2,.8) -- (1.2,0);
\draw[thick, blue] (-.2,-.3) -- (-.2,-.8);
\draw[thick, red] (.2,.3) arc (180:0:.3cm);
\draw[thick, blue] (1.2,-.3) arc (-180:0:.3cm) -- (1.8,.8);
\roundNbox{unshaded}{(0,0)}{.3}{.1}{.1}{$\psi_\varphi$}
\roundNbox{unshaded}{(1,0)}{.3}{.1}{.1}{$\overline{\psi_\varphi}$}
}
:
{}_B\Phi \boxtimes_A \underbrace{\alpha(X\otimes \overline{X})}_{{}_AL^2M{}_A} {}_A \xrightarrow{\sim} {}_B \underbrace{\beta(X\otimes \overline{X})}_{{}_BL^2N{}_B} \boxtimes_B \Phi{}_A.
$$
For $n\geq 2$,
we define $\phi_n = \phi_{(X\otimes X)^{\otimes n}}$ by concatenating $n$ copies of $\phi$ along the blue string, e.g.,
$$
\phi_3 =
\tikzmath{
\begin{scope}
\clip[rounded corners=5pt] (-.9,-.8) rectangle (.9,.8);
\fill[white] (-.7,-.8) .. controls ++(90:.1cm) and ++(-135:.2cm) .. (-.5,-.5) -- (.5,.5) .. controls ++(45:.2cm) and ++(270:.1cm) .. (.7,.8) -- (.9,.8) -- (.9,-.8);
\fill[\ATildeColor] (-.7,-.8) .. controls ++(90:.1cm) and ++(-135:.2cm) .. (-.5,-.5) -- (.5,.5) .. controls ++(45:.2cm) and ++(270:.1cm) .. (.7,.8) -- (-.9,.8) -- (-.9,-.8);
\fill[\BColor] (-.5,-.8) -- (-.5,-.5) -- (-.3,-.3) -- (-.3,-.8);
\fill[\BColor] (-.1,-.8) -- (-.1,-.1) -- (.1,.1) -- (.1,-.8);
\fill[\BColor] (.5,-.8) -- (.5,.5) -- (.3,.3) -- (.3,-.8);
\fill[\BTildeColor] (-.5,.8) -- (-.5,-.5) -- (-.3,-.3) -- (-.3,.8);
\fill[\BTildeColor] (-.1,.8) -- (-.1,-.1) -- (.1,.1) -- (.1,.8);
\fill[\BTildeColor] (.5,.8) -- (.5,.5) -- (.3,.3) -- (.3,.8);
\end{scope}
\draw (-.1,-.8) -- (-.1,.8);
\draw (.1,-.8) -- (.1,.8);
\draw (-.3,-.8) -- (-.3,.8);
\draw (.3,-.8) -- (.3,.8);
\draw (-.5,-.8) -- (-.5,.8);
\draw (.5,-.8) -- (.5,.8);
\draw[thick, blue] (-.7,-.8) .. controls ++(90:.1cm) and ++(-135:.2cm) .. (-.5,-.5);
\draw[thick, red] (-.5,-.5) -- (-.3,-.3);
\draw[thick, blue] (-.1,-.1) -- (-.3,-.3);
\draw[thick, red] (-.1,-.1) -- (.1,.1);
\draw[thick, blue] (.1,.1) -- (.3,.3);
\draw[thick, red] (.5,.5) -- (.3,.3);
\draw[thick, blue] (.5,.5) .. controls ++(45:.2cm) and ++(270:.1cm) .. (.7,.8);
}
:
{}_B\Phi \boxtimes_A \alpha((X\otimes \overline{X})^{\otimes 3}){}_A
\xrightarrow{\sim}
{}_B\beta((X\otimes \overline{X})^{\otimes 3}) \boxtimes_B \Phi{}_A.
$$
Now since $(X\otimes \overline{X})^{\otimes k} \otimes (X\otimes \overline{X})^{\otimes n} = (X\otimes \overline{X})^{\otimes (k+n)}$ suppressing associators, the coherence axiom \eqref{eq:HalfBraiding} automatically holds.
Naturality immediately follows by the recabling relations \eqref{eq:Recabling}.
\end{proof}

We now prove Theorem \ref{thm:ClassificationOfRepresentationsOfMultifusion}
in two parts.

\begin{thm}[Theorem \ref{thm:ClassificationOfRepresentationsOfMultifusion}, Part I]
\label{thm:B-partI}
Let $R$ be either the hyperfinite $\rm II_1$ or $\rm II_\infty$ factor.
The map $\alpha\mapsto\delta^\alpha$ descends to a bijection
$$
\frac{
\left\{
\parbox{5.9cm}{\rm
Representations $\alpha : \cC \to \Bim(R^{\oplus n})$
}\right\}
}
{
\parbox{6cm}{\rm
Iso $(\Phi,\phi)$ induced by
$\varphi \in \Aut(R^{\oplus n})$
}}
\cong
\left\{
\parbox{6.8cm}{\rm
Groupoid homomorphsims $\delta : \cG_{n} \to \bbR_{>0}$
}\right\}
\,.
$$
\end{thm}
\begin{proof}
\item[\underline{Well-defined:}]
As discussed in Definition \ref{defn:IsomorphismInducedByAlgebraIso} above,
isomorphisms induced by $*$-algebra isomorphisms must preserve the distortion.

\item[\underline{Injective:}]
This follows immediately from Theorem \ref{thm:RepresentationIsoInducedByAlgebraIso}.

\item[\underline{Surjective:}]
This follows immediately from Corollary \ref{cor:DistortionsAreRealizable}.
\end{proof}

\begin{thm}[Theorem \ref{thm:ClassificationOfRepresentationsOfMultifusion}, Part II]
\label{thm:B-partII}
Let $R$ be either the hyperfinite $\rm II_1$ or ${\rm II}_\infty$ factor.
Suppose $\alpha,\beta: \cC \to \Bim(R^{\oplus n})$  are two arbitrary representations, where $A$ and $B$ are either both hyperfinite type ${\rm II}_1$ or type ${\rm II}_\infty$ multifactors with $n$-dimensional centers.
Then there is an isomorphism $(\Phi, \phi)$ from $\alpha$ to $\beta$.
\end{thm}
\begin{proof}
First, since $R$ and $B(\ell^2)\otimes R$ are Morita equivalent, $\Bim(R^{\oplus n})$ and $\Bim((B(\ell^2)\otimes R)^{\oplus n})$ are equivalent.
Hence we may assume $A$ and $B$ are both $\rm II_1$ multifactors.
Second, by Proposition \ref{prop:DistortionsAreRealizable}, 
by conjugating $\alpha$ and $\beta$ by appropriate invertible $A-A$ and $B-B$ bimodules respectively, we may assume that $\delta^\alpha = \delta^\beta$.
The result now follows by Theorem \ref{thm:RepresentationIsoInducedByAlgebraIso}.
\end{proof}

\settocdepth{section}
\appendix
\section{Commuting squares of finite index finite multifactors}
\label{sec:CommutingSquares}

Consider a quadrilateral of unital finite index inclusions of finite multifactors
\begin{equation}
\label{eq:AppendixCommutingSquare}
\begin{matrix}
\xymatrix@C=5pt@R=2pt{
N  &\subset & M
\\
\cup&&\cup
\\
Q   &\subset & P
}
\end{matrix}
\end{equation}
together with 
a faithful tracial state $\tr$ on $M$.
Let $E^M_{N}: M \to N$, $E^M_{P}: M \to P$, and $E^M_{Q}: M \to Q$ be the canonical trace-preserving conditional expectations, where $N$ and $P$ are considered as subalgebras of $M$.

\begin{defn}
\label{defn:CommutingSquare}
The quadrilateral \eqref{eq:AppendixCommutingSquare} is called a \emph{commuting square} if 
$$
E^M_{N}E^M_{P} 
=
E^M_{P}E^M_{N} 
= 
E^M_{Q}.
$$
\end{defn}

\subsection{Nondegeneracy}

Recall that by \cite[3.6.4(i)]{MR999799}, given a finite index inclusion of finite multifactors $N\subset M$ together with a faithful normal trace $\tr$ on $M$ and the unique trace-preserving conditional expectation, there exists a (finite) Pimsner-Popa basis for $M$ over $N$.

\begin{lem}[{cf.~\cite[Prop.~in~1.1.5]{MR1278111}}]
\label{lem:Nondegenerate}
For a commuting square of finite index inclusions of finite multifactors \eqref{eq:AppendixCommutingSquare} the following are equivalent:
\begin{enumerate}[label=(N\arabic*)]
\item
\label{Nondegenerate:every P/Q is M/N}
Every Pimsner-Popa basis for $P$ over $Q$ 
is also a Pimsner-Popa basis for $M$ over $N$. 
\item
\label{Nondegenerate:every N/Q is M/P}
Every Pimsner-Popa basis for $N$ over $Q$ 
is also a Pimsner-Popa basis for $M$ over $P$. 
\item
\label{Nondegenerate:exists P/Q is M/N}
There is a Pimsner-Popa basis for $P$ over $Q$ 
which is a Pimsner-Popa basis for $M$ over $N$. 
\item
\label{Nondegenerate:exists N/Q is M/P}
There is a Pimsner-Popa basis for $N$ over $Q$ 
which is a Pimsner-Popa basis for $M$ over $P$. 
\item
\label{Nondegenerate:M=PN}
$M = \spann PN$. 
\item
\label{Nondegenerate:M=NP}
$M = \spann NP$. 
\end{enumerate}
\end{lem}
\begin{proof}
Clearly 
\ref{Nondegenerate:every P/Q is M/N}
$\Rightarrow$
\ref{Nondegenerate:exists P/Q is M/N}
$\Rightarrow$
\ref{Nondegenerate:M=PN}
and
\ref{Nondegenerate:every N/Q is M/P}
$\Rightarrow$
\ref{Nondegenerate:exists N/Q is M/P}
$\Rightarrow$
\ref{Nondegenerate:M=NP},
and obviously 
\ref{Nondegenerate:M=PN}$\Leftrightarrow$\ref{Nondegenerate:M=NP}.

\item[\underline{
\ref{Nondegenerate:M=PN}
$\Rightarrow$
\ref{Nondegenerate:every P/Q is M/N}:
}]
Suppose $\{b\}$ is a Pimsner-Popa basis for $P$ over $Q$, and let $x\in M$.
Since $M= \spann PN$, we can write $x=\sum_{i=1}^n p_in_i$.
We then calculate
$$
\sum_b b E_N(b^*x)
=
\sum_b\sum_{i=1}^n b E_N(b^*p_in_i)
=
\sum_b\sum_{i=1}^n b E_N(b^*p_i)n_i
=
\sum_{i=1}^n\sum_b b E_Q(b^*p_i)n_i
=
\sum_{i=1}^n p_in_i
=
x.
$$

\item[\underline{
\ref{Nondegenerate:M=NP}
$\Rightarrow$
\ref{Nondegenerate:every N/Q is M/P}
:}]
This follows by an argument similar to the above swapping the roles of $P$ and $N$.
\end{proof}

\begin{defn}[{cf.~\cite[1.1.5]{MR1278111}}]
\label{defn:Nondegenerate}
A commuting square of finite index finite multifactors \eqref{eq:AppendixCommutingSquare} is called \emph{nondegenerate} if the equivalent conditions of Lemma \ref{lem:Nondegenerate} hold.
\end{defn}

\begin{lem}[{cf.~\cite[Pf.~of~Lem.~6.1]{MR1055708}}]
\label{lem:NondegeneracyAndMarkovTrace}
Suppose we have a commuting square of finite index inclusions of finite multifactors as in \eqref{eq:AppendixCommutingSquare} above such that $N\subset (M,\tr)$ is Markov with index $d^2$.
In this setting, nondegeneracy of \eqref{eq:AppendixCommutingSquare} is equivalent to
\begin{enumerate}[label=(N7)]
\item 
\label{Nondegenerate:MarkovCondition}
The inclusion $Q\subseteq (P, \tr|_{P})$ is also Markov with index $d^2$.
\end{enumerate}
\end{lem}
\begin{proof}
The canonical map $\langle P, Q\rangle=Pe^P_QP \to \langle M,N\rangle = Me^M_N M$ by $a e^P_Q b \mapsto a e^M_N b$ is a well-defined (possibly non-unital) $*$-homomorphism  by the commuting square condition.
This homomorphism preserves the canonical commutant trace on $\langle P, Q\rangle$ given by $ae^P_Q b \mapsto \tr(ab)$.
Hence the inclusion $Q\subseteq (P, \tr|_{P})$ is always Markov.
However, it may be the case that the Markov index of $Q\subseteq (P, \tr|_{P})$ is strictly less than $d^2$.
The image of $1\in \langle P, Q\rangle$ in $\langle M, N\rangle$ under the canonical map is an orthogonal projection in $\langle M, N\rangle$.
We see that this projection is equal to $1\in \langle M, N\rangle$ if and only if \ref{Nondegenerate:exists P/Q is M/N} holds.
But we also have that this projection is equal to $1\in \langle M, N\rangle$ if and only if 
the Markov index of $Q\subseteq (P, \tr|_{P})$ is equal to $d^2$.
The result follows.
\end{proof}

\begin{facts}
\label{facts:Nondegeneracy}
\mbox{}
\begin{itemize}
\item
(Basic construction of commuting squares cf.~\cite{MR1055708} and \cite[1.1.6]{MR1278111})
Suppose we have a nondegenerate commuting square of finite index finite multifactors as in \eqref{eq:AppendixCommutingSquare}, equipped with the unique Markov trace $\tr_M$ for the inclusion $N\subset M$.
Consider the basic construction commuting square
\begin{equation}
\label{eq:BasicConstructionCommutingSquare}
\begin{matrix}
\xymatrix@C=5pt@R=2pt{
M  &\subset & \langle M,N\rangle
\\
\cup&&\cup
\\
P   &\subset & \langle P,Q\rangle
}
\end{matrix}
\end{equation}
with the canonical Markov trace $\tr_{\langle M, N\rangle}$ on $\langle M, N\rangle$ where the inclusion $\langle P, Q\rangle \hookrightarrow \langle M,N\rangle$ is given by the canonical map $ae^P_Q b \mapsto a e^M_N b$.
By nondegeneracy, there is a Pimsner-Popa basis for $P$ over $Q$ which is also a Pimsner-Popa basis for $M$ over $N$, which implies the inclusion $\langle P, Q\rangle \hookrightarrow \langle M,N\rangle$ is unital.
Viewing all algebras as subalgebras of $\langle M,N\rangle$ thus identifies $e^P_Q$ with $e^M_N$.
Thus $\langle P, Q\rangle = \spann Pe^M_N P$, so
$$
\spann M\langle P, Q\rangle
=
\spann MPe^M_NP
=
\spann Me^M_NNP
=
\spann Me^M_NM
=
\langle M, N\rangle,
$$
and the commuting square \eqref{eq:BasicConstructionCommutingSquare} is also nondegenerate by \ref{Nondegenerate:M=NP}.

If in addition all the algebras in \eqref{eq:AppendixCommutingSquare} are finite dimensional, then the Bratteli diagram for $\langle P, Q\rangle \hookrightarrow \langle M,N\rangle$ is the same as the Bratteli diagram for the inclusion $Q\subset N$ by \cite[Lem.~6.1]{MR1055708} or \cite[Lem.~5.3.3]{MR1473221}.

\item
(Composite commuting squares cf.~\cite[Cor.~in~1.1.5]{MR1278111})
A composite of two commuting squares 
of finite index finite multifactors
$$
\begin{matrix}
\xymatrix@C=5pt@R=2pt{
T  &\subset & Q &\subset & M
\\
\cup&&\cup &&\cup
\\
S   &\subset & P &\subset & N
}
\end{matrix}
$$
is again a commuting square.
The composite square is nondegenerate if and only if the two component commuting squares are nondegenerate.
\end{itemize}
\end{facts}

\subsection{Nondegenerate commuting squares from unitary multifusion categories}
\label{sec:CommutingSquaresFromMultifusion}

In this section, we study some nondegenerate commuting squares which arise from a unitary multifusion category $\cC$.
The results in this section are technical lemmas used in the proof of Theorem \ref{thm:ExistsStandardInclusion}.

\begin{defn}
For $c\in \cC$, we denote by $1_{s(c)}$ and $1_{t(c)}$ are the \emph{source} and \emph{target} summands of $1_\cC$ for $c\in \cC$, i.e., the minimal subobjects of $1_\cC$ such that $c = 1_{s(c)} \otimes c \otimes 1_{t(c)}$.
\end{defn}

\begin{assumption}
\label{ass:commutingsquare}
Fix objects $a,b,c\in \cC$.
Suppose there is a unitary dual functor $\vee$ on $\cC$ 
such that
\begin{itemize}
\item 
$
\ev_a^\ast\circ \ev_a
=
\lambda_a
\id_{1_{s(a)}}
$
and
$
\coev_c\circ \coev_c^\ast
=
\rho_c
\id_{1_{t(c)}}
$
for some positive scalars $\lambda_a, \rho_c$,
and
\item
there is a a faithful spherical state $\psi$ on $\End_\cC(1_\cC)$ 
satisfying
$\psi\circ\tr^\vee_L=\psi\circ \tr^\vee_R$
on hom categories of $\cC$.
\end{itemize}
\end{assumption}

Under the above assumptions, we have a commuting square of finite dimensional von Neumann algebras
\begin{equation}
\label{eq:CommutingSquareFromMultifusion}
\begin{matrix}
\xymatrix@C=5pt@R=2pt{
&
\tikzmath{
\draw (.075,-.4) node [right] {$\scriptstyle b$} -- (.075,.4); 
\draw[thick, blue] (-.075,-.4) node[left] {$\scriptstyle a$} -- (-.075,.4);
\ScaledRoundNbox{unshaded}{(0,0)}{.15}{0}{0}{\scriptsize{$x$}}{2}
}
&\mapsto&
\tikzmath{
\draw (.075,-.4) -- (.075,.4); 
\draw[thick, blue] (-.075,-.4) node[left] {$\scriptstyle a$} -- (-.075,.4);
\draw[thick, red] (.25,-.4) node[right] {$\scriptstyle c$} -- (.25,.4);
\ScaledRoundNbox{unshaded}{(0,0)}{.15}{0}{0}{\scriptsize{$x$}}{2}
}
\\
\tikzmath{
\draw[thick, blue] (-.25,-.4) node[left] {$\scriptstyle a$} -- (-.25,.4);
\draw (0,-.4) node[right] {$\scriptstyle b$} -- (0,.4);
\ScaledRoundNbox{unshaded}{(0,0)}{.15}{0}{0}{\scriptsize{$x$}}{2}
}
& \End_\cC(a \otimes b)  &\subset & \End_\cC(a\otimes b \otimes c) &
\!\!\!
\tikzmath{
\draw[thick, red] (.075,-.4) node[right] {$\scriptstyle c$} -- (.075,.4); 
\draw (-.075,-.4) -- (-.075,.4);
\draw[thick, blue] (-.25,-.4) node[left] {$\scriptstyle a$} -- (-.25,.4);
\ScaledRoundNbox{unshaded}{(0,0)}{.15}{0}{0}{\scriptsize{$x$}}{2}
}
\\
\rotatebox{90}{$\mapsto$}
&\cup&&\cup&
\!\!\!\rotatebox{90}{$\mapsto$}
\\
\tikzmath{
\draw (0,-.4) node[right] {$\scriptstyle b$} -- (0,.4);
\ScaledRoundNbox{unshaded}{(0,0)}{.15}{0}{0}{\scriptsize{$x$}}{2}
}
& 
\End_\cC(b)   &\subset & \End_\cC(b\otimes c)
&
\!\!\!
\tikzmath{
\draw[thick, red] (.075,-.4) node [right] {$\scriptstyle c$} -- (.075,.4); 
\draw (-.075,-.4) node[left] {$\scriptstyle b$} -- (-.075,.4);
\ScaledRoundNbox{unshaded}{(0,0)}{.15}{0}{0}{\scriptsize{$x$}}{2}
}
\\
&
\tikzmath{
\draw (0,-.4) node[right] {$\scriptstyle b$} -- (0,.4);
\ScaledRoundNbox{unshaded}{(0,0)}{.15}{0}{0}{\scriptsize{$x$}}{2}
}
&\mapsto&
\tikzmath{
\draw[thick, red] (.25,-.4) node[right] {$\scriptstyle c$} -- (.25,.4);
\draw (0,-.4) node[left] {$\scriptstyle b$} -- (0,.4);
\ScaledRoundNbox{unshaded}{(0,0)}{.15}{0}{0}{\scriptsize{$x$}}{2}
}
}
\end{matrix}
\end{equation}
where we equip $\End_\cC(a\otimes b\otimes c)$ with the faithful trace $\tr:=\psi\circ\tr^\vee_L=\psi\circ \tr^\vee_R$.
The canonical trace-preserving conditional expectations are given by the partial trace on the left/right and dividing by $\lambda_a,\rho_c$ respectively, which obviously commute.

\begin{lem}
\label{lem:MultifusionCommutingSquaresNondegenerate}
Suppose that the set of isomorphism classes of simple summands of $b$ and $b\otimes c \otimes c^\vee$ agree.
Then the commuting square \eqref{eq:CommutingSquareFromMultifusion}
is nondegenerate.
Moreover, the basic construction commuting square is given by
\begin{equation}
\label{eq:MultifusionBasicConstructionCommutingSquare}
\begin{matrix}
\xymatrix@C=5pt@R=2pt{
\End_\cC(a\otimes b\otimes c)  
&\subset & 
\End_\cC(a\otimes b\otimes c\otimes c^\vee)  
\\
\cup&&\cup
\\
\End_\cC(b\otimes c)  
&\subset & 
\End_\cC(b\otimes c\otimes c^\vee)  
}
\end{matrix}
\end{equation}
A similar statement holds if the set of isomorphism classes of simple summands of $a^\vee\otimes a \otimes b$ and $b$ agree.
\end{lem}
\begin{proof}
We prove \eqref{eq:CommutingSquareFromMultifusion}
is nondegenerate under the condition \ref{Nondegenerate:exists P/Q is M/N}, and the second is similar.
Let $\cS$ be a set of representatives for the common set of isomorphism classes of simple summands of $b$ and $b\otimes c \otimes c^\vee$.
For each $s\in \cS$, 
pick a basis $\{v_s\}$ of $\cC(s\to b\otimes c\otimes c^\vee)$ consisting of isometries with orthogonal ranges
so that 
$$
\id_{b\otimes c\otimes c^\vee}
=
\sum_{s\in \cS}\sum_{v_s} v_s\circ v_s^\ast
\qquad\text{and}\qquad
v_s^\ast \circ v_s = \id_s
\quad
\forall\,s\in \cS.
$$
For every $s\in \cS$, pick a single isometry $w_s\in \cC(s \to b)$ so that
$$
\id_{b\otimes c\otimes c^\vee}
=
\sum_{s\in \cS}
\sum_{v_s}
v_s\circ v_s^\ast
=
\sum_{s\in \cS}
\sum_{v_s}
v_s\circ w_s^\ast \circ w_s\circ v_s^\ast.
$$
Now using isotopy, for each $s\in \cS$ and isometry basis element $v_s$, we set
$$
u_s:=
\sqrt{\rho_c}
\cdot
\tikzmath{
\draw (0,-1.5) node[below] {$\scriptstyle b$} -- (0,0);
\draw (-.3,0) -- (-.3,.5) node[above] {$\scriptstyle b$};
\draw[thick,red] (0,0) -- (0,.5) node[above] {$\scriptstyle c$};
\draw[thick,red] (.3,.3) arc (180:0:.2cm) -- (.7,-1.5) node[below] {$\scriptstyle c$};
\node at (-.2,-.5) {$\scriptstyle s$};
\roundNbox{unshaded}{(0,0)}{.3}{.2}{.2}{$v_s$}
\roundNbox{unshaded}{(0,-1)}{.3}{0}{0}{$w_s^\ast$}
}
=
(\id_{b\otimes c}\otimes \coev_c^\ast) \circ (v_s\otimes \id_{c^\vee}).
$$
Then we see that $\amalg_{s\in\cS}\{u_s\}$ satisfies the condition
\begin{equation}
\label{eq:PimsenerPopaForCategory}
\id_{b\otimes c\otimes c^\vee}
=
\frac{1}{\rho_c}
\sum_{s\in \cS}\sum_{u_s}\,
\tikzmath{
\draw (-.15,-.5) -- (-.15,1.5);
\draw[thick, red] (.15,1.5) -- (.15,.7) arc (-180:0:.15cm) -- (.45,1.5);
\draw[thick, red] (.15,-.5) -- (.15,.3) arc (180:0:.15cm) -- (.45,-.5);
\roundNbox{unshaded}{(0,1)}{.3}{0}{0}{$u_s$};
\roundNbox{unshaded}{(0,0)}{.3}{0}{0}{$u_s^\ast$};
}
=
\sum_{s\in \cS}\sum_{u_s} u_s f u_s^\ast
\end{equation}
where 
$f:=\frac{1}{\rho_c} \cdot \id_b\otimes(\coev_c\circ \coev_c^\ast)
$,
which immediately implies the inclusion
$$
\End_\cC(b)\subset \End_\cC(b\otimes c) \overset{\frac{1}{\rho_c} (\id_b\otimes(\coev_c\circ \coev_c^\ast))}{\subset} \End_\cC(b \otimes c \otimes c^\vee)
$$
is standard, i.e., a Jones basic construction by \cite[Lem.~5.3.1 and Cor.~5.3.2]{MR1473221}.
The only interesting part in checking the hypotheses for this recognition lemma is checking $\End_\cC(b\otimes c\otimes c^\vee)$ is generated by $\End_\cC(b\otimes c)$ and $f$, which follows from \eqref{eq:PimsenerPopaForCategory} using the following graphical argument:
$$
\tikzmath{
\draw (-.3,-.5) -- (-.3,.5);
\draw[thick, red] (0,-.5) -- (0,.5);
\draw[thick, red] (.3,-.5) -- (.3,.5);
\roundNbox{unshaded}{(0,0)}{.3}{.2}{.2}{$x$}
}
=
\frac{1}{\rho_c^2}
\sum_{s,t\in\cS}
\sum_{u_s, \upsilon_t}
\tikzmath{
\draw (-.15,-2.5) -- (-.15,2.5);
\draw[thick, red] (.15,2.5) -- (.15,1.7) arc (-180:0:.15cm) -- (.45,2.5);
\draw[thick, red] (.15,0) -- (.15,1.3) arc (180:0:.15cm) -- (.45,-1.3) arc (0:-180:.15cm) -- (.15,0);
\draw[thick, red] (.15,-2.5) -- (.15,-1.7) arc (180:0:.15cm) -- (.45,-2.5);
\roundNbox{unshaded}{(0,2)}{.3}{0}{0}{$u_s$};
\roundNbox{unshaded}{(0,1)}{.3}{0}{0}{$u_s^\ast$};
\roundNbox{unshaded}{(0,0)}{.3}{0}{.4}{$x$};
\roundNbox{unshaded}{(0,-1)}{.3}{0}{0}{$\upsilon_t$};
\roundNbox{unshaded}{(0,-2)}{.3}{0}{0}{$\upsilon_t^\ast$};
}
=
\frac{1}{\rho_c^2}
\sum_{s,t\in\cS}
\sum_{u_s, \upsilon_t}
\tikzmath{
\draw (-.15,-1.5) -- (-.15,1.5);
\draw[thick, red] (.15,1.5) -- (.15,.7) arc (-180:0:.15cm) -- (.45,1.5);
\draw[thick, red] (.15,-1.5) -- (.15,-.7) arc (180:0:.15cm) -- (.45,-1.5);
\roundNbox{unshaded}{(0,1)}{.3}{0}{0}{$u_s$};
\roundNbox{unshaded}{(-.15,0)}{.3}{.05}{.05}{$x_{s,t}$};
\roundNbox{unshaded}{(0,-1)}{.3}{0}{0}{$\upsilon_t^\ast$};
}
=
\frac{1}{\rho_c^2}
\sum_{s,t\in\cS}
\sum_{u_s, \upsilon_t}
\tikzmath{
\draw (-.15,-1.5) -- (-.15,1.5);
\draw[thick, red] (.35,1.5) -- (.35,.7) arc (-180:0:.15cm) -- (.65,1.5);
\draw[thick, red] (.35,-1.5) -- (.35,.3) arc (180:0:.15cm) -- (.65,-1.5);
\roundNbox{unshaded}{(0,1)}{.3}{0}{.2}{$u_s$};
\roundNbox{unshaded}{(-.15,-.2)}{.3}{.05}{.05}{$x_{s,t}$};
\roundNbox{unshaded}{(0,-1)}{.3}{0}{.2}{$\upsilon_t^\ast$};
}
\in 
\langle \End_\cC(b\otimes c), f\rangle.
$$
Under this identification with the Jones basic construction, $\amalg_{s\in\cS}\{u_s\}$ forms a Pimsner-Popa basis for $\End_\cC(b\otimes c)$ over $\End_\cC(b)$.
By tensoring with $a$ on the left, we immediately get
$$
\id_{a^\vee\otimes a \otimes b\otimes c}
=
\frac{1}{\lambda_a}
\sum_v\,
\tikzmath{
\draw[thick, blue] (-.45,-.5) -- (-.45,1.5);
\draw (-.15,-.5) -- (-.15,1.5);
\draw[thick, red] (.15,1.5) -- (.15,.7) arc (-180:0:.15cm) -- (.45,1.5);
\draw[thick, red] (.15,-.5) -- (.15,.3) arc (180:0:.15cm) -- (.45,-.5);
\roundNbox{unshaded}{(0,1)}{.3}{0}{0}{$u_s$};
\roundNbox{unshaded}{(0,0)}{.3}{0}{0}{$u_s^\ast$};
}
$$
It follows that
$$
\End_\cC(a\otimes b)\subset \End_\cC(a\otimes b\otimes c) \overset{\frac{1}{\rho_c} \id_{a\otimes b}\otimes(\coev_c\circ \coev_c^\ast)}{\subset} \End_\cC(a\otimes b \otimes c \otimes c^\vee)
$$
is standard
and $\amalg_{s\in \cS}\{\id_a\otimes u_s\}$ is a Pimnser-Popa basis for $\End_\cC(a\otimes b \otimes c)$ over $\End_\cC(a\otimes b)$.
Hence \ref{Nondegenerate:exists P/Q is M/N} holds, and \eqref{eq:CommutingSquareFromMultifusion} is nondegenerate.
It is immediate that \eqref{eq:MultifusionBasicConstructionCommutingSquare} is the basic construction commuting square.
\end{proof}

\begin{example}
\label{ex:EventuallyNondegenerate}
Suppose $\cC$ has chosen generator $X$ and is equipped with the standard unitary dual functor with respect to $X$.
By (A) in the proof of Theorem \ref{thm:ExistsStandardInclusion}, there is a spherical state $\psi$ on $\End_\cC(1_\cC)$ such that $\psi \circ \tr^\vee_L = \psi \circ \tr^\vee_R$.
Taking 
$a=X^{\alt\otimes 2n-1}$, $b=\overline{X}^{\alt\otimes 2k+j}$, $c=X$ or $\overline{X}$ depending on parity of $j$,
and reflecting about the $y$-axis,
Lemma \ref{lem:MultifusionCommutingSquaresNondegenerate}
tells us
there is a $k\in \bbN$ such that for every $n\in \bbN$ and $j\geq 0$,
the commuting squares
\begin{equation*}
\xymatrix@C=5pt@R=2pt{
\End_\cC(X^{\alt\otimes 2k+2n+j+1}) &\supset & \End_\cC(X^{\alt\otimes 2k+2n+j}) &\supset & \End_\cC(X^{\alt\otimes 2k+2n+j-1})
\\\cup&&\cup&&\cup
\\
\End_\cC(\overline{X}^{\alt\otimes 2k+j+2})  &\supset & \End_\cC(\overline{X}^{\alt\otimes 2k+j+1}) &\supset & \End_\cC(\overline{X}^{\alt\otimes 2k+j})
}
\end{equation*}
are nondegenerate, and the left commuting square is the basic construction commuting square of the right commuting square.
A similar result holds for the other 3 types of composite commuting squares from the lattice \eqref{eq:CommutingSquareLattice}.
\end{example}

\begin{lem}
\label{lem:CommutationLemma}
Under Assumption \ref{ass:commutingsquare}, 
consider the commuting square
\begin{equation*}
\xymatrix@C=5pt@R=2pt{
\End_\cC(a\otimes b)
&\subset & 
\End_\cC(a^\vee \otimes a \otimes b) 
\\\cup&&\cup
\\
\End_{\cC}(a)
&\subset &
\End_\cC(a^\vee\otimes a).
}
\end{equation*}
The subalgebra
$$
(\End_\cC(a^\vee\otimes a\otimes) \id_{b})' \cap (\id_{a^\vee}\otimes \End_\cC(a\otimes b))
\subset 
\End_\cC(a^\vee\otimes a\otimes b)
$$ 
is exactly $\id_{a^\vee\otimes a}\otimes \End_\cC(b)$.
\end{lem}
\begin{proof}
Observe that the projection 
$$
e_a := 
\frac{1}{\lambda_a}
\cdot
\tikzmath{
\draw (.5,-.5) -- (.5,.5);
\draw[thick, blue] (-.3,-.5) arc (180:0:.3cm);
\draw[thick, blue] (-.3,.5) arc (-180:0:.3cm);
}
=
\frac{1}{\lambda_a}\cdot (\ev_a^\ast \circ \ev_a)\otimes \id_b
\in \End_\cC(a^\vee\otimes a)\otimes \id_b
$$
is an orthogonal projection.
Suppose that $x\in\End_\cC(a\otimes b)$ such that 
$$
\id_{a^\vee}\otimes x 
\in 
(\End_\cC(a^\vee \otimes a)\otimes \id_b)' 
\cap 
(\id_{a^\vee}\otimes \End_\cC(a\otimes b)).
$$
Then since $e_a \in \End_\cC(a^\vee \otimes a)\otimes \id_b$, $\id_{a^\vee}\otimes x$ and $e_a$ commute.
Hence
$$
\tikzmath[xscale=-1]{
	\draw (-.1,-.5) -- (-.1,.5);
	\draw[very thick, blue] (.1,-.5) -- (.1,.5);
	\roundNbox{unshaded}{(0,0)}{.25}{0}{0}{$x$}
}
=
\tikzmath[xscale=-1]{
	\draw (-.1,-.7) -- (-.1,1.7);
	\draw[very thick, blue] (.1,-.7) -- (.1,.7) arc (180:0:.15cm) -- (.4,-.4) arc (-180:0:.15cm) -- (.7,1.4) arc (0:180:.15cm) -- (.4,1.3) arc (0:-180:.15cm) -- (.1,1.7);
	\roundNbox{unshaded}{(0,0)}{.25}{0}{0}{$x$}
	\roundNbox{dashed}{(0,0)}{.4}{0}{.15}{}
	\roundNbox{dashed}{(0,1)}{.4}{0}{.15}{}
}
=
\lambda_a\cdot
\tikzmath[xscale=-1]{
	\draw (-.1,-.7) -- (-.1,1.7);
	\draw[very thick, blue] (.1,-.7) -- (.1,.7) arc (180:0:.15cm) -- (.4,-.4) arc (-180:0:.15cm) -- (.7,1.4) arc (0:180:.15cm) -- (.4,1.3) arc (0:-180:.15cm) -- (.1,1.7);
	\roundNbox{unshaded}{(0,0)}{.25}{0}{0}{$x$}
	\roundNbox{dashed}{(0,0)}{.4}{0}{.15}{}
	\roundNbox{unshaded}{(0,1)}{.4}{0}{.15}{$e_a$}
}
=
\lambda_a\cdot
\tikzmath[xscale=-1]{
	\draw (-.1,-.7) -- (-.1,1.7);
	\draw[very thick, blue] (.1,-.7) -- (.1,.7) arc (180:0:.15cm) -- (.4,-.4) arc (-180:0:.15cm) -- (.7,1.4) arc (0:180:.15cm) -- (.4,1.3) arc (0:-180:.15cm) -- (.1,1.7);
	\roundNbox{unshaded}{(0,0)}{.25}{0}{0}{$x$}
	\roundNbox{dashed}{(0,0)}{.4}{0}{.15}{}
	\roundNbox{unshaded}{(0,1)}{.4}{0}{.15}{$e_a^2$}
}
=
\lambda_a\cdot
\tikzmath[xscale=-1]{
	\draw (-.1,-1.7) -- (-.1,1.7);
	\draw[very thick, blue] (.1,.25) -- (.1,.7) arc (180:0:.15cm) -- (.4,-.7) arc (0:-180:.15cm) -- (.1,-.25);
	\draw[very thick, blue] (.1,-1.7) -- (.1,-1.3) arc (180:0:.15cm) -- (.4,-1.4) arc (-180:0:.15cm) -- (.7,1.4) arc (0:180:.15cm) -- (.4,1.3) arc (0:-180:.15cm) -- (.1,1.7);
	\roundNbox{unshaded}{(0,0)}{.25}{0}{0}{$x$}
	\roundNbox{unshaded}{(0,-1)}{.4}{0}{.15}{$e_a$}
	\roundNbox{dashed}{(0,0)}{.4}{0}{.15}{}
	\roundNbox{unshaded}{(0,1)}{.4}{0}{.15}{$e_a$}
}
=
\frac{1}{\lambda_a}\cdot
\tikzmath[xscale=-1]{
	\draw (-.1,-.7) -- (-.1,.7);
	\draw[very thick, blue] (.1,-.25) arc (-180:0:.15cm) -- (.4,.25) arc (0:180:.15cm);
	\draw[very thick, blue] (.6,-.7) -- (.6,.7);
	\roundNbox{unshaded}{(0,0)}{.25}{0}{0}{$x$}
}
\in \id_{a}\otimes \End_{\cC}(c)
\,.
$$
The result follows.
\end{proof}

\begin{cor}
\label{cor:CommutationCorollary}
Under the hypotheses of Lemma \ref{lem:CommutationLemma}, 
for the commuting square
\begin{equation*}
\xymatrix@C=5pt@R=2pt{
\End_\cC(a^{\alt \otimes 2k}\otimes b)
&\subset & 
\End_\cC(a^\vee \otimes a^{\alt \otimes 2k} \otimes b) 
\\\cup&&\cup
\\
\End_{\cC}(a^{\alt \otimes 2k})
&\subset &  
\End_\cC(a^\vee \otimes a^{\alt \otimes 2k}) 
,
}
\end{equation*}
$
(\End_\cC(a^\vee \otimes a^{\alt \otimes 2k})\otimes \id_b)' \cap (\id_{a^\vee}\otimes \End_\cC(a^{\alt \otimes 2k}\otimes b) 
=
\id_{a^\vee \otimes a^{\alt \otimes 2k}}\otimes \End_\cC(b)$.
\end{cor}
\begin{proof}
The proof follows from Lemma \ref{lem:CommutationLemma} by a simple induction argument.
\end{proof}

\subsection{Inductive limits of nondegenerate commuting squares}

For this section, we fix a nondegenerate commuting square of finite dimensional von Neumann algebras
\begin{equation}
\label{eq:FiniteDimensionalCommutingSquare}
\begin{matrix}
\xymatrix@C=5pt@R=2pt{
M_0  &\subset & M_1 
\\
\cup&&\cup
\\
N_0   &\subset & N_1
}
\end{matrix}
\end{equation}
such that the inclusion $M_0\subset (M_1,\tr_1)$ is Markov with index $d^2$.
By Lemma \ref{lem:NondegeneracyAndMarkovTrace} above, $N_0\subset (N_1, \tr_1|_{N_1})$ is also Markov with index $d^2$.
Iterating the basic construction, we get an increasing sequence of nondegenerate commuting squares.
Define $M_\infty := \varinjlim M_n$, $N_\infty := \varinjlim N_n$, and $\tr_\infty := \varinjlim \tr_n$ on $M_\infty$.
Notice that $\tr_\infty$ is faithful on $M_\infty$ since $\tr_n$ is faithful on $M_n$ for all $n$ by nondegeneracy.

It is straightforward to verify that $M_\infty$ acts on $H:=L^2(M_\infty, \tr_\infty)$ by bounded operators, where $\|x^*x\|_{B(H)} \leq \|x^*x\|_{M_n}$ for $x\in M_n$.
Let $M := M_\infty'' \subset B(H)$ and $N := N_\infty '' \subset B(H)$.
Define $\tr$ on $M$ by $\tr(x) := \langle x\Omega, \Omega\rangle$ where $\Omega \in H$ is the image of $1\in M_\infty$.

\begin{lem}
\label{lem:FaithfulTrace}
The normal state $\tr$ on $M$ is a faithful trace.
\end{lem}
\begin{proof}
Traciality follows by normality of $\tr$, SOT density of $M_\infty$ in $M$, and the Kaplansky Density Theorem.
To show $\tr$ is faithful, we use the proof in \cite[\S6.2]{JonesVNA}, which we include for completeness and convenience.

Suppose $\tr(x^*x) = 0$.
Then for all $m \in M_\infty$,
$$
\|xm\Omega \|_2^2 
= 
\|x R_m \Omega\|_2^2 
= 
\|R_m x\Omega\|_2^2 
\leq
\|R_m\|^2\| x\Omega\|_2^2 
=
\|R_m\|^2 \tr(x^*x) 
= 
0
$$
where for $a,m\in M_\infty$, $R_m a\Omega = am\Omega$ is the bounded right action as $\tr_\infty$ is a trace on $M_\infty$.
We conclude $x = 0$.
\end{proof}

We thus have an increasing sequence of finite dimensional nondegenerate commuting squares with an inductive limit inclusion of finite von Neumann algebras with the canonical inductive limit faithful normal tracial state $\tr$ on $M$.
Notice that at each iteration, identifying $N_{n+1}$ with a subalgebra of $M_{n+1}$ identifies the Jones projection $f_n$ for $N_{n-1}\subset N_n$ with the Jones projection $e_n$ for $M_n \subset M_{n-1}$.

\begin{equation}
\label{eq:JonesTowerOfCommutingSquares}
\begin{matrix}
\xymatrix@C=5pt@R=2pt{
M_0  &\subset & M_1   &\overset{e_1}{\subset} & M_2 &\overset{e_2}{\subset} & M_3 &\overset{e_3}{\subset} & \cdots &\subset & M 
\\
\cup&&\cup && \cup&&\cup &&&& \cup
\\
N_0   &\subset & N_1 &\overset{e_1}{\subset} & N_2 &\overset{e_2}{\subset} & N_3 &\overset{e_3}{\subset} & \cdots &\subset & N 
}
\end{matrix}
\end{equation}

\begin{lem}
\label{lem:II_1Inclusion}
The finite von Neumann algebras $M$ and $N$ are finite direct sums of ${\rm II}_1$ factors with $Z(M)= Z(M_0)\cap Z(M_1)$ and $Z(N) = Z(N_0)\cap Z(N_1)$.
\end{lem}
\begin{proof}
We prove the result for $M$ and the result for $N$ is similar.
The Bratteli diagram for the inclusion $M_0\subseteq M_1$ is a disjoint union of $\dim(Z(M_0)\cap Z(M_1))$ connected graphs.
Denote the minimal projections of $Z(M_0)\cap Z(M_1)$ by $\{p\}$.
Then the Bratteli diagram of each inclusion $pM_0\subseteq pM_1$ is connected with Markov trace $x\mapsto \tr_1(px)/\tr_1(p)$.
Iterating the basic construction, we have $M_n = \bigoplus pM_n$ for all $n$, and the von Neumann algebra generated by $\varinjlim pM_n$ in the GNS representation with respect to the unique Markov trace is clearly isomorphic to $pM$, as multiplication by $p$ is SOT continuous.
We know that each $pM$ is a ${\rm II}_1$ factor as the unique trace is faithful by an argument similar to Lemma \ref{lem:FaithfulTrace}.
Thus $M = \bigoplus_p pM$ is a finite direct sum of ${\rm II}_1$ factors with $Z(M) = \spann\{p\} = Z(M_0)\cap Z(M_1)$.
\end{proof}

\begin{prop}
The Watatani index of the inductive limit hyperfinite type ${\rm II}_1$ inclusion $N\subset M$ from \eqref{eq:JonesTowerOfCommutingSquares} is equal to the Watatani index $\sum_b bb^*$ of the inclusion $N_0\subset M_0$, where $\{b\}$ is any (left) Pimsner-Popa basis for $M_0$ over $N_0$.
\end{prop}
\begin{proof}
The proof is identical to \cite[Cor.~5.7.4]{MR1473221}.
By nondegeneracy and Facts \ref{facts:Nondegeneracy}, there is a Pimsner-Popa basis $\{b\}$ for $M_0$ over $N_0$ which is also a Pimsner-Popa basis for $M_n$ over $N_n$ for every $n$.
This means for every $x\in \bigcup_{n\geq 0} M_n$, we have $x = \sum_b bE_N(b^*x)$.
This equation clearly varies ultraweakly continuously in $x$, and thus $\{b\}$ is a right Pimsner-Popa basis for $M$ over $N$.
We conclude that the Watatani index of $N\subseteq M$ is equal to $\sum_b bb^*$.
\end{proof}

\begin{remark}
\label{rem:HorizontallyConnected}
If the commuting square \eqref{eq:AppendixCommutingSquare} is nondegenerate and
\emph{horizontally connected}, i.e., 
$Z(M_0)\cap Z(M_1) = \bbC$ and $Z(N_0)\cap Z(N_1) = \bbC$,
then $N\subset M$ is a ${\rm II}_1$ subfactor whose Jones index $[M:N]$ is equal to the Watatani index $\sum_b bb^*$, which by \cite[Lem.~5.3.3]{MR1473221} or \cite[Cor.~6.2]{MR1055708}, is necessarily equal to  $\|\Gamma\Gamma^T\|$ where $\Gamma$ is the bipartite adjacency matrix for the Bratteli diagram of the inclusion $N_0\subset M_0$.
\end{remark}

\section{Popa's theorem for homogeneous connected finite depth hyperfinite \texorpdfstring{$\rm II_1$}{II1} multifactor inclusions}
\label{appendix:Popa}

In this section, for completeness and convenience of the reader, we give a proof of Popa's theorem that a homogeneous finite index finite depth connected hyperfinite $\rm II_1$ multifactor inclusion is completely determined by its standard invariant.
We adapt the proof for subfactors from \cite{MR1055708}.

Suppose $A\subset (B,\tr_B)$ is a finite index connected $\rm II_1$ multifactor inclusion with its unique Markov trace.
As in Definition \ref{defn:StronglyMarkov}, $A\subset (B,\tr_B)$ is strongly Markov.

\begin{fact}[{\emph{Tunnel-tower duality} cf.~\cite[1.3.2]{MR1278111}}]
\label{fact:TunnelTowerDuality}
Suppose $B_{-1}=A\subset B=B_0$ is a homogeneous connected $\rm II_1$ multifactor inclusion of index $d^2$.
Consider the Jones tower of length $n$ together with any Jones tunnel of length $n$: 
$$
B_{-n}\subset \cdots \subset B_{-1} \subset B_0 \overset{e_0}{\subset} B_1 \subset \cdots \subset B_n.
$$
Identifying the tower with the multistep basic constuction from Facts \ref{facts:StronglyMarkov} on $L^2(B_0, \tr_0)$, we have that
$B_n = (JB_{-n}J)'$ for all $n>0$, and conjugation by $J=J_0$ moves the Jones projections as $Je_{-n}J = e_n$ as in \cite[3.2]{MR1055708}.
In particular, we have conjugation by $J$ is an anti-isomorphism between the centralizer algebras:
$$
B_{-n}' \cap B_0
\underset{\text{anti}}{\cong}
J(B_{-n}' \cap B_0)J 
= 
JB_{-n}J' \cap JB_{0}J 
=
B_n \cap B_0'.
$$
\end{fact}

\begin{lem}[{cf.\ \cite[Lemma 4.5]{MR3308880}}]
\label{lem:StandardInvariantOfReduced}
Suppose $A\subset B$ is finite index with standard invariant $(\cC,X)$.
Suppose $p \in A'\cap B$ is a minimal projection.
Then the standard invariant of the reduced subfactor $pA \subset pBp$ is equivalent to some $2\times 2$ unitary multifusion category generated by a single simple object of $\cC$.
\end{lem}
\begin{proof}
Let $p_i \in Z(A)$ be the minimal projection such that $pp_i = p$, and let $q_j \in Z(B)$ be the minimal projection such that $pq_j = p$.
Consider the bimodules
${}_{A_i} pL^2B_j{}_{B_j}$
and
${}_{pA} L^2(pBp)_{pBp}$.
Their dual Q-systems are isomorphic.
Indeed,
since $pBp = pB_jp$ and $pA = pA_i$,
observe that
\begin{align*}
{}_{pBp}
L^2(pBp) 
\boxtimes_{Ap} 
L^2(pBp){}_{pBp} 
&=
{}_{pBp}
L^2(pB_jp) 
\boxtimes_{pA_i} 
L^2(pB_jp){}_{pBp} 
\\&\cong
{}_{pBp} pL^2B \boxtimes_{B_j} L^2B_j p 
\boxtimes_{pA_i} 
pL^2B_j \boxtimes_{B_j} L^2Bp {}_{pBp}
\\&\cong
{}_{pBp} pL^2B \boxtimes_{B_j} L^2B_j p 
\boxtimes_{A_i} 
pL^2B_j \boxtimes_{B_j} L^2Bp {}_{pBp}
\end{align*}
and ${}_{pBp}pL^2B_j {}_{B_j}$ is an invertible bimodule.
We conclude that the standard invariant of 
$pA \subset pBp$ 
is equivalent to the $2\times 2$ unitary multifusion subcategory of $\cC$ generated by the simple object ${}_{A_i}pL^2B_j {}_{B_j} \in \cC_{ij}$.
\end{proof}

\begin{cor}[{cf.~\cite[Thm.~3.8]{MR1055708}}]
\label{cor:UpperBoundOfIndices}
Suppose $A\subset B$ is finite index and finite depth with standard invariant $(\cC,X)$.
Define $M := \max\set{\dim(c)^2}{c\in \Irr(\cC) }<\infty$.
\begin{enumerate}[label=(\arabic*)]
\item 
If $(A_n)_{n\in \bbN}$ is the Jones tower of $A_0 = A\subset B =A_1$, then for every 
$n\in \bbN$ and any minimal projection $p\in A_0'\cap A_n$, we have
$
[pA_np:pA_0]
\leq
M$.
\item
If moreover $A\subset B$ is homogeneous, then
for any homogeneous tunnel $(A_{-n})_{n\in \bbN}$,
for any $n \geq 0$ and any minimal projection $p\in A_{-n}'\cap B$, we have
$
[pBp:pA_{-n}]
\leq
M$.
\end{enumerate}
\end{cor}
\begin{proof}
By tunnel-tower duality from Fact \ref{fact:TunnelTowerDuality}, it suffices to prove the first statement.
This is immediate by Lemma \ref{lem:StandardInvariantOfReduced}.
\end{proof}

\begin{lem}[{cf.~\cite[Thm.~4.3]{MR1055708}}]
\label{lem:ExpectationEstimateForTunnel}
Suppose $A\subset B$ has finite depth.
There is a $\lambda>0$ such that for any homogeneous tunnel $(A_{-n})_{n\in \bbN}$, we have
$E_{A_{-n}\vee (A_{-n}'\cap B)}(x) \geq \lambda x$ for all $x\in B^+$, $n\in \bbN$.
\end{lem}
\begin{proof}
The proof of \cite[Thm.~4.3]{MR1055708} applies verbatim with Corollary \ref{cor:UpperBoundOfIndices} in place of \cite[Thm.~3.8]{MR1055708}.
(The proof of \cite[Lem.~4.2]{MR1055708} also applies verbatim when $M$ is a $\rm II_1$ multifactor.)
\end{proof}

\begin{lem}[{cf.~\cite[Lem.~4.4]{MR1055708}}]
\label{lem:TunnelContinuation}
Suppose $A\subset B$ is hyperfinite.
Suppose we have a choice of finite homogeneous tunnel 
$(A_{-n})_{n=1}^j$.
For any $\varepsilon>0$
and any finite set $F\subset A_{-j}\vee (A_{-j}'\cap B)$,
there is a $k \geq j$ and a homogeneous continuation of the tunnel $(A_{-n})_{n=1}^k$ such that
$f \in_\varepsilon A_{-k}'\cap B$ for all $f\in F$. 
\end{lem}
\begin{proof}

Since $A_{-j}'\cap B \subset A_{-k}'\cap B$ whenever $k \geq j$, it suffices to consider the case where $F \subset A_{-j}$.

Fix an isomorphism $A_{-j}\cong R^{\oplus c}$ for $c\in \{a,b\}$, and consider the diagonal embedding $\iota: R\hookrightarrow R^{\oplus c}\cong A_{-j}$.
Let $z_1,\dots, z_c$ be the minimal central projections in $A_{-j}$.
Then every $x \in A_{-j}$ can be uniquely expressed in the form $x = \sum_{i=1}^c \iota(x_i)z_i$ for some $x_1,\dots, x_c \in R$. 
Since each $z_i \in A_{-k}'\cap B$ whenever $k \geq j$, it suffices to consider the case where $F \subset \iota(R)$.

The proof now proceeds as in \cite[Lem.~4.4]{MR1055708}. View $R$ as generated by a sequence $(e_i)_{i=0}^\infty$ of Jones projections with $\lambda$ the Markov index of $A\subset B$. 
Given $\varepsilon > 0$ and a finite subset $F \subset \iota(R)$, there is $N \in \mathbb{N}$ such that $F \subset_\epsilon \mathrm{Alg}\lbrace 1, \iota(e_1), \ldots, \iota(e_{N-1})\rbrace$.

Let $(A^{(0)}_{-n})_{n=j+1}^{j+N}$ be any homogeneous continuation of the tunnel and let $(e^{(0)}_{-n+1})_{n=j+1}^{j+N}$ be the corresponding Jones projections, which have constant centre-valued trace by homogeneity.
There exists a $Z(A_{-j})$-valued trace-preserving isomorphism between the finite-dimensional algebras $\mathrm{Alg}\lbrace 1, \iota(e_1), \ldots, \iota(e_{N-1})\rbrace$ and $\mathrm{Alg}\lbrace 1, e^{(0)}_{-j-1}, \ldots, e^{(0)}_{-j-N+1}\rbrace$.
Thus there exists a unitary $u \in A_{-j}$ such that $\iota(e_i) = ue^{(0)}_{-j-i}u^*$ for each $i$.
Setting $A_{-n} := uA^{(0)}_{-n}u^*$ for $j+1 \leq n \leq j+N$ gives a homogeneous continuation of the tunnel with the desired property.
\end{proof}

\begin{cor}[{cf.~\cite[Cor.~4.5]{MR1055708}}]
\label{cor:FiniteIndexTunnel}
Suppose $A\subset B$ has finite depth and is hyperfinite.
There is a choice of homogeneous tunnel $(A_{-n})_{n\in \bbN}$ such that $C:= \left(\bigcup A_{-n}'\cap B\right)''$ has finite (Pimsner-Popa) index in $B$.
\end{cor}
\begin{proof}
The proof of \cite[Cor.~4.5]{MR1055708} with Lemma \ref{lem:ExpectationEstimateForTunnel} in place of \cite[Thm.~4.3]{MR1055708} and Lemma \ref{lem:TunnelContinuation} in place of \cite[Lem.~4.4]{MR1055708} applies verbatim, up to the final sentence.
We obtain that there is a $\lambda>0$ such that
\begin{equation}
\label{eq:PimsnerPopa-lambda2}
\|E_C(x)\|_2^2 \geq \lambda\|x\|_2^2
\qquad\qquad
\forall\, x\in B^+.
\end{equation}
Observe now that by construction, $Z(C) =Z(B)$, so $C\subset B$ is a finite direct sum of $\rm II_1$ subfactors.
Now each of these component $\rm II_1$ subfactors satisfies \eqref{eq:PimsnerPopa-lambda2}, and thus each has Jones index at most $\lambda^{-1}$ by \cite[Thm.~2.2]{MR860811}.
We conclude that $E_C: B \to C$ has finite Pimsner-Popa index.
\end{proof}

The following lemma is well known to experts.
We include a proof for the reader's convenience.

\begin{lem}
\label{lem:ConvergenceIn2Norm}
Suppose $C$ is a von Neumann algebra with a faithful normal tracial state $\tr$, and $(C_n)_{n\geq 0}$ is an increasing sequence of unital $*$-subalgebras whose union is strongly dense in $C$.
Let $E_n : C\to C_n$ be the unique trace-preserving conditional expectation.
Then for all $x\in C$, $\|x\Omega - E_n(x)\Omega\|_2 \to 0$ as $n\to \infty$.
\end{lem}
\begin{proof}
Fix $x\in C$, and consider the unital $*$-subalgebra $C^\circ := \bigcup_{n\geq 0} C_n \subset C$.
Let $X$ be the $\|\cdot\|_\infty$-closed ball of $C$ of radius $\|x\|_\infty$.
Recall from \cite[Prop.~9.1.1]{JonesVNA} that $X$ is a complete metric space in $\|\cdot\|_2$, and the $\|\cdot\|_2$ topology on $X$ agrees with the strong operator topology.
Fix $\varepsilon>0$, and let $B_\varepsilon(x)$ denote the open ball of radius $\varepsilon$ in $\|\cdot\|_2$ about $x$.
Pick an open neighborhood $U\subseteq C$ for the strong operator topology such that $U\cap X = B_\varepsilon(x) \cap X$.

By the Kaplansky density theorem, there is an $y\in C^\circ\cap X$ with $y\in B_\varepsilon(x)$.
Let $N\in \bbN$ such that $y\in C_N$.
Then since $E_n(x) \in C_n$ is the unique element in $C_n$ closest to $x$ in $\|\cdot\|_2$, for all $n\geq N$,
$\|x\Omega - E_n(x)\Omega\|_2 \leq \|x \Omega- y\Omega\|_2 < \varepsilon$.
\end{proof}

\begin{prop}
\label{prop:PPBasesForTunnel}
Suppose $A\subset B$ has finite depth and $(A_{-n})_{n\in \bbN}$ is a homogeneous tunnel for $A\subset B$
such that $C:=\left(\bigcup A_{-n}'\cap B\right)''$ has finite Pimsner-Popa index in $B$ as in Corollary \ref{cor:FiniteIndexTunnel}.
\begin{enumerate}[label=(\arabic*)]
\item
For every $n\geq k\in \bbN$, the following are commuting squares when equipped with the Markov trace-preserving conditional expectations:
\begin{equation}
\label{eq:CommutingSquaresForPimsnerPopaBases}
\begin{tikzcd}[column sep=0em, row sep=0em]
A_{-n}'\cap B & \subset & C & \subset & B
\\
\cup &&\cup && \cup
\\
A_{-n}'\cap A_{-k} & \subset &C\cap A_{-k} & \subset & A_{-k}.
\end{tikzcd}
\end{equation}
\item
For $n$ sufficiently large, there is a Pimsner-Popa basis for $A_{-n}' \cap B$ over $A_{-n}'\cap A_{-k}$ which is also a Pimsner-Popa basis for both
$C$ over $C\cap A_{-k}$ 
and
for $B$ over $A_{-k}$.
\item
\label{prop:PPBasesForTunnel:iii}
There is a Pimsner-Popa basis for $C$ over $C\cap A_{-k}$ which is also a Pimser-Popa basis for $B$ over $A_{-k}$.
\end{enumerate}
\end{prop}
\begin{proof}
\item[(1)]
By \cite[Lem.~1.2.2]{MR720738},
for every inclusion of von Neumann algebras
$N\subset M\subset P$ with $P$ type $\rm II_1$ with a faithful normal trace, the unique trace-preserving conditional expectations $E_{N'\cap P}$ and $E_M$ commute.
Hence
$E_{A_{-n}'\cap B}$ and $E_{A_{-k}}$ commute whenever $n\geq k$.
This means the large composite square commutes.

Observe
\begin{equation}
\label{eq:CompositeConditionalExpectation}
E_{A_{-n}'\cap B}= E^C_{A_{-n}'\cap B}\circ E_C
\end{equation}
by uniqueness of the trace-preserving conditional expectation.
By Lemma \ref{lem:ConvergenceIn2Norm}, 
for every $x\in B$,
$$
\|E_C(x)\Omega -  E_{A_{-j}'\cap B}(x)\Omega\|_2
=
\|E_C(x)\Omega -  E^C_{A_{-j}'\cap B}(E_C(x))\Omega\|_2
\xrightarrow{j\to \infty}
0.
$$
Since $E_{A_{-j}'\cap B}$ and $E_{A_{-k}}$ commute for all $j\geq k$,
$E_C$ and
$E_{A_{-k}}$ commute, so the square on the right of \eqref{eq:CommutingSquaresForPimsnerPopaBases} commutes.

Now since
$E^C_{A_{-n}'\cap B} = E_{A_{-n}'\cap B}|_C$
by \eqref{eq:CompositeConditionalExpectation}
and
$E^C_{C\cap A_{-k}} = E_{A_{-k}}|_C$ since the square on the right of \eqref{eq:CommutingSquaresForPimsnerPopaBases} commutes, the square on the left of \eqref{eq:CommutingSquaresForPimsnerPopaBases} commutes.

\item[(2)]
Since $A_{-k}\subset B=A_0$ has finite depth, we can choose $n$ large so that both
$$
A_{-k} \subset A_0 \overset{f}{\subset} A_k
\qquad\text{and}\qquad
A_{-n}'\cap A_{-k} \subset A_{-n}' \cap A_0 \overset{f}{\subset} A_{-n}'\cap A_k
$$ 
are multi-step basic constructions as in Facts \ref{facts:StronglyMarkov} with the same Jones projection $f$.
Hence there is a finite subset $\{b\}\subset A_{-n}'\cap A_0$ such that $\sum_b bfb^* = 1_{A_{-n}'\cap A_k}= 1_{A_k}$.
Then for all $x\in B=A_0$, we have
$
x
=
\sum_b b E_{A_{-k}}(b^*x).
$
Since $E_{A_{-k}}$ and $E_C$ commute by statement (1), we see that
$
r
=
\sum_b b E^C_{A_{-k}\cap C}(b^*r)
$
for all $r\in C$.
Hence $\{b\}$ is the desired Pimsner-Popa basis.

\item[(3)]
Observe that every inclusion in the commuting square of finite multifactors on the right of \eqref{eq:CommutingSquaresForPimsnerPopaBases} has finite index.
Indeed, $C\subset B$ was assumed to have finite index, and $A_{-k}\subset B$ has finite index as a composite of finite index submultifactors in the Jones tunnel.
By statement (2), we see $C\cap A_{-k}\subset C$ has the same (finite) Watatani index as $A_{-k}\subset B$.
We conclude that $C\cap A_{-k}\subset B$ has finite index.

The result now follows immediately by statement (2)
and Lemma \ref{lem:Nondegenerate}.
\end{proof}

\begin{thm}[{cf.~\cite[Thm.~4.9]{MR1055708}}]
\label{thm:ExistsGeneratingHomogeneousTunnel}
Suppose $A\subset B$ is a homogeneous connected hyperfinite $\rm II_1$ multifactor inclusion with finite depth.
There is a choice of homogeneous tunnel $(A_{-n})_{n\in \bbN}$ such that $B= \left(\bigcup A_{-n}'\cap B\right)''$
and
$A= \left(\bigcup A_{-n}'\cap A\right)''$.
\end{thm}
\begin{proof}
Suppose for contradiction that there is no such choice of generating homogeneous tunnel.
By Corollary \ref{cor:FiniteIndexTunnel}, we can pick a homogeneous tunnel
$(A_{-n})_{n\in \bbN}$ for $A\subset B$
such that $C:=\left(\bigcup A_{-n}'\cap B\right)''$ has finite Pimser-Popa index in $B$.
Then by \cite[Cor.~4.8]{MR1055708}, which apply verbatim to the multifactor setting, 
there exist 
a free ultrafilter $\omega$ on $\bbN$,
an $x\in B^\omega$ with $\|x\|_2=1$, and
an increasing sequence $(k_j) \subset \bbN$
such that
$x \perp \left(\prod_\omega A_{-k_j}\right) C^\omega \left(\prod_\omega A_{-k_j}\right)$.

We now proceed exactly as in the proof of \cite[Thm.~4.9]{MR1055708}.
For each $k\in \bbN$, we use \ref{prop:PPBasesForTunnel:iii} of Proposition \ref{prop:PPBasesForTunnel} to pick a Pimsner-Popa basis $\{b^k_i\}_{i=1}^{\ell_n}$ for $A_{-k}$ over $C\cap A_{-k}$ which is also a Pimsner-Popa basis for $B$ over $C$.

We may arrange so that $\ell_k = \ell$ is \emph{independent} of $k\in \bbN$.
Indeed, the Watatani indices of 
$C\cap A_{-k}\subset A_{-k}$
and
$C\subset B$ 
are both given by
$$
\sum_{i=1}^{\ell_n} b^k_i (b^n_i)^* \in Z(A_{-k})\cap Z(B) = \bbC,
$$
i.e., they are both the same scalar $d^2$.
Since $Z(C)=Z(B)$ and $Z(C\cap A_{-k})=Z(A_{-k})$,
both of the inclusions 
$C\cap A_{-k}\subset A_{-k}$
and
$C\subset B$ 
are finite direct sums of $\rm II_1$ subfactors with the same Jones index.
Hence each Pimsner-Popa basis $\{b^k_i\}$ for $A_{-k}$ over $C\cap A_{-k}$ can be chosen 
to have cardinality $\lceil d^2\rceil$ by \cite{MR860811}.

Thus for each $n\in\bbN$,
$\sum_{i=1}^\ell b^n_{i} C = B$.
Setting $b_i := (b^{k_j}_i) \in \prod_\omega A_{-k_j}$,
we have $\sum_{i=1}^\ell b_i C^\omega = B^\omega$.
But $B^\omega \ni x \perp \sum_{i=1}^\ell b_i C^\omega \in\left(\prod_\omega A_{-k_j}\right)C^\omega$, a contradiction. 
\end{proof}

\section{Ocneanu compactness}
\label{sec:OcneanuCompactness}

We now prove a more general version of the Ocneanu compactness theorem than that which appears in the literature.
Here, we adapt the proof that appears in \cite[Thm.\,\,5.7.1]{MR1473221} to apply to a more general class of commuting squares.
This result was certainly known to Asaeda-Haagerup \cite{MR1686551}, Schou \cite{1304.5907}, and Popa \cite{MR1055708} among other experts.

\begin{thm}[Ocneanu Compactness]
\label{thm:OcneanuCompactness}
Suppose we have a nondegenerate commuting square of finite dimensional von Neumann algebras as in \eqref{eq:FiniteDimensionalCommutingSquare}
such that the inclusion $M_0 \subset (M_1, \tr_1)$ is Markov.
Let $N\subset M$ be the inductive limit hyperfinite type ${\rm II}_1$ inclusion.
Then the relative commutant $N'\cap M$ is equal to $N_1'\cap M_0$ considered inside $M_1$.
\end{thm}

\begin{proof}
The outline of the proof follows \cite[Thm.\,\,5.7.1]{MR1473221} closely.
We postpone our modified proofs of the technical lemmas, which appear below.
\begin{enumerate}[label=(\Alph*)]
\item
\label{sketch:FiniteStepRelativeCommutants}
The same argument from \cite[Thm.\,\,5.7.1]{MR1473221} shows that for all $1\leq p < q$, $N_q' \cap M_p = N_1'\cap M_0$.
This immediately implies that $N_1'\cap M_0 \subseteq N'\cap M$.
\item
\label{sketch:L2convergence}
Let $E_n: M \to M_n$ be the canonical trace-preserving conditional expectation, and let $\Omega$ be the image of $1_M$ in $L^2M$.
Let $x\in M$ and set $x_n = E_n(x)$.
By Lemma \ref{lem:ConvergenceIn2Norm}, $\|x\Omega - x_n\Omega \|_2 \to 0$ as $n\to \infty$.
\item
Starting with an $x\in N'\cap M$, we want to show that the sequence $(x_n)_{n\geq 0}$ from \ref{sketch:L2convergence} is getting arbitrarily close to the finite dimensional subspace $N_1'\cap M_0$ in $\|\cdot\|_2$.
We could then conclude by \ref{sketch:L2convergence} that $x\in N_1'\cap M_0$.
We break this step up as follows.
\begin{enumerate}[label=(\roman*)]
\item
\label{sketch:PropertiesOfPhi}
For $k\geq 0$, we have a map $\Phi_k : N_0'\cap M_k \to N_2'\cap M_{k+2}$ which sends $N_j'\cap M_k$ to $N_{j+2}'\cap M_{k+2}$ and $N_j'\cap N_k$ to $N_{j+2}'\cap N_{k+2}$ for all $0\leq j\leq k$. 
We define this map explicitly in Definition \ref{defn:DefnOfPhi} below, and we prove many properties about it in Proposition \ref{prop:PropertiesOfPhi}.
Of particular importance are:
\begin{itemize}
\item 
each $\Phi_k$ is a $*$-algebra map, and
\item
for all $y\in N_0'\cap M_k$, $\Phi_k(E_{N_k}(y)) = E_{N_{k+2}}(\Phi_k(y))$.
\end{itemize}
\item
\label{sketch:CompactTraceSpace}
In general, the maps $\Phi_n$ do \emph{not} preserve the Markov trace, and thus the $\Phi_n$ are not isometries on $(N_0'\cap M_n )\Omega$.
However, $\Phi_n$ gets closer to being an isometry as $n\to \infty$.
Indeed, for $n\in\bbN$, we consider the composites $\Psi_n := \Phi_{n-1} \circ \cdots \circ \Phi_1\circ  \Phi_0$ which map $N_j'\cap M_k \to N_{2n+j}' \cap M_{2n+k}$
for all $0\leq j\leq k$.
By \ref{sketch:PropertiesOfPhi}, $\Psi_n$ maps the subalgebra $N_j'\cap N_k \to N_{2n+j}' \cap N_{2n+k}$.

Now setting $j=k=0$, $\Psi_n$ maps $Z(N_0)=N_0'\cap N_0$ to $Z(N_{2n})=N_{2n}'\cap N_{2n}$, which is a canonically isomorphic algebra as $N_0$ is Morita equivalent to $N_{2n}$ via $N_n$ and the multistep basic construction (see Facts \ref{facts:StronglyMarkov}).
Thus starting with the trace $\tau_0=\tr_0$ on $Z(N_0)$, we obtain a sequence of traces on $Z(N_0)$ by setting $\tau_n = \tr_{2n}\circ \Psi_n$.
We show that each $\tau_n$ on $Z(N_0)$ is faithful, and that $(\tau_n)_{n\geq 0}$ converges to a faithful trace $\tau_\infty$ on $Z(N_0)$.
We prove this result in Proposition \ref{prop:TraceInfinity} below in the language of densities with respect to the trace $\tr_0$ on $Z(N_0)$.
\item
\label{sketch:CompactBound}
Now since all faithful traces on a finite dimensional algebra are comparable,
for every $n\in \bbN$, there is a $C_n>0$ such that $C^{-1}_n \tr_0 \leq \tau_n \leq C_n \tr_0$ on $Z(N_0)=N_0'\cap N_0$.
Since $\set{\tau_n}{n\in \bbN} \cup\{\tau_\infty\}$ is compact by \ref{sketch:CompactTraceSpace}, there is a $C>0$ \emph{independent of $n$} such that $C^{-1} \tr_0 \leq \tau_n \leq C \tr_0$ for all $n\in \bbN$.
\item
\label{sketch:UniformNormBound}
For all $y\in N_0'\cap M_0$, 
\begin{align*}
\|\Psi_n(y)\Omega\|_2 
&= 
\tr_{2n}(\Psi_n(y)^*\Psi_n(y)) 
\underset{\text{\ref{sketch:PropertiesOfPhi}}}{=}
\tr_{2n}(\Psi_n(y^*y)) 
\\&=
\tr_{2n}(E_{N_{2n}}(\Psi_n(y^*y)))
\underset{\text{\ref{sketch:PropertiesOfPhi}}}{=}
\tr_{2n}(\Psi_n(E_{N_0}(y^*y)))
=
\tau_n(E_{N_0}(y^*y)).
\end{align*}
Thus for all $n\in \bbN$ and all $y\in N_0'\cap M_0$, by \ref{sketch:CompactBound} we have 
$$
C^{-1}\|y\Omega\|_2 \leq \|\Psi_n(y)\Omega\|_2\leq C \|y\Omega\|_2.
$$
\item
\label{sketch:WhereIs x n}
It is a simple algebraic calculation that $x_n \in N_n'\cap M_n\subset N_0'\cap M_n$ for all $n\geq 0$.
We use the notation $H_n = (N_n'\cap M_n )\Omega$ for this finite dimensional Hilbert space, and we see from \ref{sketch:FiniteStepRelativeCommutants} that for all $n\in \bbN$,
$$
H_n \cap H_{n+1} 
= 
(N_n'\cap M_n ) \cap (N_{n+1}'\cap M_{n+1}) 
= 
N_{n+1}'\cap M_n 
= 
N_1'\cap M_0.
$$
\item
\label{sketch:AnyTwoNormsEquivalent}
Given a Hilbert space $X$ with closed subspaces $Y,Z$,
we can define two norms on $X/(Y\cap Z)$ by
\begin{align*}
\|\xi\|_1
&:=
\|(\xi+Y, \xi+Z)\|_{X/Y \oplus_{\ell^1}X/Z}
=
\dist(\xi,Y)+\dist(\xi,Z)
\\
\|\xi\|_2
&:=
\|\xi + Y\cap Z\|_{X/(Y\cap Z)}
=
\dist(\xi, Y\cap Z).
\end{align*}
When $X/(Y\cap Z)$ is finite dimensional, these two norms are equivalent.
Setting $X=(N_0'\cap M_1)\Omega$, $Y=H_0=(N_0'\cap M_0)\Omega$ and $Z=H_1=(N_1'\cap M_1)\Omega$, there is a $K>0$ such that for all $y\in N_0'\cap M_1$,
$$
\dist(y\Omega, H_0\cap H_1) \leq K(\dist(y\Omega, H_0) + \dist(y\Omega, H_1)).
$$
In particular, for all $y\in H_0=N_0'\cap M_0$,
$$
\dist(y\Omega, H_0\cap H_1) \leq K\dist(y\Omega, H_1).
$$

\item
Finally, we calculate for each $x_{2n} = E_{2n}(x) \in N_{2n}'\cap M_{2n}$, since $\Psi_n(H_k) = H_{2n+k}$, 
\begin{align*}
\dist(x_{2n} \Omega,& (N_1'\cap M_0)\Omega)
\\&=
\dist(x_{2n} \Omega, H_{2n}\cap H_{2n+1})
&&\text{\ref{sketch:WhereIs x n}}
\\&\leq
C\dist(\Psi_n^{-1}(x_{2n}) \Omega, H_0 \cap H_{1})
&&\text{\ref{sketch:UniformNormBound}}
\\&\leq
CK
\dist(\Psi_n^{-1}(x_{2n}) \Omega, H_1)
&&\text{\ref{sketch:AnyTwoNormsEquivalent}}
\\&\leq
C^2K
\dist(x_{2n} \Omega, H_{2n+1})
&&\text{\ref{sketch:UniformNormBound}}
\\&\leq
C^2K\dist(x_{2n} \Omega, x_{2n+1}\Omega)
&&\text{\ref{sketch:WhereIs x n}}
\\&\to 0
\text{ as $n\to \infty$}.
&&\text{\ref{sketch:L2convergence}}
\end{align*}
\end{enumerate}
\end{enumerate}
This completes the outline of the proof.
\end{proof}

The rest of the appendix consists of the technical details of the above proof.

\subsection{The maps \texorpdfstring{$\Phi_n$}{Phi n}}

We now define the maps $\Phi_n$ which were the main tool for the difficult part of Theorem \ref{thm:OcneanuCompactness}.

\begin{defn}
\label{defn:DefnOfPhi}
Let $\{b\}$ be a Pimsner-Popa basis for $N_1$ over $N_0$.
Since the commuting square \eqref{eq:AppendixCommutingSquare} is horizontally Markov, $\{b\}$ is also a basis for $M_1$ over $M_0$.
We define $\Phi_n$ on $N_0'\cap M_n$ by 
$\Phi_n(x)= d^{2n} \sum_{b} b e_1 e_2\cdots e_{n+1} x e_n \cdots e_2 e_1 b^*$ (compare with the formula in \cite[Thm.~2.13]{MR1424954}). 
Note that $\Phi_n$ on $N_0'\cap M_n$ is independent of the choice of basis as in \cite[Rem.~2.30]{MR2812459}.
Whenever $z\in N_0'\cap M_n$,
$$
\sum_{b} b\otimes b^* \mapsto \sum_b bzb^*
$$
is well-defined, and the left hand side is independent of the choice of $\{b\}$.
Since for every $u\in U(N_0)$ $\{ub\}$ is another Pimsner-Popa basis, we see that $\Phi_n(x)\in N_0'\cap M_{n+2}$.
\end{defn}

\begin{prop}
\label{prop:PropertiesOfPhi}
The maps $\Phi_n$ on $N_0'\cap M_n$ enjoy the following properties:
\begin{enumerate}[label=(\arabic*)]
\item
For all $x\in N_0'\cap N_n$, $\Phi_n(x) \in N_{0}'\cap N_{n+2}$.
\item
$\Phi_n|_{M_k} = \Phi_k$ for all $0\leq k\leq n$.
\item
$\Phi_n$ is a $*$-algebra map.
\item
For all $0\leq k\leq n$, if $x\in N_k'\cap M_n$, then $\Phi_n(x) \in N_{k+2}'\cap M_{n+2}$.
\item
For all $x\in N_0'\cap M_n$, $E_{N_{n+2}}(\Phi_n(x))=\Phi_n(E_{N_n}(x))$.
\end{enumerate}
\end{prop}
\begin{proof}
\mbox{}
\begin{enumerate}[label=(\arabic*)]
\item
Since $\{b\}\subset N_1$, if $x\in N_0'\cap N_n$, then $\Phi_n(x)\in N_0'\cap N_{n+2}$.
\item
If $x\in M_k$, then $[x, e_j]=0$ for all $k+1\leq j\leq n+1$.
Thus 
\begin{align*}
\Phi_n(x) 
&= 
d^{2n} \sum_{b} b e_1 e_2\cdots e_{n+1} x e_n \cdots e_2 e_1 b^*
\\&= 
d^{2n} \sum_{b} b e_1 e_2\cdots e_ke_{k+1}\cdots e_ne_{n+1} e_n\cdots e_{k+1}x e_k \cdots e_2 e_1 b^*
\\&= 
d^{2k} \sum_{b} b e_1 e_2\cdots e_{k+1} x e_k \cdots e_2 e_1 b^*
\\&= 
\Phi_k(x).
\end{align*}
\item
For all $x,y\in N_0'\cap M_n$, we have
\begin{align*}
\Phi_n(x) \Phi_n(y)
&= 
d^{4n} \sum_{a,b} 
a e_1 e_2\cdots e_{n+1} x e_n \cdots e_2 e_1 a^* 
b e_1 e_2\cdots e_{n+1} y e_n \cdots e_2 e_1 b^*
\\&=
d^{4n} \sum_{a,b} 
a e_1 e_2\cdots e_{n+1} x e_n \cdots e_2 e_1 
E_{N_0}(a^*b)
e_1 e_2\cdots e_{n+1} y e_n \cdots e_2 e_1 b^*
\\&=
d^{4n} \sum_{a,b} 
aE_{N_0}(a^*b) e_1 e_2\cdots e_{n+1} x e_n \cdots e_2 e_1 e_2\cdots e_{n+1} y e_n \cdots e_2 e_1 b^*
\\&=
d^{2n} \sum_{b} 
b e_1 e_2\cdots e_{n+1} x y e_n \cdots e_2 e_1 b^*
\\&=
\Phi_n(xy).
\end{align*}
\item
Suppose $x\in N_k'\cap M_n$ for $0\leq k\leq n$, and suppose $y\in N_{k+2}$.
We calculate
\begin{align*}
\Psi_n(x)y
&=
\Psi_n(x)y 1_{N_{k+2}}
\\&=
\left(d^{2n} \sum_{a} a e_1 e_2\cdots e_{n+1}xe_n \cdots e_2 e_1 a^*\right)
y
\left(d^{2k} \sum_{b} b e_1 e_2\cdots e_ke_{k+1}e_k \cdots e_2 e_1 b^*\right)
\\&=
d^{2(n+k)} \sum_{a,b} 
(a e_1 e_2\cdots e_{n}x e_{n+1}e_n \cdots e_{k+2})
\underbrace{(e_{k+1} \cdots e_2 e_1 a^* y b e_1e_2\cdots e_{k+1})}_{z_{a,b}e_{k+1}}
(e_k \cdots e_2e_1 b^*)
\end{align*}
Since $e_{k+1}N_{k+2}e_{k+1} = N_k e_{k+1}$, for all $a\in \{a\}$ and $b\in \{b\}$, there is a $z_{a,b}\in N_k$ such that 
$(e_{k+1} \cdots e_2 e_1 a^* y b e_1e_2\cdots e_{k+1}) = z_{a,b}e_{k+1}$, as indicated in the underbrace above.
Continuing the above calculation, we obtain
\begin{align}
\Psi_n(x)y
&=
d^{2(n+k)} \sum_{a,b} 
(a e_1 e_2\cdots e_{n}x e_{n+1}e_n \cdots e_{k+2})
(z_{a,b}e_{k+1})
(e_k \cdots e_2e_1 b^*)
\notag
\\&=
d^{2(n+k)} \sum_{a,b} 
a e_1 e_2\cdots e_{n}xz_{a,b} e_{n+1}e_n \cdots e_2e_1 b^*
\label{eq:FirstNkPrime}
\end{align}
Starting with $y\Phi_n(x)$, a similar calculation shows
\begin{equation}
\label{eq:SecondNkPrime}
y\Psi_n(x)
=
1_{N_{k+2}}y\Psi_n(x)
=
d^{2(n+k)} \sum_{a,b} 
a e_1 e_2\cdots e_{n}z_{a,b} xe_{n+1}e_n \cdots e_2e_1 b^*.
\end{equation}
Since each $z_{a,b}\in N_k$ and $x\in N_k'\cap M_n$, \eqref{eq:FirstNkPrime} is equal to \eqref{eq:SecondNkPrime}, and we are finished.
\item
Suppose $x\in N_0'\cap M_n$.
Since
$$
\xymatrix@C=5pt@R=2pt{
M_n  &\subset & M_{n+2}
\\
\cup&&\cup
\\
N_n   &\subset & N_{n+2}
}
$$
is a commuting square, $E_{N_{n+2}}(x) = E_{N_n}(x)$.
Since $E_{N_{n+2}}$ is $N_{n+2}-N_{n+2}$ bilinear, we have
\begin{align*}
E_{N_{n+2}}(\Phi_n(x))
&=
E_{N_{n+2}}\left(d^{2n} \sum_{b} b e_1 e_2\cdots e_{n+1} x e_n \cdots e_2 e_1 b^*\right)
\\&=
d^{2n} \sum_{b} b e_1 e_2\cdots e_{n+1}E_{N_{n+2}}( x) e_n \cdots e_2 e_1 b^*
\\&=
d^{2n} \sum_{b} b e_1 e_2\cdots e_{n+1}E_{N_{n}}( x) e_n \cdots e_2 e_1 b^*
\\&=
\Phi_n(E_{N_n}(x)).
\qedhere
\end{align*}
\end{enumerate}
\end{proof}

\subsection{Behavior of the traces \texorpdfstring{$\tau_n$}{tau n}}

For $n\in\bbN$, we define $\Psi_n = \Phi_{n-1} \circ \cdots \circ \Phi_1\circ \Phi_0$.
We now observe the behavior of the sequence of traces $\tau_n:=\tr_{2n}\circ \Psi_n$ on $Z(N_0)=N_0'\cap N_0$, with $\tau_0 = \tr_0$ by convention.
The following lemma is a straightforward calculation.

\begin{lem}
\label{lem:InductiveDensity}
For all $x\in Z(N_0)$ and $n\in\bbN$, 
$$
\tau_n(x) = d^{-2n}\sum_{b_1,\dots, b_n\in B }\tr_0(x \cdot E_{N_0}(b_1^*E_{N_0}(b_{2}^*\cdots E_{N_0}(b_n^*b_n)\cdots b_{2})b_1))
$$
where $B$ is any Pimsner-Popa basis for $N_1$ over $N_0$.
\end{lem}

There is a unique $k\in \bbN$ such that $Z(N_0)\cong \bbC^k$.
For $x\in Z(N_0)$, we denote by  $\vec{x}$ the vector in $\bbC^k$ corresponding to $x$.
We define:
\begin{itemize}
\item
$\Lambda$ is the bipartite adjacency matrix of the Bratteli diagram for the inclusion $N_0\subset N_1$, i.e., $\Lambda_{i,j}$ is the number of times the $i$-th simple summand of $N_0$ is contained in the $j$-th simple summand of $N_1$.
\item
$\lambda_i$ is the Markov trace column vector for $N_i$, whose $j$-th entry $\lambda_i(j)$ is the trace of a minimal projection in the $j$-th simple summand of $N_i$.
This means $\Lambda \Lambda^T\lambda_0 = d^{2}\lambda_0$ and $\Lambda^T \Lambda \lambda_1 = d^{2} \lambda_1$.
Observe that since $\tr_1$ is a faithful Markov trace for the inclusion $N_0\subset N_1$, the matrices $\Lambda\Lambda^T$ and $\Lambda^T\Lambda$ are both direct sums of primitive symmetric non-negative integer matrices, all of which have the same Frobenius-Perron eigenvalue $d^2$.
\item
$m_i$ is the dimension (row) vector for $N_i$, i.e., the $j$-th simple summand of $N_i$ is a full matrix algebra of size $m_i(j)$.
Notice that $m_i \lambda_i = 1$ for $i=0,1$.
\item
$\Delta = \diag(m_0(i))_{i=1}^k$ is the diagonal $k\times k$ matrix whose $(i,i)$-th entry is $m_0(i)$.
\end{itemize}

\begin{example}
\label{ex:A4Example}
For the $A_4$ inclusion $N_0 = \bbC \oplus M_2(\bbC) \subseteq M_3(\bbC) \oplus M_2(\bbC)=N_1$ we have:
$\Lambda = \begin{bmatrix} 1 & 0 \\ 1 & 1 \end{bmatrix}$,
$\lambda_0 = \displaystyle\frac{1}{1+2\phi}\begin{bmatrix} 1 \\ \phi \end{bmatrix}$,
$\lambda_1 =\displaystyle \frac{1}{2+3\phi}\begin{bmatrix} \phi \\ 1\end{bmatrix}$,
$m_0 = \begin{bmatrix} 1 & 2\end{bmatrix}$,
$m_1 = \begin{bmatrix} 3 & 2\end{bmatrix}$, and
$\Delta = \begin{bmatrix} 1 & 0 \\ 0 & 2\end{bmatrix}$.
\end{example}

\begin{prop}
\label{prop:MarkovIteration}
There is a Pimsner-Popa basis $B$ for $N_1$ over $N_0$ such that for every $x\in Z(N_0)$, $\sum_{b\in B} E_{N_0}(b^*xb)\in Z(N_0)$.
Moreover, under the isomorphism $Z(N_0)\cong \bbC^k$, we have 
$$
\overrightarrow{\sum_{b\in B}E_{N_0}(b^*xb)}
= \Delta^{-1}\Lambda \Lambda^T \Delta\vec{x}.
$$
\end{prop}
\begin{proof}
We use the loop basis for $N_0\subseteq N_1$ afforded by \cite[\S3.1-3.2]{MR2812459}.
We label the edges of $\Lambda$ by $\varepsilon$, with source $s(\varepsilon)$ an even vertex corresponding to a simple summand of $N_0$, and target $t(\varepsilon)$ an odd vertex corresponding to a simple summand of $N_1$.
We introduce a new vertex $\star$ with edges $\eta$, with each source $s(\eta) = \star$, and target $t(\eta)$ an even vertex corresponding to a simple summand of $N_0$.
The number of edges $\eta$ from $\star$ to the $i$-th even vertex is equal to $m_0(i)$.
We denote by $\varepsilon^*$ and $\eta^*$ the edge with the reverse orientation.

We give an explicit basis for $N_0$ by loops of length 2 starting at $\star$, where adjoint is given by the conjugate linear extension of $[\eta_i\eta_j^*]^* = [\eta_j \eta_i^*]$, and multiplication is given by $[\eta_i\eta_j^*] \cdot [\eta_k \eta_\ell^*] = \delta_{j=k}[\eta_i \eta_\ell^*]$.
We give an explicit basis for $N_1$ by loops of length 4 starting at $\star$, where adjoint is given by the conjugate linear extension of $[\eta_i\varepsilon_j \varepsilon_k^*\eta_\ell^*]^* = [\eta_\ell \varepsilon_k \varepsilon_j^*\eta_i^*]$, and multiplication is given by $[\eta_i\varepsilon_j \varepsilon_k^*\eta_\ell^*] \cdot [\eta_m\varepsilon_n \varepsilon_p^*\eta_q^*] = \delta_{\ell=m}\delta_{k=n}[\eta_i \varepsilon_j \varepsilon_p^*\eta_q^*]$.
The trace $\tr_1$ on $N_1$ is given by 
$
\tr_1([\eta_i\varepsilon_j\varepsilon_k^*\eta_\ell^*])
= 
\delta_{\eta_i= \eta_\ell}  \delta_{\varepsilon_j= \varepsilon_k} \lambda_1(t(\varepsilon_j))
$, 
and the trace $\tr_0$ on $N_0$ is given by 
$
\tr_0([\eta_i\eta_j^*]) 
= 
\delta_{\eta_i=\eta_j} \lambda_0(t(\eta_i))
$. 
The unital inclusion $N_0\subseteq N_1$ is given by 
$
[\eta_i \eta_j^*] \mapsto \sum_{s(\varepsilon) 
= 
t(\eta_i)} [\eta_i\varepsilon\varepsilon^*\eta_j^*]
$, 
and the unique trace-preserving conditional expectation is given by 
$
E_{N_0}([\eta_i\varepsilon_j \varepsilon_k^*\eta_\ell^*]) 
= 
\delta_{\varepsilon_j = \varepsilon_k} \left(\frac{\lambda_1(t(\varepsilon_j))}{\lambda_0(s(\varepsilon_j))}\right) [\eta_i \eta_\ell^*]
$.

For example, the inclusion $N_0 = \bbC \oplus M_2(\bbC) \subseteq M_3(\bbC) \oplus M_2(\bbC)=N_1$ from Example \ref{ex:A4Example} could be represented in the loop basis as follows:
$$
\begin{tikzpicture}
	\node (s) at (0,0) {$\star$};
	\node (a) at (-1,1) {$\bbC$};
	\node (b) at (1,1) {$M_2(\bbC)$};
	\node (c) at (0,2) {$M_3(\bbC)$};
	\node (d) at (2,2) {$M_2(\bbC)$};
	\draw (s) to node [left] {\scriptsize{$\eta_1$}} (a);
	\draw[double] (s) to node [left] {\scriptsize{$\eta_2$}} node [right] {\scriptsize{$\eta_3$}} (b);
	\draw (a) to node [left] {\scriptsize{$\varepsilon_1$}} (c);
	\draw (b) to node [left] {\scriptsize{$\varepsilon_2$}} (c);
	\draw (b) to node [left] {\scriptsize{$\varepsilon_3$}} (d);
\end{tikzpicture}
$$

Now by \cite[Prop.~3.22]{MR2812459} and  \cite[Rem.~3.23]{MR2812459}, a Pimsner-Popa basis for $N_1$ over $N_0$ is given by $B = B_1\amalg B_2$ where  
\begin{align*}
B_{1}&=\set{ \left(\frac{\lambda_0(s(\varepsilon_2))}{d\lambda_1(t(\varepsilon_2))}\right)^{1/2}\sum\limits_{\eta:\,\, t(\eta)=s(\varepsilon_{1})}[\eta\varepsilon_1\varepsilon_2^*\eta^*]}{s(\varepsilon_1) = s(\varepsilon_2) \text{ and } t(\varepsilon_1) = t(\varepsilon_2)}
\\
B_{2}&=\set{\left(\frac{\lambda_0(s(\varepsilon_2))}{m_0(s(\varepsilon_2))d\lambda_1(t(\varepsilon_2))}\right)^{1/2}[\eta_1\varepsilon_1\varepsilon_2^*\eta_2^*]}{s(\varepsilon_1)\neq s(\varepsilon_2)}.
\end{align*}
Here, the sum in $B_1$ is over $\eta$ such that $[\eta\varepsilon_1\varepsilon_2^*\eta^*]$ forms a loop.

Now the minimal central projection in $N_0$ corresponding to the $i$-th simple summand is equal to $p_i = \sum_{\eta:\,\,t(\eta) = i}[\eta\eta^*]$.
One calculates that 
$\sum_{b\in B_1} E_{N_0}(b^*p_ib)= \sum_{j} \Lambda_{i,j}^2 p_i$, 
while
$\sum_{b\in B_2} E_{N_0}(b^*p_ib) = \sum_{i'\neq i}\sum_{j} \frac{m_0(i)}{m_0(i')}\Lambda_{i,j}\Lambda_{i',j}p_{i'}$.
Hence we have that $\sum_{b\in B} E_{N_0}(b^* p_i b)$ is in $Z(N_0)$ with corresponding vector in $\bbC^k$ equal to $\Delta^{-1}\Lambda \Lambda^T \Delta \vec{e}_i$.
Now since every element of $Z(N_0)$ is a linear combination of the $p_i$, the result follows.
\end{proof}

Equipped with this explicit Pimsner-Popa basis, we are prepared to analyze the traces $\tau_n$.
Note that an arbitrary tracial state $\tau$ on $Z(N_0)$ is always of the form $\tau(y) = \tr_0(y\cdot h)$ for some positive operator $h\in Z(N_0)$ with $\tr_0(h)=1$ called the \emph{density} of $\tau$.
Let $\cT$ be the topological space of traces on $Z(N_0)$, and note that we may identify the \emph{pointed} topological space $(\cT, \tr_0)$ with $(\set{h\in Z(N_0)}{ h\geq 0 \text{ and }\tr(h)=1}, 1_{Z(N_0)})$.

We see from Lemma \ref{lem:InductiveDensity}, using the Pimsner-Popa basis $B$ from Proposition \ref{prop:MarkovIteration}, that the densities $h_n\in Z(N_0)$ of the $\tau_n$ are given inductively by
$h_{n} = d^{-2}\sum_{b\in B}E_{N_0}(b^* h_{n-1}b)$ for all $n\in \bbN$.
Letting $\vec{h}_n \in \bbC^k$ be the vector corresponding to $h_n \in Z(N_0)$, Proposition \ref{prop:MarkovIteration} tells us that $\vec{h}_n = d^{-2} \Delta^{-1}\Lambda \Lambda^T \Delta\vec{h}_{n-1}$ for all $n\in \bbN$.
Since the density $h_0$ of $\tau_0 = \tr_0$ is $1_{Z(N_0)}$, for all $n\in \bbN$,
\begin{equation}
\label{eq:DensityFormula}
\vec{h}_n = d^{-2n}\Delta^{-1}(\Lambda \Lambda^T)^n \Delta \vec{1},
\end{equation}
where $\vec{1}\in \bbC^k$ is the vector whose entries are all $1$.

\begin{prop}
\label{prop:TraceInfinity}
The traces $\tau_n$ are faithful and converge to a faithful trace $\tau_\infty$ on $Z(N_0)$.
\end{prop}
\begin{proof}
The density vector $\vec{h}_n = d^{-2n} \Delta^{-1}(\Lambda \Lambda^T)^n\Delta \vec{1}$ from \eqref{eq:DensityFormula} has strictly positive entries, and thus $\tau_n$ is faithful for all $n$.
Second, the limit of $d^{-2n} \Delta^{-1}(\Lambda \Lambda^T)^n\Delta\vec{1}$ as $n\to \infty$ is well known to be $\Delta^{-1}\vec{\lambda}$, where $\vec{\lambda}$ is a suitably normalized Frobenius-Perron eigenvector for $\Lambda \Lambda^T$.
Since $\vec{\lambda}_1$ had all entries strictly positive, $\vec{\lambda}$ has all entries strictly positive.
(This follows by looking at the direct sum decomposition of $\Lambda\Lambda^T$ into its primitive symmetric blocks, all which have the same Frobenius-Perron eigenvalue by the existence of a Markov trace $\tr_!$ on $N_0\subset N_1$.)
Hence the densities $\vec{h}_n$ converge to $\vec{h}_\infty=\Delta^{-1}\vec{\lambda}$, which gives a faithful trace $\tau_\infty$ on $Z(N_0)$.
(Note that $\tau_\infty$ is not $\tr_0$ even if $\Delta=I_k$, since its density with respect to $\tr_0$ is $\vec{\lambda}$, which is in general not $\vec{1}$.)
\end{proof}

\bibliographystyle{amsalpha}
{\footnotesize{
\bibliography{bibliography}

\providecommand{\bysame}{\leavevmode\hbox to3em{\hrulefill}\thinspace}
\providecommand{\MR}{\relax\ifhmode\unskip\space\fi MR }
\providecommand{\MRhref}[2]{%
  \href{http://www.ams.org/mathscinet-getitem?mr=#1}{#2}
}
\providecommand{\href}[2]{#2}
\begin{thebibliography}{EGNO15}

\bibitem[AH99]{MR1686551}
Marta Asaeda and Uffe Haagerup, \emph{Exotic subfactors of finite depth with
  {J}ones indices {$(5+\sqrt{13})/2$} and {$(5+\sqrt{17})/2$}}, Comm. Math.
  Phys. \textbf{202} (1999), no.~1, 1--63, \mathscinet{MR1686551},
  \doi{10.1007/s002200050574}, \arXiv{math.OA/9803044}.

\bibitem[AP17]{ClaireSorinII_1}
Claire Anantharaman and Sorin Popa, \emph{An introduction to {${\rm II}_1$}
  factors}, 2017, preprint available at
  \url{http://www.math.ucla.edu/~popa/books.html}.

\bibitem[BDH88]{MR945550}
Michel Baillet, Yves Denizeau, and Jean-Fran\c{c}ois Havet, \emph{Indice d'une
  esp\'{e}rance conditionelle. ({F}rench) [{I}ndex of a conditional
  expectation]}, Compositio Math. \textbf{66} (1988), no.~2, 199--236,
  \mathscinet{MR945550}.

\bibitem[BDH14]{MR3342166}
Arthur Bartels, Christopher~L. Douglas, and Andr{\'e} Henriques,
  \emph{Dualizability and index of subfactors}, Quantum Topol. \textbf{5}
  (2014), no.~3, 289--345, \mathscinet{MR3342166} \doi{10.4171/QT/53}
  \arXiv{1110.5671}. \MR{3342166}

\bibitem[BHP12]{MR3405915}
Arnaud Brothier, Michael Hartglass, and David Penneys, \emph{Rigid {$C\sp
  *$}-tensor categories of bimodules over interpolated free group factors}, J.
  Math. Phys. \textbf{53} (2012), no.~12, 123525, 43, \mathscinet{MR3405915}
  \doi{10.1063/1.4769178} \arxiv{1208.5505}. \MR{3405915}

\bibitem[Bis97]{MR1424954}
Dietmar Bisch, \emph{Bimodules, higher relative commutants and the fusion
  algebra associated to a subfactor}, Operator algebras and their applications
  (Waterloo, ON, 1994/1995), 13-63, Fields Inst. Commun., 13, Amer. Math. Soc.,
  Providence, RI, 1997, \mathscinet{MR1424954}, \googlebooks{_InIRTO8Y7gC}.

\bibitem[BKLR15]{MR3308880}
Marcel Bischoff, Yasuyuki Kawahigashi, Roberto Longo, and Karl-Henning Rehren,
  \emph{Tensor categories and endomorphisms of von {N}eumann algebras---with
  applications to quantum field theory}, SpringerBriefs in Mathematical
  Physics, vol.~3, Springer, Cham, 2015, \mathscinet{MR3308880}
  \doi{10.1007/978-3-319-14301-9}. \MR{3308880}

\bibitem[BS10]{MR2664619}
John~C. Baez and Michael Shulman, \emph{Lectures on {$n$}-categories and
  cohomology}, Towards higher categories, IMA Vol. Math. Appl., vol. 152,
  Springer, New York, 2010, \mathscinet{MR2664619} \arxiv{math/0608420},
  pp.~1--68. \MR{2664619}

\bibitem[CHPS18]{1810.07049}
Desmond Coles, Peter Huston, David Penneys, and Srivatsa Srinivas, \emph{The
  module embedding theorem via towers of algebras}, 2018, \arxiv{1810.07049}.

\bibitem[Con94]{MR1303779}
Alain Connes, \emph{Noncommutative geometry}, Academic Press Inc., San Diego,
  CA, 1994, \mathscinet{MR1303779}.

\bibitem[DGG14]{MR3178106}
Paramita Das, Shamindra~Kumar Ghosh, and Ved~Prakash Gupta, \emph{Perturbations
  of planar algebras}, Math. Scand. \textbf{114} (2014), no.~1, 38--85,
  \mathscinet{MR3178106} \arxiv{1009.0186}. \MR{3178106}

\bibitem[EGNO15]{MR3242743}
Pavel Etingof, Shlomo Gelaki, Dmitri Nikshych, and Victor Ostrik, \emph{Tensor
  categories}, Mathematical Surveys and Monographs, vol. 205, American
  Mathematical Society, Providence, RI, 2015, \mathscinet{MR3242743}
  \doi{10.1090/surv/205}. \MR{3242743}

\bibitem[EK98]{MR1642584}
David~E. Evans and Yasuyuki Kawahigashi, \emph{Quantum symmetries on operator
  algebras}, Oxford Mathematical Monographs. Oxford Science Publications. The
  Clarendon Press, Oxford University Press, New York, 1998, xvi+829 pp. ISBN:
  0-19-851175-2, \mathscinet{MR1642584}.

\bibitem[GdlHJ89]{MR999799}
Frederick~M. Goodman, Pierre de~la Harpe, and Vaughan~F.R. Jones, \emph{Coxeter
  graphs and towers of algebras}, Mathematical Sciences Research Institute
  Publications, 14. Springer-Verlag, New York, 1989, x+288 pp. ISBN:
  0-387-96979-9, \mathscinet{MR999799}.

\bibitem[Gio20]{1908.09121}
Luca Giorgetti, \emph{Minimal index and dimension for inclusions of von
  {N}eumann algebras with finite-dimensional centers}, OT27 Proceedings
  (Timi\c{s}oara, 2018), 183--191, Theta, 2020, \arxiv{1908.09121}.

\bibitem[GL19]{MR3994584}
Luca Giorgetti and Roberto Longo, \emph{Minimal index and dimension for
  2-{$C^*$}-categories with finite-dimensional centers}, Comm. Math. Phys.
  \textbf{370} (2019), no.~2, 719--757, \mathscinet{MR3994584}
  \doi{10.1007/s00220-018-3266-x} \arxiv{1805.09234}. \MR{3994584}

\bibitem[GLR85]{MR808930}
P.~Ghez, R.~Lima, and J.~E. Roberts, \emph{{$W\sp \ast$}-categories}, Pacific
  J. Math. \textbf{120} (1985), no.~1, 79--109, \mathscinet{MR808930}.
  \MR{808930 (87g:46091)}

\bibitem[GY20]{RealTwoCat}
Luca Giorgetti and Wei Yuan, \emph{Realization of rigid {$\rm
  C^*$}-bicategories as bimodules over type {$\rm II_1$} von {N}eumann
  algebras}, 2020, arXiv preprint.

\bibitem[Haa75]{MR0407615}
Uffe Haagerup, \emph{The standard form of von {N}eumann algebras}, Math. Scand.
  \textbf{37} (1975), no.~2, 271--283, \mathscinet{MR0407615}. \MR{0407615 (53
  \#11387)}

\bibitem[Haa79]{MR549119}
\bysame, \emph{Operator-valued weights in von {N}eumann algebras. {II}}, J.
  Functional Analysis \textbf{33} (1979), no.~3, 339--361,
  \mathscinet{MR549119} \doi{10.1016/0022-1236(79)90072-7}. \MR{549119}

\bibitem[Hav90]{MR1086543}
Jean-Fran\c{c}ois Havet, \emph{Esp\'{e}rance conditionnelle minimale}, J.
  Operator Theory \textbf{24} (1990), no.~1, 33--55, \mathscinet{MR1086543}.
  \MR{1086543}

\bibitem[Hia88]{MR976765}
Fumio Hiai, \emph{Minimizing indices of conditional expectations onto a
  subfactor}, Publ. Res. Inst. Math. Sci. \textbf{24} (1988), no.~4, 673--678,
  \mathscinet{MR976765} \doi{10.2977/prims/1195174872}. \MR{976765}

\bibitem[HP17]{MR3663592}
Andr\'e Henriques and David Penneys, \emph{Bicommutant categories from fusion
  categories}, Selecta Math. (N.S.) \textbf{23} (2017), no.~3, 1669--1708,
  \mathscinet{MR3663592} \doi{10.1007/s00029-016-0251-0} \arxiv{1511.05226}.
  \MR{3663592}

\bibitem[HP20]{2004.08271}
Andr\'e Henriques and David Penneys, \emph{Representations of fusion categories
  and their commutants}, 2020, \arxiv{2004.08271}.

\bibitem[HPT16]{1607.06041}
Andr\'e Henriques, David Penneys, and James~E. Tener, \emph{Planar algebras in
  braided tensor categories}, 2016, \arxiv{1607.06041}, to appear {Mem. Amer.
  Math. Soc.}

\bibitem[HV19]{MR3971584}
Chris Heunen and Jamie Vicary, \emph{Categories for quantum theory}, Oxford
  Graduate Texts in Mathematics, vol.~28, Oxford University Press, Oxford,
  2019, An introduction, \mathscinet{MR3971584}
  \doi{10.1093/oso/9780198739623.001.0001}. \MR{3971584}

\bibitem[Izu03]{MR1953517}
Masaki Izumi, \emph{Canonical extension of endomorphisms of type {III}
  factors}, Amer. J. Math. \textbf{125} (2003), no.~1, 1--56,
  \mathscinet{MR1953517} \arxiv{math/0104228}. \MR{1953517}

\bibitem[Izu17]{MR3635673}
\bysame, \emph{A {C}untz algebra approach to the classification of near-group
  categories}, Proceedings of the 2014 {M}aui and 2015 {Q}inhuangdao
  conferences in honour of {V}aughan {F}. {R}. {J}ones' 60th birthday, Proc.
  Centre Math. Appl. Austral. Nat. Univ., vol.~46, Austral. Nat. Univ.,
  Canberra, 2017, \mathscinet{MR3635673} \arxiv{1512.04288}, pp.~222--343.
  \MR{3635673}

\bibitem[Jol90]{MR1073519}
Paul Jolissaint, \emph{Index for pairs of finite von {N}eumann algebras},
  Pacific J. Math. \textbf{146} (1990), no.~1, 43--70, \mathscinet{MR1073519},
  \euclid{euclid.pjm/1102645309}.

\bibitem[Jon83]{MR0696688}
Vaughan F.~R. Jones, \emph{Index for subfactors}, Invent. Math. \textbf{72}
  (1983), no.~1, 1--25, \mathscinet{MR696688}, \doi{10.1007/BF01389127}.

\bibitem[Jon99]{math.QA/9909027}
\bysame, \emph{Planar algebras {I}}, 1999, \arXiv{math.QA/9909027}.

\bibitem[Jon00]{MR1865703}
\bysame, \emph{The planar algebra of a bipartite graph}, Knots in {H}ellas '98
  ({D}elphi), Ser. Knots Everything, vol.~24, World Sci. Publ., River Edge, NJ,
  2000, \mathscinet{MR1865703}, pp.~94--117. \MR{1865703}

\bibitem[Jon15]{JonesVNA}
\bysame, \emph{Von {N}eumann algebras}, 2015,
  \url{https://math.vanderbilt.edu/jonesvf/VONNEUMANNALGEBRAS2015/VonNeumann2015.pdf}.

\bibitem[JP11]{MR2812459}
Vaughan F.~R. Jones and David Penneys, \emph{The embedding theorem for finite
  depth subfactor planar algebras}, Quantum Topol. \textbf{2} (2011), no.~3,
  301--337, \arXiv{1007.3173}, \mathscinet{MR2812459}, \doi{10.4171/QT/23}.

\bibitem[JP17]{MR3687214}
Corey Jones and David Penneys, \emph{Operator algebras in rigid {$\rm
  C^*$}-tensor categories}, Comm. Math. Phys. \textbf{355} (2017), no.~3,
  1121--1188, \mathscinet{MR3687214} \doi{10.1007/s00220-017-2964-0}
  \arxiv{1611.04620}. \MR{3687214}

\bibitem[JS91]{MR1107651}
Andr{\'e} Joyal and Ross Street, \emph{Tortile {Y}ang-{B}axter operators in
  tensor categories}, J. Pure Appl. Algebra \textbf{71} (1991), no.~1, 43--51,
  \mathscinet{MR1107651} \doi{10.1016/0022-4049(91)90039-5}. \MR{1107651
  (92e:18006)}

\bibitem[JS97]{MR1473221}
Vaughan~F.R. Jones and V.S. Sunder, \emph{Introduction to subfactors}, London
  Mathematical Society Lecture Note Series, 234. Cambridge University Press,
  Cambridge, 1997, xii+162 pp. ISBN: 0-521-58420-5, \mathscinet{MR1473221}.

\bibitem[Kos86]{MR829381}
Hideki Kosaki, \emph{Extension of {J}ones' theory on index to arbitrary
  factors}, J. Funct. Anal. \textbf{66} (1986), no.~1, 123--140,
  \mathscinet{MR829381} \doi{10.1016/0022-1236(86)90085-6}. \MR{829381
  (87g:46093)}

\bibitem[Lon89]{MR1027496}
Roberto Longo, \emph{Index of subfactors and statistics of quantum fields.
  {I}}, Comm. Math. Phys. \textbf{126} (1989), no.~2, 217--247,
  \mathscinet{MR1027496}.

\bibitem[MPS10]{MR2559686}
Scott Morrison, Emily Peters, and Noah Snyder, \emph{Skein theory for the
  {$D_{2n}$} planar algebras}, J. Pure Appl. Algebra \textbf{214} (2010),
  no.~2, 117--139, \arXiv{0808.0764} \mathscinet{MR2559686}
  \doi{10.1016/j.jpaa.2009.04.010}. \MR{MR2559686}

\bibitem[MvN43]{MR0009096}
F.~J. Murray and J.~von Neumann, \emph{On rings of operators. {IV}}, Ann. of
  Math. (2) \textbf{44} (1943), 716--808, \mathscinet{MR0009096}. \MR{0009096
  (5,101a)}

\bibitem[Ocn88]{MR996454}
Adrian Ocneanu, \emph{Quantized groups, string algebras and {G}alois theory for
  algebras}, Operator algebras and applications, Vol.\ 2, London Math. Soc.
  Lecture Note Ser., vol. 136, Cambridge Univ. Press, Cambridge, 1988,
  \mathscinet{MR996454}, pp.~119--172.

\bibitem[Pen13]{MR3040370}
David Penneys, \emph{A {P}lanar {C}alculus for {I}nfinite {I}ndex
  {S}ubfactors}, Comm. Math. Phys. \textbf{319} (2013), no.~3, 595--648,
  \mathscinet{MR3040370} \arXiv{1110.3504} \doi{10.1007/s00220-012-1627-4}.
  \MR{3040370}

\bibitem[Pen20]{MR4133163}
\bysame, \emph{Unitary dual functors for unitary multitensor categories}, High.
  Struct. \textbf{4} (2020), no.~2, 22--56, \mathscinet{MR4133163}
  \arxiv{1808.00323}. \MR{4133163}

\bibitem[Pop83]{MR720738}
Sorin Popa, \emph{Maximal injective subalgebras in factors associated with free
  groups}, Adv. in Math. \textbf{50} (1983), no.~1, 27--48,
  \mathscinet{MR720738} \doi{10.1016/0001-8708(83)90033-6}. \MR{720738}

\bibitem[Pop90]{MR1055708}
Sorin Popa, \emph{Classification of subfactors: the reduction to commuting
  squares}, Invent. Math. \textbf{101} (1990), no.~1, 19--43,
  \mathscinet{MR1055708}, \doi{10.1007/BF01231494}.

\bibitem[Pop94]{MR1278111}
\bysame, \emph{Classification of amenable subfactors of type {II}}, Acta Math.
  \textbf{172} (1994), no.~2, 163--255, \mathscinet{MR1278111},
  \doi{10.1007/BF02392646}.

\bibitem[Pop95a]{MR1334479}
\bysame, \emph{An axiomatization of the lattice of higher relative commutants
  of a subfactor}, Invent. Math. \textbf{120} (1995), no.~3, 427--445,
  \mathscinet{MR1334479} \doi{10.1007/BF01241137}.

\bibitem[Pop95b]{MR1339767}
\bysame, \emph{Classification of subfactors and their endomorphisms}, CBMS
  Regional Conference Series in Mathematics, vol.~86, Published for the
  Conference Board of the Mathematical Sciences, Washington, DC, 1995,
  \mathscinet{MR1339767}. \MR{1339767 (96d:46085)}

\bibitem[PP86]{MR860811}
Mihai Pimsner and Sorin Popa, \emph{Entropy and index for subfactors}, Ann.
  Sci. \'{E}cole Norm. Sup. (4) \textbf{19} (1986), no.~1, 57--106,
  \mathscinet{MR860811}.

\bibitem[PP88]{MR965748}
\bysame, \emph{Iterating the basic construction}, Trans. Amer. Math. Soc.
  \textbf{310} (1988), no.~1, 127--133, \mathscinet{MR965748}.

\bibitem[Sau83]{MR703809}
Jean-Luc Sauvageot, \emph{Sur le produit tensoriel relatif d'espaces de
  {H}ilbert}, J. Operator Theory \textbf{9} (1983), no.~2, 237--252,
  \mathscinet{MR703809}.

\bibitem[Sau85]{MR799587}
\bysame, \emph{Produits tensoriels de {$Z$}-modules et applications}, Operator
  algebras and their connections with topology and ergodic theory (1983),
  Lecture Notes in Math., vol. 1132, Springer, Berlin, 1985,
  \mathscinet{MR799587}, pp.~468--485.

\bibitem[Sch90]{1304.5907}
John Schou, \emph{Commuting squares and index for subfactors}, 1990,
  \arxiv{1304.5907}, Ph.D. thesis at Odense Universitet.

\bibitem[Sel11]{MR2767048}
P.~Selinger, \emph{A survey of graphical languages for monoidal categories},
  New structures for physics, Lecture Notes in Phys., vol. 813, Springer,
  Heidelberg, 2011, \mathscinet{MR2767048}
  \href{http://dx.doi.org/10.1007/978-3-642-12821-9_4}{\tt
  DOI:10.1007/978-3-642-12821-9\char`_4}, pp.~289--355. \MR{2767048
  (2012j:18011)}

\bibitem[SY17]{1705.05600}
Yusuke Sawada and Shigeru Yamagami, \emph{Notes on the bicategory of {$\rm
  W^*$}-bimodules}, 2017, \arxiv{1705.05600}.

\bibitem[Wat90]{MR996807}
Yasuo Watatani, \emph{Index for ${C}^*$-subalgebras}, Mem. Amer. Math. Soc.
  \textbf{83} (1990), no.~424, vi+117 pp., \mathscinet{MR996807},
  \googlebooks{Bp2cmONVye0C}.

\bibitem[Yam04]{MR2091457}
Shigeru Yamagami, \emph{Frobenius duality in {$C^*$}-tensor categories}, J.
  Operator Theory \textbf{52} (2004), no.~1, 3--20, \mathscinet{MR2091457}.
  \MR{2091457 (2005f:46109)}

\end{thebibliography}
}}
\end{document}